\documentclass[11pt]{article}

\usepackage[utf8]{inputenc} % set input encoding (not needed with XeLaTeX)

\usepackage{xr} 
\usepackage{comment}% http://ctan.org/pkg/comment
\usepackage{lipsum}% http://ctan.org/pkg/lipsum
\usepackage[usenames]{xcolor}
\usepackage{bm} 
\usepackage{eufrak}
\usepackage{upgreek}
\usepackage{enumitem}
\usepackage{abraces} 
\usepackage[letterpaper,
			left = 1.1truein,  
			right = 1.1truein, 
			top = 1.1truein, 
			bottom = 1.1truein]{geometry} % to change the page dimensions
\usepackage{amsthm, amsfonts, amssymb, amsmath,enumitem, bbm, mathabx}
\usepackage{mathtools}
\mathtoolsset{showonlyrefs,showmanualtags}
\numberwithin{equation}{section}

\newcommand{\nb}[1]{\textcolor{blue}{\texttt{[#1]}}}

\newcommand{\sff}{\mathsf{f}}
\newcommand{\sfg}{\mathsf{g}}
\newcommand{\tx}{\bm{x}^*}
\newtheorem{defn}{Definition}[section]
\newtheorem{thm}{Theorem}[section]
\newtheorem{lem}{Lemma}[section]
\newtheorem{cor}{Corollary}[section]

\theoremstyle{remark}
\newtheorem{rem}{Remark}[section]

\newcommand{\Overlap}{\mathrm{Overlap}}
\newcommand{\MMSE}{\mathsf{MMSE}}
\newcommand{\MSEA}{\mathsf{MSE}^{\mathsf{AMP}}}
\newcommand{\GMMSE}{\mathsf{GMMSE}}
\newcommand{\sMI}{\mathrm{I}}
\newcommand{\smmse}{\mathrm{mmse}}
\newcommand{\cbrac}[1]{\{#1\}}

\usepackage{subcaption}
\usepackage{hyperref}
\newcommand{\revsag}[1]{\textcolor{black}{#1}}
\newcommand{\norm}[1]{\left\|#1\right\|}

\newcommand{\sech}{\mbox{sech}}
\usepackage{graphicx} % support the \includegraphics command and options
\usepackage{color, float}
\usepackage{booktabs} % for much better looking tables
\usepackage{array} % for better arrays (eg matrices) in maths
\usepackage{verbatim} % adds environment for commenting out blocks of text & for better verbatim
\usepackage[numbers]{natbib} 

\title{Community Detection with Contextual Multilayer Networks}
\author{Zongming Ma\footnote{Email: zongming@wharton.upenn.edu.} 
\hspace{0.1in}and\hspace{0.1in}Sagnik Nandy\footnote{Email: sagnik@wharton.upenn.edu.}\\
\emph{University of Pennsylvania}}
\begin{document}
	\maketitle
 % \begin{frontmatter}
%\runtitle{Optimal Bounds}
% \begin{aug}
%  \author{\fnms{Zongming} \snm{Ma}}\and
%   \author{\fnms{Sagnik} \snm{Nandy}}
%   %\author{\fnms{Arun K.}   \snm{Kuchibhotla}\ead[label=e1]{arunku@wharton.upenn.edu}},
%   % \author{\fnms{Lawrence D.} \snm{Brown}},
%   % \and
%   % \author{\fnms{Andreas} \snm{Buja}}
%   % ,
%   % \author{\fnms{Edward I.} \snm{George}},
%   % \and
%   % \author{\fnms{Linda} \snm{Zhao}}
%     % \ead[label=e3]{third@somewhere.com}%
%     % \ead[label=u1,url]{http://www.foo.com}}
%
%   \runauthor{Ma \& Nandy}
%   \runtitle{Community Detection with Multilayer Networks and Covariates}
%
%   \affiliation{Department of Statistics}
%
%   \address{Department of Statistics, University of Pennsylvania} %\printead{e1}}
%\thankstext{t1}{Corresponding author.}

% \end{aug}  
%\maketitle 

\begin{abstract}
In this paper, we study community detection when we observe $m$ sparse networks and a high dimensional covariate matrix, all encoding the same community structure among $n$ subjects. 
In the asymptotic regime where the number of features $p$ and the number of subjects $n$ grow proportionally,
we derive an exact formula of asymptotic minimum mean square error (MMSE) for estimating the common community structure in the balanced two block case using an orchestrated approximate message passing algorithm. The formula implies the necessity of integrating information from multiple data sources. Consequently, it induces a sharp threshold of phase transition between the regime where detection (i.e., weak recovery) is possible and the regime where no procedure performs better than random guess. 
The asymptotic MMSE depends on the covariate signal-to-noise ratio in a more subtle way than the phase transition threshold.
In the special case of $m=1$, our asymptotic MMSE formula complements the pioneering work \cite{ContBlockMod} which found the sharp threshold when $m=1$.
A practical variant of the theoretically justified algorithm with spectral initialization leads to an estimator whose empirical MSEs closely approximate theoretical predictions over simulated examples. 
~
\end{abstract}

\section{Introduction}
\label{sec:intro}

\revsag{Network data is a prevalent form of relational data.}
It appears in many different fields such as social science,
 % \cite{social_1},  
economics,
% \cite{econ_1},
epidemiology,
% \cite{epidem_1},
biological science,
% \cite{bio_1,bio_2},
among others.
Many networks come with inherent community structures. 
Nodes within the same community connect in different ways than nodes between different communities.
% in a manner different from the way they connect with nodes from other communities.
The community labels of the nodes are usually unknown. 
It is of interest to uncover such latent community structures based on observed networks.
 % \cite{Fortunato_2010,1339264,1167344}.
This inference problem is usually called \emph{community detection} which in essence is \emph{clustering} of network nodes.
 % This problem of discovering latent community structure in random graphs is called graph clustering problem or community detection problem and is a very important problem in statistics and machine learning.
% Applications of community detection problem can be found in study of sociological interactions (\cite{Fortunato_2010}), gene expressions (\cite{1339264}), and recommendation systems (\cite{1167344}) among others.
\revsag{ The stochastic block model (SBM)} \cite{holland1983stochastic} is a popular model for studying community detection. 
% \nb{the references here need careful adjustment...}
There has been a large literature on theoretical approaches, 
% (e.g., \cite{KC,7298436,abbe2020entrywise,DBLP:conf/colt/BanksMNN16,bickel2013asymptotic,bickel2009nonparametric,Decelle_2011,AbbeMonYash,hajek2016achieving,hajek2016achieving2,massoulie2013community,mns14,MNS14a,mnsf,zhang2016minimax}),
algorithmic, 
% (e.g., \cite{Amini2014OnSR,amini2013pseudo,cai2015robust, gao2017achieving,MontSen,rohe2011spectral,Vu,Yun2014CommunityDV,YP,zhang2020theoretical})
and application 
% (e.g., \cite{1339264,Fortunato_2010,1167344})
aspects of SBM's.
% The foregoing references are limited, and
We refer interested readers to several recent survey papers \cite{abbe2017community,gao2021minimax,li2021convex} for more detailed accounts of this large and growing literature.

A more traditional clustering problem in statistics is clustering based on covariates, which is also a leading example of unsupervised learning.
% The other direction of community detection or clustering considered in statistics and machine learning is the data clustering problem in presence of co-variates.
Standard techniques include, but are not limited to $k$-means clustering, hierarchical clustering, and the EM algorithm.
The multivariate Gaussian mixture model has been a popular model for theoretical study on this front which has received renewed interest in recent years.
% See, e.g., \cite{balakrishnan2017statistical,daskalakis2017ten,jin2016local,lu2016statistical,wu2019randomly,xu2016global}
% \nb{refs! HZ's Llyod's algorithm, Wainwright, Wu \& Zhou, etc.} 
% for some recent advancements.
The model is closely related to the spiked covariance model \cite{johnstone2001distribution} which is widely adopted in Random Matrix Theory.
% Such problems are generally solved using clustering algorithms like $k$-means or hierarchical clustering. To study the computational thresholds in such problems people have considered co-variate models like equation~\eqref{eq:def_cov} defined in Section~\ref{cov_sec}. Such model was introduced by \cite{johnstone2009sparse}. The weak recovery in this set-up is well studied in \cite{baik2005} and \cite{10.2307/24307692} among others.

Ever-growing techniques for data acquisition have led us to a new paradigm where one could have multiple data sets as multiple sources of information about the community structure.
For instance, for a set of $n$ people, one could potentially have several social networks observed on them (Facebook, LinkedIn, etc.) together with a large collection of socioeconomic (and/or genomic, neuroimaging, etc.) covariates for each individual.
This poses a new challenge: 
\emph{How can we best integrate information from these multiple sources to uncover the common underlying community structure?}

When the network is at least part of the observation, there are two different scenarios.
The first is where one observes multiple networks without any covariate.
A practical example of this scenario was described in \cite{mm} where the nodes represent proteins, the edges in one network represent physical interactions between nodes and those in another network represent co-memberships in protein complexes.
This scenario has been studied in the \emph{multilayer network} literature.
The arguably more interesting scenario is where one observes one or more networks together with a collection of covariates.
In the pioneering work by Deshpande, et al.~\cite{ContBlockMod},
the authors considered the case where the available data is an $n\times n$ adjacency matrix of an SBM and a high dimensional $p\times n$ Gaussian covariate matrix, 
both containing the same balanced two block community structure. 
Under this stylized yet informative model, they rigorously established a sharp information-theoretic threshold for detecting the community structure (i.e., to uncover the community structure better than random guessing) when $p$ and $n$ tend to infinity proportionally and the average degree of the network diverges with $n$.
In addition, they proposed a heuristic algorithm which supports the information-theoretic threshold empirically.
Subsequently, the sharp threshold was extended to the case where the average degree is bounded \cite{lu2020contextual}. 
See also \cite{binkiewicz2017covariate,Abbe2020AnT} for investigations of spectral clustering and \cite{psarkar} for an SDP approach in similar models.
% \nb{talk about \citet{ContBlockMod} and Sen's more recent work.}

% Practitioners have proposed various heuristic algorithms to meet the challenge.
% \nb{separate review on multiple networks, and network plus covariate settings. add related references if possible}

The present paper is motivated by two important questions that remain unanswered by \cite{ContBlockMod}.
\begin{itemize}
\item \emph{Multiple networks with or without covariates}. 
How does the phase transition phenomenon found in \cite{ContBlockMod} exhibit itself when one observes multiple networks with or without high-dimensional covariates? 
How does the threshold depend on individual signal-to-noise ratios in these multiple data sources?

\item \emph{Precise characterization of the information-theoretic limit achieved by Bayes optimal \newline estimator}. 
Even in the special case considered by \cite{ContBlockMod} where only one network is observed together with covariates, 
it is not clear what the information-theoretic limit of the performance by the best estimator is when the signal-to-noise ratio is above the phase transition threshold.
The authors provided a spectral estimator that achieves nontrivial performance above the threshold. 
However, it is not Bayes optimal, and
the exact form of the information-theoretic limit is unknown.
\end{itemize}

In this paper, we provide affirmative answers to both of the above questions.
Without loss of generality, we propose to consider an observation model where one observes $m$ independent adjacency matrices from $m$ SBM's and a high-dimensional Gaussian data-set with $p$ covariates, all carrying the same latent community structure (balanced two block).
% Considering only one dataset with covariates is reasonable, since from an information-theoretic perspective, one could reduce multiple datasets with Gaussian covariates to one with appropriate identification of signal-to-noise ratios.
Focusing on this model and assuming that the average degrees diverge, our contributions are the following.
\begin{itemize}
\item \emph{Sharp phase transition threshold}.
We establish sharp thresholds for phase transition between the regime where detecting the community structure is feasible and the regime where no procedure performs better than random guess.

\item \emph{Exact formula for asymptotic minimum mean square error (MMSE)}. 
We give an exact formula for the asymptotic MMSE achieved by the Bayes optimal estimator of the community structure.

\end{itemize}
To facilitate the derivation of asymptotic MMSE, we also provide convergence analysis of 
an orchestrated approximate message passing (AMP) algorithm with multiple parallel and information-sharing orbits.
% approximate message passing (AMP) with Gaussian side information \textcolor{red}{that are not independent across network nodes}. 
This could be of independent interest.

Last but not least, our results continue to hold for the special case where one only observes $m$ networks and there is no covariate.
% Both the thresholds and asymptotic MMSE formula hold for networks with growing degrees and for a Gaussian version of the model.
In this case, our model is reduced to a multilayer SBM. 
Results in \cite{jog2015information,paul2016consistent} provide sharp information-theoretic thresholds between the regimes of exact recovery 
(where the best procedure uncovers the community structure perfectly) and almost exact recovery (where the best procedure only makes mistakes on a vanishing proportion of nodes). 
In addition, \cite{paul2016consistent,xu2020optimal} derived the minimax rates of convergence that are sharp in exponents in the regime of almost exact recovery under the Hamming loss.
In contrast, the present paper provides sharp thresholds for detection (a.k.a.~weak recovery) and the exact asymptotic minimax risk under the squared error loss.

On the technical front, the main novelty of the present manuscript lies in designing an 
orchestrated
AMP algorithm 
with multiple orbits
that synchronizes the extraction of information about community structure from multiple data sources.
In addition, we provide a rigorous proof of \revsag{the almost sure convergence of the AMP average sequence}.
The idea underpinning the algorithm design is potentially applicable to other settings where the integration of information from multiple sources is needed.

The setting of the present paper can be reformulated as a multi-view spiked matrix model which has been studied in \cite{9173970,9272982}. 
The results of the aforementioned papers give the asymptotic MMSE of \emph{joint} estimation of the covariates.
% co-variates. 
More specifically, these papers have used the adaptive interpolation technique described in \cite{Adaptive_Interpolation} to obtain the limit of per-vertex mutual information between data and both the community labels and the covariate means. 
Then they have used the I-MMSE identity to get the asymptotic MMSE for the joint estimation of the community labels and the covariate means. 
A different proof technique related to a similar model was described in \cite{chen2021statistical}, where the authors identified the limiting free energy as the viscosity solution to a certain Hamilton-Jacobi equation. 
In contrast to the foregoing papers, the present manuscript focuses on the optimal estimation of the community label vector \emph{only}. 
The asymptotic joint estimation MMSE results in the foregoing papers do not lead to the asymptotic community label estimation MMSE we shall derive in this paper, since different priors (Rademacher vs.~Gaussian) have been put on the community labels and the covariate means, respectively, while the connection between the joint and individual estimation MMSEs depend crucially on the choices of the priors.
Due to the generality of the models considered, the asymptotic MMSE's in \cite{9173970,9272982} were expressed in complicated variational forms with matrix arguments. 
In contrast, under the specific setting considered in the present manuscript, we
shall obtain an explicit `single-letter' characterization of the asymptotic MMSE. 
While it remains possible to derive asymptotic per-vertex mutual information between data and community labels alone by using the adaptive interpolation technique and further obtain the asymptotic MMSE result in the current paper by differentiation, 
our proof technique is constructive and hence is entirely different from the approaches in \cite{9173970,9272982,chen2021statistical}.
While those investigations found the limit of the free energy and used it to find the limit of the per-vertex mutual information,
we shall use an AMP algorithm to explicitly construct a Bayes optimal sequence of estimators and directly obtain the asymptotic MMSE as the limiting mean squared error of that sequence. The limit of the per-vertex mutual information will be obtained as a side result of our calculations.
% avoid this technicality by explicitly constructing the Bayes optimal estimator of the community label vector using a novel AMP algorithm. Thus, we provide a sharper and more explicit limit of the asymptotic MMSE using methods which are radically different from those used in \cite{9173970,9272982}.

% \nb{we need to make it clear what is new here compared with \cite{ContBlockMod} and \cite{AbbeMonYash}}

% See, e.g., \cite{bhattacharyya2018spectral,chen2020global,HXA,NIPS,
% lei2020tail,lei2020consistent,paul2016consistent,paul2018random,paul2020spectral,inproceedings_1,xu2020optimal}.
% A lot of theory and methods developed for multilayer networks have roots in the SBM literature.

\paragraph*{Paper organization}
The rest of the paper is organized as follows.
Section \ref{cov_sec} introduces our observation models and presents key results on detection threshold and asymptotic MMSE {along with introducing the new \emph{orchestrated AMP setup}.}
It also lays out three major steps in the proof of main results, which are executed in Sections \ref{inf_sec}--\ref{mmse} in order, and formally summarized in Section \ref{Main_MMSE}.
{In Section \ref{AMP} we collect the results on the asymptotics of the orchestrated AMP algorithm. 
In Section \ref{sec:numerical} we describe a practical algorithm for estimating the community labels using orchestrated AMP with appropriate spectral initialization and demonstrate its performance through simulations.
Finally, we discuss the wider applicability of our techniques and discuss some potential future research directions in Section \ref{sec:future}.
The technical proofs are deferred to the appendices.}

\section{Detection Threshold and Asymptotic MMSE}
\label{cov_sec}

\subsection{Model} 
\label{sec:prelim}

Suppose that $n$ subjects are partitioned into two disjoint groups (labeled by $\pm 1$) according to an $n$-dimensional vector $\tx\in \{\pm 1\}^n$.
Throughout the paper we assume that the elements $x^*_i$'s are i.i.d.~Rademacher random variables which take values $\pm 1$ with equal probability $\frac{1}{2}$.

The observed data consists of two parts. 
The first part is a collection of $m$ undirected networks on these $n$ subjects denoted by their adjacency matrices, $\bm{G} = \{\bm{G}^{(i)}:i\in [m]\}$.
Throughout the paper, for any positive integer $k$, we let $[k] = \{1,\dots,k\}$.
The second part is a $p\times n$ data matrix $\bm{B}$ where the $i$-th column records the observed values of $p$ covariates on the $i$-th subject.
Conditional on an instance of $\tx$, the adjacency matrix $\bm{G}^{(i)}$ has zero diagonal entries, and for all $k\neq l$, we assume
\begin{equation}
\label{eq:def_SBM_cov}
{G}^{(i)}_{kl} = {G}^{(i)}_{lk} \stackrel{ind}{\sim}
\begin{cases}
\mathrm{Bern}\Big(\frac{\,a_n^{(i)}}{n}\Big), & \mbox{if $x^*_k = x^*_l$},\\	
\mathrm{Bern}\Big(\frac{\,b_n^{(i)}}{n}\Big), & \mbox{if $x^*_k \neq x^*_l$}.
\end{cases}
% :=\begin{cases}
% 1 & \mbox{with probability $a^{(i)}/n$ if $\bm{x^{*}}_k\bm{x^{*}}_l=1$.}\\
% 0 & \mbox{with probability $1-a^{(i)}/n$ if $\bm{x^{*}}_k\bm{x^{*}}_l=1$.}\\
% 1 & \mbox{with probability $b^{(i)}/n$ if $\bm{x^{*}}_k\bm{x^{*}}_l=-1$.}\\
% 0 & \mbox{with probability $1-b^{(i)}/n$ if $\bm{x^{*}}_k\bm{x^{*}}_l=-1$.}\\
% \end{cases}
\end{equation}
For any $\rho\in [0,1]$, $\mathrm{Bern}(\rho)$ denotes a Bernoulli distribution with success probability $\rho$. Further, we assume that $a_n^{(i)}>b_n^{(i)}$ for all $n,i\ge0$.
In addition, the data matrix $\bm{B}$ is assumed to admit the representation
\begin{equation}
\label{eq:def_cov}
\bm{B} = \sqrt{\frac{\mu}{n}}\,\bm{v^{*}}(\bm{x^{*}})^\top+\bm{R},
\end{equation}
where $\bm{R}$ is an $p \times n$ matrix consisting of i.i.d.~standard Gaussian variates and $\bm{v^{*}} \sim {N}_p(\bm{0}, \bm{I}_p)$.
Finally, we assume that conditional on $\tx$,  $\bm{G}^{(1)}, \dots, \bm{G}^{(m)}$ and $\bm{B}$ are mutually independent.
In other words, conditional on $\tx$, the first part of our data consists of $m$ stochastic block models (a.k.a.~a multi-layer stochastic block model with $m$ layers~\cite{HXA}%\nb{refs}
) with a common community structure $\tx$.
Given $\bm{v}^*$, the columns of the covariate matrix $\bm{B}$ is also partitioned into two groups by $\tx$, where those corresponding to $\tx_i=+1$ are i.i.d.~realizations of a $N_p(\sqrt{{\mu}/{n}}\,\bm{v}^*, \bm{I}_p)$ distribution, and those with $\tx_i = -1$ are i.i.d.~random vectors following $N_p(-\sqrt{{\mu}/{n}}\,\bm{v}^*, \bm{I}_p)$.

Our goal is to estimate $\tx$ upon observing $\{\bm{G}^{(1)}, \dots, \bm{G}^{(m)}\}$
%\nb{should these $G$'s be in boldface?}
and $\bm{B}$.
We focus on an asymptotic regime where as $n\to\infty$, $m$ is fixed and
\begin{equation}
	\label{eq:pnlimit}
	\lim_{n\to\infty} \frac{p}{n} = \frac{1}{c} \in (0,\infty).
\end{equation}
For each $1\leq i\leq m$, define
\begin{equation}
	\label{eq:def_lambda}
\widebar{p}^{(i)}_n  =\frac{a_n^{(i)}+b_n^{(i)}}{2n}, \quad \Delta^{(i)}_n =\frac{a_n^{(i)}-b_n^{(i)}}{2n}, \quad
\lambda_n^{(i)}  = \frac{n (a^{(i)}_n-b^{(i)}_n )^2}{(a^{(i)}_n+b^{(i)}_n )(2n-a^{(i)}_n-b^{(i)}_n)}.
\end{equation}
We further assume that $\mu \geq 0$ is a fixed constant while
\begin{align}
	& \lim_{n\to\infty} n \widebar{p}^{(i)}_n (1 - \widebar{p}^{(i)}_n ) =\infty, \quad \mbox{and}
	\label{eq:degreelimit}
	\\
	\label{eq:lambdalimit}
	& \hspace{0.35in} \lambda_n^{(i)}
	= \lambda^{(i)} \in (0,\infty),\quad \mbox{for all $n, i$.}
	% \frac{n (a^{(i)}_n-b^{(i)}_n )^2}{(a^{(i)}_n+b^{(i)}_n )(2n-a^{(i)}_n-b^{(i)}_n)}
	 %= \lambda^{(i)} \in (0,\infty),\quad \mbox{for all $i$.}
	%& \lim_{n \to \infty} \lambda_n^{(i)}
	% \frac{n (a^{(i)}_n-b^{(i)}_n )^2}{(a^{(i)}_n+b^{(i)}_n )(2n-a^{(i)}_n-b^{(i)}_n)}
	 %= \lambda^{(i)} \in (0,\infty),\quad \mbox{for all $i$.}
\end{align}
For brevity, we let
\begin{equation}
	\label{eq:lambda}
\bm{\lambda}=(\lambda^{(1)},\ldots,\lambda^{(m)})
\quad \mbox{and}\quad
\lambda = \sum_{i=1}^m \lambda^{(i)}
\end{equation}
%For brevity, we let
%$\bm\lambda=(\lambda^{(1)}_n,\ldots,\lambda^{(m)}_n)$, 
%\begin{equation}
%	\label{eq:lambda}
%\lambda_n = \sum_{i=1}^m \lambda_n^{(i)}
%\quad  \mbox{and}  \quad
%\lambda = \sum_{i=1}^m \lambda^{(i)}
%\end{equation}
in the rest of this paper.
For convenience, we shall further assume that there are
some constants $r^{(1)},\ldots,r^{(m)} \in (0,1)$ such that 
\begin{equation}
\label{eq:lambdaratio}	
\sum_{i=1}^m r^{(i)}=1\quad \mbox{and} \quad \lambda^{(i)}=r^{(i)}\lambda, \quad \mbox{for all $i$.}
\end{equation}

\subsection{Detection Threshold}

For every fixed $n$, \bm{$\lambda$} and $\mu$,
we define 
\begin{equation}
\label{eq:def_overlap}
\Overlap_n(\bm{\lambda},\mu)
=\sup_{\widehat{s}_n:\mathcal{G}_n \times \mathcal{B}_n\rightarrow \{\pm 1\}^n}
\frac{1}{n}\,\mathbb{E}\left|\left\langle\bm{x^{*}},\hat{s}_n(\bm{G},\bm{B})\right\rangle\right|.
\end{equation}
Here $\mathcal{G}_n$ is the set of all collections of $m$ undirected networks on $n$ vertices, $\mathcal{B}_n$ is the set of all $p\times n$ real-valued matrices, and $\widehat{s}_n(\bm{G},\bm{B})$ is a generic estimator of the community vector $\bm{x^{*}}$ based on observing $\bm{G}$ and $\bm{B}$. 
% \nb{should estimator take value in $[-1,1]^n$ or $\{\pm 1\}^n$?}

We say \emph{detection} (a.k.a.~\emph{weak recovery}) of $\tx$ is possible if
\begin{equation}
\label{eq:weak_recov_defn}
\liminf_{n \rightarrow \infty}
\Overlap_n(\bm{\lambda},\mu) > 0. 
% \frac{1}{n}
% \mathbb{E}\left|\langle~\bm{x^{*}},\,\hat{\bm{x}}~\rangle\right|~>~0.
\end{equation}
Otherwise, we perform no better than random guessing.
Indeed,
if we simply estimate by random guessing, then our estimator is essentially a vector $\bm{x}$ with i.i.d.~Rademacher entries that is independent of $\tx$, and we have $\frac{1}{n}\mathbb{E}|\langle \tx, \bm{x} \rangle | \to 0$.

% By weak recovery of the community structure we mean getting a vector $\hat{\bm{x}} \in \{-1,1\}^n$ such that
%
% Let us observe that if
% \[
% \lim\limits_{n \rightarrow \infty}\Overlap_n(\bm{\lambda_n}) ~=~ 0,
% \]
% then weak recovery is not possible. It is reasonable to guess that if the difference between the probabilities of inter-community and intra-community edges is negligible, then community detection in any sense will be impossible. Our main result shows that this is indeed the case with weak recovery as defined in~\eqref{eq:weak_recov_defn}.
The following theorem characterizes the phase transition of detection under our model.
\begin{thm}
\label{thm:main_thm_cov}
% Consider $\bm{x^{*}}$, a collection of i.i.d Rademacher random variables. We observe a collection of $m$ random networks $\bm{G}$ defined as in~\eqref{eq:def_SBM_cov} and a $p \times n$ covariate matrix $\bm{B}$ defined as in~\eqref{eq:def_cov}. Consider the definition of $\lambda^{(i)}_n$ defined in~\eqref{eq:def_lambda}. We define $\lambda_n=\sum\limits_{i=1}^{m}\lambda^{(i)}_n$ and its limit as $n \rightarrow \infty$ by $\lambda=\sum\limits_{i=1}^{m}\lambda^{(i)}$.
%  Let us further assume $n\widebar{p}^{(i)}_n\left(1-\widebar{p}^{(i)}_n\right) \rightarrow \infty$ and $p/n \rightarrow 1/c \in (0,\infty)
Let the data be generated by \eqref{eq:def_SBM_cov} and \eqref{eq:def_cov}.
Suppose that as $n\to\infty$, \eqref{eq:pnlimit}, \eqref{eq:degreelimit}, \eqref{eq:lambdalimit} and \eqref{eq:lambdaratio} hold.
Then 
% for all $r^{(1)},\ldots,r^{(m)} \in (0,1)$ such that \eqref{eq:lambdaratio} holds,
we have
\begin{equation}
	\label{eq:detectionthreshold}
\begin{aligned}
\limsup_{n \rightarrow \infty}\Overlap_n(\bm{\lambda},\mu) = 0, &\quad \mbox{if $\lambda +\frac{\mu^2}{c} \le 1$},\\
\liminf_{n \rightarrow \infty}\Overlap_n(\bm{\lambda},\mu) > 0, &\quad \mbox{if $\lambda +\frac{\mu^2}{c} > 1$}.
\end{aligned}
\end{equation}
where $\lambda$ 
% and $\Overlap(\lambda,\mu)$ are
is defined in \eqref{eq:lambda}.
 % and \eqref{eq:def_overlap} respectively.
% \begin{itemize}
% \item For $\lambda +\frac{\mu^2}{c} \le 1$,
% \[
% \limsup\limits_{n \rightarrow \infty}\Overlap_n(\bm\lambda,\mu) ~=~ 0.
% \]
% \item For $\lambda +\frac{\mu^2}{c} > 1$,
% \[
% \liminf\limits_{n \rightarrow \infty}\Overlap_n(\bm\lambda,\mu) ~>~ 0.
% \]
% \end{itemize}
\end{thm}
\begin{proof}
	%\nb{revise!}
The theorem follows from our later Theorem~\ref{thm:main_thm_mmse} and \eqref{eq:MMSE-overlap}.
%Lemmas 4.2 and 4.6 of~\cite{AbbeMonYash}. 
\end{proof}
\begin{rem}
Note that $r^{(1)},\ldots,r^{(m)} \in (0,1)$ are allowed to take any values as long as \eqref{eq:lambdaratio} holds. 
% {\red In other words, the asymptotic limit of overlap depends on the networks only through their overall signal, regardless of how it is divided across them.}
% % Hence for all SBM ensembles $\bm{G}$ and covariate matrix $\bm{B}$, the phase transition of detection described in Theorem~\ref{thm:main_thm_cov} holds.
\end{rem}

The foregoing theorem determines a sharp detection threshold in terms of the joint signal-to-noise ratio (SNR) contained in the two different data sources, namely the $m$ networks and the data matrix.
Here $\lambda$ can be understood as the joint SNR of the $m$ networks.
The phase transition described in \eqref{eq:detectionthreshold} asserts that the joint SNR of the two parts has an additive form $\lambda + \mu^2/c$. 
In the special case of $m = 1$, Theorem \ref{thm:main_thm_cov} reconstructs the threshold found in \cite{ContBlockMod}.
When $m = 0$, it coincides with the famous Baik--Ben Arous--Peche phase  transition for PCA \cite{baik2005,baik2006eigenvalues,10.2307/24307692}.
When $\mu = 0$, we could simply discard the data matrix $\bm{B}$ as it contains no information about $\tx$ and Theorem \ref{thm:main_thm_cov} leads to a new detection threshold for multi-layer stochastic block models.

% \begin{rem}
% From Theorem 1 and 2 of~\cite{ContBlockMod} Theorem~\eqref{thm:main_thm_cov}, we note that if we consider only one random network $G^{(i)}$ then we can detect the community vector $\bm{x^{*}}$ if the signal to noise ratio ({\it snr}) $\lambda^{(i)}$ is bigger than $1$ and if we consider only $\bm{B}$ then we can detect $\bm{x^{*}}$ only if $\mu^2/c >1$. However if we have information from $m$ random networks $\bm{G}$ as well as $\bm{B}$, we can detect $\bm{x^{*}}$ more easily.
% \end{rem}

\subsection{Asymptotic MMSE}
We now seek a more precise characterization of the optimal estimator of $\tx$ based on observing $\bm{G}$ and $\bm{B}$, where the optimality is measured through the mean square error.

\paragraph{Minimum mean square error}
To this end, define the (matrix) 
minimum mean square error for estimating the community labels from $(\bm{G},\bm{B})$ as 
% where $\bm{G}$ is defined by~\eqref{eq:def_SBM_cov} and $\bm{B}$ is defined by~\eqref{eq:def_cov}.
\begin{equation}
\label{eq:MMSE_cov_graph}
\MMSE_{n}(\bm{\lambda},\mu) 
= \frac{1}{\,n^2}\mathbb{E}\left[\|\bm{x^{*}}(\bm{x^{*}})^\top-\mathbb{E}\left[\bm{x^{*}}(\bm{x^{*}})^\top\middle|\,\bm{G},\bm{B}\right]\|^2_F\right].
\end{equation}
%\begin{equation}
%\label{eq:MMSE_cov_graph}
%\MMSE_{n}(\bm{\lambda},\mu) 
%= \frac{1}{n(n-1)}\mathbb{E}\left[\|\bm{x^{*}}(\bm{x^{*}})^\top-\mathbb{E}\left[\bm{x^{*}}(\bm{x^{*}})^\top\middle|\,\bm{G},\bm{B}\right]\|^2_F\right].
%\end{equation}
Following the lines of the proof of Lemma 4.6 in \cite{AbbeMonYash}, one has
\begin{equation}
	\label{eq:MMSE-overlap}
	1 - \MMSE_n(\bm{\lambda},\mu) + O(n^{-1})
	\leq \Overlap_{n}(\bm{\lambda}, \mu) \leq
	\sqrt{1 - \MMSE_n(\bm{\lambda},\mu)} + O(n^{-1/2}).
\end{equation}
% appropriate functions of $\MMSE_n(\bm\lambda,\mu)$ provide controls of $\Overlap_n(\bm\lambda,\mu)$ from both sides.
In particular, $\Overlap_n(\bm \lambda,\mu)\to 0$ if and only if $\MMSE_n(\bm{\lambda},\mu) \to 1$.
Therefore, a more precise characterization of the overlap than that in Theorem \ref{thm:main_thm_cov} can be made if we could describe the exact asymptotic behavior of $\MMSE_n(\bm \lambda,\mu)$.

\paragraph{A scalar Gaussian model} 
As a useful device for describing the asymptotic behavior of $\MMSE_n(\bm \lambda,\mu)$, we follow~\cite{1633802,6875223,AbbeMonYash} 
%\nb{need more references which use this one letter description} 
to introduce the following scalar Gaussian model:
\begin{equation}
	\label{eq:scalar-Gaussian}
Y = Y(\eta) = \sqrt{\eta}X_0+Z_0,	
\end{equation}
where $X_0 \sim \mbox{Rademacher}$ and $Z_0 \sim N(0,1)$. 
In \eqref{eq:scalar-Gaussian}, every term is a scalar. 
We assume knowledge of $\eta$ and the goal is to estimate $X_0$ based on the observed $Y$.
For this model, we can compute the mutual information between $X_0$ and $Y$ and the minimum mean square error for estimating $X_0$ respectively as
\begin{align}
\label{eq:inf_scalar}
\sMI(\eta) & = 
\mathbb{E}\left[\frac{dp_{Y\mid X_0}(Y(\eta)\mid X_0)}{dp_{Y}(Y(\eta))}\right] = \eta-\mathbb{E}\left[ \log\cosh(\eta+\sqrt{\eta}Z_0) \right],
\\
% \end{equation}
% \begin{equation}
\label{eq:mmse}
\smmse(\eta) & = 
\mathbb{E}\left[X_0-\mathbb{E}(X_0\mid Y(\eta))\right]^2 = 1-\mathbb{E}\left[\tanh^2(\eta+\sqrt{\eta}Z_0)\right].
\end{align}

\paragraph{Representation of the asymptotic MMSE}
With the foregoing definitions in \eqref{eq:scalar-Gaussian}, \eqref{eq:inf_scalar} and \eqref{eq:mmse}, we define $z_* = z_{*}(\lambda,\mu)$ as the largest non-negative solution to
\begin{equation}
\label{eq:zstar}
z = 1-\smmse\left(\lambda z + \frac{\mu^2}{c}\;\frac{z}{1+\mu z}\right).
\end{equation}
%\begin{equation}
%\label{eq:zstar}
%\frac{\gamma}{\lambda} = 1-\smmse\left(\gamma+\frac{\mu}{c}\cdot\frac{\mu \gamma}{\lambda+\mu \gamma}\right),
%\end{equation}
%and $\theta_*=\theta_*(\lambda,\mu)$ as the largest non-negative solution to
%\begin{equation}
%\label{eq:tstar}
%\frac{\theta}{\mu} = 1-\smmse\left(\frac{\mu}{c}\cdot\frac{\theta}{\theta+1}+\frac{\lambda}{\mu}\theta\right).
%\end{equation}
The following theorem gives a precise characterization of the limiting behavior of $\MMSE_n(\bm\lambda, \mu)$.
\begin{thm}
\label{thm:main_thm_mmse}
Suppose that the conditions in Theorem \ref{thm:main_thm_cov} hold.
%Let $z_* = z_{*}(\lambda,\mu)$ denote the largest non-negative solution to \eqref{eq:zstar}. 
Then we have
\begin{equation}
	\label{eq:MMSE-limit}
\lim\limits_{n \rightarrow \infty}
\MMSE_n(\bm\lambda,\mu) = 1-z^2_{*}(\lambda,\mu),
\end{equation}
where $\lambda$ is defined by \eqref{eq:lambda}.
This implies
\begin{enumerate}
\item If $\lambda + \mu^2/c \le 1$,
\begin{equation}
\lim\limits_{n \rightarrow \infty}\MMSE_n(\bm\lambda,\mu) =1,
\end{equation}
\item If $\lambda + \mu^2/c > 1$,
\begin{equation}
\lim\limits_{n \rightarrow \infty}\MMSE_n(\bm\lambda,\mu) <1.
\end{equation}
\end{enumerate}
\end{thm}

\begin{figure}[t]
\includegraphics[width=0.49\textwidth]{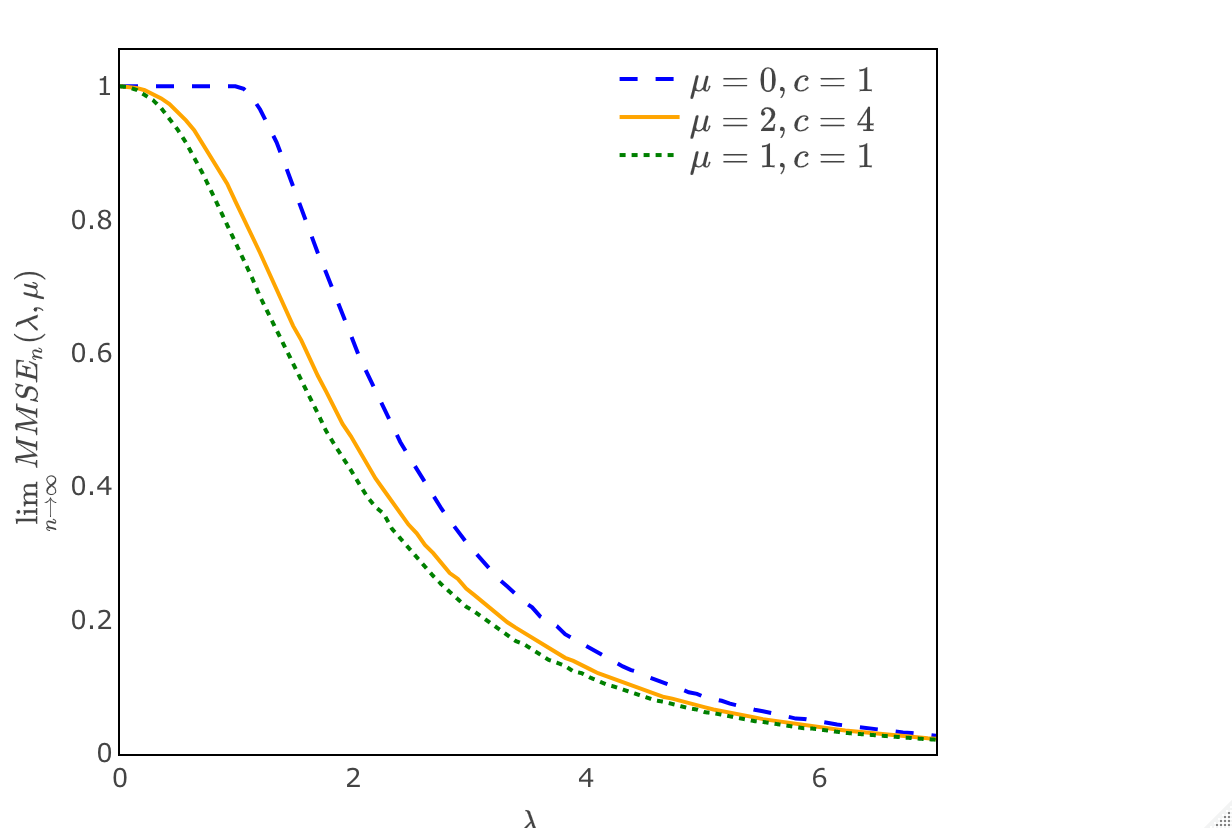}
\includegraphics[width=0.49\textwidth]{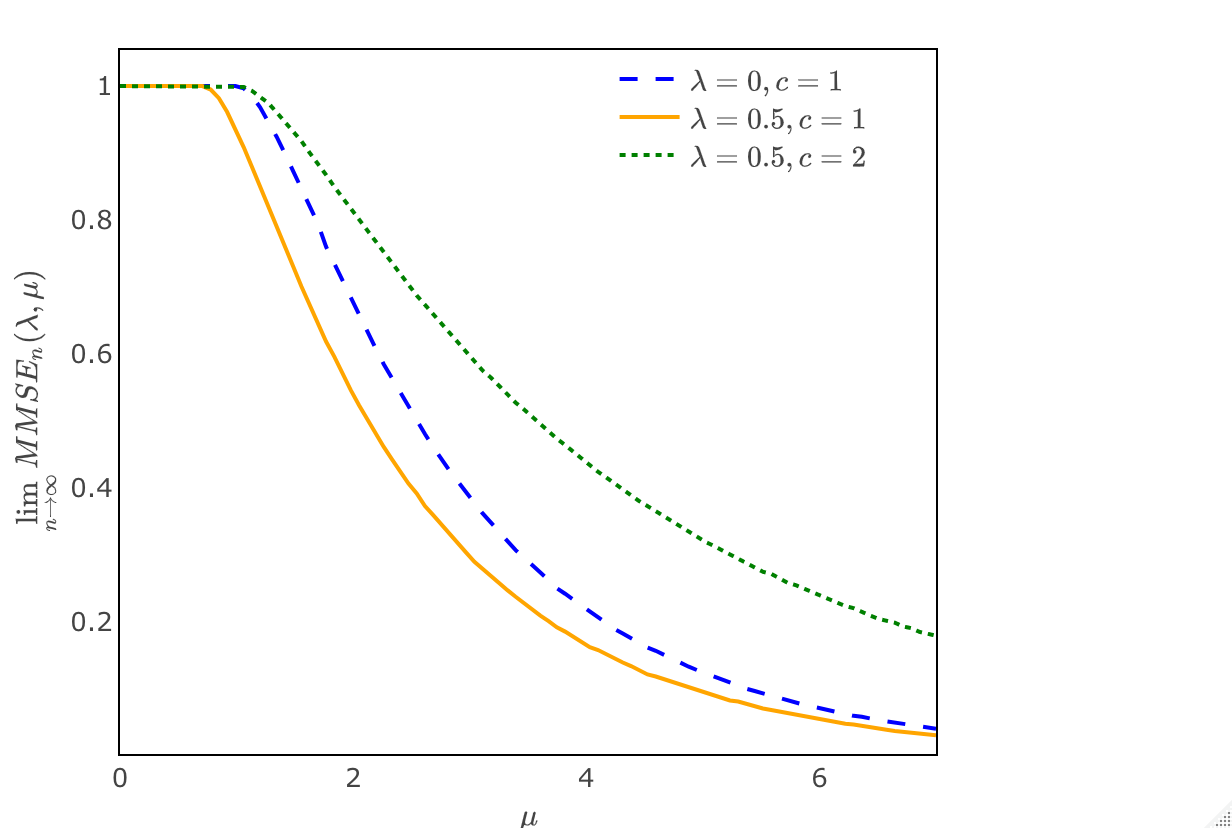}
\centering
\caption{Plots of asymptotic MMSE. 
Left panel: Asymptotic MMSE as a function of $\lambda$ for fixed $\mu$ and $c$ combinations. 
Right panel: Asymptotic MMSE as a function of $\mu$ for fixed $\lambda$ and $c$ combinations. 
% \nb{Increase font size. 20210405zm}
% as a function of $\lambda$ for various values of $\mu$ and $c$.
}
\label{fig:mmse}
\end{figure}

\begin{rem}
If $\lambda + \mu^2/c > 1$, then $z_{*}(\lambda,\mu)$, the largest non-negative solution to \eqref{eq:zstar} is strictly greater than zero. Consequently, the limit of MMSE is strictly less than $1$, or equivalently, we strictly perform better than random guessing.
\end{rem}
\begin{rem}
Together with \eqref{eq:MMSE-overlap}, phase transition of matrix MMSE for estimating $\tx$ in Theorem \ref{thm:main_thm_mmse} implies the phase transition of the overlap in Theorem \ref{thm:main_thm_cov}.
\end{rem}
\begin{rem}
By Theorems \ref{thm:main_thm_cov} and \ref{thm:main_thm_mmse}, \revsag{the parameter $\bm{\lambda}$} {affects the asymptotic behavior of} both $\Overlap_n$ and $\MMSE_n$ only through $\lambda$.
So from here on, we shall slightly abuse notation to also write $\Overlap_n(\lambda,\mu)$ and $\MMSE_n(\lambda,\mu)$,
which can be viewed as fixing a set of $r^{(i)}$'s in \eqref{eq:lambdaratio} and hence treating $\Overlap_n$ and $\MMSE_n$ as functions of $\lambda$ and $\mu$ only for this fixed set of $r^{(i)}$'s.
\end{rem}

% \begin{rem}
In Figure \ref{fig:mmse}, we show how $\lim_{n \rightarrow \infty}\MMSE_n(\lambda,\mu)$ behaves as a function of $\lambda$ for different values of $\mu$ and $c$, and as a function of $\mu$ for different values of $\lambda$ and $c$. 
It is worth noting that even for the same values of $\lambda$ and $\mu^2/c$, asymptotic MMSE's could differ, as its dependence on the three parameters is more subtle than the phase transition threshold.
% Since the curves for $\mu =1, c=1$ and $\mu=2, c=4$ are different, we can conclude that the said limit is not dependent on $\mu$ only through $\mu^2/c$ but the dependence is more subtle. }
% \end{rem}
%\nb{shall we insert the plots of MMSEs here? in addition, we need to discuss that the MMSE formula depends on $\mu$ and $c$ in a more subtle way than the threshold does.}

% \begin{rem}
% Note that, setting $\lambda =0$ in \eqref{eq:zstar}, we get the BBAP phase transition, described in~\cite{baik2005} and~\cite{10.2307/24307692}, as a special case of Theorem \ref{thm:main_thm_mmse}. Similarly, setting $\mu = 0$, we recover the phase transition described in \cite{AbbeMonYash}.
% \end{rem}
%

\subsection{Outline of Proof}
\label{sec:outline}
% \nb{ZM started here. 20210318}
We now outline the three major steps in the proof of Theorem \ref{thm:main_thm_mmse}. We follow the same overall proof structure as in \cite{AbbeMonYash}. Compared with \cite{ContBlockMod} and \cite{AbbeMonYash},
our major novelty lies in the proposal of an orchestrated AMP algorithm which synchronizes updates about community structure from multiple data sources.
% \nb{we need to briefly comment on what is new here compared with \cite{ContBlockMod} and \cite{AbbeMonYash}. 20210318}

\paragraph{Gaussian approximation}
In the first step, we define a Gaussian observation model whose asymptotic per-vertex mutual information about $\tx$ is the same as that in the 
original observation
% ``multi-layer network + covariate'' 
model in \eqref{eq:def_SBM_cov}--\eqref{eq:def_cov}.
To this end, 
let $\left\{\bm{Y}^{(i)}: i \in [m]\right\}$ be a collection of symmetric Gaussian matrices defined as
\begin{equation}
\label{eq:def_Gauss_Chann}
\bm{Y}^{(i)}=\sqrt{\frac{\,\lambda^{(i)}}{n}}\,\bm{x^{*}}(\bm{x^{*}})^\top
+\bm{Z}^{(i)},\quad i\in [m],
\end{equation}
where $\bm{Z}^{(i)}$'s are i.i.d.~Gaussian Wigner matrices.
In other words, each $\bm{Z}^{(i)}$ is symmetric and
\begin{equation}
\label{eq:def_Z}
Z^{(i)}_{kl} = Z^{(i)}_{lk} \sim 
\begin{cases}
N\left(0,1\right) & \mbox{if $k\neq l$.}\\
N\left(0,2\right) & \mbox{if $k = l$.}
\end{cases}
\end{equation}
We shall denote the collection of $\bm{Y}^{(i)}$'s by $\bm{Y}$, i.e.
\begin{equation*}
	\bm{Y} = \{\bm{Y}^{(i)}: i\in [m] \}.
\end{equation*}
Using Lindeberg's interpolation argument, we show that the per-vertex mutual information between $\tx$ and the Gaussian observation model $\{\bm{Y},\bm{B}\}$ is asymptotically the same as that between $\tx$ and the original observation $\{\bm{G}, \bm{B} \}$, in the sense that
\begin{equation}
	\label{eq:info-asymp-equiv}
\frac{1}{n}I(\tx;\bm{G},\bm{B}) - \frac{1}{n}I(\tx;\bm{Y},\bm{B}) \to 0
\quad \mbox{as} \quad
n\to\infty.
\end{equation}
As a final reduction in our first step, we shall show that 
\begin{equation}
\label{eq:i_bet_matrices}
	I(\tx;\bm{T},\bm{B}) = I(\tx;\bm{Y},\bm{B})
\end{equation}
where for $\lambda$ defined in \eqref{eq:lambda} and $\bm{Z}$, a Gaussian Wigner matrix, as in \eqref{eq:def_Z}
\begin{equation}
\label{eq:T}
\bm{T} = \bm{T}(\lambda) = \sqrt{\frac{\lambda}{n}}\,\bm{x^{*}}(\bm{x^{*}})^\top + \bm{Z}.
\end{equation}

\paragraph{Asymptotic I-MMSE relation}
For Gaussian observation models, one has the famous I-MMSE relation \cite{guo2005mutual}. 
For instance, for the Gaussian observation $(\bm{T},\bm{B})$, define 
%\nb{use a slightly different notation here for Gaussian models, instead of using subscripts}
\begin{equation}
\label{eq:Gmmse}
	\GMMSE_n(\lambda,\mu) = \frac{1}{\,n^2}\, \mathbb{E}\left\| \tx(\tx)^\top - \mathbb{E}[ \tx(\tx)^\top \,|\, \bm{T},\bm{B}]\right\|_F^2.
\end{equation}
Then the I-MMSE relation refers to the identity 
%\nb{not exactly sure. double check later.}
% \nb{this identity seems problematic... one needs to account for the fact that the diagonal entries are known, and the leftside should ideally be $I(\tx(\tx)^\top; \bm{Y},\bm{B})$}
\begin{equation}
\label{eq:I_MMSE}
	\frac{1}{n}\frac{d}{d\lambda}I\left(\tx(\tx)^\top;\bm{T},\bm{B}\right) = \frac{1}{4}\GMMSE_n(\lambda,\mu).
\end{equation}
See Section~\ref{I_MMSE} for a proof of \eqref{eq:I_MMSE}.
On the other hand, we shall derive the following asymptotic counterpart of \eqref{eq:I_MMSE} for the original observation model $\cbrac{\bm{G},\bm{B}}$:
\begin{equation}
	\label{eq:I-MMSE-lim-G} 
 \frac{1}{n}\frac{d}{d\lambda}I(\tx;\bm{G},\bm{B}) - \frac{1}{4}\MMSE_n(\lambda,\mu) \to 0.
\end{equation}
% identity similar to~\eqref{eq:I_MMSE} characterizing the derivative of $I(\tx;\bm{G},\bm{B})$ with respect to $\lambda$.
Furthermore, (263) of~\cite{AbbeMonYash} implies for $\bm{K}= \bm{Y}$ and $\bm{T}$
\begin{equation}
\label{eq:I-MMSE-lim-Y} 
 \frac{1}{n}I(\tx;\bm{K},\bm{B}) - \frac{1}{n}I\left(\tx(\tx)^\top;\bm{K},\bm{B}\right)  \to 0.
\end{equation}
%Since both terms on the leftside have finite limits, the last display, together with \eqref{eq:i_bet_matrices} and \eqref{eq:I_MMSE} further implies that 
%\begin{equation}
%	\label{eq:I-MMSE-lim-Y} 
%	\frac{1}{n}\frac{d}{d\lambda}I(\tx;\bm{T},\bm{B}) - \frac{1}{4}\GMMSE_n(\lambda,\mu) \to 0.
%\end{equation}
Together with \eqref{eq:info-asymp-equiv} and the fundamental theorem of calculus, 
\eqref{eq:I-MMSE-lim-G} and \eqref{eq:I-MMSE-lim-Y} establish the asymptotic equivalence of the MMSE's in the two models.
Hence, the proof of Theorem~\ref{thm:main_thm_mmse} reduces to
finding the exact asymptotic limit of $\GMMSE_n(\lambda,\mu)$.
To this end, we turn to Approximate Message Passing (AMP).
%\nb{further argue the reduction to the $\cbrac{\bm{T},\bm{B}}$ model or simply work on the $\cbrac{\bm{T},\bm{B}}$ model directly?}
 % suffices to prove .

%\nb{motivate the establishment of an analogous identity that holds asymptotically for the network model}

%\nb{state clearly that the purpose is then to establish the asymptotic equivalence between the MMSEs in the two models}

%\nb{== start here ==}
%
%In Section 5, we derive a differentiation result akin to I-MMSE identity, connecting the matrix mean squared error of estimating $\bm{x^{*}}(\bm{x^{*}})^\top$ using the networks and the co-variates with the per-vertex mutual information of $\bm{x^{*}}$ and the random networks plus co-variates. Then we finally combine the results from Section 3 and 4 with the differentiation result to derive the limit of $\MMSE_n(\bm\lambda,\mu)$ (defined in~\eqref{eq:MMSE_cov_graph}). Then we can conclude the proof of Theorem~\ref{thm:main_thm_cov} using 
%Lemma 4.2 and Lemma 4.6 of~\cite{AbbeMonYash}. 

\paragraph{Approximate message passing and MMSE in the Gaussian observation model}
To obtain the ``large $n$'' limit of $\GMMSE_n(\lambda,\mu)$, 
we design the following orchestrated
% approximate message passing (AMP)
AMP algorithm where we 
extract
% integrate 
information about $\tx$ from both the data sources.
% treat the covariate part $\bm{B}$ as side information.

Let $\bm{u}^0=\bm{x}^0=\bm{0}$, $\bm{T}$ be as in~\eqref{eq:T} and $\bm{B}$ be as in~\eqref{eq:def_cov}. 
Fix any $\varepsilon \in (0,1)$.
We define 
% \sout{a sequence of AMP iterates} 
two companion AMP orbits $\{\bm{v}^t,\bm{u}^{t+1}\}$ and $\{\bm{x}^{t+1}\}$, for $t=0,1,2,\ldots,$ characterized by the sensing matrices $\bm T$ 
% $\bm A$
and $\bm B$ respectively, as follows
\begin{equation}
\begin{aligned}
\label{eq:AMP_shift_main_1}
\bm{v}^{t} &= \frac{\bm{B}}{\sqrt{p}} f_{t}(\bm{u}^t,\bm{x}^t,\bm{x}_0(\varepsilon))-p_{t}g_{t-1}(\bm{v}^{t-1},\bm{v}_0(\varepsilon)),\\
\bm{u}^{t+1} &= \frac{\,\bm{B}^\top}{\sqrt{p}} g_t(\bm{v}^t,\bm{v}_0(\varepsilon))-c_tf_{t}(\bm{u}^t,\bm{x}^t,\bm{x}_0(\varepsilon)),\\
\end{aligned}
\end{equation}
and
\begin{equation}
\label{eq:AMP_shift_main_1_1}
\bm{x}^{t+1} = \frac{\bm{T}}{\sqrt{n}} f_{t}(\bm{u}^t,\bm{x}^t,\bm{x}_0(\varepsilon))-d_{t}f_{t-1}(\bm{u}^{t-1},\bm{x}^{t-1},\bm{x}_0(\varepsilon)).
\end{equation}

%\begin{equation}
%\label{eq:AMP_updates}
%\bm{x}^{t+1}=\frac{\bm{T}}{\sqrt{n}}\bm{f}_t(\bm{x}^t,\bm{B})-b^t\bm{f}_{t-1}(\bm{x}^{t-1},\bm{B}),
%\end{equation}
Here $\bm{x}_0(\varepsilon)=(x_{0,1}(\varepsilon),\ldots,x_{0,n}(\varepsilon))^\top
=(B_1x^*_1,\ldots,B_n x^*_n)^\top$ 
% \nb{shall we change these to column vectors?}
where $B_i$'s are i.i.d.~$\mathrm{Bern}(\varepsilon)$ and $\bm{v}_0(\varepsilon)=(v_{0,1}(\varepsilon),\ldots,v_{0,p}(\varepsilon))^\top
=(\widetilde{B}_1 v^*_1,\ldots,\widetilde{B}_p v^*_p)^\top$ 
where $\widetilde{B}_j$'s are  i.i.d.~$\mathrm{Bern}(\varepsilon)$. 
Next, we define \begin{equation}
\label{eq:ft-gt}
\begin{aligned}
g_t(\bm{v}^t,\bm{v}_0)&=(g_t(v^t_1,v_{0,1}(\varepsilon)),\ldots,g_t(v^t_p,v_{0,p}(\varepsilon)))^\top,\\
f_{t}(\bm{u}^t,\bm{x}^t,\bm{x}_0)&=(f_{t}(u^t_1,x^t_1,x_{0,1}(\varepsilon)),\ldots,f_{t}(u^t_n,x^t_n,x_{0,n}(\varepsilon)))^\top,
\end{aligned}
\end{equation}
and
\begin{equation}
\label{eq:def_correc}	
% \begin{aligned}
% c_t &= \frac{1}{p}\sum\limits_{i=1}^{p} \frac{\partial g_t}{\partial v}(v^t_i,{\red v}_{0,i}(\varepsilon)),\\
% p_t &= \frac{c}{n}\sum\limits_{i=1}^{n} \frac{\partial f_{t}}{\partial u}(u^t_i,{\red x}^t_i,x_{0,i}(\varepsilon)),\\
% d_t &= \frac{1}{n}\sum\limits_{i=1}^{n} \frac{\partial f_t}{\partial y}(u^t_i,{\red x}^t_i,x_{0,i}(\varepsilon)).
% \end{aligned}
c_t = \frac{1}{p}\sum\limits_{i=1}^{p} \frac{\partial g_t}{\partial v}(v^t_i,v_{0,i}(\varepsilon)),\quad
p_t = \frac{c}{n}\sum\limits_{i=1}^{n} \frac{\partial f_{t}}{\partial u}(u^t_i,x^t_i,x_{0,i}(\varepsilon)),\quad
d_t = \frac{1}{n}\sum\limits_{i=1}^{n} \frac{\partial f_t}{\partial x}(u^t_i,x^t_i,x_{0,i}(\varepsilon)),
\end{equation}
where $g_{-1}$ is the zero function, and $f_t:\mathbb{R}^3 \rightarrow \mathbb{R}$ and $g_t:\mathbb{R}^2\rightarrow \mathbb{R}$ for $t \in \mathbb{N}\cup\{0\}$ are defined as follows:
%\begin{equation}
%\label{eq:divergence}
%b^t =\mbox{div}\bm\,\bm{f}(\bm{x},\bm{B})=\sum\limits_{i=1}^{n}\frac{\partial}{\partial x_j}\bm{f}_j(\bm{x}_j,\bm{B}).
%\end{equation} 
% Now we choose the sequence of update functions
%\nb{Note that the $\alpha$'s, $\tau$'s, etc.~are not defined!}
%{\red At each coordinate, we define}
\begin{align}
\label{eq:def_f}
f_t(x,y,z) & =
\mathbb{E}\left[X_0|\alpha_{t-1}X_0+\tau_{t-1}Z_1=x,\mu_tX_0+\sigma_tZ_2=y,X_0(\varepsilon)=z\right],\\
% \end{equation}
% \begin{equation}
\label{eq:def_g}
g_t(x,z) & = \mathbb{E}\left[V_0|\beta_tV_0+\vartheta_tZ_3=x,V_0(\varepsilon)=z\right].
% \quad t = 0,1,2,\dots,
\end{align}
In the above definitions \eqref{eq:def_f} and \eqref{eq:def_g}
\begin{equation}
	\label{eq:limit-quant}
\begin{aligned}
& \mbox{$X_0 \sim \mbox{Rademacher}$, $X_0(\varepsilon)=B(\varepsilon)X_0$ with $B(\varepsilon) \sim \mathrm{Bern}(\varepsilon)$}, 	\\
& \mbox{$V_0(\varepsilon)=\widetilde{B}(\varepsilon)V_0$ with $\widetilde{B}(\varepsilon) \sim \mathrm{Bern}(\varepsilon)$,
$V_0, Z_1, Z_2, Z_3 \sim N(0,1)$, and} \\
& \mbox{$X_0, V_0, B(\varepsilon), \widetilde{B}(\varepsilon), Z_1,Z_2$ and $Z_3$ are mutually independent.}
\end{aligned}
\end{equation}
In addition, the quantities $\alpha_t,\tau_t,\mu_t,\sigma_t,\beta_t$ and $\vartheta_t$'s are recursively defined as follows.
Let $\mu_0=\sigma_0=\alpha_{-1}=\tau_{-1}=0$ and $\smmse(\cdot)$ be defined as in \eqref{eq:mmse}, then for all $t \ge 0$
\begin{equation}
	\label{eq:state_ev_rec}
\begin{aligned}	
\mu_{t+1} = \sqrt{\lambda}\;\left(1-(1-\varepsilon)\smmse\left(\frac{\alpha^2_{t-1}}{\tau^2_{t-1}}+\frac{\mu^2_{t}}{\sigma^2_{t}}\right)\right),
& \quad
\sigma^2_{t+1} = 1-(1-\varepsilon)\smmse\left(\frac{\alpha^2_{t-1}}{\tau^2_{t-1}}+\frac{\mu^2_{t}}{\sigma^2_{t}}\right);\\             
\alpha_t = \sqrt{\frac{\mu}{c}}\;\left((1-\varepsilon)\frac{\beta^2_t}{\beta^2_t+\vartheta^2_t}\right),
% \\
& \quad
\tau^2_{t} =(1-\varepsilon)\frac{\beta^2_t}{\beta^2_t+\vartheta^2_t}; \\     
\beta_t = \sqrt{\frac{\mu}{c}}\;c \left(1-(1-\varepsilon)\smmse\left(\frac{\alpha^2_{t-1}}{\tau^2_{t-1}}+\frac{\mu^2_{t}}{\sigma^2_{t}}\right)\right),
& \quad
\vartheta^2_{t} = c \left(1-(1-\varepsilon)\smmse\left(\frac{\alpha^2_{t-1}}{\tau^2_{t-1}}+\frac{\mu^2_{t}}{\sigma^2_{t}}\right)\right).
\end{aligned}
\end{equation}

\begin{rem}
These AMP iterations can be viewed as a corrected version of the power iteration to simultaneously estimate the leading eigenvector of 
$\bm T$
% $\bm A$ 
and the leading singular vectors of $\bm B$. 
However, studying the asymptotics of the iterates of the power iteration 
% iterates 
is difficult because of the dependence introduced in each step. \revsag{This difficulty is overcome by subtracting a so-called ``\emph{Onsager term}" in every
iteration, which ensures that $x^t_i$ for $i=1,\ldots,n$ are ``almost independent". Further, we deviate from using the linear version of the power iteration and specific non-linear functions $f_t$, $g_t$ (tailored to the priors in the model) are applied componentwise/row-wise to previous iterates before post multiplying the iterates to the matrices $\bm A$ and $\bm B$, so as to obtain asymptotically Bayes
optimal estimates of $\bm x^*$.} One can refer to \cite{BM11journal,JM12} for further understanding of AMP in general.
\end{rem}

\begin{rem}
The AMP iterates \eqref{eq:AMP_shift_main_1} and \eqref{eq:AMP_shift_main_1_1} are based on the $\varepsilon$-revelation of the truth $\bm{x^{*}}$ and $\bm{v^{*}}$, which is adopted here to eliminate the degenerate case where all updates $\bm{u}^t=\bm{x}^t=\bm{v}^t =0$ for $t \ge 0$.
Alternatively, such degeneracy could potentially be avoided by considering spectral initialization (e.g., \cite{montanari2021}). 
Since our primary goal here is to use AMP for bounding $\GMMSE_n(\lambda,\mu)$, we choose to work with the $\varepsilon$-revelation approach as its theoretical analysis is cleaner.
% \nb{start here!!}
% Note that the partial revealing of the community labels $\tx$ and $\bm w$ in the construction of the update functions $f_t$ and $g_t$ is essential to remove the degenerate case of the updates being $\bm{u}^t={\red \bm{x}}^t=\bm{v}^t =0$ for all $t \ge 0$. This can occur otherwise, if we initialize the AMP by setting $\bm{u}^0={\red \bm{x}}^0=\bm{0}$. Such degeneracies can also be removed by considering spectral initialization of $\bm u^t$ and $\bm {\red x}^t$ as described in \cite{montanari2021}, but we choose this method as it leads to an easier analysis.
\end{rem}

\begin{rem}
% While the AMP used in~\cite{AbbeMonYash}, had one Gaussian matrix $\bm{T}$
% our AMP has two Gaussian matrices $\bm{T}$ and $\bm{B}$.
The major difficulty in designing the AMP iterates lies in the effective integration of information from 
multiple %two 
data sources. 
One possibility
% possible way to handle it 
is to treat $\bm{T}$ as the main information and $\bm{B}$ as the side information, or vice versa.
Although AMP with side information has been considered in~\cite{Rush} in the context of signal recovery from noisy observations, the generic approach in~\cite{Rush} does not work in the present context. 
In \cite{Rush}, the side information, contained in a set of random variables
% \sout{is separable 
% in the sense that if} 
$\{S_1,\ldots,S_n\}$ where $S_i$ contains the side information for node $i$, has the special property that they are mutually independent. 
In our case, the side information is in the form of $\{\bm b_1, \ldots, \bm b_n\}$ where $\bm b_i$ is the $i$-th column of the matrix $\bm B$. 
% It is easy to check that $\bm b_i$'s 
They
are not independent whenever $\mu > 0$, and hence the side information 
% \sout{is non-separable} 
are not independent across nodes,
% This substantially complicates the situation, 
and a direct application of the results in \cite{BM11} as in \cite{Rush} is impossible. 
An alternative approach
% Another approach 
is to construct
% consider 
a \revsag{sequence of AMP recursions with matrix valued iterates} as in (28)-(29) of \cite{JM12}. 
However, we can verify that this leads to a version of the AMP that is not Bayes optimal. 
This is because $\bm x^*$ is essentially estimated in two separate iterations using $\bm T$ 
% $\bm A$
and $\bm B$. 
This nonsynchronized iteration is
the root of the
% to update the estimates of $\bm x$ leads to 
sub-optimal performance. 
Our proposed iterates \revsag{\eqref{eq:AMP_shift_main_1}-\eqref{eq:AMP_shift_main_1_1} are designed to resolve} this issue by running two parallel AMP orbits with sensing matrices $\bm T$ and $\bm B$, respectively, while sharing information between each other at each iteration.
This is achieved by the \revsag{use of a synchronized update} function $f_t$ which takes both $\bm u^t$ and $\bm x^t$ in its arguments.
% To resolve this issue, we introduce the iterations \eqref{eq:AMP_shift_main_1} and \eqref{eq:AMP_shift_main_1_1}, where we run two parallel AMP's with the sensing matrices $\bm A$ and $\bm B$ sharing information between each other in each iteration. This is achieved by the use of synchronized update functions to update $\bm h^t$ and $\bm x^t$.} 
%\nb{need some more detail on what ``separable'' means here.} 
 % in presence of side information related to the true signal.
% We regard the matrix $\bm{T}$ as the main information and $\bm{B}$ as the side information.
% %The non-separability arises from the use of $\bm{B}$ as side information in the update function.
% An AMP with side information has been considered in~\cite{Rush} in the context of signal recovery from noisy observations in presence of side information related to the true signal. But in their case the side information was separable, whereas, in our case, $\bm{B}$ is a non-separable side information matrix.
\end{rem}

Finally, we define a sequence of estimates of $\tx$ based on the AMP iterates as
\begin{equation}
\label{eq:amp_final_estim}
\widehat{\bm{x}}^{t} = f_{t-1}(\bm{u}^{t-1},\bm{x}^{t-1},\bm{x}_0(\varepsilon)),
\end{equation} 
As these estimates are functions of $\bm{T}$, $\bm{B}$, $\bm{x}_0(\varepsilon)$ and $\bm{v}_0(\varepsilon)$, 
the mean square errors of $\widehat{\bm{x}}^{t}(\widehat{\bm{x}}^{t})^\top$ to estimate $\bm{x^{*}}(\bm{x^{*}})^\top$ 
provide a sequence of upper bounds for 
 % for $\GMMSE_n(\lambda,\mu,\varepsilon)$, where,
\begin{equation}
\label{eq:Gmmse_eps}
	\GMMSE_n(\lambda,\mu,\varepsilon) = \frac{1}{\,n^2}\, \mathbb{E}\left\| \tx(\tx)^\top - \mathbb{E}[ \tx(\tx)^\top \,|\, \bm{T},\bm{B},\bm{x}_0(\varepsilon),\bm{v}_0(\varepsilon)]\right\|_F^2.
\end{equation}
%\nb{should be $\GMMSE(\lambda,\mu)$ if $\bm{T}$ is used in observation}. 
% Using the results from the first part and the I-MMSE identity described in the second part,
We shall show that
% it will be shown that
the mean square errors of $\widehat{\bm{x}}^{t}(\widehat{\bm{x}}^{t})^\top$ converge to the same limit as $\GMMSE_n$ in the ``large $n$, large $t$' asymptotics. 
We analyze the asymptotics of the mean square errors of $\widehat{\bm{x}}^{t}(\widehat{\bm{x}}^{t})^\top$ by analyzing the AMP defined in \eqref{eq:AMP_shift_main_1} and \eqref{eq:AMP_shift_main_1_1}.
To this end, we augment the techniques in \cite{BM11} to 
handle multiple communicating orbits.
% incorporate side information.
Last but not least,
we argue that as $\varepsilon$ goes to $0$, the mean square errors of $\widehat{\bm{x}}^{t}(\widehat{\bm{x}}^{t})^\top$ approximate
the limit of $\GMMSE_n(\lambda,\mu)$.
%adapting the results in~\cite{Berthier} to our setting. 
Therefore,
by showing that the ``large $n$, large $t$, small $\varepsilon$'' limit of mean square errors of the AMP iterates is exactly the same as that in Theorem \ref{thm:main_thm_mmse}, we complete the proof.
\section{Gaussian Approximation and Asymptotic Per-Vertex Mutual Information}
\label{inf_sec}
The results spelt out in this section closely follow the results of Section 5 in \cite{AbbeMonYash}. We list them here for the paper to be self contained.

\subsection{Mutual Information in the Gaussian Model}

Let us recall the Gaussian model given by $\bm{Y}$, the collection of Gaussian random matrices defined in~\eqref{eq:def_Gauss_Chann}; the SBM ensemble $\bm{G}$ defined by~\eqref{eq:def_SBM_cov}; and the covariate matrix $\bm{B}$ defined in~\eqref{eq:def_cov}. We shall show that as $n \to \infty$, the per-vertex mutual information between $\tx$ and the model $\{\bm{Y},\bm{B}\}$ is 
asymptotically the same as
% approximately equal to 
the the per-vertex mutual information between $\tx$ and the model $\{\bm{G},\bm{B}\}$. 

To this end, we begin by defining the Hamiltonian function $\mathcal{H}$ for $m$ arbitrary $n \times n$ symmetric matrices $\bm{V}^{(1)},\bm{V}^{(2)},\ldots,\bm{V}^{(m)}$: 
\begin{equation}
\label{eq:cov_H}
\mathcal{H}(\bm{x},\bm{x^{*}},\bm{v},\bm{V},\bm{B},\bm{\lambda},\mu,n,p):= \mathcal{H}'\left(\bm{x},\bm{x^{*}},\bm{V},\bm{\lambda},n\right)-
\frac{1}{2}\norm{\bm{B}-\sqrt{\frac{\mu}{n}}\bm{v}\bm{x}^\top}^2_F
\end{equation}
where
\begin{equation}
\label{eq:def_H}
% \begin{split}
\mathcal{H}'\left(\bm{x},\bm{x^{*}},\bm{V},\bm{\lambda},n\right)
:=\sum_{i=1}^{m}\sum_{k<l}V^{(i)}_{kl}\left(x_kx_l-x^*_kx^*_l\right)
+\sum_{i=1}^{m}\sum_{k<l}\frac{\lambda^{(i)}}{n}x_kx_lx^*_kx^*_l
% \end{split}
\end{equation}
with $\bm{V}:=\left(\bm{V}^{(1)},\ldots,\bm{V}^{(m)}\right)$.
 % and $\bm{\lambda}:=\left(\lambda^{(1)},\ldots,\lambda^{(m)}\right)$.
Further, define
\begin{equation}
\label{eq:cov_phi}
\begin{aligned}
& \phi(\bm{x^{*}},\bm{B},\bm{V},\bm{\lambda},\mu,n,p)\\ 
& ~~~~~ = \log\Bigg\{\sum_{\bm{x} \in \{\pm 1\}^{n}}\int_{\mathbb{R}^p}\exp(\mathcal{H}(\bm{x},\bm{x^{*}},\bm{v},\bm{V},\bm{B},\bm{\lambda},\mu,n,p))
\exp\left(-\frac{\|\bm{v}\|^2}{2} \right)d\bm{v}\Bigg\}.	
\end{aligned}
\end{equation}
Then we have the following lemma.
\begin{lem}
\label{lem:cov_gauss_sbm}
Let us consider $\bm{Y}=\{\bm{Y}^{(i)}: i \in [m]\}$ defined in~\eqref{eq:def_Gauss_Chann} and $\bm{B}$ defined in~\eqref{eq:def_cov}. Then we have 
\begin{align*}
I(\bm{x^{*}};\bm{Y},\bm{B})  & =  n\log 2 +\frac{n-1}{2}\sum_{i=1}^{m}\lambda^{(i)}\\
&\quad\quad +\mathbb{E}\log\Bigg(\int_{\mathbb{R}^p}\exp\Bigg(-\frac{1}{2}
\norm{\bm{B}-\sqrt{\frac{\mu}{n}}\bm{v}(\bm{x^{*}})^\top}^2_F\Bigg)
\exp\Bigg(-\frac{\|\bm{v}\|^2}{2} \Bigg)d\bm{v}\Bigg)\\
&\quad\quad-\mathbb{E}[\phi(\bm{x^{*}},\bm{B},\bm{W},\bm{\lambda},\mu,n,p)]
\end{align*}
where $\phi(\bm{x^{*}},\bm{B},\bm{W},\bm{\lambda},\mu,n,p)$ is defined in~\eqref{eq:cov_phi} and $\bm{W}=(\sqrt{\lambda^{(1)}/n}\bm{Z}^{(1)},\ldots,\sqrt{\lambda^{(m)}/n}\bm{Z}^{(m)})$.
\end{lem}
\begin{proof}
See Section \ref{proof_lem_1}.
\end{proof}
Furthermore, if we consider the random matrix $\bm{T}(\lambda)$ defined by~\eqref{eq:T}, the following lemma shows that the mutual information between $\bm{x^{*}}$ and $\{\bm{Y}, \bm{B}\}$ is the same as the mutual information between $\bm{x^{*}}$ and $\{\bm{T},\bm{B}\}$.
\begin{lem}
\label{lem:join_all_matrices}
If we consider $\bm{T}(\lambda)$ defined in~\eqref{eq:T}, $\bm{Y}$ defined in~\eqref{eq:def_Gauss_Chann} and $\bm{B}$ defined in~\eqref{eq:def_cov} \revsag{then we have}
\[
I(\bm{x^{*}};\bm{Y},\bm{B}) = I(\bm{x^{*}};\bm{T}(\lambda),\bm{B}).
\]
\end{lem}
\begin{proof}
See Section \ref{proof_lem_2}.
\end{proof}
This shows that it is equivalent to study the model $\{\bm{T},\bm{B}\}$ or $\{\bm{Y},\bm{B}\}$. It is easier to study the model $\{\bm{T},\bm{B}\}$ as instead of dealing with an $n$-vector of parameters $\bm{\lambda}$ in $\{\bm{Y},\bm{B}\}$, in $\{\bm{T},\bm{B}\}$ we can study the model with respect to a single parameter $\lambda$.

\subsection{Mutual Information in the Original Model}
Next, we observe that the entries of the adjacency matrix $G^{(i)}_{kl}$ of the adjacency matrices $\bm{G}^{(i)}$ are given by 
\[
G^{(i)}_{kl}:=\begin{cases}
1 & \mbox{with probability $\widebar{p}^{(i)}_n+\Delta^{(i)}_nx^*_kx^*_l$,}\\
0 & \mbox{with probability $1-\widebar{p}^{(i)}_n-\Delta^{(i)}_nx^*_kx^*_l$.}
\end{cases}
\]
We define the function $\mathcal{H}_{SBM}$, the Hamiltonian with respect to the multilayer SBM as follows.
\begin{equation}
\label{eq:H_SBM_cov}
\mathcal{H}_{SBM}(\bm{x},\bm{x^{*}},\bm{u},\bm{G},\bm{B},\bm{\lambda},\mu,n,p)) ~=~ \mathcal{H}_{SBM}'\left(\bm{x},\bm{x^{*}},\bm{G},\bm{\lambda},n\right)
-\frac{1}{2}\norm{\bm{B}-\sqrt{\mu\over n}\bm{v}\bm{x}^\top}^2_F,
\end{equation}
where
\begin{equation}
\label{eq:H_cov}
\begin{aligned}
\mathcal{H}_{SBM}'\left(\bm{x},\bm{x^{*}},\bm{G},\bm{\lambda},n\right)
:=\sum_{i=1}^{m}\sum_{k<l}\Bigg[& G^{(i)}_{kl}\log\left(\frac{\widebar{p}^{(i)}_n+\Delta^{(i)}_nx_kx_l}{\widebar{p}^{(i)}_n+\Delta^{(i)}_nx^*_kx^*_l}\right)\\
& +\left(1-G^{(i)}_{kl}\right)\log\left(\frac{1-\widebar{p}^{(i)}_n-\Delta^{(i)}_nx_kx_l}{1-\widebar{p}^{(i)}_n-\Delta^{(i)}_nx^*_kx^*_l}\right)
\Bigg].
\end{aligned}
\end{equation}
Let us define
\begin{equation}
\label{eq:cov_phi_g}
\begin{aligned}
& \psi(\bm{x^{*}},\bm{B},\bm{G},\bm{\lambda},\mu,n,p)\\
& \quad = \log\Bigg\{\sum_{\bm{x} \in \{\pm 1\}^{n}}\int_{\mathbb{R}^p}\exp(\mathcal{H}_{SBM}(\bm{x},\bm{x^{*}},\bm{v},\bm{G},\bm{B},\bm{\lambda},\mu,n,p))
\exp\left(-\frac{\,\|\bm{v}\|^2}{2}\right)d\bm{v}\Bigg\}.
\end{aligned}
\end{equation}
Then we have the following lemma characterizing the mutual information between $\tx$ and $\{\bm{G},\bm{B}\}$.
\begin{lem}
\label{lem:inf_cov_sbm}
Let us consider $\bm{B}$ defined in~\eqref{eq:def_cov} and $\bm{G}=\{\bm{G}^{(i)}: i \in [m]\}$ defined in~\eqref{eq:def_SBM_cov}. Then we have
\begin{align*}
I(\bm{x^{*}};\bm{G},\bm{B})&= n\log2+
\mathbb{E}\log\Bigg(\int_{\mathbb{R}^p}\exp\left(-\frac{1}{2}\norm{\bm{B}-
\sqrt{\mu\over n}\bm{v}(\bm{x^{*}})^\top}^2_F\right)
\exp\left(-\frac{\|\bm{v}\|^2}{2}\right)d\bm{v}\Bigg)\\
% & \mathbb{E}\log\Bigg(\int_{\mathbb{R}^p}\exp\left(-\frac{1}{2}\|\bm{B}-\sqrt{(\mu/n)}\bm{u}(\bm{x^{*}})^\top\|^2_F\right)\exp(-\|\bm{u}\|^2/2)d\bm{u}\Bigg)\\
&\quad\quad-\mathbb{E}[\psi(\bm{x}^{*},\bm{B},\bm{G},\bm{\lambda},\mu,n,p)]
\end{align*}
where $\psi(\bm{x}^{*},\bm{B},\bm{G},\bm{\lambda},\mu,n,p)$ is defined in~\eqref{eq:cov_phi_g}.
\end{lem}
\begin{proof}
See Section \ref{proof_lem_3}.
\end{proof}
To connect $I(\bm{x^{*}};\bm{G},\bm{B})$ to $I(\bm{x^{*}};\bm{Y},\bm{B})$, we use Lindeberg's Interpolation Argument. For that purpose, let us define the auxiliary random matrices $\widetilde{\bm{G}}^{(i)}$ where
\begin{equation}
\label{eq:G_aux_def}
\widetilde{G}^{(i)}_{kl}:=
\frac{\Delta^{(i)}_n}{\widebar{p}^{(i)}_n(1-\widebar{p}^{(i)}_n)}
\left(G^{(i)}_{kl}-\widebar{p}^{(i)}_n-\Delta^{(i)}_nx^*_kx^*_l\right).
\end{equation}
By $\widetilde{\bm{G}}$ we refer to the collection of random matrices $\{\widetilde{\bm{G}}^{(1)},\ldots,\widetilde{\bm{G}}^{(m)}\}$. The mutual information between $\tx$ and $\{\bm{G},\bm{B}\}$ is related to $\tx$ and $\{\widetilde{\bm{G}},\bm{B}\}$ in the following way.
\begin{lem}
\label{lem:I_SBM_cov}
Let us consider $\tilde{\bm{G}}$ defined in~\eqref{eq:G_aux_def}. Then with \revsag{$n\widebar{p}^{(i)}_n (1-\widebar{p}^{(i)}_n )\rightarrow \infty$ for $i=1,\ldots,m$}, we have the following identity
\begin{align*}
I(\bm{x^{*}};\bm{G},\bm{B})  &= n\log2+\frac{n-1}{2}\sum_{i=1}^{n}\lambda^{(i)}
\\
& \qquad +
\mathbb{E}\log\Bigg(\int_{\mathbb{R}^p}\exp\left(-\frac{1}{2}\norm{\bm{B}-
\sqrt{\mu\over n}\bm{v}(\bm{x^{*}})^\top}^2_F\right)
\exp\left(-\frac{\|\bm{v}\|^2}{2}\right)d\bm{v}\Bigg)\\
& \qquad - \mathbb{E}[\phi(\bm{x}^{*},\bm{B},\bm{\widetilde{G}},\bm{\lambda},\mu,n,p)] + O\left(\sum_{i=1}^{m}\frac{n(\lambda^{(i)})^{3/2}}{\sqrt{n\widebar{p}^{(i)}_n
(1-\widebar{p}^{(i)}_n)}}\right).
\end{align*}
\end{lem}
\begin{proof}
See Section \ref{proof_lem_4}.
\end{proof}
\subsection{Gaussian Approximation}

Next, we use Lindeberg's interpolation to approximate $\mathbb{E}[\phi(\bm{x}^{*},\bm{B},\bm{\widetilde{G}},\bm{\lambda},\mu,n,p)]$ by $\mathbb{E}[\phi(\bm{x}^{*},\bm{B},$\newline$\bm{W},\bm{\lambda},\mu,n,p)]$. 
\begin{lem}
\label{lem:cov_connect_phi}
Suppose \revsag{$n\bar{p}^{(i)}_n(1-\bar{p}^{(i)}_n) \rightarrow \infty$ for $i=1,\ldots,m$}. Then we have 
\begin{align*}
\mathbb{E}[\phi(\bm{x}^{*},\bm{B},\widetilde{\bm{G}},\bm{\lambda},\mu,n,p)] &=\mathbb{E}[\phi(\bm{x}^{*},\bm{B},\bm{W},\bm{\lambda},\mu,n,p)]+O\left(\sum_{i=1}^{m}\frac{n (\lambda^{(i)})^{3/2}}{\sqrt{n\widebar{p}^{(i)}_n(1-\widebar{p}^{(i)}_n )}}\right).
\end{align*}
\end{lem}
\begin{proof}
See Section \ref{proof_lem_5}.
\end{proof}
Finally we get the following theorem showing the asymptotic equivalence of the per-vertex mutual information in the two models.
\begin{thm}
\label{thm:connect_cov_inf}
Let us consider $\bm{Y}=\{\bm{Y}^{(i)}: i \in [m]\}$ defined in~\eqref{eq:def_Gauss_Chann}, $\bm{B}$ defined in~\eqref{eq:def_cov}, and $\bm{G}=\{\bm{G}^{(i)}: i \in [m]\}$ defined in~\eqref{eq:def_SBM_cov}. If for all $i \in [m]$ we have \revsag{$n\bar{p}^{(i)}_n(1-\bar{p}^{(i)}_n) \rightarrow \infty$  for $i=1,\ldots,m$}, then we have
\[
\left|\frac{1}{n}I(\bm{x}^{*};\bm{Y},\bm{B})-\frac{1}{n}I(\bm{x}^{*};\bm{G},\bm{B})\right| ~\le~ O\left(\sum_{i=1}^{m}\frac{(\lambda^{(i)})^{3/2}}{\sqrt{n\bar{p}^{(i)}_n(1-\bar{p}^{(i)}_n)}}\right).
\]
\end{thm}
\begin{proof}
The proof easily follows using Lemma~\ref{lem:cov_gauss_sbm}, Lemma~\ref{lem:I_SBM_cov} and Lemma~\ref{lem:cov_connect_phi}.
\end{proof}
\begin{rem}
The above theorem shows that as $n \rightarrow \infty$ the asymptotic per-vertex mutual information about $\bm{x^{*}}$ obtained from $(\bm{Y},\bm{B})$ is same as that obtained from $(\bm{G},\bm{B})$. 
\end{rem}
An immediate corollary to the above theorem is as follows.
\begin{cor}
\label{cor:inf_combined}
Consider $\bm{T}(\lambda)$ defined by~\eqref{eq:T} and $\lambda=\sum_{i=1}^{m}\lambda^{(i)}$. If for all $i \in [m]$ we have \revsag{$n\bar{p}^{(i)}_n(1-\bar{p}^{(i)}_n) \rightarrow \infty$  for $i=1,\ldots,m$}, then we have the following inequality.
\[\left|\frac{1}{n}I(\bm{x}^{*};\bm{T}(\lambda),\bm{B})-\frac{1}{n}I(\bm{x}^{*};\bm{G},\bm{B})\right| ~\le~ O\left(\sum_{i=1}^{m}\frac{(\lambda^{(i)})^{3/2}}{\sqrt{n\bar{p}^{(i)}_n(1-\bar{p}^{(i)}_n)}}\right).\]
\end{cor}
\begin{proof}
This corollary immediately follows from Theorem~\ref{thm:connect_cov_inf} and Lemma~\ref{lem:join_all_matrices}.
\end{proof}

\section{An Asymptotic I-MMSE Relation}
\label{sec_5}
Let us begin by observing that the collection of SBM's $\bm G^{(1)},\ldots,\bm G^{(m)}$ can be represented as the collection of random variables $\{G^{(i)}_{kl}: 1 \le i \le m, 1 \le k < l \le n \}$. 
% We shall refer to this collection as $\mathcal{G}$, i.e.,
% \begin{equation}
% \label{eq:cal_g}
% \mathbf{\mathcal{G}}=\left\{A^{(i)}_{kl}: 1 \le i \le m, 1 \le k < l \le n\right\}.
% \end{equation}
Instead of considering $\{0,1\}$ valued random variables $G^{(i)}_{kl}$, we shall consider $\{-1,1\}$ valued random variables $L^{(i)}_{kl}=2G^{(i)}_{kl}-1$.
This collection will be called $\mathcal{L}$, that is, 
\begin{equation}
\label{eq:linear_transformed variables}
\mathcal{L}=\left\{2G^{(i)}_{kl}-1: 1 \le i \le m, 1 \le k < l \le n\right\}.
\end{equation}
Since the elements of $\mathcal{L}$ are linear transformations of the elements of $\{G^{(i)}_{kl}:i\in [m], 1\leq k<l \leq n \}$, 
% it is easy to observe that
we have
\begin{equation}
\label{eq:cond_entrop_l}
H(\bm{x^{*}}|\bm{G},\bm{B})=H(\bm{x^{*}}|\mathcal{L},\bm{B}),
\end{equation}
where $H(\bm{x^{*}}|\mathcal{L},\bm{B})$ is the conditional entropy of $\bm{x^{*}}$ given $(\mathcal{L},\bm{B})$. This implies that
\begin{equation}
	\label{eq:MI-equiv}
\frac{1}{n}I(\bm{x^{*}};\bm{G},\bm{B}) = \frac{1}{n}I(\bm{x^{*}};\mathcal{L},\bm{B}).	
\end{equation}
Then an asymptotic I-MMSE identity for the differentiation of $I(\tx;\mathcal{L},\bm{B})$ is given by the following lemma.
\begin{lem}
\label{lem:diff_cov_inf}
Let $\bm \lambda$, $\lambda$ be \revsag{as} defined in \eqref{eq:lambda} and $r^{(i)}$ \revsag{for $i=1,\ldots,m$} be \revsag{as} defined in \eqref{eq:lambdaratio}. 
If $n\widebar{p}^{(i)}_n (1-\widebar{p}^{(i)}_n ) \rightarrow \infty$, 
then there is a positive constant $C$ such that 
\[
\left|\frac{1}{n}\,\frac{dI(\bm{x^{*}};\mathcal{L},\bm{B})}{d\lambda}
-\frac{1}{4}\,\MMSE_n(\lambda,\mu)\right| 
\le C\left(\sum\limits_{i=1}^{m}\sqrt{\frac{ r^{(i)} \lambda }{n\widebar{p}^{(i)}_n (1-\widebar{p}^{(i)}_n)}}\right).
\]
\end{lem}
\begin{proof}
See Section \ref{proof_i_mmse}.
\end{proof}
Together with \eqref{eq:MI-equiv}, the above lemma implies
\begin{equation}
\label{eq:I-MMSE-Graph}
\left|\frac{1}{n}\,\frac{dI(\bm{x^{*}};\bm{G},\bm{B})}{d\lambda}
-\frac{1}{4}\,\MMSE_n(\lambda,\mu)\right| \to 0 \quad \mbox{as $n \to \infty$.}
\end{equation}

\section{Asymptotic MMSE in the Gaussian Model}
\label{mmse}
% last modified: 02/18/2021 ZM

In this section, we derive the asymptotic limit of the quantity $\GMMSE_n(\lambda,\mu)$ defined in \eqref{eq:Gmmse}. 
To this end, we 
% construct a sequence of estimates $\widehat{\bm{x}}^{t}$ using approximate message passing such
shall show that, for the AMP iterate $\widehat{\bm{x}}^t$ defined in \eqref{eq:amp_final_estim}, the mean square error in estimating $\bm{x^{*}}(\bm{x^{*}})^\top$ by $\widehat{\bm{x}}^{t}(\widehat{\bm{x}}^{t})^\top$ in the Gaussian observation model \eqref{eq:T} is asymptotically the same as the limit of $\GMMSE_n(\lambda,\mu)$ as $\varepsilon$ goes to zero and $n$, $t$ goes to infinity.
 % where $\GMMSE_n(\lambda,\mu)$ defined in \eqref{eq:Gmmse} with $\lambda$ as in~\eqref{eq:lambda}.
The matrix mean square error in estimating $\bm{x^{*}}(\bm{x^{*}})^\top$ by $\widehat{\bm{x}}^{t}(\widehat{\bm{x}}^{t})^\top$,referred to as $\MSEA_n(t;\lambda,\mu,\varepsilon)$,
is defined by
\begin{equation}
\label{eq:MMSE_AMP_ep}
% \MSEA_n
\MSEA_n(t;\lambda,\mu,\varepsilon) ~=~ \frac{1}{\,n^2}\,\mathbb{E}\left[\|\bm{x^{*}}(\bm{x^{*}})^\top-\widehat{\bm{x}}^{t}(\widehat{\bm{x}}^{t})^\top\|^2_F\right].
\end{equation}
We show that in \revsag{the} ``large $n$, large $t$, small $\varepsilon$" limit 
%\nb{Shall we change this globally? We always take $n\to\infty$ first.},
this sequence of estimators is asymptotically Bayes optimal in the sense that
 % as $t,n \rightarrow \infty$,
 $\MSEA_n(t;\lambda,\mu,\varepsilon)$ converges to 
the same limit as 
$\GMMSE_n(\lambda,\mu)$. 
Hence, from the properties of the AMP iterates that we shall derive in this section, we can characterize the precise limit of $\GMMSE_n(\lambda,\mu)$ 
(and hence of $\MMSE_n(\lambda,\mu)$) as $n \rightarrow \infty$. 

%\nb{add some words on $\varepsilon$?}

As a byproduct, we obtain an explicit formula of the asymptotic limit of the per-vertex mutual information in the Gaussian observation model.
% ``Gaussian+Side Information" model.
By Corollary \ref{cor:inf_combined}, it also gives the asymptotic limit of the per-vertex mutual information in the original model \eqref{eq:def_SBM_cov}--\eqref{eq:def_cov}.
 % where we observe $m$ networks and one data matrix.
% multi-layer Contextual Stochastic Block Model described by $(\bm G, \bm B)$.

\paragraph{State evolution of the AMP iterates}
% \nb{== start here ==}
Recall AMP iterates $\bm u^{t}, \bm x^{t}$ and $\bm v^t$ defined in \eqref{eq:AMP_shift_main_1} and \eqref{eq:AMP_shift_main_1_1}, 
and state evolution \eqref{eq:state_ev_rec}.
% For the update functions we consider $f_t,g_t$ as defined by \eqref{eq:def_f} and \eqref{eq:def_g}, and $\alpha_t,\tau_t,\mu_t,\sigma_t,\beta_t$ and $\vartheta_t$ defined by the state evolution recursions \eqref{eq:state_ev_rec}.
From \eqref{eq:state_ev_rec}, we obtain the following
\begin{align*}
% \[
\frac{\mu^2_{t+1}}{\sigma^2_{t+1}} & = \lambda\left(1-(1-\varepsilon)\smmse\left(\frac{\alpha^2_{t-1}}{\tau^2_{t-1}}+\frac{\mu^2_{t}}{\sigma^2_{t}}\right)\right),\\
% \]
% \[
\frac{\beta^2_{t}}{\vartheta^2_{t}} & = \mu\left(1-(1-\varepsilon)\smmse\left(\frac{\alpha^2_{t-1}}{\tau^2_{t-1}}+\frac{\mu^2_{t}}{\sigma^2_{t}}\right)\right),\\
% \]
% and
% \[
\frac{\alpha^2_t}{\tau^2_t} & = (1-\varepsilon)\frac{\mu}{c} \frac{\beta^2_t}{\beta^2_t+\vartheta^2_t}.
% \]
\end{align*}
Define $\theta_t:=\beta^2_t/\vartheta^2_t$ and $\gamma_t:=\mu^2_t/\sigma^2_t$. Then we have the following
\begin{align}
% \begin{equation}
\label{eq:gamma_t}
\gamma_{t+1} & = \lambda\left(1-(1-\varepsilon)\smmse\left(\gamma_t+(1-\varepsilon)\frac{\mu}{c}\frac{\theta_{t-1}}{1+\theta_{t-1}}\right)\right), \\
% \end{equation}
% and
% \begin{equation}
\label{eq:theta_t}
\theta_{t} & = \mu\left(1-(1-\varepsilon)\smmse\left(\gamma_t+(1-\varepsilon)\frac{\mu}{c}\frac{\theta_{t-1}}{1+\theta_{t-1}}\right)\right).
% \end{equation}
\end{align}
Further, define
\begin{equation}
\label{eq:init_z_t}
z_t = \frac{\gamma_{t+1}}{\lambda} = \frac{\theta_t}{\mu}.
\end{equation}
Then the state evolution recursion reduces to
\begin{equation}
\label{eq:z_t}
z_{t+1} = 1-(1-\varepsilon)\,\smmse\left(\lambda z_t+(1-\varepsilon)\frac{\mu^2}{c}\frac{z_t}{1+\mu z_t}\right).
\end{equation}
%If $\mu=0$, substituting $\theta_t$ defined by \eqref{eq:theta_t} in \eqref{eq:gamma_t} we get,
%\begin{equation}
%\label{eq:z_t}
%\gamma_{t+1} = \lambda\left(1-(1-\varepsilon)\smmse\left(\gamma_t\right)\right).
%\end{equation}
Since the function on the right side of \eqref{eq:z_t} 
%\nb{As a function of which argument? This needs clarification.} 
is concave, increasing monotonically and bounded as a function of $z_t$ (as we shall show later in the proof of Theorem \ref{thm:MSE_AMP_main} in Section \ref{sec:proof-mse-amp}), we have 
\[
z_t \rightarrow z_*(\lambda,\mu,\varepsilon), 
\quad \mbox{as} \quad
t \rightarrow \infty.
\] 
This implies that $z_*(\lambda,\mu,\varepsilon)$ satisfies
\begin{equation}
\label{eq:rec_final}
z_{*}(\lambda,\mu,\varepsilon) = 1-(1-\varepsilon)\smmse\left(\lambda z_*(\lambda,\mu,\varepsilon)+(1-\varepsilon)\frac{\mu^2}{c}\frac{z_*(\lambda,\mu,\varepsilon)}{1+\mu z_*(\lambda,\mu,\varepsilon)}\right).
\end{equation}
% \nb{need to mention that $z_*$ is the largest satisfying the foregoing equation.}

\paragraph{Limit of MMSE}
% With the AMP iterates defined and the asymptotics of its state evolutions investigated, we now the AMP based estimates \eqref{eq:amp_final_estim}.
As a first step, we have the following theorem that characterizes the asymptotics of $\MSEA_n(t;\lambda,\mu,\varepsilon)$. 
%\nb{do we still need $\lambda \neq 0$?}
%where
%\begin{equation}
%\label{eq:MMSE_AMP_ep}
%\MSEA_n(t;\lambda,\mu,\varepsilon) ~=~ \frac{1}{\,n^2}\,\mathbb{E}\left[\|\bm{x^{*}}(\bm{x^{*}})^\top-\widehat{\bm{x}}^{t}(\widehat{\bm{x}}^{t})^\top\|^2_F\right]
%\end{equation}
\begin{thm}
\label{thm:MSE_AMP_main}
Let $\mathsf{MSE^{AMP}_n}(t;\lambda,\mu,\varepsilon)$ be defined as in~\eqref{eq:MMSE_AMP_ep}.  
Then we have
\begin{equation}
\label{eq:MSE_AMP_final_1}
\lim\limits_{n \rightarrow \infty}\mathsf{MSE^{AMP}_n}(t;\lambda,\mu,\varepsilon) = 1-z^2_{t},
\end{equation}
where $z_t$ is defined by \eqref{eq:init_z_t}. 
Taking $t\to\infty$, we have
\begin{equation}
\label{eq:MSE_AMP_final_2}
\lim\limits_{t \rightarrow \infty}\lim\limits_{n \rightarrow \infty} \mathsf{MSE^{AMP}_n}(t;\lambda,\mu,\varepsilon) = 1-z^2_{*}(\lambda,\mu,\varepsilon),
\end{equation}
where $z_{*}(\lambda,\mu,\varepsilon)$ is the largest non-negative solution to \eqref{eq:rec_final}. 
As $\varepsilon \rightarrow 0$, we get
\begin{equation}
\lim\limits_{\varepsilon \rightarrow 0}\lim\limits_{t \rightarrow \infty}\lim\limits_{n \rightarrow \infty} \mathsf{MSE^{AMP}_n}(t;\lambda,\mu,\varepsilon) = 1-z^2_{*}(\lambda,\mu),
\end{equation}
where $z_{*}(\lambda,\mu)$ is the largest non-negative solution to \eqref{eq:zstar}. 
\end{thm}
\begin{proof}
See Section \ref{sec:proof-mse-amp}.
\end{proof}
Next, we have the following theorem characterizing asymptotic $\GMMSE_n(\lambda,\mu)$ 
and per-vertex mutual information.
\begin{thm}
\label{thm:MMSE_main}
Consider $\mathsf{GMMSE}_n(\lambda,\mu)$ defined in~\eqref{eq:Gmmse}. Then \revsag{for all $\lambda,\mu \ge 0$,}
\begin{equation}
\label{eq:final_mmse_amp}
\lim\limits_{n \rightarrow \infty} \mathsf{GMMSE}_n(\lambda,\mu) = 1-z^2_{*}(\lambda,\mu),
\end{equation}
where $z_{*}(\lambda,\mu)$ is the largest non-negative solution to \eqref{eq:zstar}.
Further,
\begin{equation}
\label{eq:lim_inf}
\lim\limits_{n \rightarrow \infty} \frac{1}{n}I(\tx;\bm{T}(\lambda),\bm{B}) = \xi(z_{*}(\lambda,\mu),\lambda,\mu),
\end{equation}
where $\sMI(\cdot)$ is defined in \eqref{eq:inf_scalar} and
\begin{equation}
\label{eq:def_xi}
\begin{aligned}
\xi(z,\lambda,\mu) & = \frac{\lambda z^2}{4} - \frac{\lambda z}{2} + \frac{\lambda}{4} + \frac{1}{2c}\log(1+\mu z)+\frac{1}{2c}\,\frac{(1+\mu)}{(1+\mu z)}\\
& ~~~~~~~~~+\sMI\left(\lambda z + \frac{\mu^2}{c}\frac{z}{1+\mu z}\right)-\frac{1}{2c}\log(1+\mu)-\frac{1}{2c}.	
\end{aligned}
\end{equation}
\end{thm}
\begin{proof}
See Section \ref{proof_thm_5_2}.
\end{proof}
Finally, as an immediate corollary, we obtain the following limit for per-vertex mutual information in the original observation model.
\begin{cor}
\label{cor:asym_mutual_inf}
Consider $\bm{G}$ defined by~\eqref{eq:def_SBM_cov}. 
Then we have the following
% the following identities
\[
\lim\limits_{n \rightarrow \infty}\frac{1}{n}I(\bm{x^{*}};\bm{G},\bm{B}) = 
\xi(z_*(\lambda,\mu),\lambda,\mu),
\]
where $z_*(\lambda,\mu)$ is the largest non-negative solution to \eqref{eq:zstar}, and $\xi(z,\lambda,\mu)$ is defined by \eqref{eq:def_xi}.
\end{cor}
\begin{proof}
Follows from Corollary \ref{cor:inf_combined} and Theorem \ref{thm:MMSE_main}.
\end{proof}
\section{Proof of Theorem~\ref{thm:main_thm_mmse}}
\label{Main_MMSE}
% modified 02/18/2021 by zm

With all the ingredients collected in Sections \ref{inf_sec}--\ref{mmse}, we now give a formal proof of Theorem \ref{thm:main_thm_mmse} according to the outline laid out in Section \ref{sec:outline}. 

Using (263) of \cite{AbbeMonYash}, we get
\[
\lim_{n \rightarrow \infty}\left[\frac{1}{n}I\left(\tx; \bm{T}(\lambda), \bm{B}\right) - \frac{1}{n}I\left(\tx(\tx)^\top; \bm{T}(\lambda), \bm{B}\right)\right] = 0.
\]
Now using the same arguments as those in Section \ref{I_MMSE},
we have
\begin{equation}
\begin{aligned}
\label{eq:immse}
& \lim_{n \rightarrow \infty}
\frac{1}{n}\left(I(\bm{x^{*}};\bm{T}(\lambda_1),\bm{B})-I(\bm{x^{*}};\bm{T}(\lambda_2),\bm{B})\right)\\
& =\lim_{n \rightarrow \infty}\frac{1}{n}\left(I(\tx(\tx)^\top; \bm{T}(\lambda_1),\bm{B})
-I(\tx(\tx)^\top;\bm{T}(\lambda_2),\bm{B})\right)\\
& =\lim_{n \rightarrow \infty}\int_{\lambda_1}^{\lambda_2}\frac{1}{4}\GMMSE_{n}(\theta,\mu)d\theta	
\end{aligned}
\end{equation}
where 
$\GMMSE_n(\theta,\mu)=\frac{1}{n^2}\mathbb{E}\|\bm{x^{*}}(\bm{x^{*}})^\top-\mathbb{E}\left[\bm{x^{*}}(\bm{x^{*}})^\top|\bm{T}(\theta),\bm{B}\right]\|^2_F$ 
and $\bm{T}(\theta)$ are defined in~\eqref{eq:T}.

% Next consider $\MMSE_n(\bm{\lambda},\mu)$ defined in~\eqref{eq:MMSE_cov_graph}.
Fix a set of $r^{(i)}$'s defined in \eqref{eq:lambdaratio}.
For any $\theta > 0$, we can write $\MMSE_n(\theta,\mu)$ for $\MMSE_n(\bm{\theta},\mu)$ where the $i$th element of $\bm{\theta}$ is $\theta r^{(i)}$.
Using Lemma~\ref{lem:diff_cov_inf} and \eqref{eq:degreelimit}, we get for all finite $\lambda_1$ and $\lambda_2$
%\nb{need more explanation here!}
\begin{equation}
\lim_{n \rightarrow \infty}\int_{\lambda_1}^{\lambda_2}\frac{1}{4}\MMSE_n(\theta,\mu)d\theta
=\lim_{n \rightarrow \infty}\frac{1}{n}\left(I(\bm{x^{*}}; \bm{G}(\lambda_2),\bm{B})
-I(\bm{x^{*}}; \bm{G}(\lambda_1),\bm{B})\right)
\end{equation}
where $I(\bm{x^{*}}; \bm{G}(\lambda),\bm{B})$ refers to the mutual information between $\bm{x^{*}}$ and $(\bm{G},\bm{B})$.
Then using Corollary~\ref{cor:inf_combined} and~\eqref{eq:immse} we get for all $\lambda \ge 0$ 
and $\mu\geq 0$
%\begin{equation}
%\lim_{n \rightarrow \infty}\frac{1}{n}\left(I(\mathbf{X},\mathbf{G}(\lambda_2))-I(\mathbf{X},\mathbf{G}(\lambda_1))\right)=\lim_{n \rightarrow \infty}\frac{1}{n}\left(I(\mathbf{X},Y^{*}(\lambda_1))-I(\mathbf{X},Y^{*}(\lambda_2))\right)
%\end{equation}
%This implies
\[
\lim_{n \rightarrow \infty}\GMMSE_n(\lambda,\mu) = \lim_{n \rightarrow \infty}\MMSE_n(\lambda,\mu).
\]
Now using Theorem~\ref{thm:MMSE_main} we get
\begin{equation}
\label{eq:MMSE_connect}
\lim_{n \rightarrow \infty}\MMSE_n(\lambda,\mu)=1-z^2_*(\lambda,\mu),
\end{equation}
where $z_*(\lambda,\mu)$ is the largest non-negative solution to~\eqref{eq:zstar}. Define
\[
G(z) = 1-\smmse\left(\lambda z + \frac{\mu^2}{c} \;\frac{z}{1+\mu z}\right),
\]
and
\[
S(z) = 1-\smmse(z).
\]
Then using Lemma 6.1 
%\nb{I do not find Remark 6.1 there. Only Lemma 6.1.}
of \cite{AbbeMonYash} we get that $G$ is increasing, concave, $G(0)=0$ and there is an unique positive solution of \eqref{eq:zstar} if and only if
\[
G^\prime(0)= \left(\lambda + \frac{\mu^2}{c}\right)S^\prime(0) = \lambda + \frac{\mu^2}{c} > 1.
\]
In this case, $z_*(\lambda,\mu) < 1$ by its definition in \eqref{eq:zstar}. Otherwise, if $\lambda + \mu^2/c \le 1$, then the only non-negative solution to \eqref{eq:zstar} is $0$. 

Since the foregoing arguments hold for any fixed $r^{(1)},\dots, r^{(m)}$, this implies if $\lambda + \mu^2/c \le 1$
\begin{equation}
\lim\limits_{n \rightarrow \infty}\MMSE_n(\bm\lambda,\mu) =1,
\end{equation}
and if $\lambda + \mu^2/c > 1$
\begin{equation}
\lim\limits_{n \rightarrow \infty}\MMSE_n(\bm\lambda,\mu) <1.
\end{equation}
This completes the proof.

%\nb{Shall we move the proof of Theorem \ref{thm:main_thm_mmse} here?}
\section{Orchestrated Approximate Message Passing}
\label{AMP}
% last modified 02/19/2021 zm

This section collects the key results on the SLLN type behavior of the orchestrated AMP iterates with multiple parallel orbits.
These results are used to derive the properties of the sequence of estimators $\widehat{\bm x}^t$ in Section \ref{mmse}, and they are potentially of independent interest.
Although we focus on the case of two orbits in this section, the arguments could be extended to more than two orbits.

To fully accommodate the $\varepsilon$-revelation approach we have taken in \eqref{eq:AMP_shift_main_1}--\eqref{eq:AMP_shift_main_1_1}, we need to introduce some additional technicalities for the function classes that we establish convergence results on.
The details are spelled out in Section \ref{sec:func-class}.
The SLLN-type behavior of the iterates in AMP without and with signal is established in Sections \ref{sec:amp-side} and \ref{sec:amp-side-signal}, respectively.

\subsection{Partially Pseudo-Lipschitz and Partially Lipschitz Functions}
\label{sec:func-class}

Traditionally, while analyzing the convergence of the AMP iterates, one considers pseudo-Lipschitz functions \cite{BM11journal,JM12,AbbeMonYash}. 
However, 
% this is a quite strong assumption and
many update functions which are intuitive may not belong to this function class.
% satisfy this requirement. 
For example, in our case, \revsag{while the sequence of update functions $f_t$ (in \eqref{eq:AMP_shift_main_1} and \eqref{eq:AMP_shift_main_1_1}) are pseudo-Lipschitz, the functions $g_t$ are not. 
Fortunately, the asymptotics of the AMP iterates that we have designed can be analyzed with a weaker 
requirement}
% assumption 
in the update functions. 

To motivate our definition, observe that $g_t:\mathbb{R}^2 \rightarrow \mathbb{R}$ in \eqref{eq:def_g} is given by
\begin{equation}
	\label{eq:def_g_explicit}
g_t(x,z) = \begin{cases}\frac{\beta_t}{\beta^2_t+\vartheta^2_t}x & \mbox{if $z=0$,}\\ z & \mbox{if $z \neq 0$.} \end{cases}
\end{equation} 
This function is discontinuous at $(x,0)$ for all $x \in \mathbb{R}\setminus\{0\}$. 
 Hence, it is not pseudo-Lipschitz. 
However, if we fix the last argument and view the function as only a function of the remaining arguments, then it becomes pseudo-Lipschitz.
%However, the pseudo-Lipschitz requirement is violated only at a lower dimensional subspace $\{(x,z):z=0\}$. 
%Moreover, if we consider the mapping $z \mapsto g_t(0,z)$, the resulting function is Lipschitz continuous in $z$. 
Such functions are sufficiently smooth for the AMP iterates to behave properly in the asymptotic regime. 
\revsag{In view of this} we define the following {\it partially pseudo-Lipschitz} functions.
%\nb{Shall we switch to ``partially pseudo-Lipschitz''? 20210318}
{ \begin{defn}
Let $\bm{a}=(a_1,\ldots,a_{k})^\top$, $\bm{b}=(b_1,\ldots,b_{k})^\top$, and 
$z \in \mathbb{R}$. 
A function $\varphi:\mathbb{R}^{k+1} \rightarrow \mathbb{R}$ is called 
\emph{partially pseudo-Lipschitz} 
if there is an absolute constant $C>0$ such that for all $\bm a, \bm b \in \mathbb{R}^k$ and $z \in \mathbb{R}$,
% it satisfies,
\begin{equation}
\label{eq:varphi}
|\varphi(\bm a, z)-\varphi(\bm b, z)| \le C\left(1+\sum\limits_{i=1}^{k}|a_i|+\sum\limits_{i=1}^{k}|b_i|+|z|\right)\|\bm a -\bm b\|.
\end{equation}
Further there exists $C_1>0$, {such that for all $z \in \mathbb{R}$,}
\begin{equation}
\label{eq:varphi_1}
|\varphi(\bm 0, z))| \le C_1\left(1+|z|^2\right).
\end{equation}
% for some absolute constant $C>0$ free of $x$
%and for all $z_1,z_2 \in \mathbb{R}$,
%\begin{equation}
%\label{eq:varphi_1}
%|\varphi(\bm 0, z_1)-\varphi(\bm 0,z_2)| \le C\left(1+|z_1|+|z_2|\right)|z_1-z_2|.
%\end{equation}
%\nb{this definition is problematic!}
\end{defn}}
In a partially pseudo-Lipschitz function,
the first $k$ variables are the main {\it variables}, and the last is called the {\it offset variable}. 

Similar to pseudo-Lipschitz functions, 
partially pseudo-Lipschitz functions also 
form a function class on which one has the desired SLLN type behavior.
In the same spirit, we define \emph{partially Lipschitz} functions as follows.
%\nb{Switch to ``partially Lipschitz''? 20210318} 
{ 
\begin{defn}
Consider $\bm{a}=(a_1,\ldots,a_{k})^\top$, $\bm{b}=(b_1,\ldots,b_{k})^\top$ and $z \in \mathbb{R}$. 
A function $f:\mathbb{R}^{k+1} \rightarrow \mathbb{R}$ is called 
\emph{partially Lipschitz} if there is an absolute constant $C > 0$ such that for all $\bm a, \bm b \in \mathbb{R}^k$ and $z \in \mathbb{R}$,
\begin{equation}
% \label{eq:varphi}
|f(\bm a, z)-f(\bm b, z)| \le C\|\bm a -\bm b\|.
\end{equation}
Further there exists $C_1>0$, {such that for all $z \in \mathbb{R}$,}
\begin{equation}
\label{eq:varphi_2}
% for some absolute constant $C>0$ free of $x$
%and for all $z_1,z_2 \in \mathbb{R}$, 
%% we have,
%\begin{equation}
%% \label{eq:varphi_1}
|f(\bm 0, z)| \le C_1(1+|z|).
\end{equation}
%\nb{this definition is problematic!}
\end{defn}}
\begin{rem} 
All pseudo-Lipschitz functions are partially pseudo-Lipschitz. 
This implies that all Lipschitz functions are partially pseudo-Lipschitz. 
Furthermore, if ${ f(x_1,\ldots,x_{k},z)}:\mathbb{R}^{k+1}\rightarrow \mathbb{R}$, is Lipschitz, then the functions $f^2$ and $x_if$ for $i \in [k]$ are partially pseudo-Lipschitz. 
For two Lipschitz functions $f,g:\mathbb{R}^{k+1}\rightarrow \mathbb{R}$, the function $fg$ is partially pseudo-Lipschitz. 
%\nb{This is a bit too vague. Do $f$ and $g$ have the same set of arguments? 20210401zm}
Finally, by Lemma \ref{lem:Lipschitz},
% \nb{the lemma does not really address the last argument}, 
the sequence of functions $f_t(x,y,z)$ defined by \eqref{eq:def_f} is Lipschitz. 
Hence, $f^2_t(x,y,z)$, $xf_t(x,y,z)$, $yf_t(x,y,z)$ and $f_t(x,y,z)f_s(x,y,z)$ are all partially pseudo-Lipschitz. 
The same is true for $\partial f_t / \partial x$ and $\partial f_t/\partial y$ for all $t$, as they are Lipschitz continuous.
\end{rem}
{ Next we observe the following properties of partially Lipschitz and partially pseudo-Lipschitz functions.
\begin{lem}
\label{lem:prop_pl}
Consider two partially Lipschitz functions $f,g:\mathbb{R}^{k+1} \rightarrow \mathbb{R}$. Then they satisfy the following properties.
\begin{enumerate}
\item The function $h(x_1,\ldots,x_k,z)=f(x_1,\ldots,x_k,z)g(x_1,\ldots,x_k,z)$ is partially pseudo-\newline Lipschitz.
\item Consider a random variable $X$ with finite expectation.
For any fixed $x_1,\ldots,x_{r-1},\newline x_{r+1},\ldots,x_k$, let \[H(x_1,\ldots,x_{r-1},x_{r+1},\ldots,x_k,z)=\mathbb{E}_{X}\left\{\phi(x_1,\ldots,x_{r-1},X,x_{r+1},\ldots,x_k,z)\right\},\] where $\phi$ is partially pseudo-Lipschitz. Then the function $H:\mathbb{R}^k\to \mathbb{R}$ is partially pseudo-Lipschitz.
\end{enumerate}
\end{lem}
\begin{rem}
Recall that $g_t(x,z)$ defined in \eqref{eq:def_g} satisfies \eqref{eq:def_g_explicit}.
% , then,
% \[
% g_t(x,z)=\begin{cases}\frac{\beta_t}{\beta^2_t+\vartheta^2_t}x & \mbox{if $z = 0$,}\\
% z & \mbox{if $z \neq 0$.}
% \end{cases}
% \]
It is straightforward to check that $g_t$'s are partially Lipschitz. 
As all partially Lipschitz functions are partially pseudo-Lipschitz,  $g_t$'s are partially pseudo-Lipschitz. Further, $\partial g_t/\partial x$'s are also partially Lipschitz for all $t$ and hence partially pseudo-Lipschitz. 
Furthermore, $g^2_t(x,z)$, $xg_t(x,z)$, and $g_t(x,z)g_s(x,z)$ are all partially pseudo-Lipschitz.
\end{rem}}
%Henceforth, we shall slightly abuse notation to refer a function {\it pseudo-Lipschitz with respect to the first $k$ arguments} as a {\it pseudo-Lipschitz} function where $k$ should be understood from the context. 
%Let us define, $\phi:(\mathbb{R}^{k+1})^n \rightarrow \mathbb{R}$ as
%\begin{equation}
%\label{eq:phi}
%\phi(\bm x_1, \ldots, \bm x_n)=\frac{1}{n}\sum\limits_{i=1}^{n}\varphi(x_{i,1},\ldots,x_{i,k+1}).
%\end{equation}
%Then we have the following lemma characterizing the {\it pseudo-Lipschitz} behavior of $\phi$.
%\begin{lem}\bm{\xi}
%\label{lem:lem_5_1}
%Consider $\varphi$ defined in ~\eqref{eq:varphi} and $\phi$ defined in ~\eqref{eq:phi}. Then we have $C_1 > 0$ such that for all 
%\begin{equation}
%	\label{eq:pseudo-Lip-vec}
%|\phi(\bm{x})-\phi(\bm{y})| \le C_1\,\left(1+\frac{\|\bm{x}\|}{\sqrt{n}}+\frac{\|\bm{y}\|}{\sqrt{n}}\right)\frac{\|\bm{x}-\bm{y}\|}{\sqrt{n}}.	
%\end{equation}
%\end{lem}

\subsection{Orchestrated AMP with Mean Zero Gaussian Sensing Matrices}
% \sout{Gaussian Side Information}}
\label{sec:amp-side}
Let $\bm{L}$ be a $p \times n$ random matrix where 
\begin{equation}
\label{eq:L-def}
L_{ij} \overset{iid}{\sim} N(0,1/p),
\end{equation}
and let $\bm{N}$ be a scaled GOE($n$) matrix where 
\begin{equation}
\label{eq:N-def}
\mbox{$N_{ii}\stackrel{iid}{\sim} N(0,2/n)$ and $N_{ij} = N_{ji}\stackrel{iid}{\sim} N(0,1/n)$ when $i\neq j$. }
\end{equation}
In addition, assume that $\bm{L}$ and $\bm{N}$ are mutually independent. We want to construct two orchestrated AMP orbits based on the matrices 
$\bm L$ and $\bm N$
% $\bm A$ 
% and $\bm B$ 
with information sharing between them in each iteration.

\paragraph{Construction of orchestrated AMP orbits}
% \sout{AMP iterates}} 
Consider a sequence of update functions $\sff_t:\mathbb{R}^4 \rightarrow \mathbb{R}$, where for all integers $t \ge 0$, 
\begin{equation}
	\label{eq:ft-req}
\begin{aligned}
&\mbox{${ \sff_t}$'s are partially Lipschitz and their partial derivatives
with respect to}\\
&\mbox{the first two variables are also partially Lipschitz.}	
\end{aligned}
\end{equation}
Let us consider another sequence of update functions ${\sfg_t}:\mathbb{R}^3 \rightarrow \mathbb{R}$, where for any integer $t \geq 0$, 
\begin{equation}
\label{eq:gt-req}
\begin{aligned}
&\mbox{${\sfg_t}$'s are partially Lipschitz and their partial derivatives with respect to}\\
&\mbox{the first argument are also partially Lipschitz.}
\end{aligned}
\end{equation}
In addition, let ${ \sff_{-1}}$ and ${ \sfg_{-1}}$ be zero functions. 
% \nb{I am afraid that we are switching the meanings of $f_t$ and $g_t$ too frequently in this section, and it could be confusing to the readers.
% How about keeping $f_t,g_t$ here, and using $\sff_t, \sfg_t$ [mathsf fonts] in AMPs with rank one perturbation?
% Depending on the amount of work, we can go the other way round.
% 202010405zm}

Starting with $\bm{h}^0=\bm{y}^0=0$, 
we consider the following two AMP orbits:
%\nb{== start here == 20210318}

\begin{equation}
\label{eq:orig_AMP}	
\begin{aligned}
\bm{b}^{t} &= \bm{L} {\sff_{t}}(\bm{h}^t,\bm{y}^t,\bm{\xi}_0,\bm{x}_0)
-p_{t}{\sfg_{t-1}}(\bm{b}^{t-1},\bm{\omega}_0,\bm{v}_0),\\
\bm{h}^{t+1} &= \bm{L}^\top {\sfg_t}(\bm{b}^t,\bm{\omega}_0,\bm{v}_0)-c_t{\sff_{t}}(\bm{h}^t,\bm{y}^t,\bm{\xi}_0,\bm{x}_0),\\
\end{aligned}
\end{equation}
and
\begin{equation}
\label{eq:orig_AMP_1}
\bm{y}^{t+1} = \bm{N} {\sff_{t}}(\bm{h}^t,\bm{y}^t,\bm{\xi}_0,\bm{x}_0)-d_{t}{\sff_{t-1}}(\bm{h}^{t-1},\bm{y}^{t-1},\bm{\xi}_0,\bm{x}_0),
\end{equation}

where $\bm{\xi}_0 = (\xi_{0,1},\ldots,\xi_{0,n})^\top$ and $\bm{x}_0 = (x_{0,1},\ldots,x_{0,n})^\top$ with 
$(\xi_{0,i},x_{0,i}) \stackrel{iid}{\sim} P_{\xi,x}$
which has a finite second moment.
% where both the marginals have finite { second} moments .
% for all $i \in [n]$. Further for all $i \in [n]$, we have,
%and
%\begin{align}
%\label{eq:assn_5}
%\mathbb{E}[|\xi_{0,i}||x_{0,i}] & \le D_1(1+|x_{0,i}|),\\
%% \end{equation}
%% and
%% \begin{equation}
%\label{eq:assn_6}
%\mathbb{E}[|\xi_{0,i}|^2|x_{0,i}] & \le D_2(1+|x_{0,i}|^2),
%\end{align}
%almost surely for some positive constants $D_1$ and $D_2$.
Similarly $\bm{\omega}_0 = (\omega_{0,1},\ldots,\omega_{0,p})^\top$ and $\bm{v}_0 = (v_{0,1},\ldots,v_{0,p})^\top$ with $(\omega_{0,j},v_{0,j}) \stackrel{iid}{\sim} P_{\omega,v}$
{which also has finite second moment}.
 % where both the marginals have finite { second} moments.
%% . Further for all $j \in [p]$, we have,
%\begin{align}
%\label{eq:assn_7}
%\mathbb{E}[|\omega_{0,j}||v_{0,j}] & \le L_1(1+|v_{0,j}|),\\
%% \end{equation}
%% and
%% \begin{equation}
%\label{eq:assn_8}
%\mathbb{E}[|\omega_{0,j}|^2|v_{0,j}] & \le L_2(1+|v_{0,j}|^2),
%\end{align}
%almost surely for some positive constants $L_1$ and $L_2$. 
We further assume that $(\bm{\xi}_0,\bm{x}_0)$, and $(\bm{\omega}_0,\bm{v}_0)$ are independent of $\bm{L}$ and $\bm{N}$. 
Moreover, in \eqref{eq:orig_AMP}
\begin{align*}
{\sfg_t}(\bm{b}^t,\bm{\omega}_0,\bm{v}_0)&=({\sfg_t}(b^t_1,\omega_{0,1},v_{0,1}),\ldots,{\sfg_t}(b^t_p,\omega_{0,p},v_{0,p}))^\top,\\
{\sff_{t}}(\bm{h}^t,\bm{y}^t,\bm{\xi}_0,\bm{x}_0)&=({\sff_{t}}(h^t_1,y^t_1,\xi_{0,1},x_{0,1}),\ldots,{\sff_{t}}(h^t_n,y^t_n,\xi_{0,n},x_{0,n}))^\top,
\end{align*}
and
\begin{equation}
	\label{eq:cpd-generic}	
\begin{aligned}
c_t & = \frac{1}{p}\sum\limits_{i=1}^{p} \frac{\partial {\sfg_t}}{\partial b}(b^t_i,\omega_{0,i},v_{0,i}),\\
% \quad
% \\
p_t &= \frac{c}{n}\sum\limits_{i=1}^{n} \frac{\partial {\sff_{t}}}{\partial h}(h^t_i,y^t_i,\xi_{0,i},x_{0,i})
,\\
% \\
% \quad
d_t & = \frac{1}{n}\sum\limits_{i=1}^{n} \frac{\partial {\sff_t}}{\partial y}(h^t_i,y^t_i,\xi_{0,i},x_{0,i}),
\end{aligned}
\end{equation}
where $c = \lim\limits_{n\to\infty} n/p$.
Note that in construction of the above (partially) pseudo-Lipschitz functions, elements of $\bm{x}_0$ and $\bm{v}_0$ are offset variables, while those of $\bm{\xi}_0$ and $\bm{\omega}_0$ belong to the main variables. 
Finally, denote
\begin{equation}
\label{eq:func_AMP}
\bm{m}^t = {\sfg_t}(\bm{b}^t,\bm{\omega}_0,\bm{v}_0)
% ,
\hspace{0.1in} \mbox{and} \hspace{0.1in} \bm{q}^t = {\sff_{t}}(\bm{h}^t,\bm{y}^t,\bm{\xi}_0,\bm{x}_0).
\end{equation}

\begin{rem}
Compared to \eqref{eq:AMP_shift_main_1} and \eqref{eq:AMP_shift_main_1_1}, for AMP iterations without signal \eqref{eq:orig_AMP}, we have increased the number of arguments in both update function sequences by one to accommodate later analysis. 
Therefore, we change the notation to $\sff_t$ and $\sfg_t$ to alert the readers that the number of arguments has increased.
Their connection with the update functions $\{f_t,g_t:t\geq 0\}$ used when signal is present will be made explicit in Remark \ref{rem:signal}.
\end{rem}

\begin{rem}
The use of the same update function while updating $\bm h^t$ and $\bm y^t$ is not necessary.
% for the method to run. 
% However, in the current problem under consideration, 
We have considered this setup because it helps in the analysis of the estimate $\widehat{\bm x}^t$.
\end{rem}

%\begin{rem}
%{
%Although the AMP iterates \eqref{eq:orig_AMP} can be defined for any choice of $\{f_t, g_t:t \geq 0\}$, 
%we are most interested in cases where
%for all $t \ge 0$, $f_t$ is Lipschitz in its first two arguments and its partial derivatives with respect to the first two variables are also Lipschitz with respect to the first two arguments. 
%For $\{g_t:t\geq 0\}$, we shall assume that $g_t$ is Lipschitz with respect to the first argument and that its partial derivative with respect to the first argument is Lipschitz with respect to the first argument. }
%\end{rem}

\paragraph{State evolution}
\revsag{For notational simplicity}, we define for any vector $\bm{u},\bm{v} \in \mathbb{R}^m$, 
\[
\langle \bm{u}, \bm{v} \rangle_m = \frac{1}{m}\sum\limits_{i=1}^{m}u_iv_i
% ,
% \]
\quad \mbox{and} \quad
% \[
\langle \bm{u} \rangle_m = \frac{1}{m}\sum\limits_{i=1}^{m}u_i.
\]
The asymptotics of the foregoing AMP can be analyzed by its state evolution described below. 
Let $\tau^2_{-1}=\sigma^2_0=0$, $\vartheta^2_0 = c\, \lim\limits_{n \rightarrow \infty}\langle\bm{q}^0,\bm{q}^0\rangle_n$, and $\sigma^2_1 = \lim\limits_{n \rightarrow \infty}\langle\bm{q}^0,\bm{q}^0\rangle_n$. 
For all integer $t\geq 1$, we define recursively
\begin{equation}
\label{eq:var}
\begin{aligned}
\sigma^2_t &= \mathbb{E}\{{\sff_{t-1}}(\tau_{t-2}Z_1,\sigma_{t-1}Z_2,\widetilde{\Xi}_0,\widetilde{X}_0)^2\},\\
\tau^2_{t-1} & = \mathbb{E}\{{\sfg_{t-1}}(\vartheta_{t-1}Z_3,\widetilde{\Omega}_0,\widetilde{V}_0)^2\}, \\
\vartheta^2_t &= c\,\mathbb{E}\{{\sff_{t}}(\tau_{t-1}Z_1,\sigma_{t}Z_2,\widetilde{\Xi}_0,\widetilde{X}_0)^2\}.
\end{aligned}
\end{equation}
Here, $Z_1,Z_2,Z_3 \overset{iid}{\sim} N(0,1)$, $(\widetilde{\Xi}_0,\widetilde{X}_0) \sim P_{\xi,x}$ and $(\widetilde{\Omega}_0,\widetilde{V}_0) \sim P_{\omega,v}$, 
and they are mutually independent.
As before, $c = \lim_{n\to\infty} n/p$.
%\nb{The use of symbols $X_0$ and ${ V_0}$ is EXTREMELY confusing here given \eqref{eq:def_f} and \eqref{eq:def_g}! I would suggest changing them to $\widetilde{X}$ and $\widetilde{W}$. This also requires us to change notation in Section E.1--E.3!!}

% Let us consider pseudo-Lipschitz functions $\phi_h: \mathbb{R}^{3} \rightarrow \mathbb{R}$ and $\psi_b: \mathbb{R}^{2} \rightarrow \mathbb{R}$.
With the foregoing definitions,
the following theorem characterizes the SLLN type behavior of the ``large $n$'' averages of partially  pseudo-Lipschitz functions applied on AMP iterates.
\begin{thm}
\label{thm:thm_6_1}
Consider $\bm{L}$ and $\bm{N}$ defined in \eqref{eq:L-def} and \eqref{eq:N-def} 
that are mutually independent, and the AMP iterates \eqref{eq:orig_AMP} 
{ satisfying \eqref{eq:ft-req} and \eqref{eq:gt-req}}.
Let $\bm{\xi}_0 = (\xi_{0,1},\ldots,\xi_{0,n})^\top$ and $\bm{x}_0 = (x_{0,1},\ldots,x_{0,n})^\top$ 
with $(\xi_{0,i},x_{0,i}) \stackrel{iid}{\sim} P_{\xi,x}$
 % and $x_{0,i} \stackrel{iid}{\sim} P_x$ for some distributions $P_\xi$ and $P_x$ respectively,
which has finite second moment. 
Similarly, $\bm{\omega}_0 = (\omega_{0,1},\ldots,\omega_{0,p})^\top$ and $\bm{v}_0 = (v_{0,1},\ldots,v_{0,p})^\top$ 
with $(\omega_{0,j},v_{0,j}) \stackrel{iid}{\sim} P_{\omega,v}$
 % and $v_{0,j} \stackrel{iid}{\sim} P_v$ for some distribution $P_\omega$ and $P_v$ respectively,
which also has finite second moment.
% \label{eq:assn_7}
% \mathbb{E}[|\omega_{0,j}||v_{0,j}] & \le L_1(1+|v_{0,j}|),\\
% % \end{equation}
% % and
% % \begin{equation}
% \label{eq:assn_8}
% \mathbb{E}[|\omega_{0,j}|^2|v_{0,j}] & \le L_2(1+|v_{0,j}|^2),
% \end{align}
In addition, suppose $(\bm{\xi}_0, \bm{x}_0)$ and $(\bm{\omega}_0, \bm{v}_0)$ are independent of both $\bm{L}$ and $\bm{N}$.
% \nb{more conditions on the independence structures! 20210326zm}
Furthermore, 
% consider $\bm{m}^t,\bm{q}^t$ defined via \eqref{eq:func_AMP}, and
let $\tau^2_t, \vartheta^2_t, \sigma^2_t$ be defined by the recursions \eqref{eq:var} with initializations $\bm{y}^0=0,\bm{h}^0=0$, $\vartheta^2_0= c \lim_{p \rightarrow \infty}\langle \bm{q}^0, \bm{q}^0 \rangle_n$, and $\sigma^2_1= \lim_{n \rightarrow \infty}\langle \bm{q}^0, \bm{q}^0 \rangle_n$.
% \nb{Spell out the explicit conditions of the theorem here. 20210328 zm}
For any partially  pseudo-Lipschitz functions $\phi_h: \mathbb{R}^{4} \rightarrow \mathbb{R}$ and $\psi_b: \mathbb{R}^{3} \rightarrow \mathbb{R}$ in the sense of \eqref{eq:varphi}, we have
\begin{equation}
\frac{1}{n}\sum\limits_{i=1}^{n}\phi_h(h^{t+1}_i,y^{t+1}_i,\xi_{0,i},x_{0,i}) \overset{a.s.}{\longrightarrow} \mathbb{E}\left\{\phi_h(\tau_{t}Z_1,\sigma_{t+1}Z_2,\widetilde{\Xi}_0,\widetilde{X}_0)\right\},
\end{equation}
and
\begin{equation}
\frac{1}{p}\sum\limits_{i=1}^{p}\psi_b(b^{t}_i,\omega_{0,i},v_{0,i}) \overset{a.s.}{\longrightarrow} \mathbb{E}\left\{\psi_b(\vartheta_{t}Z_3,\widetilde{\Omega}_0,\widetilde{V}_0)\right\}.
\end{equation}
Here $Z_1, Z_2, Z_3\stackrel{iid}{\sim} N(0,1)$, $(\widetilde{\Xi}_0,\widetilde{X}_0) \sim P_{\xi,x}$, $(\widetilde{\Omega}_0,\widetilde{V}_0) \sim P_{\omega,v}$,
and they are mutually independent.  
%and 
%they are mutually independent \nb{Not true. $\widetilde{X}_0$ and $\widetilde{\Xi}_0$ are not independent. Nor are $\widetilde{\Omega}_0$ and $\widetilde{V}_0$!}.
\end{thm}

\begin{rem}
The presence of 
% \sout{side information and the crucial choice} 
two AMP orbits 
% characterized by $\bm A$ and $\bm B$ 
% \nb{what does it mean by ``characterized by $\bm A$ and $\bm B$''?}
and the use of 
orchestrated
iterates $\bm y^{t}$ and $\bm h^{t}$ in each ${ \sff_t}$ prevents us from directly using existing AMP convergence results in \cite{BM11journal}, \cite{JM12} or \cite{Berthier}.
To resolve this issue, we shall modify the proof of Lemma 1 in \cite{BM11journal} by  using the conditioning technique in \cite{Bol12} directly and prove an analogous lemma in Subsection \ref{AMP_1} (Lemma \ref{lem:AMP_lem_tech}) suitable for \eqref{eq:orig_AMP}.
The lemma will then be used to prove Theorem \ref{thm:thm_6_1}.
% Note that, to prove this theorem it is not enough to use the techniques of \cite{BM11journal}. This is because the use of side information in the construction of update functions. In the construction of the updates $\bm y^t$ and $\bm h^t$ we use the same iterates $\bm x^{t-1}$ and $\bm h^{t-1}$. This prevents the use of traditional techniques as described in \cite{BM11journal}.
\end{rem}

\subsection{Orchestrated AMP with Rank-One Deformed Sensing Matrices}
\label{sec:amp-side-signal}

We now turn back to the AMP iterates $\bm{u}^t,\bm{v}^t$ and $\bm{x}^t$ defined by \eqref{eq:AMP_shift_main_1}  and \eqref{eq:AMP_shift_main_1_1} with some generic update functions $f_t$ and $g_t$.
 % with any update functions $f_t$ and $g_t$ satisfying appropriate smoothness conditions.
With a slight abuse of notation, we 
define $\mu_0=\sigma_0=\alpha_{-1}=\tau_{-1}=0$ and 
\[
\beta_0=\sqrt{\frac{\mu}{c}}\,\mathbb{E}\left\{X_0f_0(0,0,X_0(\varepsilon))\right\}, \quad \vartheta^2_0 = c\,\lim\limits_{n \rightarrow \infty}\langle f_{0}(\bm{u}^0,\bm{x}^0,\bm{x}_0),f_{0}(\bm{u}^0,\bm{x}^0,\bm{x}_0)\rangle_n.
\] 
Then we define for all $t \ge 1$
\begin{equation}
\label{eq:state-ev-generic}
\begin{aligned}
\mu_t &= \sqrt{\lambda}\;\mathbb{E}\left\{X_0f_{t-1}(\tau_{t-2}Z_{1}+\alpha_{t-2}X_0,\sigma_{t-1}Z_{2}+\mu_{t-1}X_0,X_0(\varepsilon))\right\},\\
\alpha_{t-1} &= \sqrt{\frac{\mu}{c}}\;\mathbb{E}\left\{V_0 g_{t-1}(\vartheta_{t-1}Z_{3}+\beta_{t-1}V_0,V_0(\varepsilon))\right\},\\
\beta_t &= c\,\sqrt{\frac{\mu}{c}}\;\mathbb{E}\left\{X_0f_{t}(\tau_{t-1}Z_{1}+\alpha_{t-1}X_0,\sigma_{t}Z_{2}+\mu_{t}X_0,X_0(\varepsilon))\right\},\\
\sigma^2_t &= \mathbb{E}\left\{\left[f_{t-1}(\tau_{t-2}Z_{1}+\alpha_{t-2}X_0,\sigma_{t-1}Z_{2}+\mu_{t-1}X_0,X_0(\varepsilon))\right]^2\right\},\\
\tau^2_{t-1} &= \mathbb{E}\left\{\left[g_{t-1}(\vartheta_{t-1}Z_{3}+\beta_{t-1}V_0,V_0(\varepsilon))\right]^2\right\},\\
\vartheta^2_{t} &=c\,\mathbb{E}\left\{\left[f_{t}(\tau_{t-1}Z_{1}+\alpha_{t-1}X_0,\sigma_{t}Z_{2}+\mu_{t}X_0,X_0(\varepsilon))\right]^2\right\},
\end{aligned}
\end{equation}
where $X_0, X_0(\varepsilon), V_0, V_0(\varepsilon), Z_1, Z_2$ and $Z_3$ satisfy \eqref{eq:limit-quant}.
% %
% $X_0 \sim B(\varepsilon)X^*_0$, where $B(\varepsilon) \sim \mathrm{Bern}(\varepsilon)$, $X^*_0 \sim \mathrm{Rademacher}$  and ${ V_0} \sim \widetilde{B}(\varepsilon)U^*_0$ where $\widetilde{B}(\varepsilon) \sim \mathrm{Bern}(\varepsilon)$, $U^*_0 \sim N(0,1)$.

\begin{rem}
	\label{rem:signal}
Here we slightly \revsag{abuse notation} in the sense that we use $\alpha_t, \beta_t,\mu_t, \sigma_t^2,\tau_t^2$ and $\vartheta_t^2$ for state evolution with generic $f_t$ and $g_t$, whereas they were originally defined only for the specific $f_t$ and $g_t$ in \eqref{eq:def_f}--\eqref{eq:def_g}. 
Effectively, we could think of these quantities as functions of $\{f_t, g_t: t\geq 0\}$. In this way, the notation could be unified.
{Furthermore, the notation $\sigma_t^2$, $\tau^2_{t-1}$ and $\vartheta_t^2$ are in accordance with that of \eqref{eq:var} by identifying $(X_0, X_0(\varepsilon))$ with $(\widetilde{\Xi}_0, \widetilde{X}_0)$, $(V_0, V_0(\varepsilon))$ with $(\widetilde{\Omega}_0, \widetilde{V}_0)$, $\sff_t(x_1,x_2,y,z)$ with $f_t(x_1 + \alpha_{t-1}y,x_2 + \mu_t y, z)$
and $\sfg_{t-1}(x, y, z)$ with $g_{t-1}(x + \beta_{t-1}y, z)$.}
\end{rem}

% Recall pseudo-Lipschitz functions defined in \eqref{eq:varphi}. Then for such pseudo-Lipschitz functions, we have the following theorem.{ V_0}
The following theorem establishes the SLLN type behavior for AMP iterates defined by \eqref{eq:AMP_shift_main_1} and \eqref{eq:AMP_shift_main_1_1} with generic $f_t$ and $g_t$ satisfying certain smoothness conditions.
\begin{thm}
\label{thm:slln_shifted}
Consider partially pseudo-Lipschitz functions $\phi:\mathbb{R}^3 \rightarrow \mathbb{R}$ and $\psi:\mathbb{R}^2 \rightarrow \mathbb{R}$ in the sense of \eqref{eq:varphi}.
 % and the AMP iterates defined by \eqref{eq:AMP_shift}, where $\bm{u}^0=\bm{y}^0=0$.
Suppose that, in \eqref{eq:AMP_shift_main_1} and \eqref{eq:AMP_shift_main_1_1}, the update functions $f_t$ and its partial derivatives with respect to the first two variables are partially Lipschitz for all $t \ge 0$. 
Further for all $t \ge 0$, $g_t$ and its partial derivative with respect to the first argument  are also partially  Lipschitz. In addition, let $f_{-1}$ and $g_{-1}$ be zero functions. 
Then for all $t \in \mathbb{N}$, 
{ and for $\mu_t,\alpha_{t-1},
\beta_t, \sigma_t^2, \tau_{t-1}^2$ and $\vartheta_t^2$ defined in \eqref{eq:state-ev-generic},}
we have the following identities:
\begin{equation}
\label{eq:slln_1}
\lim\limits_{n \rightarrow \infty}\frac{1}{n}\sum\limits_{i=1}^{n}\phi(u^t_i, x^t_i, x_{0,i}) \overset{a.s.}{=} \mathbb{E}\left\{\phi\left(\alpha_{t-1}X_0+\tau_{t-1}Z_1, \mu_{t}X_0+\sigma_{t}Z_2, X_0(\varepsilon)\right)\right\},
\end{equation}
and
\begin{equation}
\label{eq:slln_2}
\lim\limits_{p \rightarrow \infty}\frac{1}{p}\sum\limits_{i=1}^{p}\psi(v^t_i, v_{0,i}) \overset{a.s.}{=} \mathbb{E}\left\{\psi\left(\beta_{t}V_0+\vartheta_{t}Z_3, V_0(\varepsilon)\right)\right\}.
\end{equation}
Here $X_0, X_0(\varepsilon), V_0, V_0(\varepsilon), Z_1, Z_2$ and $Z_3$ satisfy \eqref{eq:limit-quant}.
\end{thm}
\section{Numerical Experiments}
\label{sec:numerical}

\revsag{The AMP algorithm defined by the recursions \eqref{eq:AMP_shift_main_1} and \eqref{eq:AMP_shift_main_1_1} is asymptotically Bayes optimal for estimating $\bm x^*(\bm x^*)^\top$ in the Gaussian model. }
However, its dependence on the partial revelation of the truth $\bm x^*$ and $\bm v^*$ makes it impractical. 
% One can eliminate the dependence on $\bm x^*$ and $\bm v^*$, if one uses a non-zero initialization. 
% However, using any initialization may not give the desired state evolution limits. One popular initialization for AMP with a single sensing matrix is the spectral initialization. 
In this section, we investigate the empirical performance of a practically implementable variant of \eqref{eq:AMP_shift_main_1}-- \eqref{eq:AMP_shift_main_1_1}: 
we initialize with a spectral estimator and force $\varepsilon = 0$ in the AMP iterates \eqref{eq:AMP_shift_main_1}-- \eqref{eq:AMP_shift_main_1_1}.

To this end, we propose to initialize both $\bm{x}^0$ and $\bm{u}^0$ with {$\sqrt{n}\,\bar{\bm{e}}$, where $\bar{\bm{e}}$ is} the leading eigenvector of $\bm{T} + a_0 \bm{B}\bm{B}^\top$ for some constant $a_0$ defined below. 
% and $\bm{u}^0$ with \nb{blah, need to specify}.
Throughout this section, we define $a_0$ as the unique solution to the following equation: 
\begin{equation}
\label{eq:initialization}
\frac{\mu}{c\lambda}=\frac{-\lambda+(ca^2 + \mu a^2)+\sqrt{(\lambda+ca^2 + \mu a^2)^2 - 4\lambda ca^2 }}{2\mu}. 
\end{equation}
It can be shown that if $\lambda+\mu^2/c>1$, then the leading eigenvector of $\bm{T} + a_0 \bm{B}\bm{B}^\top$ is asymptotically correlated with $\bm x^*$, and they are asymptotically orthogonal if $\lambda+\mu^2/c\le 1$.
Rigorous proofs of the properties of this initializer and related issues are beyond the scope of the present paper, and they are being investigated in \cite{fan_yang_calculation}. 

In the rest of this section, we first conduct a simulation study of the above algorithm under the Gaussian observation model \eqref{eq:T} and \eqref{eq:def_cov}.
Next, we study its performance under the original multilayer network plus covariate model with one and two layers.
In both cases, the empirical performance of this practical variant agrees well with the theoretical predictions in Theorems \ref{thm:MMSE_main} and \ref{thm:main_thm_mmse}, respectively.

\subsection{The Gaussian Observation Model}
\label{sec:gauss-simu}
{We take $n=1500$ and $p=900$ and} consider two settings: ``fixed $\mu$ varying $\lambda$'' and ``fixed $\lambda$ varying $\mu$''. 
 
 In the first setting, \revsag{we fix $\mu \in \{0.5,0.7,0.9\}$}, respectively. 
 At each fixed value $\mu$, we vary $\lambda$ across \revsag{$25$ equally-spaced values in the interval $[0.5, 4.5]$. }
 For each combination $(\lambda,\mu)$, we generate $25$ i.i.d.~copies of $(\bm T, \bm B)$ pairs. For each $(\bm T, \bm B)$ pair, we run the iterates in \eqref{eq:AMP_shift_main_1} and \eqref{eq:AMP_shift_main_1_1} with $\varepsilon = 0$ for $100$ iterations after {initializing $\mu_0,\sigma_0, \alpha_{-1}$ and $\tau_{-1}$ randomly in the interval $[4, 10]$ and using the spectral initialization for $\bm u^0$ and $\bm x^0$.} 
 We construct the AMP estimate $\widehat{\bm{x}}^{100}$ as in \eqref{eq:amp_final_estim}. 
%  We compute the empirical MMSEs of the estimate $\widehat{\bm{x}}^{100}$ for $25$ $(\bm T, \bm B)$ pairs. 
{The upper panel of} Figure \ref{fig:fig_amp_lambda} reports the average and spread of the empirical MMSEs on 25 replications at each $\lambda$ for all three fixed values of $\mu$ and compares the average with the theoretical prediction \eqref{eq:MMSE-limit}.
These plots show that the MMSEs of the spectral initialized AMP iterates agree well with the theoretical limits across all $(\lambda,\mu)$ value pairs.

In the second setting, \revsag{we fix $\lambda \in \{0.3,0.6,0.9\}$}, respectively. 
At each fixed value $\lambda$, we vary $\mu$ across $25$ equally-spaced values in the interval $[0.5, 4.5]$.
For each combination $(\lambda,\mu)$, the other simulation details are the same as in the first setting.
The lower panel of Figure \ref{fig:fig_amp_lambda} reports the average and spread of the empirical MMSEs over 25 replications at each $\mu$ for the three fixed values of $\lambda$ and compares the average with the theoretical prediction \eqref{eq:MMSE-limit}.
As in the previous setting, the empirical MMSEs agree well with theoretical predictions.
 
%  We plot this average in a series of six plots (Figures \ref{fig:fig_amp_lambda} and \ref{fig:fig_amp_mu}), first fixing $\mu$ and plotting the empirical MMSE versus $\lambda$ and next fixing $\lambda$ and plotting empirical MMSE versus $\mu$. We also plot the predicted asymptotic MMSE for the $(\lambda,\mu)$ combinations, computed using the formula \eqref{eq:MMSE-limit} for comparison. From these two figures it is evident that the MMSE of the AMP based estimator using spectral initialization converges to the theoretical MMSE pretty fast. Hence, for Gaussian matrices, we recommend the spectral initialization along with our AMP orbits to recover the leading eigenvector of $\bm T$.

\begin{figure}[tb]
\begin{subfigure}{.33\textwidth}
  \centering
  \includegraphics[width=\linewidth]{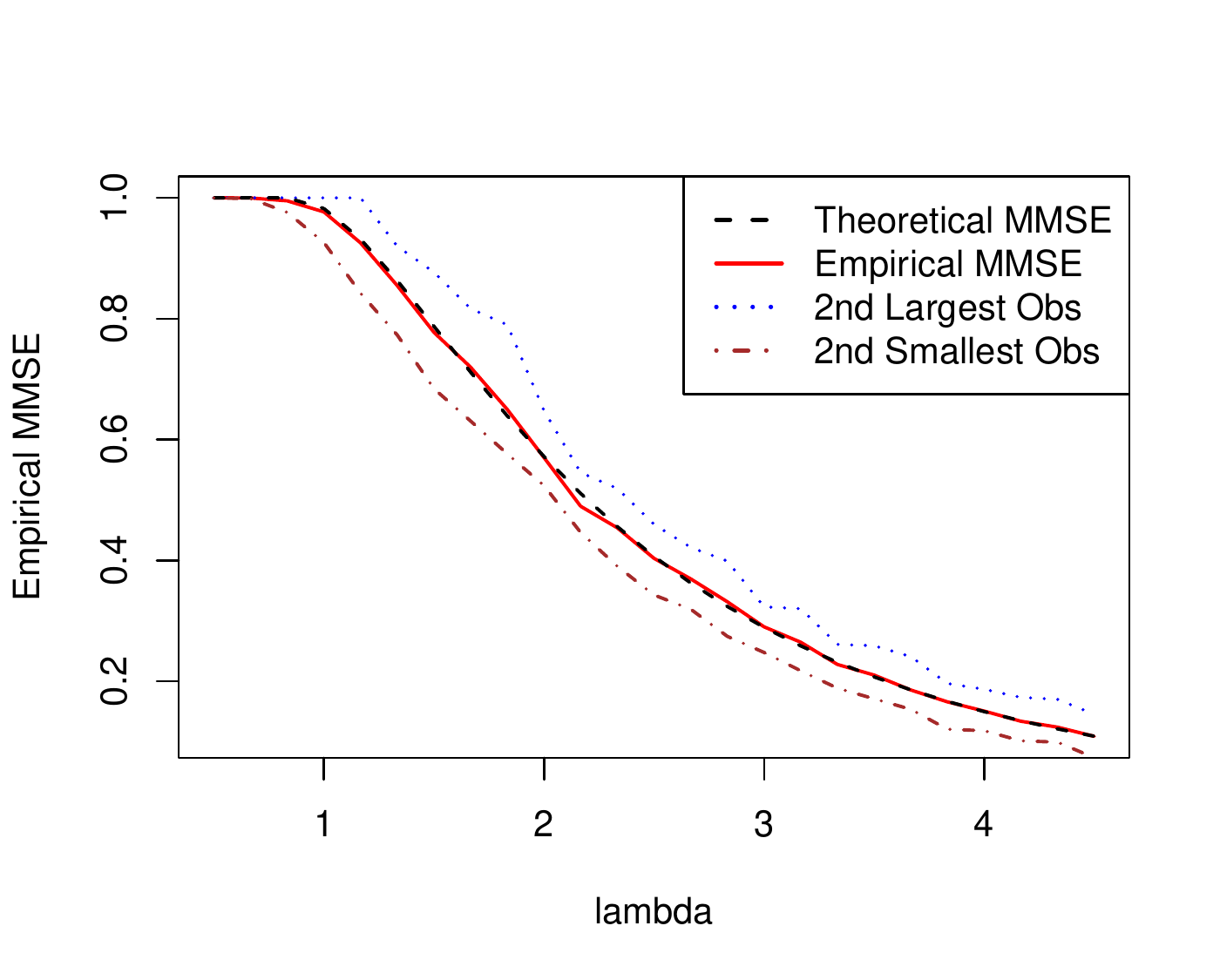}
  \caption{$\mu=0.5$}
  \label{fig:sfig11}
\end{subfigure}%
\begin{subfigure}{.33\textwidth}
  \centering
  \includegraphics[width=\linewidth]{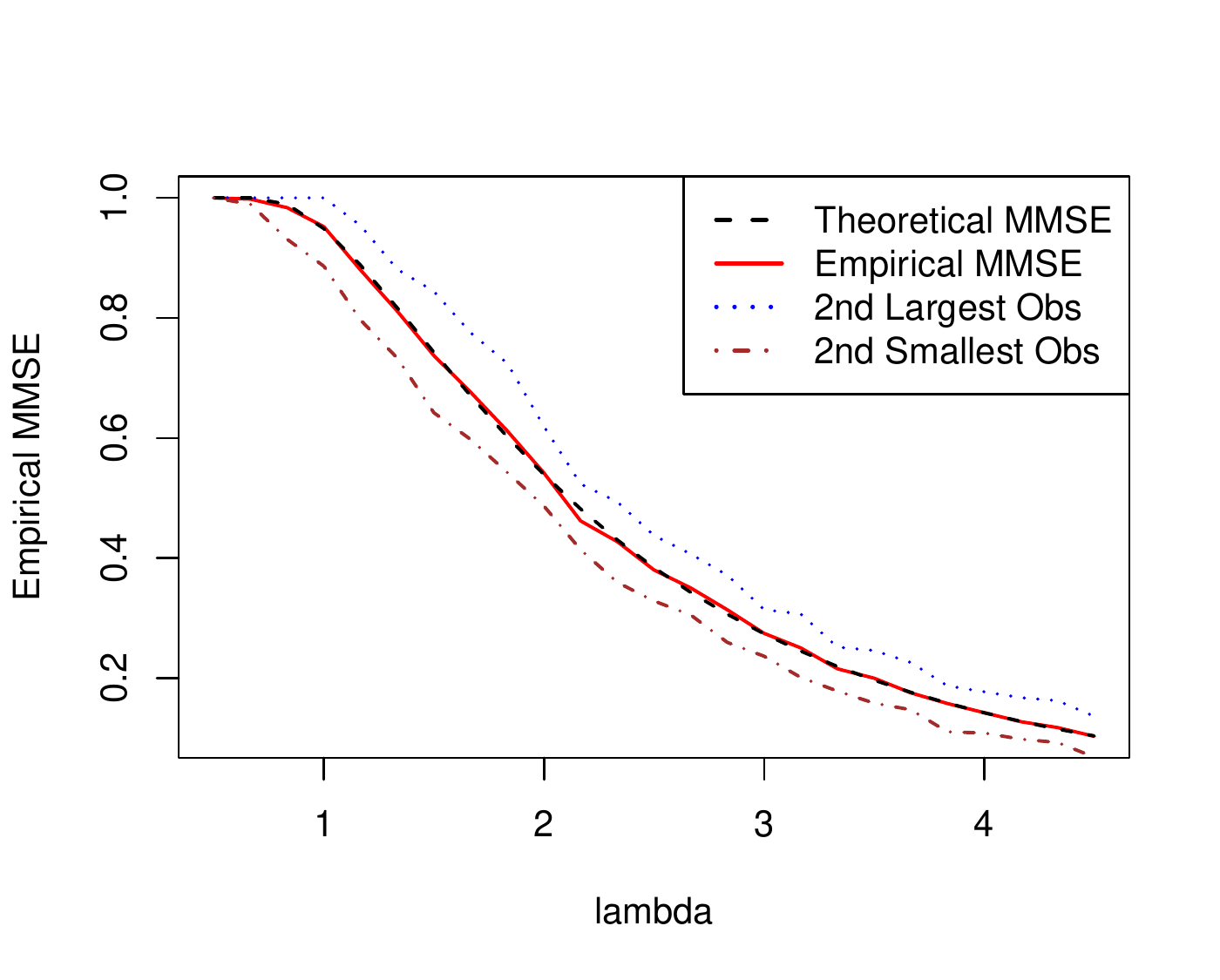}
  \caption{$\mu=0.7$}
  \label{fig:sfig12}
\end{subfigure}%
\begin{subfigure}{.33\textwidth}
  \centering
  \includegraphics[width=\linewidth]{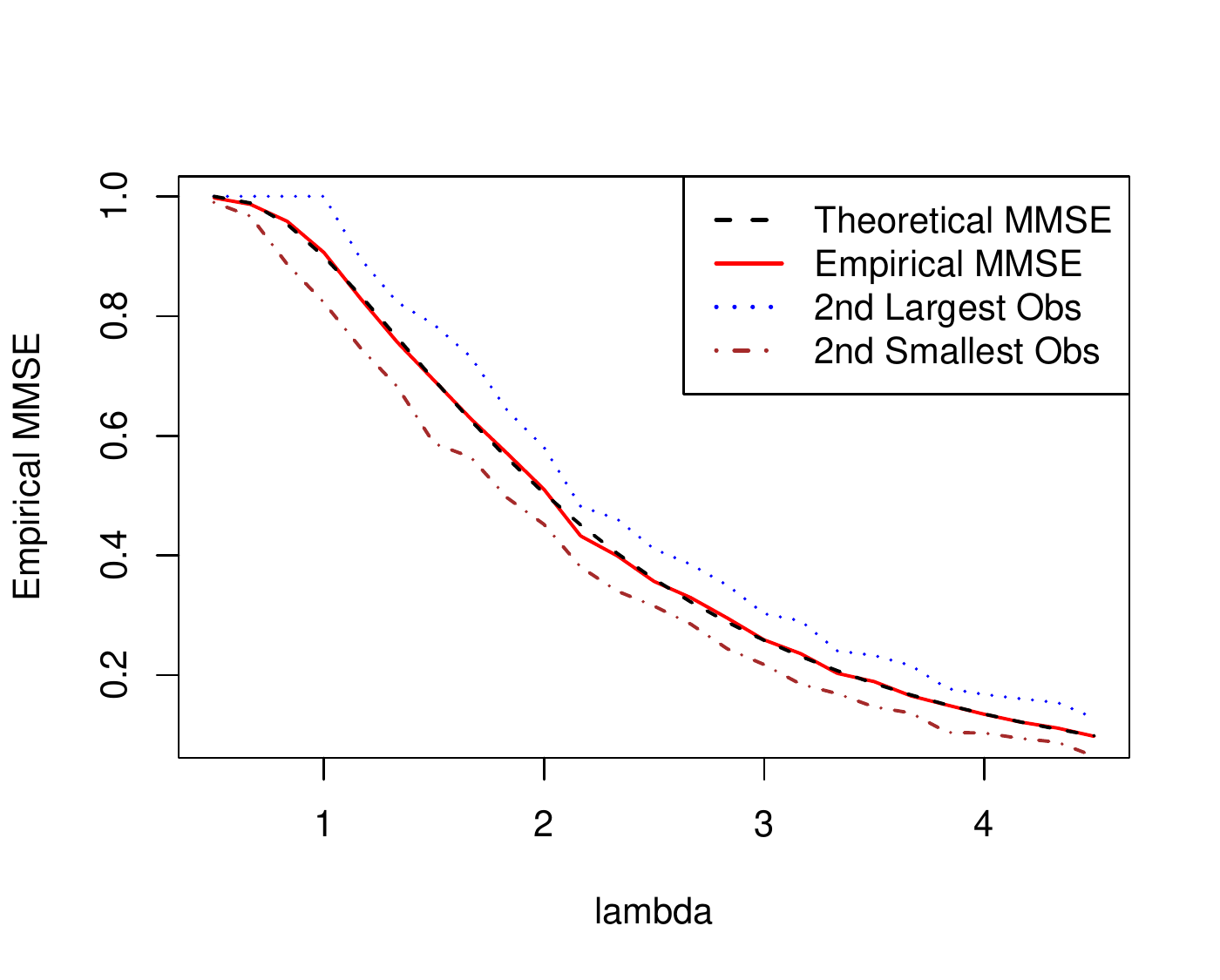}
  \caption{$\mu=0.9$}
  \label{fig:sfig13}
\end{subfigure}\\
\begin{subfigure}{.33\textwidth}
  \centering
  \includegraphics[width=\linewidth]{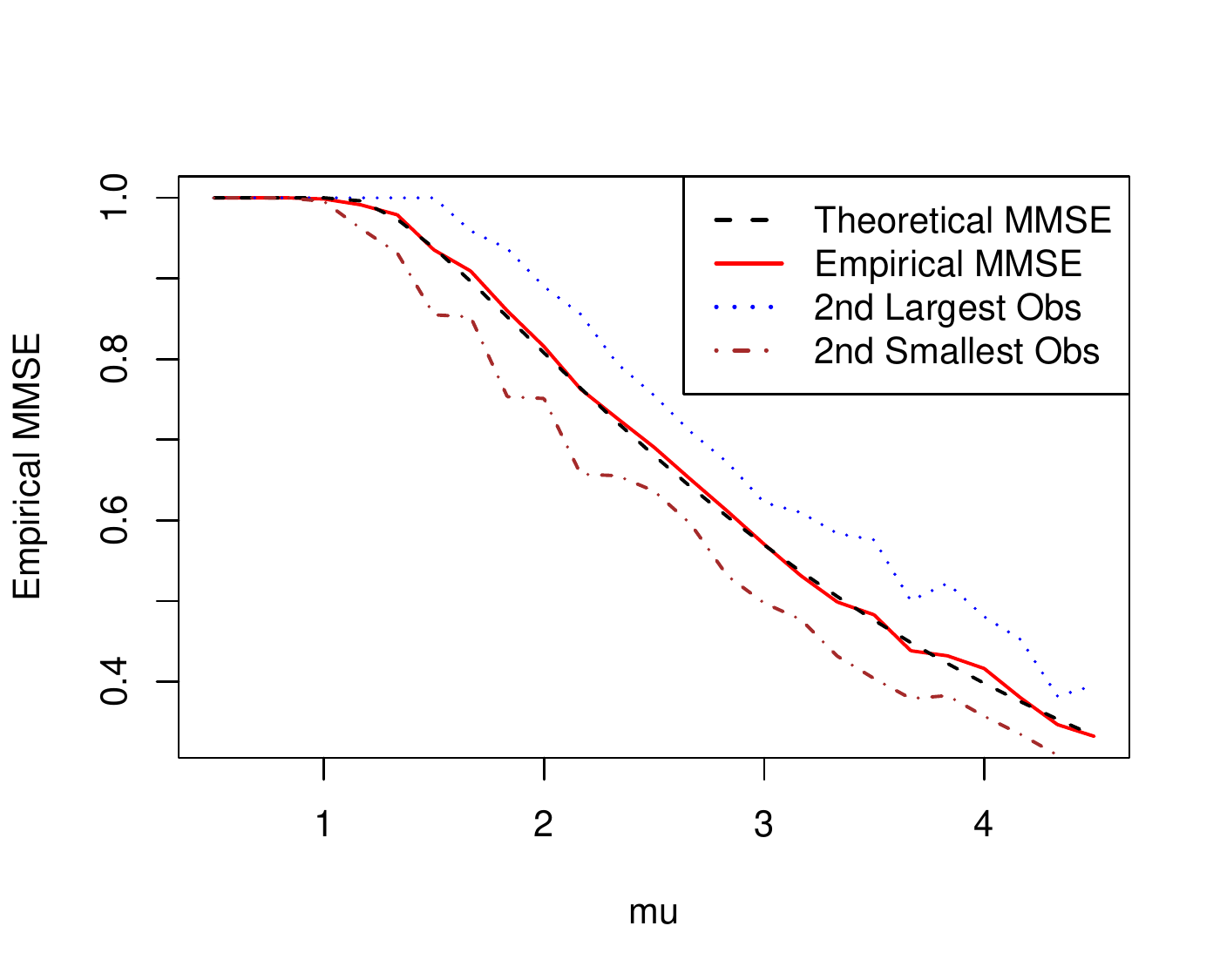}
  \caption{$\lambda=0.3$}
  \label{fig:sfig1}
\end{subfigure}%
\begin{subfigure}{.33\textwidth}
  \centering
  \includegraphics[width=\linewidth]{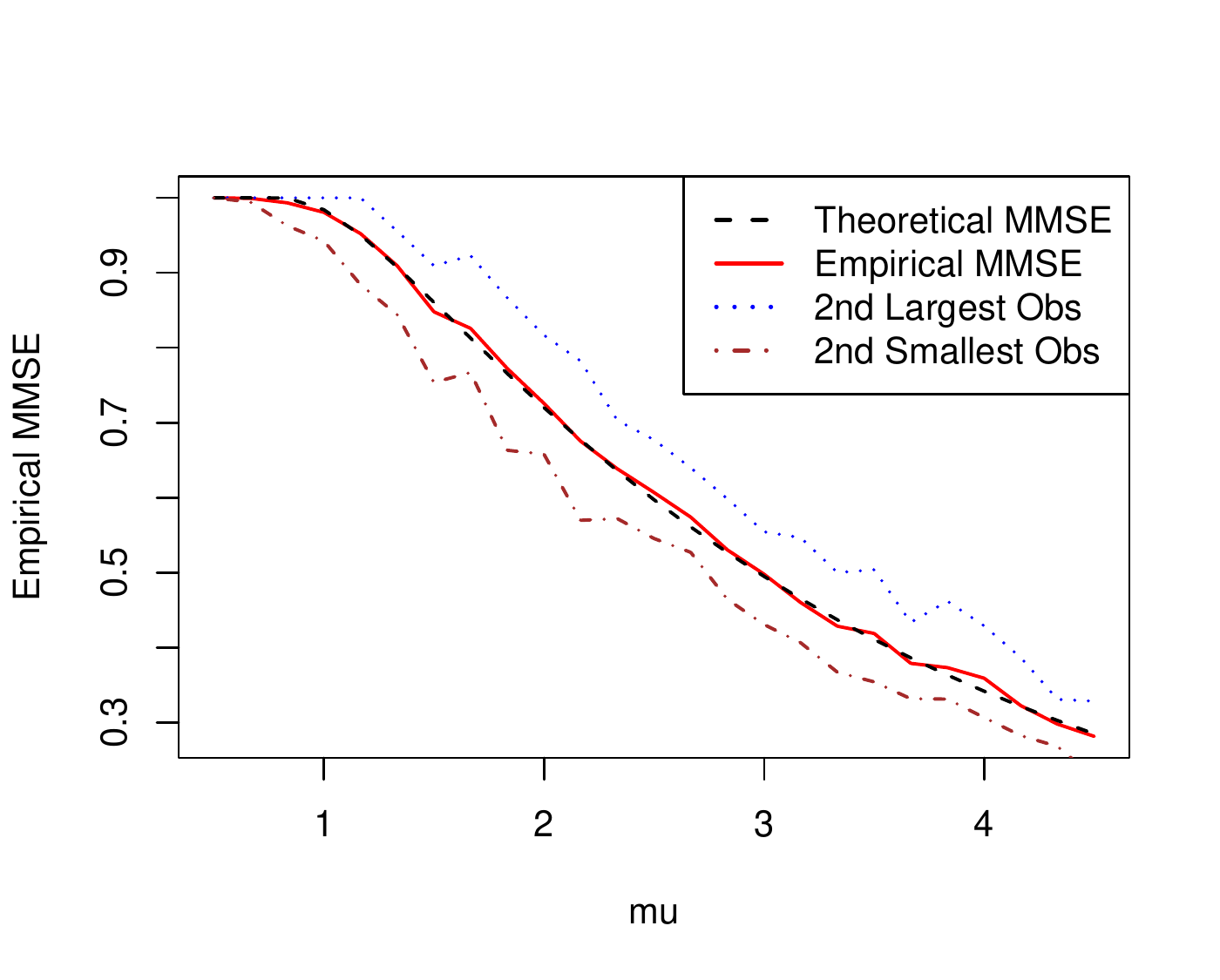}
  \caption{$\lambda=0.6$}
  \label{fig:sfig2}
\end{subfigure}%
\begin{subfigure}{.33\textwidth}
  \centering
  \includegraphics[width=\linewidth]{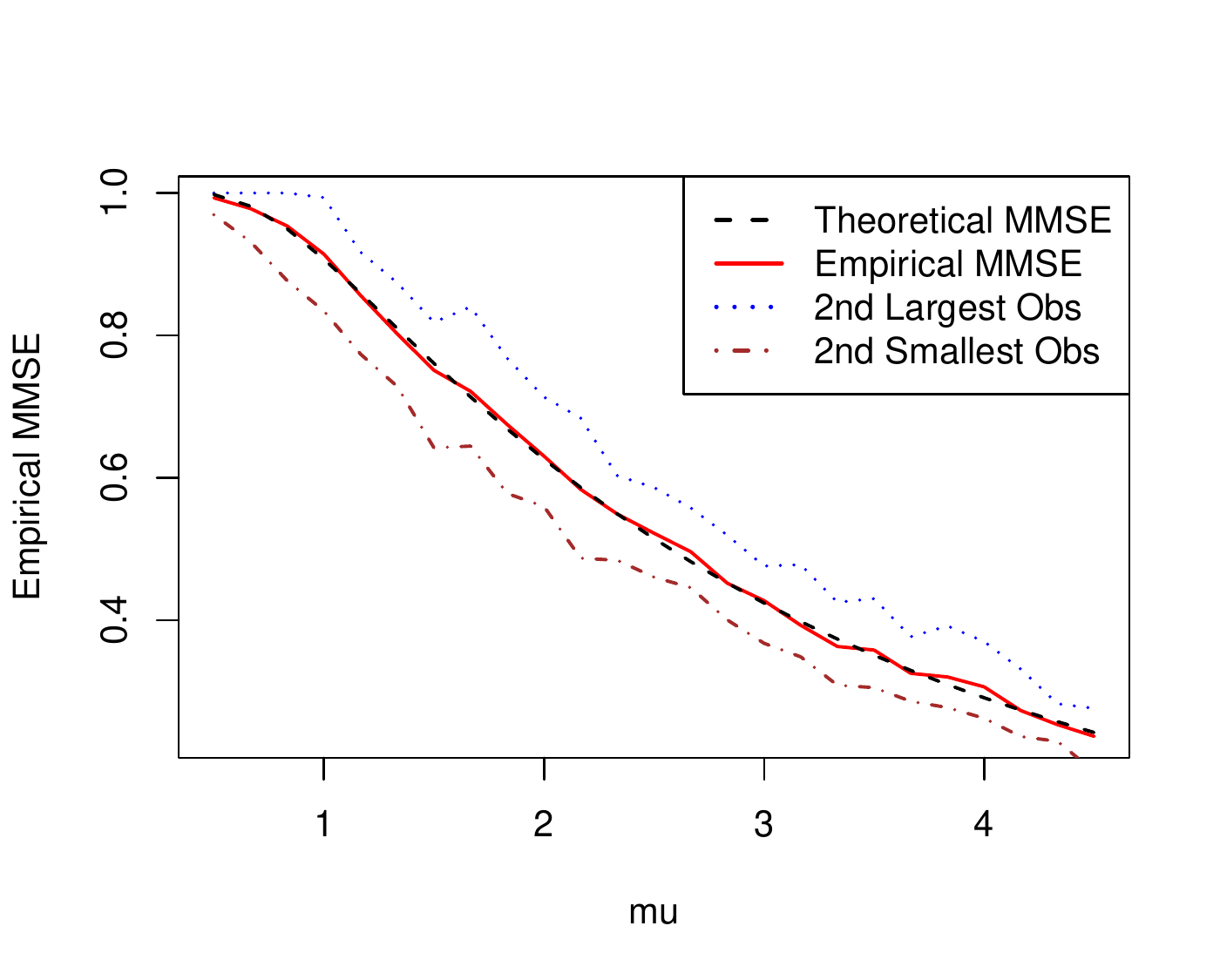}
  \caption{$\lambda=0.9$}
  \label{fig:sfig3}
\end{subfigure}
\caption{Upper Panel: Empirical MMSE plots of the AMP estimator versus $\lambda$ for different fixed $\mu$'s.
Lower Panel: Empirical MMSE plots of the AMP estimator versus $\mu$ for different fixed $\lambda$'s.}
\label{fig:fig_amp_lambda}
\end{figure}

\subsection{The Original Observation Model with One Layer}
\label{Single-layer Network plus Covariate Model}
We now consider the original observation model \eqref{eq:def_SBM_cov}--\eqref{eq:def_cov} with $m = 1$. 
This special case is also known as \revsag{the} contextual SBM.
Although the AMP algorithm defined by \eqref{eq:AMP_shift_main_1} and \eqref{eq:AMP_shift_main_1_1} is designed for a Gaussian sensing matrix, but \revsag{arguments of \cite{wang_zhong_fan} show that} the same state evolution limits can be obtained if instead of $\bm T$, one considers the matrix
\[
\bm A = \frac{\bm G-\widebar{p}_n\bm 1\bm 1^\top}{\sqrt{n\widebar{p}_n(1-\widebar{p}_n)}},
\]
where $\bm G$ is the adjacency matrix of the one layer network.
Thus, we could apply the practical algorithm presented at the beginning of this section with $\bm T$ replaced with $\bm A$.

% We do not prove any theoretical results for the orchestrated AMP state evolution using the graph adjacency matrices. Instead, we present some numerical simulations to demonstrate the efficiency of our AMP algorithm in detecting communities in the Contextual SBM setup using the network adjacency matrix and the Gaussian covariate matrix. In this case we have taken $m=1$. 

We take $n=2000$ and $p=3000$ and $\bar{p}_n = 0.7/\sqrt{n}$, and all other simulation details are identical to those used in Section \ref{sec:gauss-simu}.
% \begin{equation}
% \begin{aligned}
% \label{eq:AMP_shift_main_graph_1}
% \bm{v}^{t} &= \frac{\bm{B}}{\sqrt{p}} f_{t}(\bm{u}^t,\bm{x}^t)-p_{t}g_{t-1}(\bm{v}^{t-1}),\\
% \bm{u}^{t+1} &= \frac{\,\bm{B}^\top}{\sqrt{p}} g_t(\bm{v}^t)-c_tf_{t}(\bm{u}^t,\bm{x}^t),\\
% \end{aligned}
% \end{equation}
% and
% \begin{equation}
% \label{eq:AMP_shift_main_graph_1_1}
% \bm{x}^{t+1} = \frac{\bm{A}}{\sqrt{n}} f_{t}(\bm{u}^t,\bm{x}^t)-d_{t}f_{t-1}(\bm{u}^{t-1},\bm{x}^{t-1}),
% \end{equation}
% where
% \[
% \bm{A}:=\frac{\bm G-\widebar{p}_n\bm 1\bm 1^\top}{\sqrt{n\widebar{p}_n(1-\widebar{p}_n)}},
% \]
% where $\bm G$ is the adjacency matrix of an SBM with $\widebar{p}_n=0.7/\sqrt{n}$. Again we consider two settings like before. \textcolor{red}{In the first setting, we vary $\mu$ in $0.5,0.7$ and $0.9$ respectively. At each value of $\mu$, we vary $\lambda$ between $0.8$ and $4.5$. For each combination of $(\lambda,\mu)$ we generate $25$ copies of $(\bm G,\bm B)$ pairs. For each $(\bm G,\bm B)$ pair we initialize $\mu_0,  \sigma_0, \alpha_{-1}$ and $\tau_{-1}$ as in Section \ref{Single-layer Network plus Covariate Model}. To initialize $\bm u^0$ and $\bm x^0$, we consider the leading eigenvector $\bar{e}_1$ of $\bm G+a_0\bm B\bm B^\top$, where $a_0$ is the unique solution of \eqref{eq:initialization} and initialize both $\bm u^0$ and $\bm x^0$ by $\sqrt{n}*\bar{e}_1$. We run the iterates in \eqref{eq:AMP_shift_main_graph_1} and \eqref{eq:AMP_shift_main_graph_1_1} for $100$ iterations. 
In the upper panel of Figure \ref{fig:fig_amp_mu_g} we plot the average and the spread of the empirical MMSE of the estimator defined by \eqref{eq:amp_final_estim} over 25 iterates for each value of $\lambda$ at three fixed values of $\mu$ and compare it against the theoretical prediction given by \eqref{eq:MMSE-limit}. 
In the lower panel of Figure \ref{fig:fig_amp_mu_g}, we repeat the experiment across different $\mu$ values at three fixed $\lambda$ values.
In both settings, 
we see the same pattern as in the Gaussian observation model: the empirical MMSEs of the practical algorithm approximate the theoretical predictions well across all $(\lambda,\mu)$ combinations that we consider.
% We observe that the empirical MMSE of our estimator compares very well with the theoretical prediction.
% }

% \textcolor{red}{
% In the second setting, we consider the same setup as before with the same initializers except, in this case we vary $\lambda$ in $0.3,0.6$ and $0.9$ respectively and at each value of $\lambda$, we vary $\mu$ between $0.8$ and $4.5$. 
% In the lower panel of Figure \ref{fig:fig_amp_mu_g} we plot the average and the spread of the empirical MMSE of the estimator defined by \eqref{eq:amp_final_estim} across 25 iterates for different values of $\mu$ holding $\lambda$ fixed in three different values. Like the first setting, in this plot also, we observe that the empirical MMSE of our estimator compares very well with the theoretical prediction.}

\begin{figure}[h]
\begin{subfigure}{.33\textwidth}
  \centering
  \includegraphics[width=\linewidth]{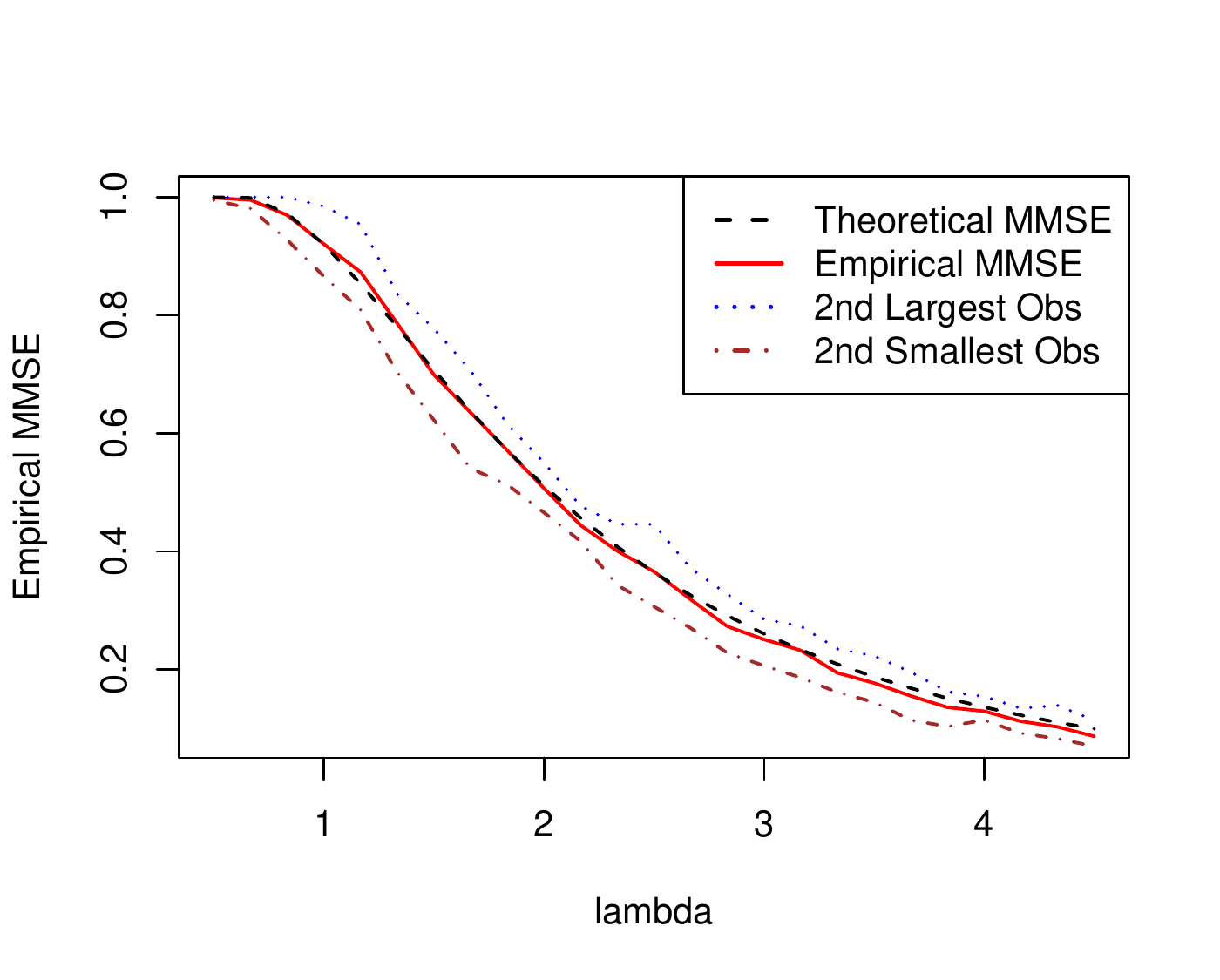}
  \caption{$\mu=0.5$}
  \label{fig:sfig11_g}
\end{subfigure}%
\begin{subfigure}{.33\textwidth}
  \centering
  \includegraphics[width=\linewidth]{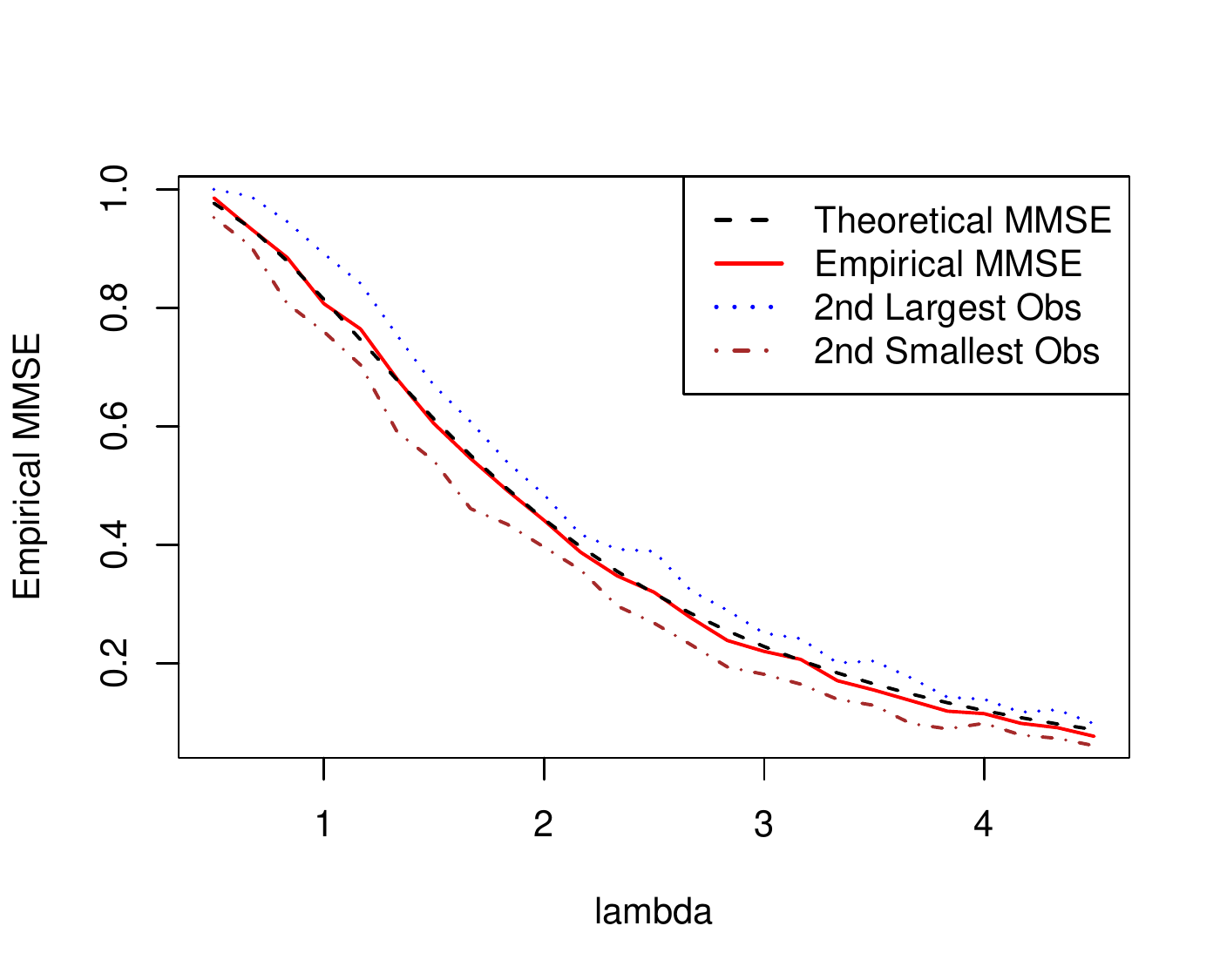}
  \caption{$\mu=0.7$}
  \label{fig:sfig12_g}
\end{subfigure}%
\begin{subfigure}{.33\textwidth}
  \centering
  \includegraphics[width=\linewidth]{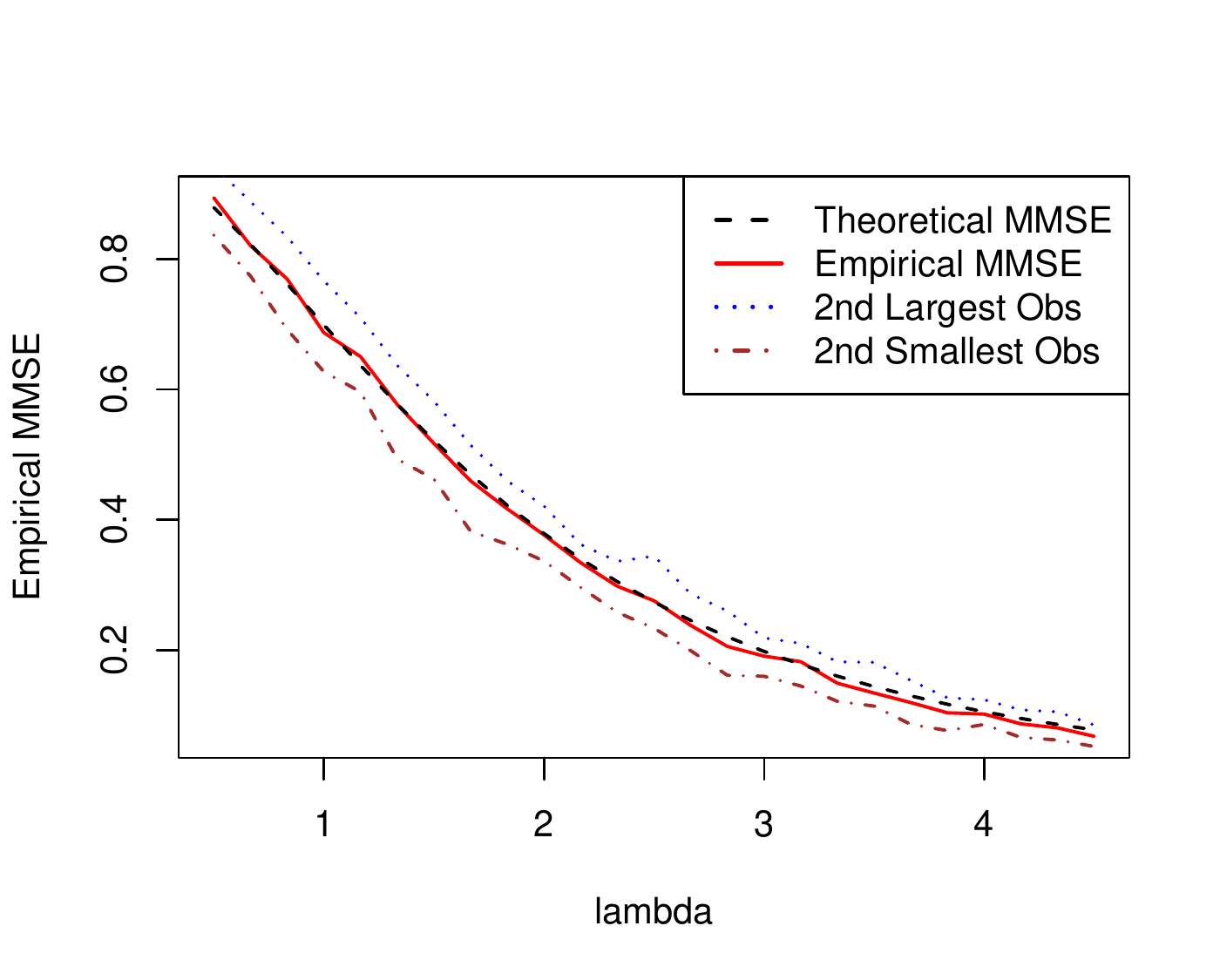}
  \caption{$\mu=0.9$}
  \label{fig:sfig13_g}
\end{subfigure}
\label{fig:fig_amp_lambda_g}
\\
\begin{subfigure}{.33\textwidth}
  \centering
  \includegraphics[width=\linewidth]{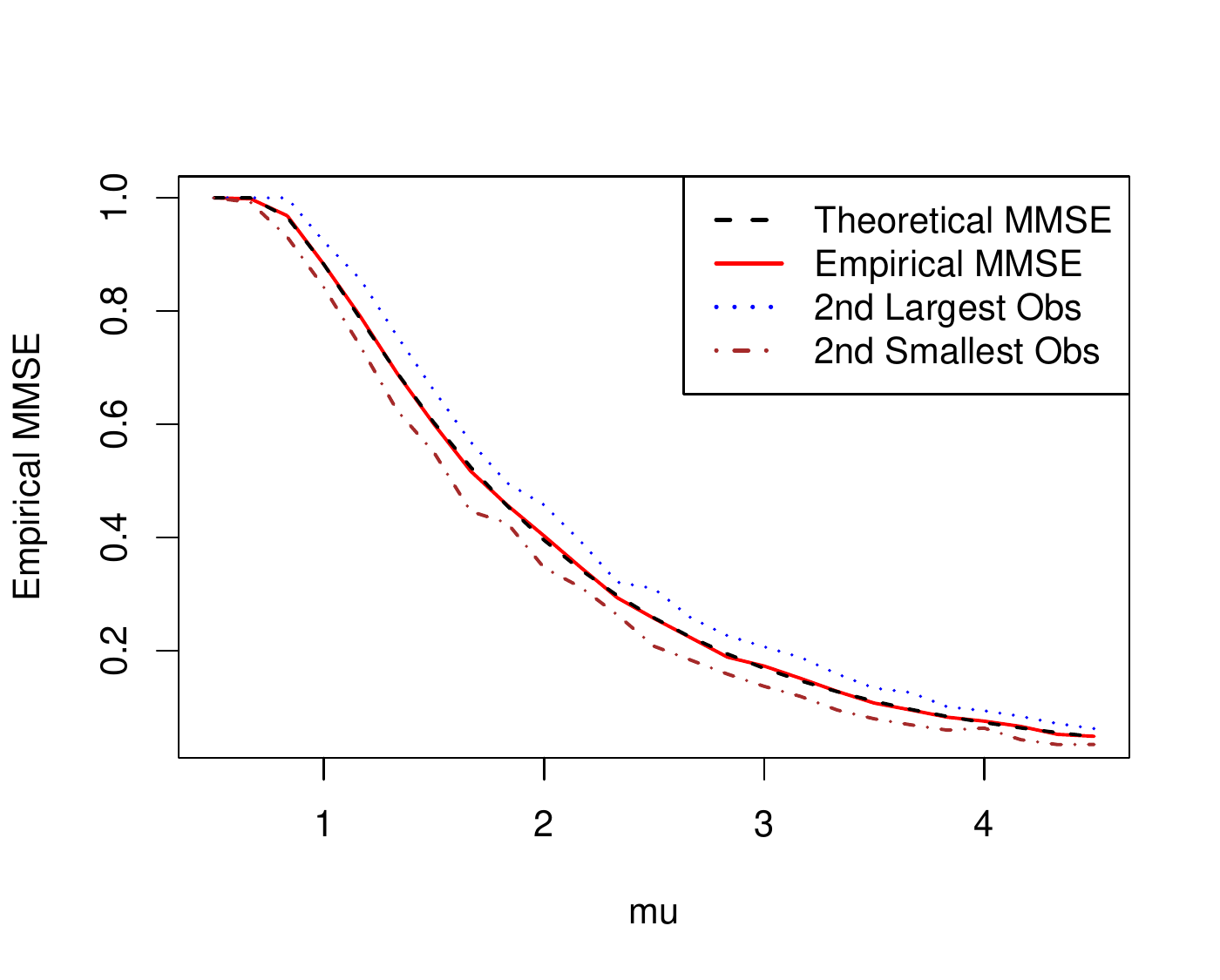}
  \caption{$\lambda=0.3$}
  \label{fig:sfig1_g}
\end{subfigure}%
\begin{subfigure}{.33\textwidth}
  \centering
  \includegraphics[width=\linewidth]{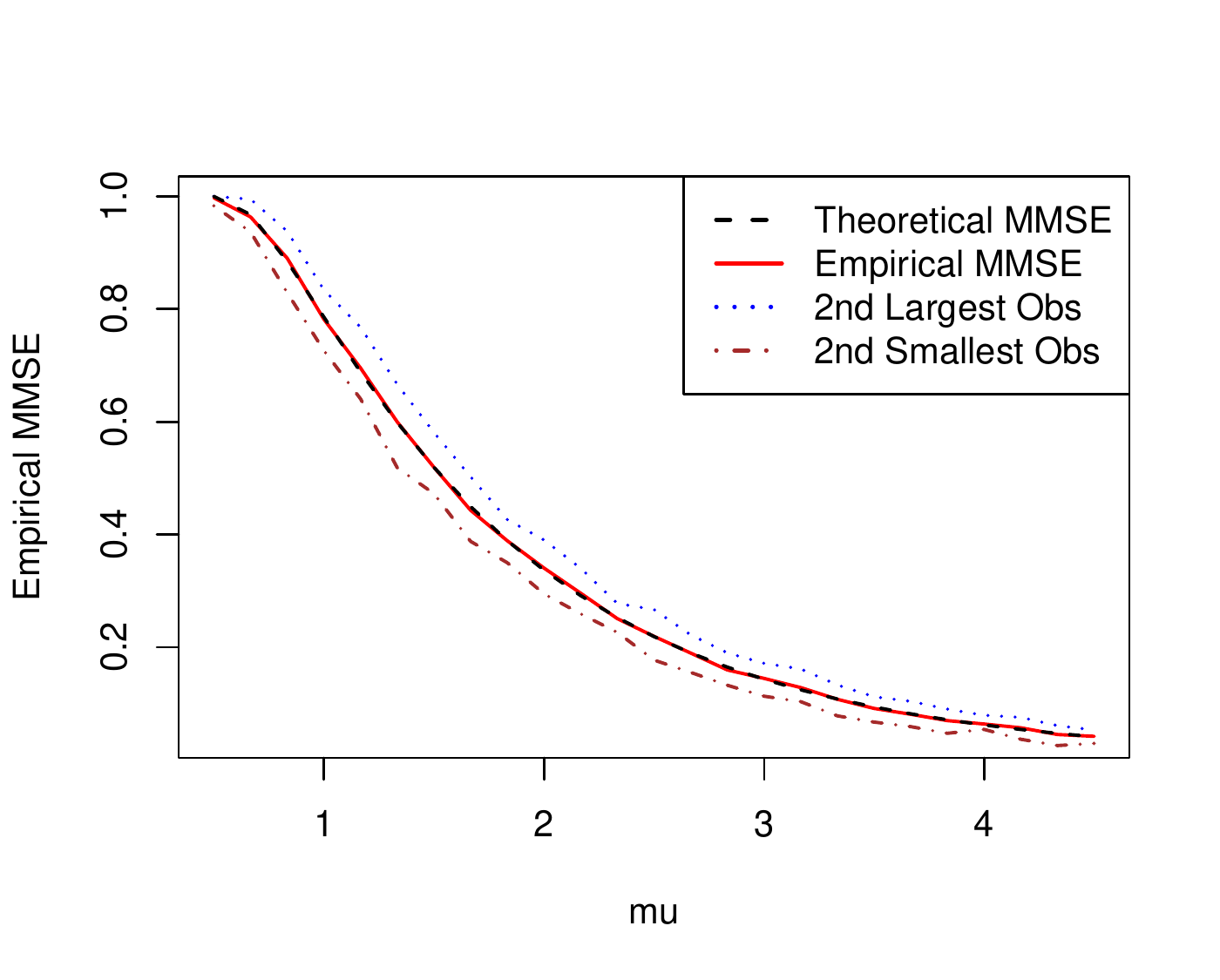}
  \caption{$\lambda=0.6$}
  \label{fig:sfig2_g}
\end{subfigure}%
\begin{subfigure}{.33\textwidth}
  \centering
  \includegraphics[width=\linewidth]{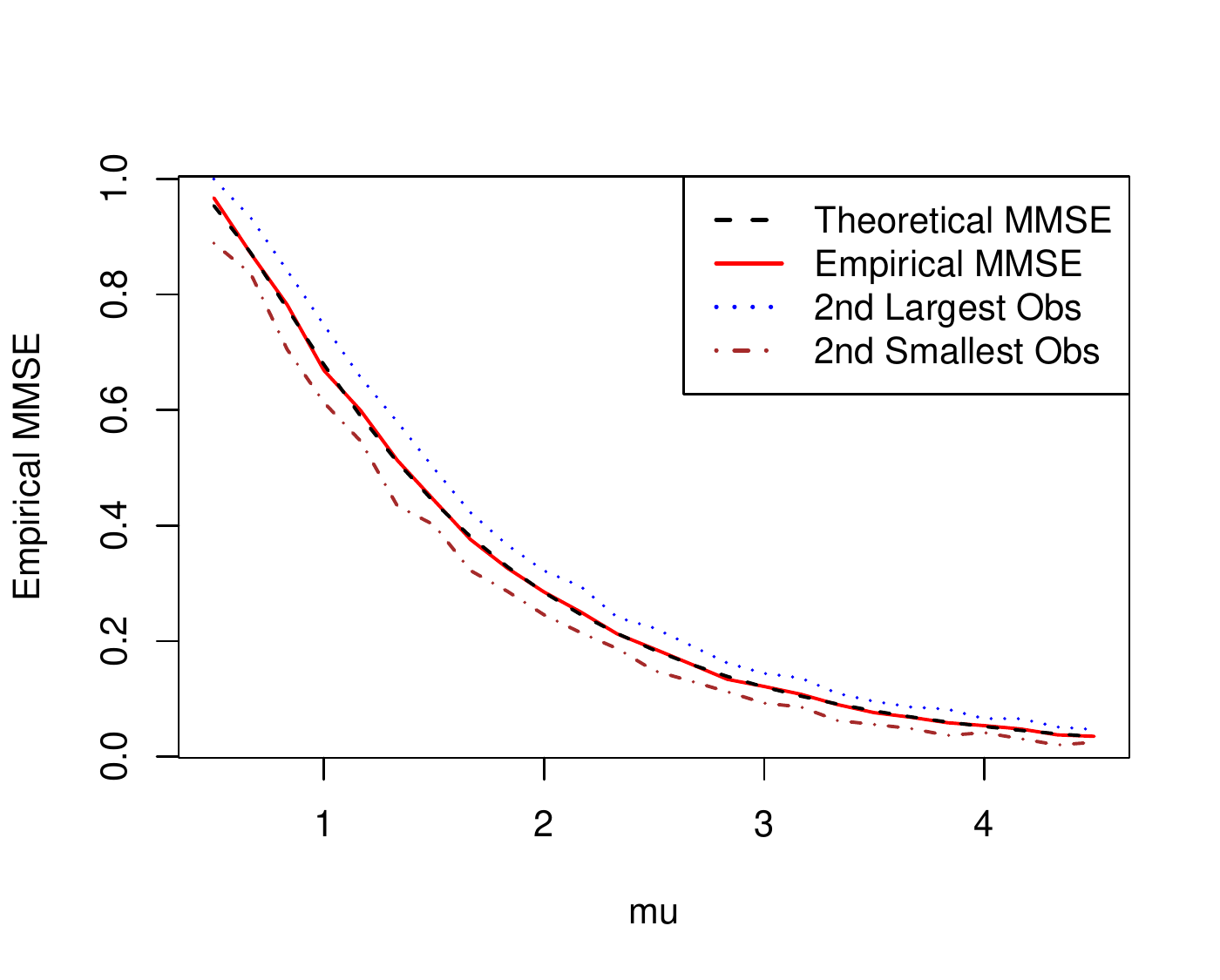}
  \caption{$\lambda=0.9$}
  \label{fig:sfig3_g}
\end{subfigure}
\caption{Upper Panel: Empirical MMSE plots of the AMP estimator based on Graph Adjacency matrix versus $\lambda$ for different fixed $\mu$'s.
Lower Panel: Empirical MMSE plots of the AMP estimator based on Graph Adjacency matrix versus $\mu$ for different fixed $\lambda$'s.}
\label{fig:fig_amp_mu_g}
\end{figure}

\subsection{The Original Observation Model with Three Layers}
\label{Multi-layer Network plus Covariate Model}
In our last set of simulations, 
we turn to the original observation model \eqref{eq:def_SBM_cov}-\eqref{eq:def_cov} with \revsag{$m = 3$.
% problem our observations were $m$ adjacency matrices of $m$ different SBM's with a common community structure and a Gaussian covariate matrix $\bm B$. In this subsection we define a set of orchestrated AMP orbits for this setup of multi-layer networks with the covariate matrix. 
We take $n=2000$, $p=3000$, and consider three SBMs with $\bar{p}^{(1)}_n=0.7/\sqrt{n}$, $\bar{p}^{(2)}_n=0.4/\sqrt{n}$ and $\bar{p}^{(3)}_n=0.3/\sqrt{n}$. 
In addition, we keep the SNR fractions
$r^{(1)}$, $r^{(2)}$ and $r^{(3)}$ in \eqref{eq:lambdaratio} at $0.6$, $0.2$ and $0.2$, respectively. }
The adjacency matrices of the SBMs are denoted by $\bm G_1$ and $\bm G_2$. 
To find the counterpart for $\bm T$ to be used in the practical algorithm, 
we first define the centered and scaled adjacency matrices:
\[
\bm A_i = \frac{\bm G_i-\widebar{p}^{(i)}_n\bm 1\bm 1^\top}{\sqrt{n\widebar{p}^{(i)}_n(1-\widebar{p}^{(i)}_n)}},
\qquad i = 1,2,3.
\]
% and
% \[
% \bm A_2 = \frac{\bm G_2-\widebar{p}^{(2)}_n\bm 1\bm 1^\top}{\sqrt{n\widebar{p}^{(2)}_n(1-\widebar{p}^{(2)}_n)}}.
% \]
\revsag{Simple} algebra suggests that we should replace $\bm T$ in the Gaussian model with 
\begin{equation*}
    {\bm A}:=\sqrt{\lambda_1/\lambda}\bm A_1+\sqrt{\lambda_2/\lambda}\bm A_2+\sqrt{\lambda_3/\lambda}\bm A_3
\end{equation*}
Other than the foregoing modification, the other experiment details are identical to what we have used in the previous two subsections.

% \textcolor{red}{
% We consider the following combined sensing matrix $A:=\sqrt{\lambda_1/\lambda}\bm A_1+\sqrt{\lambda_2/\lambda}\bm A_2$. Combining $\bm A_1$ and $\bm A_2$ in the foregoing fashion matches the mean of $\bm A$ with that of $\bm T$. Next, we consider the following orchestrated AMP orbits.
% \begin{equation}
% \begin{aligned}
% \label{eq:AMP_shift_main_graph_multi_1}
% \bm{v}^{t} &= \frac{\bm{B}}{\sqrt{p}} f_{t}(\bm{u}^t,\bm{x}^t)-p_{t}g_{t-1}(\bm{v}^{t-1}),\\
% \bm{u}^{t+1} &= \frac{\,\bm{B}^\top}{\sqrt{p}} g_t(\bm{v}^t)-c_tf_{t}(\bm{u}^t,\bm{x}^t),\\
% \end{aligned}
% \end{equation}
% and
% \begin{equation}
% \label{eq:AMP_shift_main_graph_1_2}
% \bm{x}^{t+1} = \frac{\bm{A}}{\sqrt{n}} f_{t}(\bm{u}^t,\bm{x}^t)-d_{t}f_{t-1}(\bm{u}^{t-1},\bm{x}^{t-1}).
% \end{equation}
% To initialize $\bm x^0$ and $\bm u^0$ we take the leading eigenvector of $\bm A$ and multiply it by $\sqrt{n}$. We initialize $\mu_0,  \sigma_0, \alpha_{-1}$ and $\tau_{-1}$ as in Section \ref{Single-layer Network plus Covariate Model}. Then we consider two settings. In the first setting, we vary $\mu$ in $0.5,0.7$ and $0.9$ and $\lambda$ between $0.8$ and $4.5$. For each pair of $(\lambda,\mu)$, we generate $25$ copies of $(\bm G_1, \bm G_2, \bm B)$ and run the recursions \eqref{eq:AMP_shift_main_graph_multi_1}. 
In the upper panel of Figure \ref{fig:fig_amp_mu_m_g} we plot the average and the spread of the empirical MMSE of the estimator defined by \eqref{eq:amp_final_estim} over 25 iterations at each value of $\lambda$ for three fixed values of $\mu$. 
We compare empirical MMSEs with the theoretical prediction \eqref{eq:MMSE-limit}. 
% In the second setting we repeat the same exercise except varying $\lambda$ in $0.3,0.6$ and $0.9$ and $\mu$ between $0.8$ and $4.5$ for each value of $\lambda$. 
In the lower panel of Figure \ref{fig:fig_amp_mu_m_g} we switch the roles of $\lambda$ and $\mu$, that is, we fix $\lambda$ and vary $\mu$. 
% In both the settings, we see that the empirical MMSE of our estimator closely approximates the predicted optimal value.
% }
In both settings, 
we see the same pattern as in the Gaussian observation model and in the contextual SBM: the empirical MMSEs of the practical algorithm approximate the theoretical predictions well across all $(\lambda,\mu)$ value pairs that we consider.

\begin{figure}[h]
\begin{subfigure}{.33\textwidth}
  \centering
  \includegraphics[width=\linewidth]{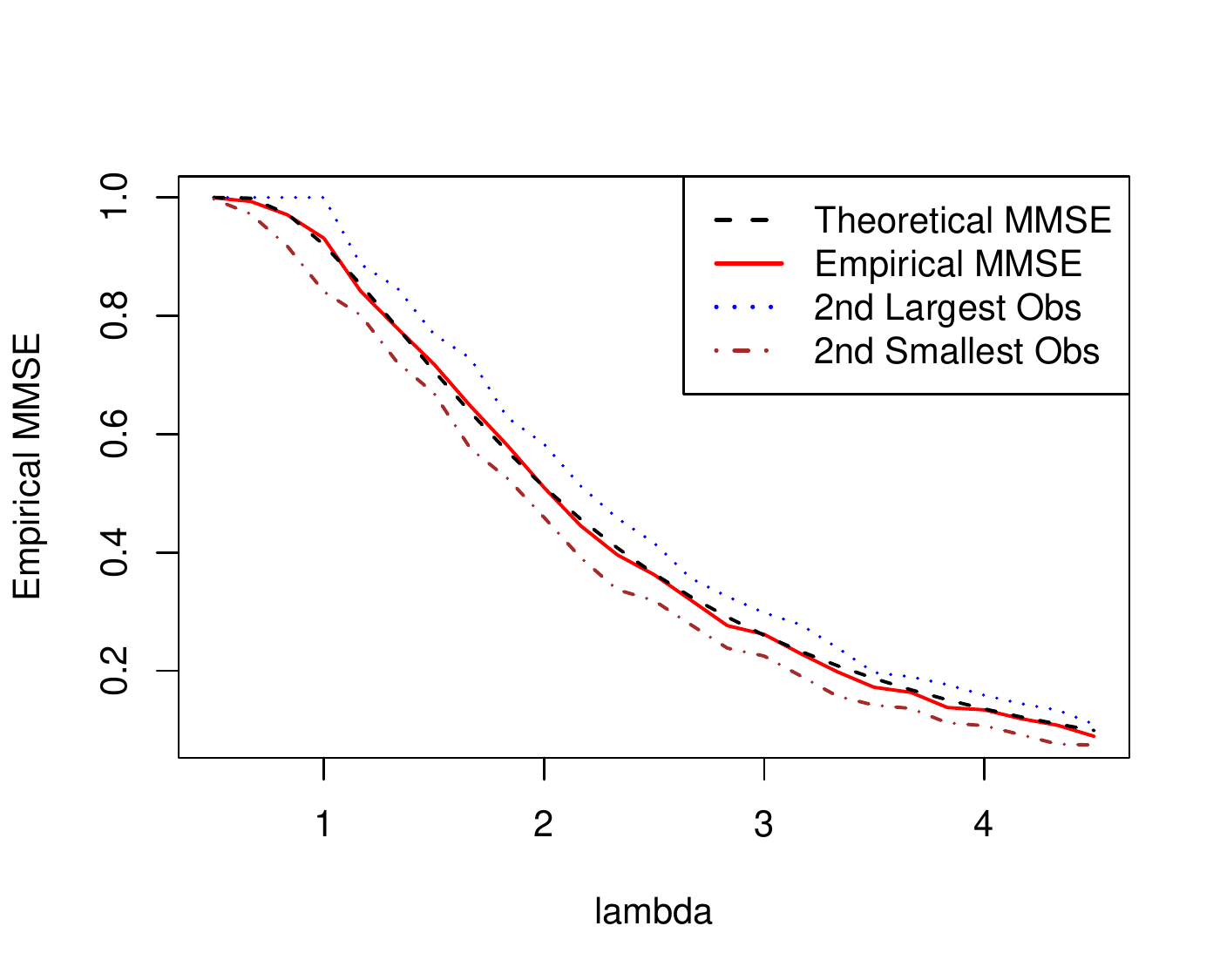}
  \caption{$\mu=0.5$}
  \label{fig:sfig11_m_g}
\end{subfigure}%
\begin{subfigure}{.33\textwidth}
  \centering
  \includegraphics[width=\linewidth]{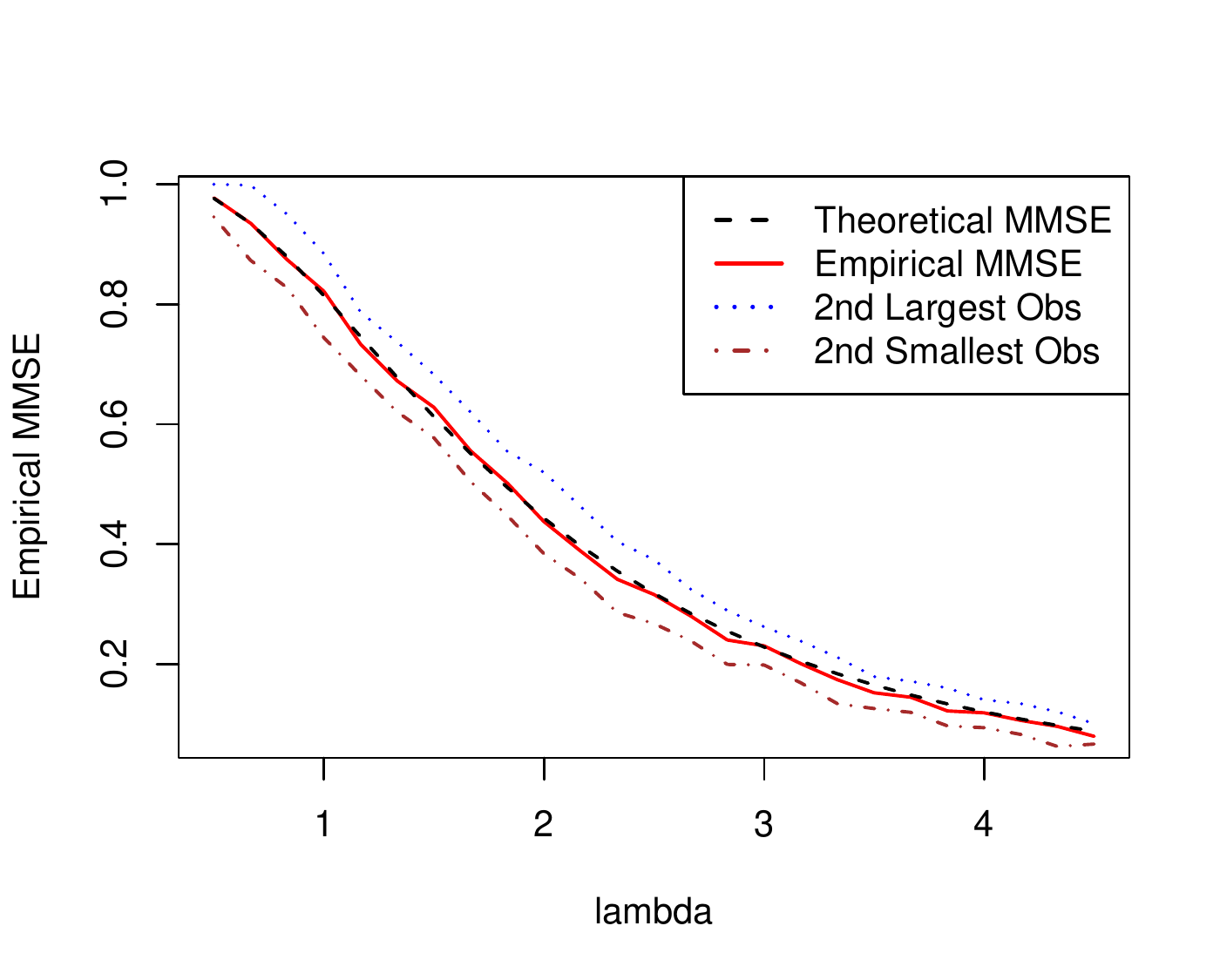}
  \caption{$\mu=0.7$}
  \label{fig:sfig12_m_g}
\end{subfigure}%
\begin{subfigure}{.33\textwidth}
  \centering
  \includegraphics[width=\linewidth]{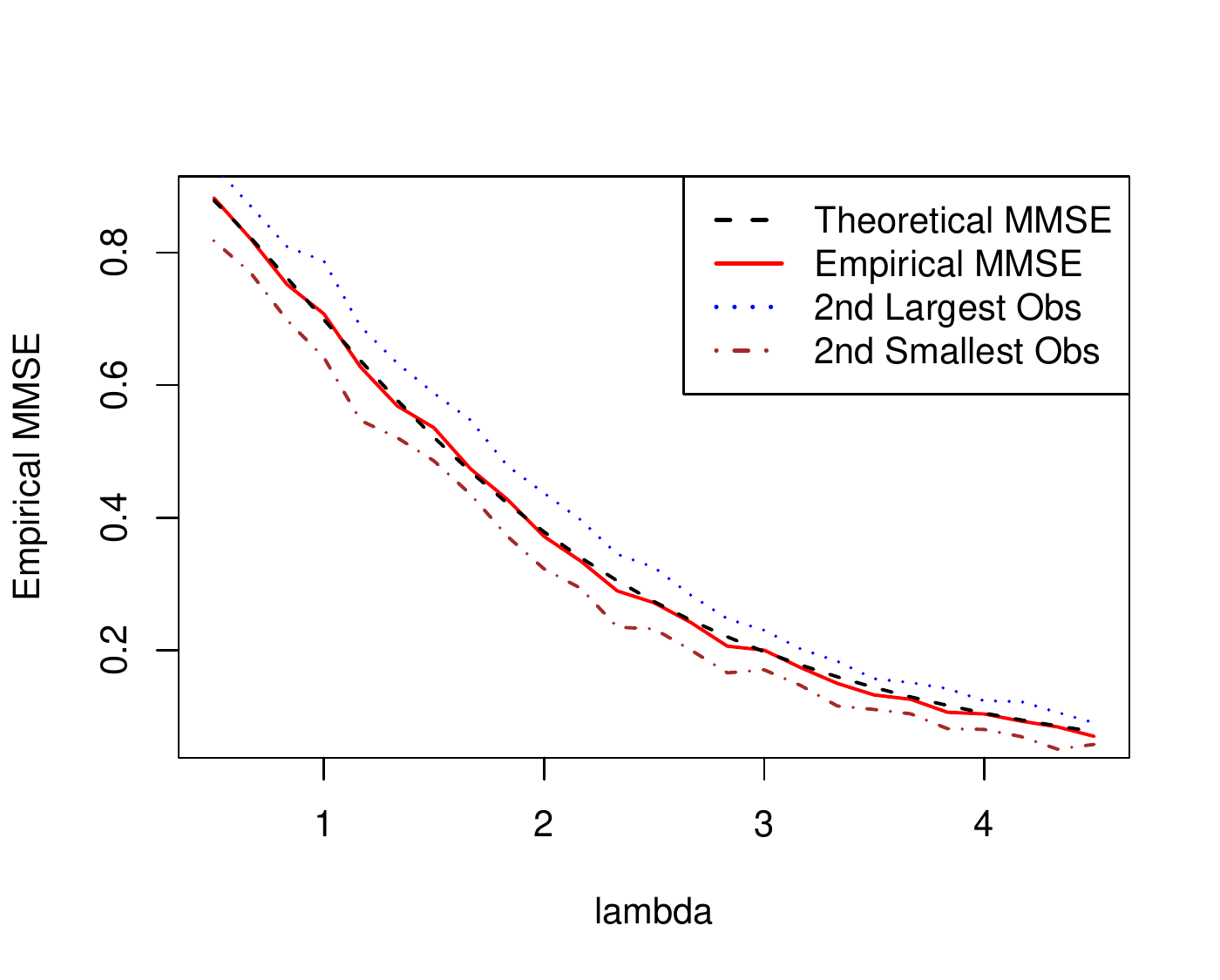}
  \caption{$\mu=0.9$}
  \label{fig:sfig13_m_g}
\end{subfigure}
\\
\begin{subfigure}{.33\textwidth}
  \centering
  \includegraphics[width=\linewidth]{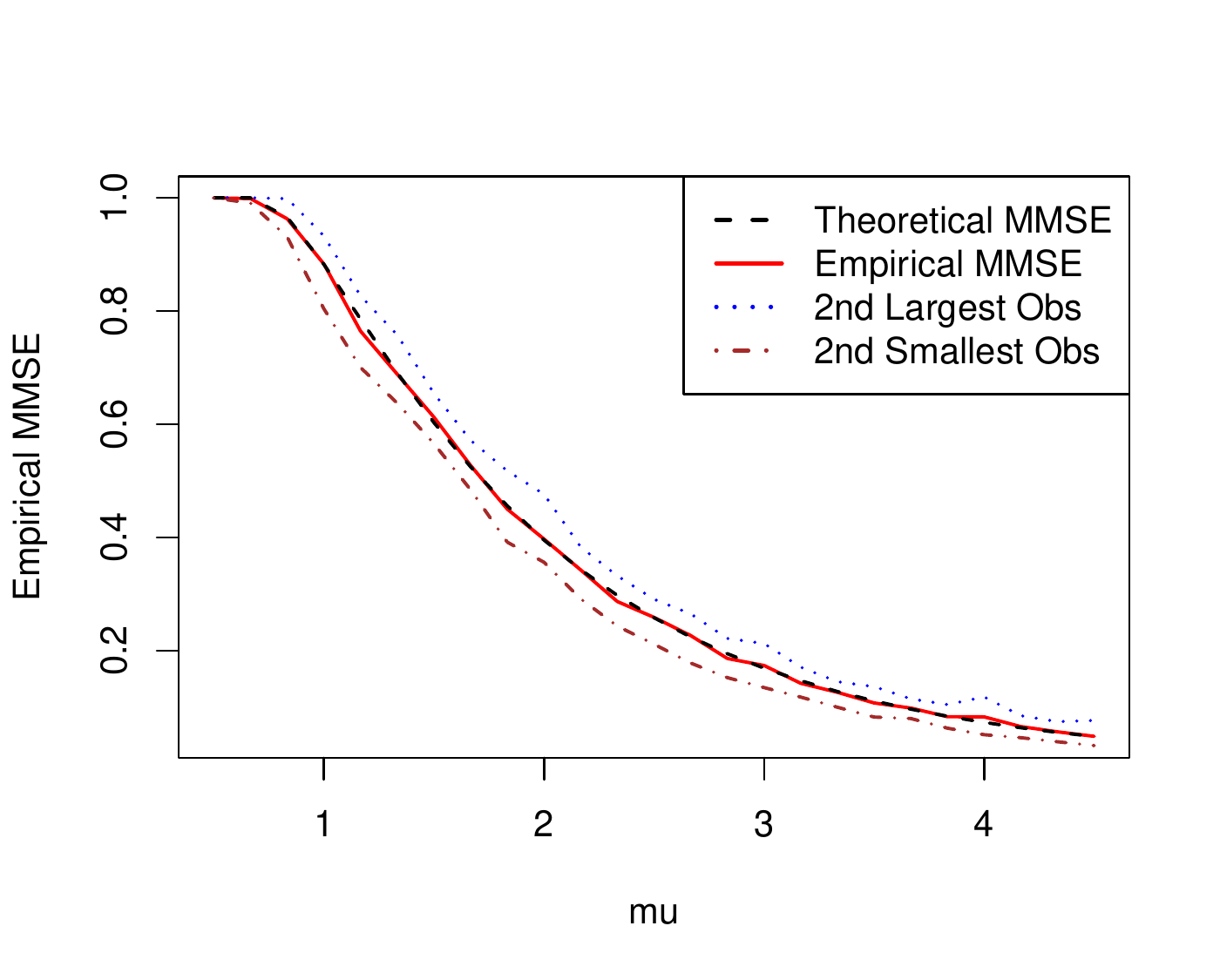}
  \caption{$\lambda=0.3$}
  \label{fig:sfig1_m_g}
\end{subfigure}%
\begin{subfigure}{.33\textwidth}
  \centering
  \includegraphics[width=\linewidth]{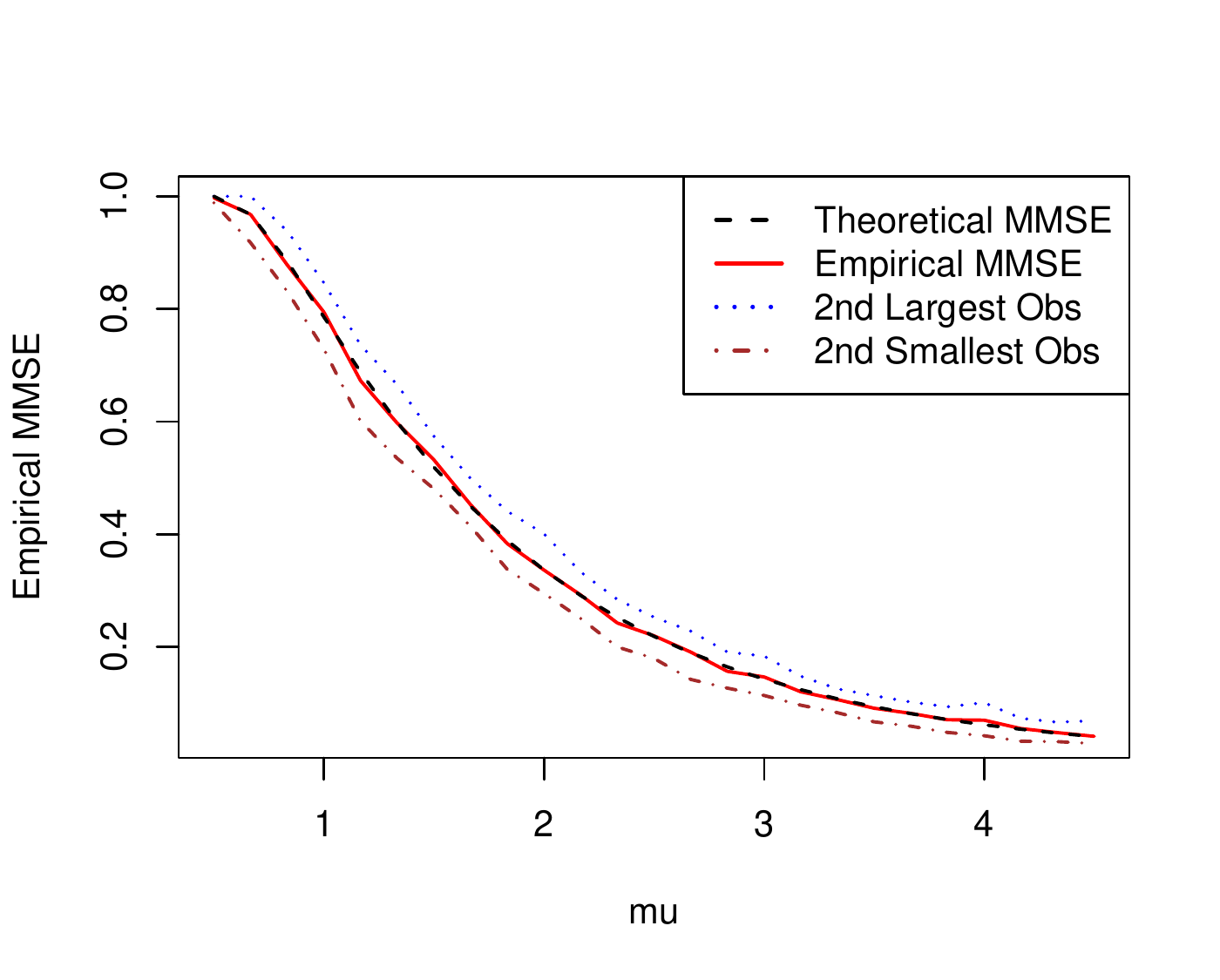}
  \caption{$\lambda=0.6$}
  \label{fig:sfig2_m_g}
\end{subfigure}%
\begin{subfigure}{.33\textwidth}
  \centering
  \includegraphics[width=\linewidth]{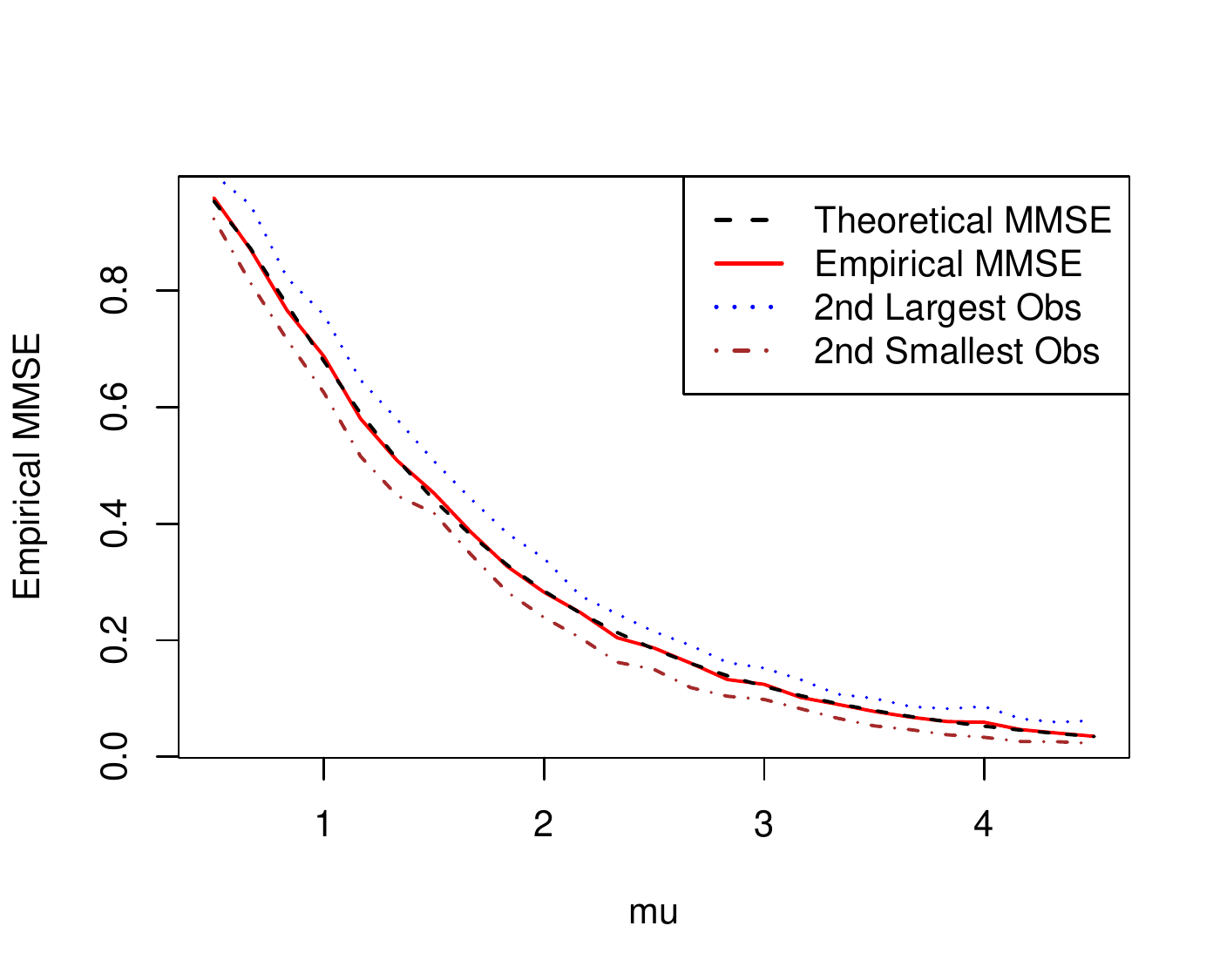}
  \caption{$\lambda=0.9$}
  \label{fig:sfig3_m_g}
\end{subfigure}
\caption{Upper Panel: Empirical MMSE plots of the AMP estimator based on three Graph Adjacency matrices versus $\lambda$ for different fixed $\mu$'s. Lower Panel: Empirical MMSE plots of the AMP estimator based on two Graph Adjacency matrices versus $\mu$ for different fixed $\lambda$'s.}
\label{fig:fig_amp_mu_m_g}
\end{figure}

\section{Concluding Remarks}
\label{sec:future}

In this paper, we have designed an orchestrated AMP algorithm with two orbits and $\varepsilon$-revelation to establish the exact asymptotic limit of MMSE for estimating $\bm{x}^*$. 
The theoretically justified version is not practical due to its dependence on the true parameter values through $\varepsilon$-revelation. 
In Section \ref{sec:numerical}, a practical variant with spectral initialization leads to empirical estimation errors that {closely approximate the} theoretically predicted optimal values over all simulated examples. 
To fully establish its generality, it is of great interest to show mathematically that this practical algorithm indeed reaches the asymptotic MMSE. 
{An alternative practical algorithm for community detection in single layer Contextual SBM has been described in Section 4 of \cite{ContBlockMod}. 
This algorithm can potentially be modified to handle multilayer networks and covariates. However, this is beyond the scope of the current manuscript and we leave it for future research.}

% \nb{insert some summary on what we have done}
% One important question that arises based on our work is the existence and optimality of any polynomial-time algorithm for (optimal) weak recovery of the community label vector above the detection threshold. \textcolor{red}{The AMP algorithm described in the Section \ref{sec:numerical} performs quite well empirically. However, at present we don't have any theory characterizing its performance. We leave this to future research.}

We have focused exclusively on the balanced two-block setting, which is the simplest non-trivial case for community detection. 
In addition, we have considered Gaussian covariates. 
These assumptions could be relaxed. 
The extension of our results to general sub-Gaussian sensing matrices can be derived by directly using the techniques described in \cite{wang_zhong_fan}. Therefore, we do not describe it in detail. 
In addition, we could consider the balanced $k$ block setting with $k > 2$, in which case the signal would be encoded by a matrix of rank $k-1$. 
See, for instance, the multiple spiked models (1.1) and (1.2) in \cite{montanari2021}.
The AMP algorithms for handling such cases can be developed by following the principle in the present paper and generalizing our techniques along the line of \cite{montanari2021}. 
% AMP algorithms can also be designed for signal detection in spiked tensor models with covariate information using our techniques. We leave these extensions to future research.
Furthermore, we could consider detection threshold in a sparse setting with non-diverging average degrees. We could also consider the optimal rate of community detection in these models under a Hamming loss as in \cite{10.1214/15-AOS1428} in the regime of weak consistency as opposed to detection.
In these settings, different techniques from AMP are expected to be needed to achieve the information-theoretically optimal performance. 
That being said, the idea of developing orchestrated parallel estimation sequences with information sharing at each step for different data sources could still be useful.
We leave the aforementioned potential extensions for future research.

In the present work, we have considered an orchestrated AMP algorithm with two orbits as we have two sources of information about the estimand. 
It could be generalized to more than two orbits when additional sources of information are present.
For example, suppose that there are $n$ vertices in total. 
In addition to \revsag{network and covariate information for all vertices} as we have considered in this work, there may be an additional network on a subset of vertices and some additional covariates on a different subset, represented by an $n_1\times n_1$ adjacency matrix and an $n_2\times p'$ covariate matrix, where $n_1, n_2 < n$.
To pool all the information together, we anticipate that an orchestrated AMP algorithm with four orbits would be needed to achieve an information-theoretically optimal estimation error.
We think that the study of \revsag{orchestrated AMP algorithms in more general settings} would be an interesting future research topic.

% The orchestrated AMP algorithm \revzm{we have proposed could potentially be adapted to other scenarios involving information extraction from multiple noisy channels}.
% % in the proportional growth asymptotic regime. 
% These channels may not be of the form described in our paper, but can be very generic. For instance, we might consider two sets of covariates in our setup. One set of covariates might be the Gaussian mixture model type covariate as described above, while the other set may be another $n_1 \times p_1$ data matrix, encoding some special information about a particular set of nodes. This additional information can also be leveraged using our orchestrated AMP set-up while detecting the signal. The detection threshold for the rank-one signal in such a set-up can be explored in the future using our generic method.

\section*{Acknowledgment}
The authors would like to thank Fan Yang for helpful \revsag{discussions} and communication \revsag{that lead} to \eqref{eq:initialization} and Galen Reeves for interesting \revsag{discussions} on the connection of the present work to estimation under the multi-view spiked matrix models.

\newpage
\appendix
\section{Proof of Results in Section~\ref{inf_sec}}
\subsection{Proof of Lemma~\ref{lem:cov_gauss_sbm}}
\label{proof_lem_1}
From the definition of mutual information we have
\[ I(\bm{x^{*}};\bm{Y},\bm{B})
=\mathbb{E}\left[\log\frac{dp_{\bm{Y},\bm{B}|\bm{x^{*}}}(\bm{Y},\bm{B}|\,\bm{x^{*}})}{dp_{\bm{Y},\bm{B}}(\bm{Y},\bm{B})}\right].
\]
Let 
\[
\mathcal{E}(\bm{x},\bm{v},\bm{Y},\bm{B})
=\exp\left(\sum_{i=1}^{m}-\frac{1}{4}\norm{\bm{Y^{(i)}}-
\sqrt{\frac{\lambda^{(i)}}{n}}\bm{x}\bm{x}^\top}^2_F-
\frac{1}{2} \norm{\bm{B}-\sqrt{\frac{\mu}{n}}\bm{v}\bm{x}^\top}^2_F
\right).
\]
Then using the property of Gaussian channel we get
\begin{equation} 
\label{eq:Gauss_Chann_Inform_cov}
\begin{split}
& I(\bm{x^{*}};\bm{Y},\bm{B}) \\
& = \mathbb{E}\log\left\{
\frac{\int_{\mathbb{R}^p}\mathcal{E}(\bm{x^{*}},\bm{v},\bm{Y},\bm{B})
\exp(-\frac{1}{2}{\|\bm{v}\|^2})d\bm{v}}{\sum_{\bm{x} \in \{\pm 1\}^{n}}\int_{\mathbb{R}^p}{2^{-n}}\mathcal{E}(\bm{x},\bm{v},\bm{Y},\bm{B})\exp(-\frac{1}{2}\|\bm{v}\|^2)d\bm{v}}\right\}\\
& = n\log 2 + \mathbb{E}\log\Bigg(\int_{\mathbb{R}^p}
\exp\left(-\frac{1}{2}\norm{\bm{B}-\sqrt{\frac{\mu}{n}}\bm{v}(\bm{x^{*}})^\top}^2_F\right)\exp\left(-\frac{\,\|\bm{v}\|^2}{2} \right)d\bm{v}\Bigg)\\
&\quad
- \mathbb{E}\log\Bigg\{\sum_{\bm{x} \in \{\pm 1\}^{n}}\int_{\mathbb{R}^p}\exp\Bigg(\sum_{i=1}^{m}\Bigg[\sum_{k<l}\Bigg\{-\frac{1}{2}\Bigg(Z^{(i)}_{kl} - \sqrt{\frac{\lambda^{(i)}}{n}}(x_kx_l-x^*_kx^*_l)\Bigg)^2+\frac{1}{2}(Z^{(i)}_{kl})^2\Bigg\}\Bigg]\\
&\hskip 19em
-\frac{1}{2}\norm{\bm{B}-\sqrt{\frac{\mu}{n}}\bm{v}\bm{x}^\top}^2_F\Bigg)
\exp\left(-\frac{\,\|\bm{v}\|^2}{2}\right)d\bm{v}\Bigg\}.
\end{split}
\end{equation}
Furthermore, if we note that 
\[
\sum_{k<l}\left(x_kx_l-x^*_kx^*_l\right)^2=n(n-1)-2\sum_{k<l}x_kx_lx^*_kx^*_l,
\]
then we easily get
\begin{align*}
I(\bm{x^{*}};\bm{Y},\bm{B})
&=n\log 2 +\frac{n-1}{2}\sum_{i=1}^{m}\lambda^{(i)}\\
& \quad +\mathbb{E}\log\Bigg(\int_{\mathbb{R}^p}
\exp\left(-\frac{1}{2}\norm{\bm{B}-\sqrt{\frac{\mu}{n}}\bm{v}(\bm{x^{*}})^\top}^2_F\right)
\exp\left(-\frac{\,\|\bm{v}\|^2}{2} \right)
d\bm{v}\Bigg)\\
&\quad-\mathbb{E}[\phi(\bm{x^{*}},\bm{B},\bm{W},\bm{\lambda},\mu,n,p)].
\end{align*}
% \normalsize
\subsection{Proof of Lemma~\ref{lem:join_all_matrices}}
\label{proof_lem_2}
A careful inspection of the expression of $I(\bm{x^{*}};\bm{Y},\bm{B})$ in Lemma \ref{lem:cov_gauss_sbm} shows that the mutual information depends on $\{\lambda^{(i)}\}$ and $\{Z^{(i)}\}$ only through 
\begin{equation*}
\lambda = \sum_{i=1}^m \lambda^{(i)}
\quad \mbox{and} \quad
\sum_{i=1}^m\sqrt{\frac{\lambda^{(i)}}{n}} Z_{kl}^{(i)}, 
\quad \mbox{for all $k<l$}.
\end{equation*}
The proof is simply completed by noting that 
\begin{equation*}
\sum_{i=1}^m\sqrt{\frac{\lambda^{(i)}}{n}} Z_{kl}^{(i)} 
\stackrel{d}{=} 
\sqrt{\frac{\lambda}{n}}\, Z_{kl},
\quad \mbox{for all $k<l$}.
\end{equation*}

\subsection{Proof of Lemma~\ref{lem:inf_cov_sbm}}
\label{proof_lem_3}
Let us define
\begin{align}
\mathcal{F}(\bm{x},\bm{v},\bm{G},\bm{B})
&=\Bigg[\prod\limits_{i=1}^{m}
\prod\limits_{k<l}(\widebar{p}^{(i)}_n+\Delta_n^{(i)}x_kx_l)^{G^{(i)}_{kl}}
(1-\widebar{p}^{(i)}_n-\Delta_n^{(i)}x_kx_l)^{1-G^{(i)}_{kl}}\Bigg]\\
& \hspace{0.4in} \exp\left(-\frac{1}{2}\norm{\bm{B}-\sqrt{\frac{\mu}{n}}\bm{v}\bm{x}^\top}^2_F\right).
\end{align}
Then from the definition of mutual information we have 
\begin{align*}
& I(\bm{x^{*}};\bm{G},\bm{B}) \\
& = \mathbb{E}\log\left\{\hspace{-0.05in}\frac{\int_{\mathbb{R}^p}\mathcal{F}(\bm{x^{*}},\bm{v},\bm{G},\bm{B})\exp(-\frac{1}{2}\|\bm{v}\|^2)d\bm{v}}{\sum_{\bm{x} \in \{\pm 1\}^{n}}
\int_{\mathbb{R}^p}{2^{-n}}\mathcal{F}(\bm{x},\bm{v},\bm{G},\bm{B})\exp(-\frac{1}{2} \|\bm{v}\|^2)d\bm{v}} \right\}\\
& = 
n\log2 
+\mathbb{E}\log\Bigg(\int_{\mathbb{R}^p}
\exp\left(-\frac{1}{2}\norm{\bm{B}-\sqrt{\mu\over n}\bm{v}(\bm{x^{*}})^\top}^2_F\right)
\exp\left(- \frac{\,\|\bm{v}\|^2}{2} \right)d\bm{v}\Bigg)\\
&\quad-\mathbb{E}\log\Bigg\{
\sum_{x \in \{\pm 1\}^{n}}
\int_{\mathbb{R}^p}\exp\Bigg(\sum_{i=1}^{m}\sum_{k<l}
\Bigg[G^{(i)}_{kl}\log\Bigg(\frac{\widebar{p}^{(i)}_n+\Delta_n^{(i)}x_kx_l}{\widebar{p}^{(i)}_n+\Delta_n^{(i)}x^*_kx^*_l}\Bigg)\\
&\quad+(1-G^{(i)}_{kl})\log\Bigg(\frac{1-\widebar{p}^{(i)}_n-\Delta_n^{(i)}x_kx_l}{1-\widebar{p}^{(i)}_n-\Delta_n^{(i)}x^*_kx^*_l}\Bigg)
\Bigg]
-\frac{1}{2}\norm{\bm{B}-\sqrt{\mu\over n}\bm{v}\bm{x}^\top}^2_F\Bigg)
\exp\left(- \frac{\,\|\bm{v}\|^2}{2} \right)d\bm{v}\Bigg\}
% -\frac{1}{2}\|\bm{B}-\sqrt{(\mu/n)}\bm{v}\bm{x}^\top\|^2_F\Bigg)\exp(-\|\bm{v}\|^2/2)d\bm{v}\Bigg\}
\\
& = n\log2
+\mathbb{E}\log\Bigg(\int_{\mathbb{R}^p}
\exp\left(-\frac{1}{2}\norm{\bm{B}-\sqrt{\mu\over n}\bm{v}(\bm{x^{*}})^\top}^2_F\right)
\exp\left(- \frac{\,\|\bm{v}\|^2}{2} \right)d\bm{v}\Bigg)
% +\mathbb{E}\log\Bigg(\int_{\mathbb{R}^p}\exp\left(-\frac{1}{2}\|\bm{B}-\sqrt{(\mu/n)}\bm{v}(\bm{x^{*}})^\top\|^2_F\right)\exp(-\|\bm{v}\|^2/2)d
\Bigg)\\
&-\mathbb{E}\psi(\bm{x^{*}},\bm{B},\bm{G},\mu,n,p).
\end{align*}
 \normalsize
\subsection{Proof of Lemma~\ref{lem:I_SBM_cov}}
\label{proof_lem_4}
We begin by noting that if $x\in\{\pm 1\}$ then
\[
\log\left(c+dx\right)=\frac{1}{2}\log(c+d)(c-d)+\frac{x}{2}\log\frac{c+d}{c-d}.
\]
Now from~\eqref{eq:H_cov} we get
\begin{equation}
\label{eq:simpl_inf_SBM}
\begin{split}
\mathcal{H}_{SBM}'\left(\bm{x},\bm{x^{*}},\bm{G},\bm \lambda,n\right)&=\sum_{k<l}(x_kx_l-x^*_kx^*_l)\left[\sum_{i=1}^{m}\frac{G^{(i)}_{kl}}{2}\log\left(\frac{1+\Delta^{(i)}_n/\widebar{p}^{(i)}_n}{1-\Delta^{(i)}_n/\widebar{p}^{(i)}_n}\right)\right]\\
&\quad + \sum_{k<l}(x_kx_l-x^*_kx^*_l)\left[\sum_{i=1}^{m}\frac{1-G^{(i)}_{kl}}{2}\log\left(\frac{1-\Delta^{(i)}_n/\left(1-\widebar{p}^{(i)}_n\right)}{1+\Delta^{(i)}_n/\left(1-\widebar{p}^{(i)}_n\right)}\right)\right].
\end{split}
\end{equation}
For large values of $n$, there exists some sufficiently small $c_0< \frac{1}{2}$ such that for all $i \in [m]$
% \nb{This is not guaranteed by the current set of assumptions. We need something stronger, such as $\max_i \widebar{p}_n^{(i)}\to 0$.}
\[
\max\left(\frac{\Delta^{(i)}_n}{\widebar{p}^{(i)}_n},\frac{\Delta^{(i)}_n}{1-\widebar{p}^{(i)}_n}\right)~\le~c_0.
\]
Using Taylor approximation for $z \in [0,c_0]$ we have
% \nb{whether the multiplier on the right side is $1$ depends on how close to $0$ the value $c_0$ is.}
\[
\left|\frac{1}{2}\log\left(\frac{1+z}{1-z}\right)-z\right|~\le~z^3.
\]
By triangle inequality
\[
\mathcal{H}_{SBM}'
\left(\bm{x},\bm{x^{*}},\bm{G}, \bm{\lambda}, n\right)
=\sum_{i=1}^{m}\left[\sum_{k<l}\left(\left(x_kx_l-x^*_kx^*_l\right)\left(\frac{\Delta^{(i)}_nG^{(i)}_{kl}}{\widebar{p}^{(i)}_n}-\frac{\Delta^{(i)}_n(1-G^{(i)}_{kl})}{1-\widebar{p}^{(i)}_n}\right)\right)+\mbox{err}^{(i)}_n\right]
\]
where $\mbox{err}^{(i)}_n$ satisfies, for all $i \in [m]$
\begin{equation}
	\label{eq:err-control}
\begin{split}
\left|\mbox{err}^{(i)}_n\right|
& \le 
C_1\left(\frac{\Delta^{(i)}_n}{\widebar{p}^{(i)}_n}\right)^3
\Big( \big|\bm{x}^\top \bm G^{(i)}\bm{x}\big|
+\big| (\bm{x^{*}})^\top \bm G^{(i)}\bm{x^{*}} \big|\Big)
\\
&\quad +C_2\left(\frac{\Delta^{(i)}_n}{1-\widebar{p}^{(i)}_n}\right)^3
\Big(\big| \bm{x}^\top(\bm{1}\bm{1}^\top-\bm G^{(i)} ) \bm{x}\big|+\big| (\bm{x^{*}})^\top(\bm{1}\bm{1}^\top-\bm G^{(i)} )\bm{x^{*}}\big|\Big)
\end{split}
\end{equation}
with $C_1$ and $C_2$ absolute positive constants.
Furthermore
\begin{equation}
\begin{split}
& \sum_{k<l}\left(x_kx_l-x^*_kx^*_l\right)
\left(\frac{\Delta^{(i)}_nG^{(i)}_{kl}}{\widebar{p}^{(i)}_n}
-\frac{\Delta^{(i)}_n (1-G^{(i)}_{kl})}{1-\widebar{p}^{(i)}_n}\right)\\
& \qquad\quad =\sum_{k<l}\left(x_kx_l-x^*_kx^*_l\right)
\left(\widetilde{G}^{(i)}_{kl}
+\frac{(\Delta^{(i)}_n)^2x^*_kx^*_l}{\widebar{p}^{(i)}_n (1-\widebar{p}^{(i)}_n)}\right)\\
&\qquad \quad =-\frac{n-1}{2}\lambda^{(i)}+\sum_{k<l}\widetilde{G}^{(i)}_{kl}(x_kx_l-x^*_kx^*_l)
% \\
% &
+\sum_{k<l}\frac{\,\lambda^{(i)}}{n}x_kx_lx^*_kx^*_l.
\end{split}
\end{equation}
This implies
\begin{equation}
\label{eq:lambda_new}
\mathcal{H}_{SBM}'\left(\bm{x},\bm{x^{*}},\bm{G},\bm \lambda,n\right)
=-\frac{n-1}{2}\sum_{i=1}^{m}\lambda^{(i)}
+\mathcal{H}'(\bm{x},\bm{x^{*}},\widetilde{\bm{G}},\bm{\lambda},n )
+\sum_{i=1}^{m}\mbox{err}^{(i)}_n.
\end{equation}
where $\mbox{err}^{(i)}_n$ satisfies \eqref{eq:err-control}
for all $i \in [m]$.
% \begin{equation}
% \begin{split}
% \left|\mbox{err}^{(i)}_n\right|& \le~C_1\left(\frac{\Delta^{(i)}_n}{\widebar{p}^{(i)}_n}\right)^3\left(\left|\left\langle\bm{x},\bm A^{(i)}\bm{x}\right\rangle\right|+\left|\left\langle\bm{x^{*}},\bm A^{(i)}\bm{x^{*}}\right\rangle\right|\right)\\
% &~+C_2\left(\frac{\Delta^{(i)}_n}{1-\widebar{p}^{(i)}_n}\right)^3\left(\left|\left\langle\bm{x},\left(\bm{1}\bm{1}^\top-\bm A^{(i)}\right)\bm{x}\right\rangle\right|+\left|\left\langle\bm{x^{*}},\left(\bm{1}\bm{1}^\top-\bm A^{(i)}\right)\bm{x^{*}}\right\rangle\right|\right)
% \end{split}
% \end{equation}
% where $C_1$ and $C_2$ are absolute constants.
Using~\eqref{eq:H_SBM_cov} we get
\begin{align*}
& \mathcal{H}_{SBM}(\bm{x},\bm{x^{*}},\bm{v},\bm{G},\bm{B},\bm \lambda,\mu,n,p)) \\
& \quad = -\frac{n-1}{2}\sum_{i=1}^{m}\lambda^{(i)}
+\mathcal{H}'(\bm{x},\bm{x^{*}},\widetilde{\bm{G}},\bm{\lambda},n)
-\frac{p}{2}\norm{\bm{B}-\sqrt{\mu\over n}\bm{v}\bm{x}^\top}^2_F
+\sum_{i=1}^{m}\mbox{err}^{(i)}_n.
% (\bm{x},\bm{x^{*}},\bm{G},\bm{\lambda}).
\end{align*}
Furthermore, by Remark 5.4 of \cite{AbbeMonYash} and~\eqref{eq:cov_phi_g}, 
we obtain the following.
\begin{align*}
& \mathbb{E}\psi(\bm{x^{*}},\bm{B},\bm{G},\bm \lambda, \mu,n,p)
\\
&= \hspace{-0.04in}-\frac{n-1}{2}\sum_{i=1}^{m}\lambda^{(i)} \hspace{-0.04in}+ \mathbb{E}\log\Bigg\{\sum_{\bm{x} \in \{\pm 1\}^n}\int_{\mathbb{R}^p}\exp(\mathcal{H}(\bm{x},\bm{x^{*}},\bm \lambda, \bm{v},\bm{\tilde{G}},\bm{B},\bm{\lambda},\mu,n,p))
\exp\left(- \frac{\|\bm{v}\|^2}{2} \right)d\bm{v}\Bigg\}\\
&+ O\left(\sum_{i=1}^{m}\frac{n\left(\lambda^{(i)}\right)^{3/2}}{\sqrt{n\widebar{p}^{(i)}_n (1-\widebar{p}^{(i)}_n )}}\right)
\end{align*}
We complete the proof by \revsag{applying}
Lemma~\ref{lem:inf_cov_sbm}.
% we get the result.

\subsection{Proof of Lemma~\ref{lem:cov_connect_phi}}
\label{proof_lem_5}
We begin by observing that  
\begin{equation}
\label{eq:first_mom_comp_cov}
\mathbb{E}\big[\widetilde{G}^{(i)}_{kl} \big|\;\bm{x^{*}}\big]
= 0
= \mathbb{E}\Big[Z^{(i)}_{kl}\sqrt{\frac{\lambda^{(i)}}{n}}
\,\Big|\,\bm{x^{*}}\Big].
\end{equation}
% Further for the second moments
Following the arguments of Lemma 5.5 of \cite{AbbeMonYash}, we obtain
\begin{equation}
\label{eq:second_mom_comp_cov}
\begin{split}
& \bigg|
\mathbb{E}\Big[ (\widetilde{G}^{(i)}_{kl} )^2
-(Z^{(i)}_{kl})^2\frac{\lambda^{(i)}}{n}\, \Big|\,\bm{x^{*}}
\Big]
\bigg|
\\
& \qquad = \left|\frac{(\Delta^{(i)}_n )^2}{(\widebar{p}^{(i)}_n )^2 (1-\widebar{p}^{(i)}_n )^2} (\widebar{p}^{(i)}_n+\Delta^{(i)}_nx^*_ix^*_j)(1-\widebar{p}^{(i)}_n+\Delta^{(i)}_nx^*_ix^*_j)-\frac{\lambda^{(i)}}{n}\right|\\
& \qquad \le \frac{\lambda^{(i)}}{n}\left(\sqrt{\frac{\lambda^{(i)}}{n\widebar{p}^{(i)}_n
(1-\widebar{p}^{(i)}_n )}}+\frac{\lambda^{(i)}}{n}\right),
\end{split}
\end{equation}
and
% For the third and fourth moments again following the arguments of Lemma 5.5 of \cite{AbbeMonYash} we get
\begin{align}
\label{eq:third_mom_comp_cov}
\left|\mathbb{E}\big[(\widetilde{G}^{(i)}_{kl} )^3 \,|\,\bm{x^{*}}\big]\right|
& \leq
\frac{3(\lambda^{(i)})^{3/2}}{n^{3/2}\sqrt{\widebar{p}^{(i)}_n
(1-\widebar{p}^{(i)}_n)}}\, ,\\
% \end{equation}
% and
% \begin{equation}
\label{eq:fourth_mom_comp_cov}
\mathbb{E}\big[ (\widetilde{G}^{(i)}_{kl} )^4\,|\,\bm{x^{*}}\big]
& \leq
\frac{2 (\lambda^{(i)} )^2}{n^2\widebar{p}^{(i)}_n (1-\widebar{p}^{(i)}_n )}\, . 
\end{align} 
Now we shall use Lindeberg's generalization theorem (Theorem 5.6 of \cite{AbbeMonYash}) to establish the desired result. 
We consider the collections of random variables $\{\widetilde{G}^{(i)}_{kl} \}$ 
and $ \{Z^{(i)}_{kl}\sqrt{\lambda^{(i)}/n}\}$. 
To this end, we regard the function $\phi(\bm{x^{*}},\bm{B},\bm{V},\bm \lambda,\mu,n,p)$ as a function of $\bm{V}=\left(\bm V^{(1)},\bm V^{(2)},...,\bm V^{(m)}\right)$.
% , a collection of $m$ matrices of dimension $n \times n$.
% and apply Lemma 5.6 of \cite{AbbeMonYash} to it.
Define
\[
\partial^{r}_{k,l;i}\phi:=\frac{\partial^r}{\partial(V^{(i)}_{kl})^r}
\phi(\bm{x^{*}},\bm{B},\bm{V},\bm \lambda,\mu,n,p).
\]
Let us consider the measure on $\{\pm 1\}^n$ given by
\[
\bm{m}(\bm{x}) = \frac{\int_{\mathbb{R}^p}\exp(\mathcal{H}(\bm{v},\bm{x^{*}},\bm{v},\bm{W},\bm{B},\bm{\lambda},\mu,n,p))\exp(-\|\bm{v}\|^2/2)d\bm{v}}{\sum_{x \in \{\pm 1\}^n}\int_{\mathbb{R}^p}\exp(\mathcal{H}(\bm{x},\bm{v},\bm{V},\bm{B},\bm{\lambda},\mu,n,p))\exp(-\|\bm{v}\|^2/2)d\bm{v}},
\quad \bm{x}\in \{\pm 1\}^n.
\]
Then it is easy to verify using induction that for all $i,k,l$, and any $r>1$
\begin{align*}
\partial^{1}_{k,l;i}\phi
& =\mathbb{E}_{\bm{m}}\left[x_kx_l-x^*_kx^*_l\right],\\
\partial^{r}_{k,l;i}\phi
& =\mathbb{E}_{\bm{m}}\left[\left(x_kx_l-x^*_kx^*_l\right)
-\mathbb{E}_{\bm{m}}\left[x_kx_l-x^*_kx^*_l\right]\right]^r.
\end{align*}
We note that in this case also the expressions of the partial derivatives \revsag{are} equivalent to the polynomial representations $p_r$ mentioned in Lemma 5.5 of~\cite{AbbeMonYash}. Since $|x_kx_l-x^*_kx^*_l|~\le~2$ for all $i,k,l$ there exists a constant $C$ such that
\[
|\partial^{r}_{k,l;i}\phi|~\le~C,
\]
for all $r \le 4$. 
Then, using Theorem 5.6 of \cite{AbbeMonYash} and \eqref{eq:first_mom_comp_cov},~\eqref{eq:second_mom_comp_cov},~\eqref{eq:third_mom_comp_cov} and \eqref{eq:fourth_mom_comp_cov} we get
\begin{equation}
\label{eq:main_eq_comb_Gauss_SBM_cov}
\mathbb{E}
% _{\bm{x^{*}},\bm{B},\tilde{\bm{G}}}
\phi(\bm{x^{*}},\bm{B},\tilde{\bm{G}},\bm{\lambda},\mu,n,p)
=
\mathbb{E}
% _{\bm{x^{*}},\bm{B},\bm{Z}}
\phi(\bm{x^{*}},\bm{B},\bm{W},\bm{\lambda},\mu,n,p)+~O\left(\sum_{i=1}^{m}\frac{n(\lambda^{(i)})^{3/2}}{\sqrt{n\widebar{p}^{(i)}_n(1-\widebar{p}^{(i)}_n)}}\right).
\end{equation}
\section{Proofs of Results in Section~\ref{sec_5}}

% Consider $\mathcal{G}$ defined in~\eqref{eq:cal_g} and $\mathcal{L}$ defined in~\eqref{eq:linear_transformed variables}. We also consider $\lambda^{(i)}$ from \eqref{eq:lambdalimit} and $\lambda$ from \eqref{eq:lambda}. Further recall $r^{(1)},\ldots,r^{(m)} \in (0,1)$ from \eqref{eq:lambdaratio}. Let us define
% \begin{equation}
% p^{(i)}_n=\widebar{p}^{(i)}_n+\Delta^{(i)}_n \quad\quad q^{(i)}_n=\widebar{p}^{(i)}_n-\Delta^{(i)}_n.
% \end{equation}
We start with some definitions.
For $i\in [m]$, we let $E(\bm{G}^{(i)})$ denote all ${n\choose 2}$ unordered pairs of nodes in the $i$th graph.
For a pair of vertices $e=(k,l)$, consider the following random variables.
\begin{equation*}
x_e:=x^*_kx^*_l \quad \mbox{and} \quad \ell^{(i)}_e:=L^{(i)}_{k,l}, 
\end{equation*} 
where $L^{(i)}_{k,l}=2G^{(i)}_{k,l}-1$. 
Further, we define
% the function $\pi^{(i)}(\lambda;y_e,x_e):$ is defined as
\begin{equation}
\label{eq:def_pi} 
\pi^{(i)}(\lambda;\ell^{(i)}_e,x_e):=\begin{cases}\widebar{p}^{(i)}_n+x_e\Delta^{(i)}_n & \mbox{if $\ell^{(i)}_e=+1$}\\1-(\widebar{p}^{(i)}_n+x_e\Delta^{(i)}_n) & \mbox{if $\ell^{(i)}_e=-1$.}\end{cases}
\end{equation}
and 
\begin{equation}
\label{eq:marg_prob}
p(y_e)=\mathbb{P}_{x_e}\left(x_e=y_e\right).
\end{equation}
By definitions in Section \ref{sec:prelim}
\begin{equation*}
	\Delta_n^{(i)} = \sqrt{ \frac{\lambda r^{(i)}\widebar{p}_n^{(i)}(1-\widebar{p}_n^{(i)})}{n} },
\end{equation*}
then it is easy to show that
\begin{equation}
\label{eq:deriv_prob}
\frac{d\pi^{(i)}(\lambda;\ell^{(i)}_e,x_e)}{d\lambda}
=\frac{1}{2}\sqrt{\frac{\widebar{p}^{(i)}_n
(1-\widebar{p}^{(i)}_n )r^{(i)}}{n\lambda}}\;x_e\ell^{(i)}_e.
\end{equation}
%$\tx_e=\tx_k\tx_l$ if $e=(k,l)$, and
%\[
%p^{(i)}_{e}(\lambda;\bm{x^{*}}):=\widebar{p}^{(i)}_n+\Delta^{(i)}_n(\lambda)\tx_e,
%\]
%where
%\[
%\widebar{p}^{(i)}_n:=\frac{a^{(i)}+b^{(i)}}{2n}, \quad \Delta^{(i)}_n(\theta):=\sqrt{\frac{\theta\widebar{p}^{(i)}_n(1-\widebar{p}^{(i)}_n)}{n}}.
%\]
%This implies for $e=(k,l)$
%\[
%p^{(i)}_{e}(\lambda;\bm{x^{*}})=\mathbb{P}\left[A^{(i)}_{k,l}=1|\tx_k=x^*_k,\tx_l=x^*_l\right].
%\]
%$\mathbb{P}\left[A^{(i)}_{k,l}=1|\tx_k=x^*_k,\tx_l=x^*_l\right]=p^{(i)}_{kl}(r^{(i)}\lambda;\bm{x^{*}})$ where $A^{(i)}_{k,l}$ is defined in~\eqref{eq:def_SBM_cov}.
By definition, it is immediate that $\MMSE_n(\lambda,\mu)$ defined in~\eqref{eq:MMSE_cov_graph} is a function of $\lambda$ and $\mu$ only. For the purpose of this lemma, we consider $\mu \in (0,\infty)$ to be a fixed number
% satisfying $\mu^2/c < 1$ \nb{do we really need this condition?}
and study the function as a function of $\lambda$ only.
We note that
\begin{equation}
\label{eq:joint_cond_entrop}
H(\bm{x^{*}}|\bm{G},\bm{B})
% =H(\bm{x^{*}}|\mathcal{G},\bm{B})
=-\mathbb{E}_{\bm{x^{*}},\bm{G},\bm{B}}\left[\log \pi\left(\bm{x^{*}}|\bm{G},\bm{B}\right)\right],
\end{equation}
where $\pi\left(\bm{x^{*}}|\bm{G},\bm{B}\right)$ is the posterior density of $\bm{x^{*}}$ given $\bm{G}$ and $\bm{B}$. 
By the definition of mutual information,
 % it is easy to see that
\[
I\left(\bm{x^{*}};\bm{G},\bm{B}\right)
% =I\left(\bm{x^{*}};\mathcal{G},\bm{B}\right)
=H(\bm{x^{*}})-H(\bm{x^{*}}|\bm{G},\bm{B}),
\]
where $H(\bm{x^{*}})$ is the entropy of $\bm{x^{*}}$. It is easy to observe that $H(\bm{x^{*}})$ equals $n\log 2$ as $\bm{x^{*}}$ is a $n$ vector made of i.i.d Rademacher random variables. Therefore
\[
\frac{dI\left(\bm{x^{*}};\bm{G},\bm{B}\right)}{d\lambda}=
-\frac{dH(\bm{x^{*}}|\bm{G},\bm{B})}{d\lambda}.
\]
%Let us recall the random variables $L^{(i)}_{k,l}=2A^{(i)}_{k,l}-1$. 
By \eqref{eq:cond_entrop_l} we have
\[
H(\bm{x^{*}}|\bm{G},\bm{B})=H(\bm{x^{*}}|\mathcal{L},\bm{B}),
\] 
where $H(\bm{x^{*}}|\mathcal{L},\bm{B})$ is the conditional entropy of $\bm{x^{*}}$ given $\mathcal{L}$ and $\bm{B}$. 
Finally, for $e=(k,l)$ we define
\begin{equation}
\label{eq:del_t}
\mathcal{L}^{(i)}_{-e}=\mathcal{L}\setminus \{L^{(i)}_{k,l} \}.
\end{equation}
%and 
%\begin{equation}
%\label{eq:del_t}
%p^{(i)}_{kl}\left(\ell^{(i)}_{k,l}|x^*_kx^*_l\right)=\mathbb{P}\left(A^{(i)}_{k,l}=(\ell^{(i)}_{k,l}+1)/2|\bm{x^{*}}_k=x^*_k, \bm{x^{*}}_l=x^*_l\right),
%\end{equation}
%where $\ell^{(i)}_{k,l} \in \{\pm 1\}$.
%
\subsection{Proof of Lemma~\ref{lem:diff_cov_inf}}
\label{proof_i_mmse}
%We recall $p^{(i)}_{kl}\left(L^{(i)}_{k,l}=1|x^*_kx^*_l\right)=\widebar{p}^{(i)}_n+\sqrt{\widebar{p}^{(i)}_n\left(1-\widebar{p}^{(i)}_n\right)r^{(i)}\lambda/n}\;x^*_kx^*_l$. This implies
%\[
%\frac{dp^{(i)}_{kl}\left(\ell^{(i)}_{k,l}|x^*_kx^*_l\right)}{d\lambda}=\frac{1}{2}\sqrt{\left(\widebar{p}^{(i)}_n\left(1-\widebar{p}^{(i)}_n\right)r^{(i)}/n\lambda\right)}\;x^*_kx^*_l\ell^{(i)}_{k,l}.
%\] 
Let us define $\widehat{x}_{e}(\mathcal{L}^{(i)}_{-e},\bm{B})=
\mathbb{E}[x_e\big|\mathcal{L}^{(i)}_{-e},\bm{B}]$ and 
\begin{equation}
\label{eq:mathcal_p}
\mathcal{P}^{(i)}_{e}(\mathcal{L}^{(i)},\bm{B},y_e,\ell^{(i)}_e)
=\mathbb{P}(x_e=y_e |\mathcal{L}^{(i)}_{-e},\bm{B} )
\log\left(\sum\limits_{x \in \{\pm 1\}}\pi^{(i)}(\lambda;\ell^{(i)}_e,x)
p(x|\mathcal{L}^{(i)}_{-kl},\bm{B})\right).
\end{equation}
Using Lemma~\ref{lem:lem_diff} below in Section \ref{sec:lem_diff}, we have
\begin{equation}
\begin{split}
& \frac{dH(\bm{x^{*}}|\mathcal{L},\bm{B})}{d\lambda}\\
& = \frac{1}{2}\sum\limits_{i=1}^{m}
\sqrt{\frac{\widebar{p}^{(i)}_n(1-\widebar{p}^{(i)}_n)r^{(i)}}{n\lambda}}
\sum\limits_{e \in E(G^{(i)})}
\sum\limits_{\substack{y^{(i)}_e\in\{\pm 1\} \\ \ell^{(i)}_{e} \in \{\pm 1\}}}
\ell^{(i)}_{e}y^{(i)}_e\,
\mathbb{E}_{\mathcal{L}^{(i)}_{-e},\bm B}
\left[\mathcal{P}^{(i)}_{e}(\mathcal{L}^{(i)},\bm{B},y^{(i)}_e,\ell^{(i)}_e)\right]\\
& \quad
-\frac{1}{2}\sum\limits_{i=1}^{m}
% \sqrt{\left(\widebar{p}^{(i)}_n\left(1-\widebar{p}^{(i)}_n\right)r^{(i)}/n\lambda\right)}
\sqrt{\frac{\widebar{p}^{(i)}_n(1-\widebar{p}^{(i)}_n)r^{(i)}}{n\lambda}}
\sum\limits_{e \in E(G^{(i)})}\sum\limits_{\substack{y^{(i)}_e\in\{\pm 1\} \\ \ell^{(i)}_{e} \in \{\pm 1\}}}
\ell_{e}^{(i)} y^{(i)}_e
p(y^{(i)}_e)
% \pi^{(i)}(\lambda;\ell^{(i)}_e,x)
\log \pi^{(i)}(\lambda;\ell^{(i)}_e, y_e)\\
& =\frac{1}{2}\sum\limits_{i=1}^{m}
% \sqrt{\left(\widebar{p}^{(i)}_n\left(1-\widebar{p}^{(i)}_n\right)r^{(i)}/n\lambda\right)}
\sqrt{\frac{\widebar{p}^{(i)}_n(1-\widebar{p}^{(i)}_n)r^{(i)}}{n\lambda}}
% \hspace{-0.08in}
\sum\limits_{e \in E(G^{(i)})}
% \hspace{-0.13in}
\mathbb{E}_{\mathcal{L}^{(i)}_{-e},\bm{B}}
\left\{ \widehat{x}_{e}(\mathcal{L}^{(i)}_{-e},\bm{B})
% \hspace{-0.04in}
\log
% \hspace{-0.03in}
\left[\frac{\sum_x
% \limits_{x \in \{\pm 1\}}
\pi^{(i)}(\lambda;1,x)p(x|\mathcal{L}^{(i)}_{-e},\bm{B})}{\sum_x
% \limits_{x \in \{\pm 1\}}
\pi^{(i)}(\lambda;-1,x)p(x|\mathcal{L}^{(i)}_{-e},\bm{B})}
\right]\right\}\\
&\quad -\frac{1}{4}\sum\limits_{i=1}^{m}
% \sqrt{\left(\widebar{p}^{(i)}_n\left(1-\widebar{p}^{(i)}_n\right)r^{(i)}/n\lambda\right)}
\sqrt{\frac{\widebar{p}^{(i)}_n(1-\widebar{p}^{(i)}_n)r^{(i)}}{n\lambda}}
{n \choose 2}\log\left\{\frac{\pi^{(i)}(\lambda;-1,-1)\pi^{(i)}(\lambda;+1,+1)}{\pi^{(i)}(\lambda;+1,-1)\pi^{(i)}(\lambda;-1,+1)}\right\}.
\end{split}
\end{equation}
Next we observe 
\begin{equation}
\label{eq:one_way_p}
\begin{split}
\sum_{x \in \{\pm 1\}}\pi^{(i)}(\lambda;1,x)
p(x|\mathcal{L}^{(i)}_{-e},\bm{B})
& =p^{(i)}_n p(1|\mathcal{L}^{(i)}_{-e},\bm{B})
+q^{(i)}_n p(-1|\mathcal{L}^{(i)}_{-e},\bm{B}) \\
& = 
p^{(i)}_n p(1|\mathcal{L}^{(i)}_{-e},\bm{B})
+q^{(i)}_n[1-p(1|\mathcal{L}^{(i)}_{-e},\bm{B} ) ]\\
& =
q^{(i)}_n+ (p^{(i)}_n-q^{(i)}_n)
p(1|\mathcal{L}^{(i)}_{-e},\bm{B} ).
\end{split}
\end{equation}
We also have
\begin{equation}
\label{eq:two_way_p}
\begin{split} 
\sum_{x \in \{\pm 1\}}
\pi^{(i)}(\lambda;1,x)p(x|\mathcal{L}^{(i)}_{-e},\bm{B} )
& = 
p^{(i)}_n [1 - p(-1|\mathcal{L}^{(i)}_{-e},\bm{B} )]
+q^{(i)}_n p(-1|\mathcal{L}^{(i)}_{-e},\bm{B} )\\
& = p^{(i)}_n - (p^{(i)}_n-q^{(i)}_n )
p(-1|\mathcal{L}^{(i)}_{-e},\bm{B} ).
\end{split}
\end{equation}
Combining~\eqref{eq:one_way_p} and~\eqref{eq:two_way_p} we get
\begin{equation}
\label{eq:first_rel}
2\sum_{x \in \{\pm 1\}}
\pi^{(i)}(\lambda;1,x)p(x|\mathcal{L}^{(i)}_{-e},\bm{B} )
= (p^{(i)}_n+q^{(i)}_n )
+\widehat{x}_{e}(\mathcal{L}^{(i)}_{-e},\bm{B}) (p^{(i)}_n-q^{(i)}_n ).
\end{equation}
By similar calculations, we also get
\begin{equation}
\label{eq:second_rel}
2\sum\limits_{x \in \{\pm 1\}}\pi^{(i)}(\lambda;-1,x)
p(x|\mathcal{L}^{(i)}_{-e},\bm{B})
= (2-p^{(i)}_n-q^{(i)}_n )-\widehat{x}_{e}(\mathcal{L}^{(i)}_{-e},\bm{B})
(p^{(i)}_n-q^{(i)}_n ).
\end{equation}
From~\eqref{eq:first_rel} and~\eqref{eq:second_rel} we get 
\begin{equation}
\begin{split}
\frac{\sum_{x \in \{\pm 1\}}
\pi^{(i)}(\lambda;1,x)p(x|\mathcal{L}^{(i)}_{-e},\bm{B})}{\sum_{x \in \{\pm 1\}}\pi^{(i)}(\lambda;-1,x)p(x|\mathcal{L}^{(i)}_{-e},\bm{B} )}
&= \frac{(p^{(i)}_n+q^{(i)}_n)+\widehat{x}_{e}(\mathcal{L}^{(i)}_{-e},\bm{B})(p^{(i)}_n-q^{(i)}_n)}{(2-p^{(i)}_n-q^{(i)}_n)-\widehat{x}_{e}(\mathcal{L}^{(i)}_{-e},\bm{B})(p^{(i)}_n-q^{(i)}_n)}\\
&=\frac{\widebar{p}^{(i)}_n}{1-\widebar{p}^{(i)}_n}
\left[
\frac{1+(\Delta^{(i)}_n/\widebar{p}^{(i)}_n )\widehat{x}_{e}(\mathcal{L}^{(i)}_{-e},\bm{B} )}{1- (\Delta^{(i)}_n/(1-\widebar{p}^{(i)}_n ) )\widehat{x}_{e}(\mathcal{L}^{(i)}_{-e},\bm{B} )}
\right].
\end{split}
\end{equation}
In addition,
it is easy to observe that 
\[
\frac{\pi^{(i)}(\lambda;-1,-1)\pi^{(i)}(\lambda;+1,+1)}{\pi^{(i)}(\lambda;+1,-1)\pi^{(i)}(\lambda;-1,+1)}
=
\frac{(1+(\Delta^{(i)}_n/\widebar{p}^{(i)}_n )) (1+(\Delta^{(i)}_n/(1-\widebar{p}^{(i)}_n ) ) )}{ (1-  (\Delta^{(i)}_n/\widebar{p}^{(i)}_n  ) )  (1-  (\Delta^{(i)}_n/ (1-\widebar{p}^{(i)}_n ) ) )}.
\]
Since we have $|\Delta^{(i)}_n/\widebar{p}^{(i)}_n |, 
|\Delta^{(i)}_n/(1-\widebar{p}^{(i)}_n) | 
\le \sqrt{\lambda r^{(i)}/(n\widebar{p}_n^{(i)}(1-\widebar{p}_n^{(i)}))}\rightarrow 0$, 
and 
$|\widehat{x}_{e}(\mathcal{L}^{(i)}_{-e},\bm{B})| \le 1$, 
% we obtain the following bounds by Taylor expansion we have
for each $\lambda_{max} \in \mathbb{R}$ there exists a $n_0(\lambda_{max})$ such that for $n \ge n_0(\lambda_{max})$ we have for all $1 \le i \le m$ and $e \in E(G^{(i)})$
\[
\left|\log\left[\frac{\sum_{x \in \{\pm 1\}}\pi^{(i)}(\lambda;1,x)p(x|\mathcal{L}^{(i)}_{-e},\bm{B})}{\sum_{x \in \{\pm 1\}}\pi^{(i)}(\lambda;-1,x)p(x|\mathcal{L}^{(i)}_{-e},\bm{B})}\right]
-B^{(i)}_0
-\frac{\Delta^{(i)}_n \widehat{x}_{e}(\mathcal{L}^{(i)}_{-e},\bm{B})}{\widebar{p}^{(i)}_n(1-\widebar{p}^{(i)}_n)}\right| \le C_1\frac{\lambda r^{(i)}}{n\widebar{p}^{(i)}_n (1-\widebar{p}^{(i)}_n )},
\]
and
\[
\left|\log\left[\frac{\pi^{(i)}(\lambda;-1,-1)\pi^{(i)}(\lambda;+1,+1)}{\pi^{(i)}(\lambda;+1,-1)\pi^{(i)}(\lambda;-1,+1)}\right]
-\frac{2\Delta^{(i)}_n}{\widebar{p}^{(i)}_n(1-\widebar{p}^{(i)}_n)}\right| \le C_2\frac{\lambda r^{(i)}}{n\widebar{p}^{(i)}_n (1-\widebar{p}^{(i)}_n )},
\]
where $B^{(i)}_0=\log(\widebar{p}^{(i)}_n/(1-\widebar{p}^{(i)}_n ))$ and $C_1,C_2$ are positive constants depending on $\lambda_{max}$. 
We observe that $\mathbb{E}[\widehat{x}_{e}(\mathcal{L}^{(i)}_{-e},\bm{B} )]=\mathbb{E}\left[x_e\right]=0$. 
This implies
% \nb{an additional $r^{(i)}$ multiplier can be squeezed in for each term on the RHS if necessary.}
\[
\left|\frac{1}{n}\,\frac{dH(\bm{x^{*}}|\mathcal{L},\bm{B})}{d\lambda}
+\frac{1}{2n^2}\sum\limits_{i=1}^{m}\sum\limits_{e \in E(G^{(i)})}r^{(i)}\left(1-\mathbb{E}[\widehat{x}_{e}(\mathcal{L}^{(i)}_{-e},\bm{B})]^2\right)\right| \le C_2\sum\limits_{i=1}^{m}\sqrt{\frac{\lambda r^{(i)}}{n\widebar{p}^{(i)}_n(1-\widebar{p}^{(i)}_n)}}.
\] 
Let $\widehat{x}_{e}(\mathcal{L},\bm{B})=\mathbb{E}\left[x_e|\mathcal{L},\bm{B}\right]$. 
Recall that $L^{(i)}_{e}=L^{(i)}_{k,l}$ for $e=(k,l)$, then by Bayes' formula we have
\begin{equation}
\begin{split}
p^{(i)}\left(y_e|\mathcal{L},\bm{B}\right)
& =\frac{p^{(i)}(L^{(i)}_{e}, \mathcal{L}^{(i)}_{-e},\bm{B}, x_e=y_e )}{\sum_{x \in \{\pm 1\}}p^{(i)}(L^{(i)}_{e}, \mathcal{L}^{(i)}_{-kl},\bm{B}, x_e=x )}\\
& \overset{(1)}{=} \frac{\pi^{(i)}(\lambda;L^{(i)}_{e},y_e)p^{(i)}(y_e|\mathcal{L}^{(i)}_{-e},\bm{B} )}{\sum_{x \in \{\pm 1\}}\pi^{(i)}(\lambda;L^{(i)}_{e},x)p^{(i)}(x|\mathcal{L}^{(i)}_{-e},\bm{B})}. 
\end{split}
\end{equation}
Here, equality (1) follows as conditional on $x_e$, $L^{(i)}_{e}$ is independent of $\mathcal{L}^{(i)}_{-e}$.
Let us define 
\[
b^{(i)}(L^{(i)}_{e})=\frac{\pi^{(i)}(\lambda;L^{(i)}_{e},+1)-\pi^{(i)}(\lambda;L^{(i)}_{e},-1)}{\pi^{(i)}(\lambda;L^{(i)}_{e},+1)+\pi^{(i)}(\lambda;L^{(i)}_{e},-1)}.
\]
We note that 
\begin{equation}
\label{eq:b_def_1}
\frac{\widehat{x}_{e}(\mathcal{L}^{(i)}_{-e},\bm{B})+b^{(i)}(L^{(i)}_{e})}{2} ~=~ \frac{\pi^{(i)}(\lambda;L^{(i)}_{e},+1)p^{(i)}(+1|\mathcal{L}^{(i)}_{-e},\bm{B})-\pi^{(i)}(\lambda;L^{(i)}_{e},-1)p^{(i)}(-1|\mathcal{L}^{(i)}_{-e},\bm{B})}{\pi^{(i)}(\lambda;L^{(i)}_{e},+1)+\pi^{(i)}(\lambda;L^{(i)}_{e},-1)}
\end{equation}
and
\begin{equation}
\label{eq:b_def_2}
\frac{1+b^{(i)}(L^{(i)}_{e})\widehat{x}_{e}(\mathcal{L}^{(i)}_{-e},\bm{B})}{2} ~=~ \frac{\pi^{(i)}(\lambda;L^{(i)}_{e},+1)p^{(i)}(+1|\mathcal{L}^{(i)}_{-e},\bm{B})+\pi^{(i)}(\lambda;L^{(i)}_{e},-1)p^{(i)}(-1|\mathcal{L}^{(i)}_{-e},\bm{B})}{\pi^{(i)}(\lambda;L^{(i)}_{e},+1)+\pi^{(i)}(\lambda;L^{(i)}_{e},-1)}.
\end{equation}
Then from~\eqref{eq:b_def_1} and~\eqref{eq:b_def_2} we get
\begin{equation}
\label{eq:cond_expec_delete_connect}
\widehat{x}_{e}(\mathcal{L},\bm{B})=\frac{\widehat{x}_{e}(\mathcal{L}^{(i)}_{-e},\bm{B})+b^{(i)}(L^{(i)}_{e})}{1+b^{(i)}(L^{(i)}_{e})\widehat{x}_{e}(\mathcal{L}^{(i)}_{-e},\bm{B})}.
\end{equation}
From the definition of $\pi^{(i)}(\lambda;L^{(i)}_{e},x_e)$ it follows that
\[
b^{(i)}(L^{(i)}_{e})=\begin{cases}\sqrt{(1-\widebar{p}^{(i)}_n)\lambda r^{(i)}/(n\widebar{p}^{(i)}_n)} & \mbox{if $L^{(i)}_{e}=1$,}\\ -\sqrt{\widebar{p}^{(i)}_n\lambda r^{(i)}/(n(1-\widebar{p}^{(i)}_n))} & \mbox{if $L^{(i)}_{e}=-1$.}
\end{cases}
\]
This in particular gives us $|b^{(i)}(L^{(i)}_{e})| \le \sqrt{\lambda r^{(i)}/(n\widebar{p}^{(i)}_n (1-\widebar{p}^{(i)}_n))} \to 0$.
 % for large values of $n$.
In addition, $|\widehat{x}_{e}(\mathcal{L}^{(i)}_{-e},\bm{B} ) |\le1$. 
Thus 
\begin{equation}
\label{eq:connect_cond_expec}
\begin{aligned}
\left|\widehat{x}_{e}(\mathcal{L},\bm{B})-
\widehat{x}_{e}(\mathcal{L}^{(i)}_{-e},\bm{B} )\right| 
& = \left|\frac{b^{(i)}(L^{(i)}_{e})(1-\widehat{x}_{e}(\mathcal{L}^{(i)}_{-e},\bm{B})^2)}{1+b^{(i)}(L^{(i)}_{e})\,\widehat{x}_{e}(\mathcal{L}^{(i)}_{-e},\bm{B} )^2} \right|\\
& \le |b^{(i)}(L^{(i)}_{e})| 
\le \sqrt{\frac{\lambda r^{(i)}}{n\widebar{p}^{(i)}_n(1-\widebar{p}^{(i)}_n )}}\,.
\end{aligned}
\end{equation}
Using $|\widehat{x}_{e}(\mathcal{L}^{(i)}_{-e},\bm{B} ) |\le 1$ and~\eqref{eq:connect_cond_expec} we get 
%\nb{Where does the extra term $\frac{1}{n}$ come from?}
\[
\left|\frac{1}{n}\,\frac{dH(\bm{x^{*}}|\mathcal{L},\bm{B})}{d\lambda}+\frac{1}{2n^2}\sum\limits_{i=1}^{m}\sum\limits_{e \in E(G^{(i)})}r^{(i)}\left(1-\mathbb{E}\left[\widehat{x}_{e}(\mathcal{L},\bm{B})\right]^2\right)\right| \le C_2\sum\limits_{i=1}^{m}\left(\sqrt{\frac{\lambda r^{(i)}}{n\widebar{p}^{(i)}_n(1-\widebar{p}^{(i)}_n )}}\right).
\]
It is easy to observe that for $e=(k,l)$
\[
\left(1-\mathbb{E}\left[\widehat{x}_{e}(\mathcal{L},\bm{B})\right]^2\right)=\mathbb{E}\left[\left(\bm{x^{*}}_k\bm{x^{*}}_l-\mathbb{E}(\bm{x^{*}}_k\bm{x^{*}}_l|\mathcal{L},\bm{B})\right)^2\right]
\]
Then using $\sum_{i=1}^{m}r^{(i)}=1$ we get
\[
\left|\frac{1}{n}\,\frac{dH(\bm{x^{*}}|\mathcal{L},\bm{B})}{d\lambda}+\frac{1}{4}\MMSE_n(\lambda,\mu)\right| \le C_2\left(\sum\limits_{i=1}^{m}\sqrt{\frac{\lambda r^{(i)}}{n\widebar{p}^{(i)}_n(1-\widebar{p}^{(i)}_n)}}\right).
\]
As $H(\bm{x^{*}})=n\log 2$ we get using definition of conditional entropy and mutual information 
\[
\left|\frac{1}{n}\,\frac{dI(\bm{x^{*}}; \mathcal{L},\bm{B})}{d\lambda}-\frac{1}{4}\MMSE_n(\lambda,\mu)\right| \le C_2\left(\sum\limits_{i=1}^{m}\sqrt{\frac{\lambda r^{(i)}}{n\widebar{p}^{(i)}_n(1-\widebar{p}^{(i)}_n )}}\right),
\]
which implies the lemma.

\subsection{Results Used to Prove Lemma~\ref{lem:diff_cov_inf}}
\label{sec:lem_diff}
\begin{lem}
\label{lem:lem_diff}
% Let $H(\bm{x^{*}}|\mathcal{L})$ be defined as in~\eqref{eq:linear_transformed variables}.
Let $\pi^{(i)}(\lambda;\ell^{(i)}_e,x_e)$, $p(y^{(i)}_e)$, $\mathcal{P}^{(i)}_{e}(\mathcal{L}^{(i)},\bm{B},y^{(i)}_e,\ell^{(i)}_e)$ be defined in~\eqref{eq:def_pi}, \eqref{eq:marg_prob} and \eqref{eq:mathcal_p} respectively,
then 
\begin{equation}
\begin{split}
\frac{dH(\tx|\mathcal{L}^{(i)},\bm{B})}{d\lambda}
&= 
-\sum\limits_{i=1}^{m}\sum\limits_{e \,\in E(G^{(i)})}\sum\limits_{\substack{y^{(i)}_e\in\{\pm 1\}\\ \ell^{(i)}_e \in \{\pm 1\}}}
p(y_e^{(i)})
\log \pi^{(i)}(\lambda;\ell^{(i)}_e,y^{(i)}_e)\frac{d\pi^{(i)}(\lambda;\ell^{(i)}_e,x_e)}{d\lambda}\\
&\quad 
+\sum\limits_{i=1}^{m}\sum\limits_{e \,\in E(G^{(i)})}\sum\limits_{\substack{y^{(i)}_e\in\{\pm 1\}\\ \ell_e \in \{\pm 1\}}}\frac{d\pi^{(i)}(\lambda;\ell^{(i)}_e,y^{(i)}_e)}{d\lambda}\mathbb{E}_{\mathcal{L}^{(i)}_{-e},\bm B}\left[\mathcal{P}^{(i)}_{e}(\mathcal{L}^{(i)},\bm{B},y^{(i)}_e,\ell^{(i)}_e)\right].
\end{split}
\end{equation}
\end{lem}
\begin{proof}
We begin by observing that $H\left(\bm{x^{*}}|\mathcal{L}^{(i)},\bm{B}\right)$ is a function of $\lambda$ through $\pi^{(i)}(\lambda;\ell^{(i)}_e,x_e)$ for $e \in E(G^{(i)})$ and $1 \le i \le m$. 

By the chain rule and linearity of differentiation, 
it suffices to assume that only $\pi^{(i)}(\lambda;\ell^{(i)}_e,x_e)$ depends on $\lambda$. 
Then, for $e=(k,l)$, we have
\begin{equation}
\begin{split}
H(\bm{x^{*}}|\mathcal{L},\bm{B})
+H(L^{(i)}_{e}|\mathcal{L}^{(i)}_{-e},\bm{B} ) & = 
H(\bm{x^{*}};L^{(i)}_{e}|\mathcal{L}^{(i)}_{-e},\bm{B} )
\\
&\overset{(1)}{=}H(\bm{x^{*}}|\mathcal{L}^{(i)}_{-e},\bm{B})+H(L^{(i)}_{e}|\bm{x^{*}},\mathcal{L}^{(i)}_{-e},\bm{B})\\
&\overset{(2)}{=}H(\bm{x^{*}}|\mathcal{L}^{(i)}_{-e},\bm{B})+H(L^{(i)}_{e}|x_e).
\end{split}
\end{equation}
Here, equality (1) follows by writing the entropy in two different forms using the chain rule, 
and equality (2) follows from observing that given $x_{e}$, $L^{(i)}_{e}$ is independent of everything else.
This implies 
\[
\frac{dH(\bm{x^{*}}|\mathcal{L},\bm{B})}{d\lambda}=\frac{dH(L^{(i)}_{e}|x_e)}{d\lambda}-\frac{dH(L^{(i)}_{e}|\mathcal{L}^{(i)}_{-e},\bm{B})}{d\lambda},
\]
because $H(\bm{x^{*}}|\mathcal{L}^{(i)}_{-e},\bm{B})$ does not depend on $\pi^{(i)}(\lambda;\ell^{(i)}_e,x_e)$, hence not on $\lambda$ under the simplifying assumption that only $\pi^{(i)}(\lambda;\ell^{(i)}_e,x_e)$ depends on $\lambda$. 

Next let us observe
\begin{equation}
\label{eq:cond_entrop_2}
\begin{split}
\frac{dH(L^{(i)}_{e}|x_e)}{d\lambda} & = -\frac{d}{d\lambda}\sum\limits_{\small{\ell^{(i)}_{e} \in \{\pm 1\}}}\mathbb{E}_{\bm{x^{*}}}\left[\pi^{(i)}(\lambda;\ell^{(i)}_e,x_e)\log \pi^{(i)}(\lambda;\ell^{(i)}_e,x_e)\right]\\
& = -\sum\limits_{\small{\ell^{(i)}_{e} \in \{\pm 1\}}}\mathbb{E}_{\bm{x^{*}}}\left[\frac{d\pi^{(i)}(\lambda;\ell^{(i)}_e,x_e)}{d\lambda}\log \pi^{(i)}(\lambda;\ell^{(i)}_e,x_e) + \frac{d\pi^{(i)}(\lambda;\ell^{(i)}_e,x_e)}{d\lambda}\right]\\
&= -\sum\limits_{\small{\ell^{(i)}_{e} \in \{\pm 1\}}}\mathbb{E}_{\bm{x^{*}}}\left[\frac{d\pi^{(i)}(\lambda;\ell^{(i)}_e,x_e)}{d\lambda}\log \pi^{(i)}(\lambda;\ell^{(i)}_e,x_e)\right]\\
&= -\sum\limits_{\substack{y^{(i)}_e\in\{\pm 1\} \\ \ell^{(i)}_{e} \in \{\pm 1\}}}\mathbb{P}_{x_e}\left(x_e=y^{(i)}_e\right)\frac{d\pi^{(i)}(\lambda;\ell^{(i)}_e,y^{(i)}_e)}{d\lambda}\log \pi^{(i)}(\lambda;\ell^{(i)}_e,y^{(i)}_e).
\end{split}
\end{equation}
Here the third equality holds since
\[
\sum\limits_{\small{\ell^{(i)}_{e} \in \{\pm 1\}}}\frac{d\pi^{(i)}(\lambda;\ell^{(i)}_e,x_e)}{d\lambda}=0, \qquad  x_e\in\{\pm 1\}.
\]
Next, we note that
\begin{equation*}
H(L_e^{(i)}| \mathcal{L}_{-e}^{(i)}, \bm{B})
= -\sum_{\ell_e^{(i)}\in \{\pm 1\}} 
\mathbb{E}_{\mathcal{L}_{-e}^{(i)}, \bm{B} }\left[ \sum_{y_e^{(i)}\in \{\pm 1\}}\pi^{(i)}(\lambda;\ell_e^{(i)}, y_e^{(i)})
\mathcal{P}_e^{(i)}(\mathcal{L},\bm{B}, y_e^{(i)},\ell_e^{(i)}) \right].
\end{equation*}
So we have
\begin{align}
& \frac{dH(L^{(i)}_{e}|\mathcal{L}^{(i)}_{-e},\bm{B})}{d\lambda} 
\nonumber \\
&= 
-\sum\limits_{\substack{y^{(i)}_e \in\{\pm 1\} \\ \ell^{(i)}_{e} \in \{\pm 1\}}}
\frac{d \pi^{(i)}(\lambda;\ell_e^{(i)},x_e)}{d\lambda}
\mathbb{E}_{\mathcal{L}_{-e}^{(i)}, \bm{B} }
\left[
\mathcal{P}_e^{(i)}(\mathcal{L},\bm{B}, y_e^{(i)},\ell_e^{(i)}) \right]
\nonumber \\
& \quad
-  \sum\limits_{\substack{y^{(i)}_e \in\{\pm 1\} \\ \ell^{(i)}_{e} \in \{\pm 1\}}}
\mathbb{E}_{\mathcal{L}_{-e}^{(i)}, \bm{B} }
\left[ \frac{\pi^{(i)}(\lambda;\ell_e^{(i)}, y_e^{(i)})p(y_e^{(i)}|\mathcal{L}_{-e}^{(i)},\bm{B})}{p(\ell_e^{(i)}|\mathcal{L}_{-e}^{(i)},\bm{B} )} 
\sum_{x\in \{\pm 1\} } \frac{d\pi^{(i)}(\lambda;\ell_e^{(i)},x) }{d\lambda}
p(x|\mathcal{L}_{-e}^{(i)},\bm{B})
\right]
\nonumber \\
& = -\sum\limits_{\substack{y^{(i)}_e \in\{\pm 1\} \\ \ell^{(i)}_{e} \in \{\pm 1\}}}
\frac{d \pi^{(i)}(\lambda;\ell_e^{(i)},x_e)}{d\lambda}
\mathbb{E}_{\mathcal{L}_{-e}^{(i)}, \bm{B} }
\left[
\mathcal{P}_e^{(i)}(\mathcal{L},\bm{B}, y_e^{(i)},\ell_e^{(i)}) \right].
\label{eq:cond_entrop_1}
\end{align} 
Here, the second equality holds since
\begin{align*}
&  \sum\limits_{\substack{y^{(i)}_e \in\{\pm 1\} \\ \ell^{(i)}_{e} \in \{\pm 1\}}}
\mathbb{E}_{\mathcal{L}_{-e}^{(i)}, \bm{B} }
\left[ \frac{\pi^{(i)}(\lambda;\ell_e^{(i)}, y_e^{(i)})p(y_e^{(i)}|\mathcal{L}_{-e}^{(i)},\bm{B})}{p(\ell_e^{(i)}|\mathcal{L}_{-e}^{(i)},\bm{B} )} 
\sum_{x\in \{\pm 1\} } \frac{d\pi^{(i)}(\lambda;\ell_e^{(i)},x) }{d\lambda}
p(x|\mathcal{L}_{-e}^{(i)},\bm{B})
\right] \\
& = \sum_{\ell_e^{(i)}\in \{\pm 1\} } 
\mathbb{E}_{\mathcal{L}_{-e}^{(i)}, \bm{B} }
\left[ \sum_{x\in \{\pm 1\} } \frac{d\pi^{(i)}(\lambda;\ell_e^{(i)},x) }{d\lambda}
p(x|\mathcal{L}_{-e}^{(i)},\bm{B}) \right] \\
& = \mathbb{E}_{\mathcal{L}_{-e}^{(i)}, \bm{B} }
\left[ \frac{d}{d\lambda}\left( \sum_{x\in \{\pm 1\} }
\sum_{ \ell_e^{(i)} \in \{\pm 1\} } \pi^{(i)}(\lambda;\ell_e^{(i)},x) 
p(x|\mathcal{L}_{-e}^{(i)},\bm{B}) \right) \right] \\
& = \mathbb{E}_{\mathcal{L}_{-e}^{(i)}, \bm{B} } \left[ \frac{d}{d\lambda} 1 \right] = 0.
\end{align*}
% If we denote \[\mathcal{Q}(\mathcal{L}^{(i)},\bm{B})=\hspace{-0.1in}\sum\limits_{x_e\in\{\pm 1\}}\pi^{(i)}(\lambda;\ell^{(i)}_e,x_e)p\left(x_e|\mathcal{L}^{(i)}_{-kl},\bm{B}\right)\] and \[\mathcal{R}(\mathcal{L}^{(i)},\bm{B})=\hspace{-0.04in}\sum\limits_{x\in\{\pm 1\}}\pi^{(i)}(\lambda;\ell^{(i)}_e,x)p\left(x|\mathcal{T}^{(i)}_{-kl},\bm{B}\right),\] then we have
% So we have
% \begin{equation}
% \label{eq:cond_entrop_1}
% \begin{split}
% \frac{dH\left(L^{(i)}_{e}|\mathcal{L}^{(i)}_{-e},\bm{B}\right)}{d\lambda} &= -\frac{d}{d\lambda}\sum\limits_{\small{\ell^{(i)}_{e} \in \{\pm 1\}}}\mathbb{E}_{\mathcal{L}^{(i)}_{-e},\bm{B}}\left[\mathcal{Q}(\mathcal{L}^{(i)},\bm{B})\log\left(\mathcal{R}(\mathcal{L}^{(i)},\bm{B})\right)\right]\\
% & \overset{(1)}{=} -\sum\limits_{\substack{y^{(i)}_e \in\{\pm 1\} \\ \ell^{(i)}_{e} \in \{\pm 1\}}}\frac{d\pi^{(i)}(\lambda;\ell^{(i)}_e,x_e)}{d\lambda}\mathbb{E}_{\mathcal{L}^{(i)}_{-e},\bm{B}}\left[p\left(x_e=y^{(i)}_e|\mathcal{L}^{(i)}_{-e},\bm{B}\right)\log\left(\mathcal{R}(\mathcal{L}^{(i)},\bm{B})\right)\right].
% \end{split}
% \end{equation}
% Here equality (1) follows from the fact that $p\left(x_e=y_e|\mathcal{L}^{(i)}_{-e},\bm{B}\right)$ is free of $\pi^{(i)}(\lambda;\ell^{(i)}_e,x_e)$ hence free of $\lambda$ under the simplifying assumption that only $\pi^{(i)}(\lambda;\ell^{(i)}_e,x_e)$ depends on $\lambda$ and
% \[
% \sum\limits_{\small{\ell^{(i)}_{e} \in \{\pm 1\}}}\frac{d}{d\lambda}\left(\sum\limits_{x\in\{\pm 1\}}\pi^{(i)}(\lambda;\ell^{(i)}_e,x)p\left(x|\mathcal{L}^{(i)}_{-e},\bm{B}\right)\right)=0.
% \]
Combining \eqref{eq:cond_entrop_2} and \eqref{eq:cond_entrop_1}, we complete the proof of the lemma.
\end{proof}
\section{Proof of I-MMSE Identity in the Gaussian Model}
\label{I_MMSE}
 Let us consider the vector $\bm{t}$ containing $\{T_{i,j}\}_{i < j}$ and the vector $\bm{r}$ containing $\{x^*_ix^*_j\}_{i < j}$. Then from the definition of $\bm{T}$ we have
\[
\bm{t}=\sqrt{\frac{\lambda}{n}}\bm{r}+\bm{q},
\]
where $\bm{q}\sim N_{C^{n}_{2}}(0,\bm{I}_{C^{n}_{2}})$ ($C^{n}_{2}={n \choose 2}$). As the diagonal entries of $\tx(\tx)^\top$ are all $1$, let us consider 
\[
\bm{s}=\sqrt{\frac{\lambda}{n}}+\bm q_1,
\]
where $\bm q_1 \sim N_{n}(0,2\bm{I}_{n})$. Let $\mathcal{Y}=(\bm t, \bm s, \bm B)$ and $\bm B=(\bm b_1,\cdots,\bm b_p)$ where $\bm b_i = (\bm I + \frac{\mu}{n}\tx(\tx)^\top)^{-1}\widetilde{\bm z}_i$, where $\widetilde{\bm z}_i\sim N_n(\bm 0, \bm I_n)$. From the definition, it is clear that,
\[
I(\tx(\tx)^\top;\bm T, \bm B) =I(\tx(\tx)^\top;\mathcal{Y}) 
\]
Next we have
\[
I(\tx(\tx)^\top;\mathcal{Y}) = H(\mathcal{Y})-H(\mathcal{Y}|\tx(\tx)^\top).
\]
Note that
\[
H(\mathcal{Y}|\tx(\tx)^\top) = H(\bm q, \bm q_1, \{\widetilde{\bm z}_i\}_{i=1}^p).
\]
Observe that the right hand side is free of $\lambda$ and hence we get
\[
\frac{d}{d\lambda}I(\tx(\tx)^\top;\mathcal{Y}) = \frac{d}{d\lambda}H(\mathcal{Y}).
\]
Since the density of $\mathcal{Y}$ can be written as
\[
f_\mathcal{Y}(\mathcal{Y})=\sum\limits_{\bm x \in \{\pm 1\}^n}\frac{1}{2^n}\Bigg[\Bigg\{\prod\limits_{i<j}\phi\left(T_{i,j}-\sqrt{\frac{\lambda}{n}}x_ix_j\right)\Bigg\}\Bigg\{\prod\limits_{i=1}^{n}\phi\left(\frac{T_{i,i}-\sqrt{\lambda/n}}{\sqrt{2}}\right)\Bigg\}f(\bm{B}|\bm x)\Bigg],
\]
we have
\begin{align}
\frac{d}{d\lambda}H(\mathcal{Y}) & = \frac{d}{d\lambda}\Bigg\{-\mathbb{E}_{\mathcal{Y},\tx}[\log f_\mathcal{Y}(\mathcal{Y})]\Bigg\}\\
& = \frac{1}{2\sqrt{\lambda n}}\mathbb{E}_{\mathcal{Y},\tx}\Bigg[\frac{1}{2^n}\sum\limits_{\bm x \in \{\pm 1\}^n}\sum\limits_{i<j}\frac{(T_{i,j}-\sqrt{\lambda/n}x_ix_j)(x^*_ix^*_j-x_ix_j)f_\mathcal{Y}(\mathcal{Y}|\bm x)}{f_\mathcal{Y}(\mathcal{Y})}\Bigg]\\
& = \frac{1}{2\sqrt{\lambda n}}\mathbb{E}_{\mathcal{Y},\tx}\Bigg[\sum\limits_{\bm x \in \{\pm 1\}^n}\sum\limits_{i<j}(T_{i,j}-\sqrt{\lambda/n}x_ix_j)(x^*_ix^*_j-x_ix_j)f(\bm x|\mathcal{Y})\Bigg]\\
& = \frac{1}{2\sqrt{\lambda n}}\mathbb{E}_{\mathcal{Y},\tx}\Bigg[\sum\limits_{i<j}\Bigg\{T_{i,j}x^*_ix^*_j-T_{i,j}\mathbb{E}[x^*_ix^*_j|\mathcal{Y}]-\sqrt{\frac{\lambda}{n}}x^*_ix^*_j\mathbb{E}[x^*_ix^*_j|\mathcal{Y}]\\
&\hskip 1.5in +\sqrt{\frac{\lambda}{n}}\mathbb{E}[(x^*_ix^*_j)^2|\mathcal{Y}]\Bigg]\\
& = \frac{1}{2n}\sum\limits_{i<j}\mathbb{E}\Bigg[(x^*_ix^*_j-\mathbb{E}[x^*_ix^*_j|\mathcal{Y}])^2\Bigg]\\
& = \frac{1}{4n^2}\mathbb{E}\|\tx(\tx)^\top-\mathbb{E}[\tx(\tx)^\top|\mathcal{Y}]\|^2_F\\
& = \frac{1}{4}\GMMSE_n(\lambda,\mu).
\end{align}
This implies
\begin{equation}
\frac{1}{n}\frac{d}{d\lambda}I(\tx(\tx)^\top;\bm T, \bm B)=\frac{1}{4}\GMMSE_n(\lambda,\mu).
\end{equation}

\section{Proof of Results in Section \ref{mmse}}
\subsection{Proof of Theorem \ref{thm:MSE_AMP_main}}
\label{sec:proof-mse-amp}
We start by observing that
\begin{equation}
\label{eq:mse_amp_eqn}
\begin{split}
& \mathsf{MSE^{AMP}_n}(t;\lambda,\mu,\varepsilon) \\
& = 1-2\;\mathbb{E}\langle\widehat{\bm{x}}^{t}, \tx\rangle^2_n +\frac{1}{n^2}\mathbb{E}\|\widehat{\bm{x}}^{t}\|^4\\
&= 1-2\;\mathbb{E}\langle f_{t-1}(\bm{u}^{t-1},\bm{y}^{t-1},\bm{x}_0(\varepsilon)),\tx\rangle^2_n +\frac{1}{n^2}\mathbb{E}\|f_{t-1}(\bm{u}^{t-1},\bm{y}^{t-1},\bm{x}_0(\varepsilon))\|^4.\\
\end{split}
\end{equation}
Since $f_{t-1}$ is Lipschitz (by Lemma \ref{lem:Lipschitz}), 
$(x,y,w,z)\mapsto w f_{t-1}(x,y,z)$ is partially pseudo-Lipschitz. 
This implies using Theorem~\ref{thm:slln_shifted}, we get
\begin{equation}
\label{eq:sigma_2_amp}
\lim_{n \rightarrow \infty}\langle\widehat{\bm{x}}^{t}, \bm{x^{*}} \rangle^2_n 
= 
\mathbb{E}\left[X_0\mathbb{E}
\big[X_0|\alpha_{t-2}X_0+\tau_{t-2}Z_0, \mu_{t-1}X_0+\sigma_{t-1}\widetilde{Z}_0, X_0(\varepsilon)\big] \right],
\end{equation}
where $Z_0,\widetilde{Z}_0~\stackrel{iid}{\sim}~N(0,1)$, 
$X_0 \sim \mbox{Rademacher}$ and $X_0(\varepsilon)=B_0X_0$ where $B_0 \sim \mathrm{Bern}(\varepsilon)$ is independent of all other random variables. 
Similarly, $(x,y,w,z) \mapsto f^2_{t-1}(x,y,z)$ is partially pseudo-Lipschitz, 
and Theorem~\ref{thm:slln_shifted} implies
\begin{equation}
% \label{eq:rec_final_1}
\lim_{n\rightarrow \infty}\frac{1}{n}\|\bm{f}_{t-1}(\bm{u}^{t-1},
\bm{y}^{t-1},
\bm{x}_0(\varepsilon))\|^2 = \mathbb{E}\left[
\big(\mathbb{E}\big[X_0|\alpha_{t-2}X_0+\tau_{t-2}Z_0, \mu_{t-1}X_0+\sigma_{t-1}\widetilde{Z}_0,X_0(\varepsilon)\big]\big)^2 \right]. 
\end{equation}
Then using \revsag{the} dominated convergence theorem, property of conditional expectation and \eqref{eq:mse_amp_eqn}, we obtain the following.
\[
\lim_{n\rightarrow \infty} \mathsf{MSE^{AMP}_n}(t;\lambda,\mu,\varepsilon) 
= 1-\left(\mathbb{E}\Big[\big(\mathbb{E}\big[X_0|\alpha_{t-2}X_0+\tau_{t-2}Z_0, \mu_{t-1}X_0+\sigma_{t-1}\widetilde{Z}_0, X_0(\varepsilon)\big]\big)^2 \Big]\right)^2.
\]
Now using \eqref{eq:state_ev_rec}, and \eqref{eq:init_z_t} we get
\[
\lim_{n\rightarrow \infty} \mathsf{MSE^{AMP}_n}(t;\lambda,\mu,\varepsilon) = 1-z^2_t.
\]
It is immediate to show that
\[
G_\varepsilon(z)= 1-(1-\varepsilon)\mathsf{mmse}\left(\lambda z + (1-\varepsilon)\frac{\mu^2}{c}\frac{z}{1+\mu z}\right)
\]
\revsag{is continuous on $[0,\infty)$,} $\lim_{z \rightarrow \infty}G_\varepsilon(z)=1$ and $G_\varepsilon(0)=\varepsilon$. Using the fact that the function $t \mapsto \mathsf{mmse}(t)$ is monotone decreasing and $\varepsilon \in [0,1]$, it is easy to show that $G_\varepsilon(z)$ is monotone increasing in $z$. Further, using Lemma 6.1 of \cite{AbbeMonYash}, it can be concluded that $G_\varepsilon(z)$ is strictly concave in $[0,\infty)$. From these observations we have
\[
\lim_{t \rightarrow \infty}\lim_{n\rightarrow \infty} \mathsf{MSE^{AMP}_n}(t;\lambda,\mu,\varepsilon) = 1-z^2_*(\lambda,\mu,\varepsilon),
\]
where $z_*(\lambda,\mu,\varepsilon)$ is the largest non-negative solution to \eqref{eq:rec_final}. 
Note that
\[
G_\varepsilon(z)= 1+(\varepsilon-1)\mathsf{mmse}\left(\lambda z + (1-\varepsilon)\frac{\mu^2}{c}\frac{z}{1+\mu z}\right).
\]
As $\varepsilon \mapsto \mathsf{mmse}\left(\lambda z + (1-\varepsilon)(\mu^2/c)(z/(1+\mu z))\right)$ is increasing in $\varepsilon$, we have 
$G_\varepsilon(z)$ is increasing as a function of $\varepsilon$. From this observation and boundedness of $G_\varepsilon(z)$, we have
\begin{equation}
\label{eq:5_4_4}
z_{*}(\lambda,\mu,\varepsilon) \rightarrow z_{*}(\lambda,\mu)
\end{equation}
as $\varepsilon \rightarrow 0$, where $z_{*}(\lambda,\mu)$ satisfies \eqref{eq:zstar} and hence
\[
\lim_{\varepsilon \rightarrow 0}\lim_{t \rightarrow \infty}\lim_{n\rightarrow \infty} \mathsf{MSE^{AMP}_n}(t;\lambda,\mu,\varepsilon) = 1-z^2_*(\lambda,\mu).
\]

\subsection{Proof of Theorem \ref{thm:MMSE_main}}
\label{proof_thm_5_2}
Begin by noting that
\begin{multline}
\left|\frac{1}{n}I\left(\bm{x^{*}};\bm{T}(\lambda), \bm{B},\bm{x}_0(\varepsilon),\bm{w}_0(\varepsilon)\right)-
\frac{1}{n}I\left(\bm{x^{*}};\bm{T}(\lambda), \bm{B}\right)\right| 
= \frac{1}{n}I\left(\bm{x^{*}};\bm{x}_0(\varepsilon),\bm{w}_0(\varepsilon)|\bm{T}(\lambda), \bm{B}\right) \\ \le \frac{1}{n} H(\bm{x}_0(\varepsilon),\bm{w}_0(\varepsilon)) \le \varepsilon \log 2 +\frac{p}{2n}\varepsilon\log(2\pi\,e) \rightarrow 0, 
\end{multline}
as $\varepsilon \rightarrow 0$. 
%\nb{need more details on why the last inequality holds.}
Further, using techniques similar to the proof of Remark 6.5 in \cite{AbbeMonYash}, we can show
\begin{equation}
\label{eq:Inf_3}
\lim_{\lambda \rightarrow \infty}\lim_{n \rightarrow \infty}\frac{1}{n}I\left(\tx(\tx)^\top;\bm{T}(\lambda), \bm{B}\right) = \log 2,
\end{equation}
where $\bm{T}(\lambda)$ is defined in~\eqref{eq:T}. 
From Lemma~\ref{lem:mse_2_1}, we also have
\begin{align}
\label{eq:Inf_4}
&\lim_{n \rightarrow \infty}\frac{1}{n}I\left(\tx(\tx)^\top; \bm{T}(0), \bm{B}\right)\\
&=\frac{1}{2c}\log(1+\mu \gamma_*)+\frac{1}{2c}\,\frac{(1+\mu)}{(1+\mu \gamma_*)}+\sMI\left(\frac{\mu^2}{c}\frac{\gamma_*}{1+\mu \gamma_*}\right)-\frac{1}{2c}\log(1+\mu)-\frac{1}{2c}\\
&= \upkappa(\mu,\gamma_*), 
\end{align}
where $\gamma_*$ satisfies
\begin{equation}
\label{eq:rec_final_1}
\gamma_{*} = 1-\smmse\left(\frac{\mu^2}{c}\frac{\gamma_*}{1+\mu \gamma_*}\right).
\end{equation}
From (263) of~\cite{AbbeMonYash} we have for all $\lambda \ge 0$ 
%\nb{Why do we need this assumption?}
\[
\lim_{n \rightarrow \infty}\left[\frac{1}{n}I\left(\tx; \bm{T}(\lambda), \bm{B}, \bm{x}_0(\varepsilon),\bm{w}_0(\varepsilon)\right) - \frac{1}{n}I\left(\tx(\tx)^\top; \bm{T}(\lambda), \bm{B}, \bm{x}_0(\varepsilon),\bm{w}_0(\varepsilon)\right)\right] = 0.
\]
Next observe that for all $\lambda,\mu,\varepsilon > 0$
\begin{equation}
\label{eq:5_4_1}
\mathsf{GMMSE}_n(\lambda,\mu,\varepsilon) \le \mathsf{MSE^{AMP}_n}(t;\lambda,\mu,\varepsilon),
\end{equation}
where 
\[
\GMMSE_n(\lambda,\mu,\varepsilon) = \frac{1}{\,n^2}\, \mathbb{E}\left\| \tx(\tx)^\top - \mathbb{E}[ \tx(\tx)^\top \,|\, \bm{T},\bm{B},\bm x_0(\varepsilon), \bm w_0(\varepsilon)]\right\|_F^2.
\]
%Using conditional version of I-MMSE identity of~\cite{1412024} we get
With the same techniques used to prove \eqref{eq:I_MMSE}, we have
\begin{equation}
\label{eq:5_4_2}
\frac{1}{n}\frac{d}{d\lambda}I\left(\tx(\tx)^\top;\bm{T}(\lambda), \bm{B},\bm{x}_0(\varepsilon),\bm{w}_0(\varepsilon)\right) = \frac{1}{4}\GMMSE_n(\lambda,\mu,\varepsilon).
\end{equation}
Using Lemma~\ref{lem:mse_2_3}, 
%\nb{fill in reference!}, 
we further have
\begin{equation}
\label{eq:5_4_3}
\xi(z_{*}(\lambda,\mu),\lambda,\mu) = \xi(\gamma_{*},0,\mu) + \int_{0}^{\lambda}\frac{1}{4}\left(1-z^2_{*}(t,\mu)\right)\,dt.
\end{equation}
Then using Theorem \ref{thm:MSE_AMP_main}, \eqref{eq:Inf_3}, \eqref{eq:Inf_4}, \eqref{eq:5_4_1}, \eqref{eq:5_4_2}, \eqref{eq:5_4_3}
\begin{equation}
\label{eq:equiv_2}
%\begin{split}
\begin{aligned}
\log 2 - \upkappa(\mu,\gamma_*) & =  \liminf_{n \rightarrow \infty}\lim_{\lambda \rightarrow \infty}\Bigg[\frac{1}{n}I\left(\tx(\tx)^\top;\bm{T}(\lambda), \bm{B}\right)-\frac{1}{n}I\left(\tx(\tx)^\top; \bm{T}(0), \bm{B}\right)\Bigg]\\
&=\liminf_{n \rightarrow \infty}\lim_{\lambda \rightarrow \infty}\lim_{\varepsilon \rightarrow 0}\Bigg[\frac{1}{n}I\left(\tx(\tx)^\top;\bm{T}(\lambda), \bm{B},\bm{x}_0(\varepsilon),\bm{w}_0(\varepsilon)\right)\\
&\hspace{2in}-\frac{1}{n}I\left(\tx(\tx)^\top; \bm{T}(0), \bm{B},\bm{x}_0(\varepsilon),\bm{w}_0(\varepsilon)\right)\Bigg]\\
    								   & = \liminf_{n \rightarrow \infty}\lim_{\lambda \rightarrow \infty}\lim_{\varepsilon \rightarrow 0} \int_{0}^{\lambda}\frac{1}{4}\mathsf{GMMSE}_{n}(u,\mu,\varepsilon)\,du\\
								   & \le \limsup_{t \rightarrow \infty}\limsup_{n \rightarrow \infty}\lim_{\lambda \rightarrow \infty}\lim_{\varepsilon \rightarrow 0}\int_{0}^{\lambda}\frac{1}{4}\mathsf{MSE}^{\scriptsize{\mathsf{AMP}}}_n(t;u,\mu,\varepsilon)\,du\\
								   &= \lim_{\varepsilon \rightarrow 0}\lim_{\lambda \rightarrow \infty}\int_{0}^{\lambda}\frac{1}{4}\left(1-z^2_{*}(u,\mu,\varepsilon)\right)\,du\\& \hspace{2 in} \left(\mbox{where $z_{*}(\lambda,\mu,\varepsilon)$ satisfies~\eqref{eq:rec_final}}\right)\\
								   &= \lim_{\lambda \rightarrow \infty}\int_{0}^{\lambda}\frac{1}{4}\left(1-z^2_{*}(u,\mu)\right)\,du\\& \hspace{2 in} \left(\mbox{where $z_{*}(\lambda,\mu)$ satisfies~\eqref{eq:zstar}}\right)\\
								   & = \lim_{\lambda \rightarrow \infty} \xi(z_{*}(\lambda,\mu),\lambda,\mu)-\xi(\gamma_{*},0,\mu) \hspace{0.1 in}\\& \hspace{2in} \left(\mbox{where $\xi(z,\lambda,\mu)$ is defined by~\eqref{eq:def_xi}}\right)\\
								   & = \log 2 - \upkappa(\mu,\gamma_*).
								   %\end{split}
\end{aligned}
\end{equation}
This implies that all inequalities in~\eqref{eq:equiv_2} are equalities, which, in turn, implies 
\[
\lim_{n \rightarrow \infty}\mathsf{GMMSE}_n(\lambda,\mu) = \lim_{\varepsilon \rightarrow 0}\lim_{t \rightarrow \infty}\lim_{n \rightarrow \infty}\mathsf{MSE}^{\scriptsize{\mathsf{AMP}}}_{n}(t;\lambda,\mu,\varepsilon) = 1-z^2_{*}(\lambda,\mu).
\]
By the definition of $\xi$
\[
\lim_{n \rightarrow \infty}\frac{1}{n}I\left(\tx(\tx)^\top; \bm{T}(0), \bm{B}\right) = \xi(\gamma_{*},0,\mu).
\]
Finally using Theorem \ref{thm:MSE_AMP_main} and Lemma~\ref{lem:mse_2_3}, 
%\nb{fill in reference}
we have
\begin{equation}
\begin{split}
\hspace{-0.5in}\lim_{n \rightarrow \infty}\frac{1}{n}I\left(\tx(\tx)^\top;\bm{T}(\lambda), \bm{B}\right) & = \xi(\gamma_{*},0,\mu)+ \lim_{n \rightarrow \infty}\int_{0}^{\lambda}\frac{1}{4}\mathsf{GMMSE}_n(t,\mu)\,dt\\
																			   & = \xi(\gamma_{*},0,\mu)+\hspace{-0.05in} \lim_{n \rightarrow \infty}\int_{0}^{\lambda}\frac{1}{4}\mathsf{MSE}^{\scriptsize{\mathsf{AMP}}}_{n}(t;\lambda,\mu)\,dt\\
																			   & = \xi(\gamma_{*},0,\mu)+ \int_{0}^{\lambda}\frac{1}{4}\left(1-z^2_{*}(t,\mu)\right)\,dt\\
																			   & = \xi(z_{*}(\lambda,\mu),\lambda,\mu).
\end{split}
\end{equation}
This completes the proof. 

\subsection{Lemmas Used to Prove Results in Section \ref{mmse}}
\begin{lem}
\label{lem:Lipschitz}
{ Consider $f_t$ defined in \eqref{eq:def_f}. Then $f_t$ and its partial derivatives with respect to the first and second arguments are Lipschitz for all $t \ge 0$.  }
\end{lem}
\begin{proof}
Let $\bm{x}=(x,y)$ and $\bm{a}=(a,b)$. We begin by observing that
\begin{equation}
\label{eq:f_t_lips}
f_t(x,y,z)=\begin{cases}1 & \mbox{if $z=1$}\\-1 & \mbox{if $z=-1$}\\ \tanh\left(-\frac{\alpha_{t-1}}{\tau^2_{t-1}}x-\frac{\mu_{t}}{\sigma^2_{t}}y\right) & \mbox{if $z=0$} \end{cases}
\end{equation}
Using \eqref{eq:f_t_lips}, we get
\[
|f_t(x,y,1)-f_t(a,b,1)| = 0 \le \|(\bm{x},1)-(\bm{a},1)\|,
\]
and
\[
|f_t(x,y,-1)-f_t(a,b,-1)| = 0 \le \|(\bm{x},1)-(\bm{a},1)\|.
\] 
Further, since $|\tanh^\prime(x)| \le 1$ for all $x$, using multivariate mean-value theorem, we have
\[
|f_t(x,y,0)-f_t(a,b,0)| \le (\alpha_{t-1}/\tau^2_{t-1} + \mu_{t}/\sigma^2_{t})^{1/2} \|(\bm{x},0)-(\bm{a},0)\|.
\]
{ Again as $|\tanh(x)|\le 1$, it can be easily shown that
\[
|f_t(x,y,-1)-f_t(a,b,0)| \le C \|(\bm{x},-1)-(\bm{a},0)\|,
\] 
and
\[
|f_t(x,y,1)-f_t(a,b,0)| \le C \|(\bm{x},1)-(\bm{a},0)\|.
\] 
Again, by definition
\[
|f_t(x,y,1)-f_t(a,b,-1)| =2 \le C \|(\bm{x},1)-(\bm{a},-1)\|.
\] }
Next observe that 
\begin{equation}
\label{eq:f_t_1}
\frac{\partial f_t(x,y,z)}{\partial x}=\begin{cases}0 & \mbox{if $z=1$,}\\ 0 & \mbox{if $z=-1$,}\\-\frac{\alpha_{t-1}}{\tau^2_{t-1}}\sech^2\left(-\left(\frac{\alpha_{t-1}}{\tau^2_{t-1}}x + \frac{\mu_{t}}{\sigma^2_{t}}y\right)\right) & \mbox{if $z=0$;}\end{cases}
\end{equation}
Observing that $|\sech(x)| \le 1$ and $|\sech(x)\tanh(x)| \le 1$, and using arguments similar to those previously used, we can show that $\frac{\partial f_t(x,y,z)}{\partial x}$ is Lipschitz. Similarly, we can also show $\frac{\partial f_t(x,y,z)}{\partial y}$ is Lipschitz.
\end{proof}
\begin{lem}
\label{lem:mse_2_1}
We have
\[
\lim_{n \rightarrow \infty}\frac{1}{n}I\left(\tx; \bm{B}\right) = \upkappa(\mu,\gamma_*),
\]
where $\upkappa(\mu,\gamma_*)$ is as defined in \eqref{eq:Inf_4}.
\end{lem}
\begin{proof}
Let us observe that if $\lambda = 0$, we have $\bm B$ is the transpose of the matrix $\bm Y$ described in (1) of \cite{miolane2018} with $\bm U = \tx$, $\bm V= \bm v^{*}$, $\lambda = \mu$ and $\alpha = 1/c$. Since
\[
\frac{1}{n} I(\bm v^{*}(\tx)^\top; \bm B) = \frac{1}{n} I((\bm v^{*}, \tx); \bm B) = \frac{1}{n} I((\bm v^{*}, \tx); \bm B^\top),
\]
the result of \cite{miolane2018} directly \revsag{apply} in our case. Let us consider the scalar model
\[
Y = \sqrt{\gamma}\,X+Z,
\]
where $Z \sim N(0,1)$ and $X \in \{U,V\}$ where $V \sim N(0,1)$ and $U \sim \mbox{Rademacher}$. Recall
\[
\mathcal{Z}(Y) = \int d\,P_X e^{\gamma xX + \sqrt{\gamma}xZ-\frac{\gamma x^2}{2}}
\]
and the function $F_{P_X}(\gamma)$ defined in (9) of \cite{miolane2018} given by
\[
F_{P_X}(\gamma) = \mathbb{E}\left[\frac{1}{\mathcal{Z}(Y)}\int xXe^{\gamma xX + \sqrt{\gamma}xZ-\frac{\gamma x^2}{2}}d\,P_X\right].
\]
For $X = V$ it can be easily verified that
\[
F_{P_X}(\gamma) = \frac{\gamma}{1+\gamma}.
\]
By definition of $\Gamma(\mu,c)$ as in (11) of \cite{miolane2018}, in our case we have
\[
\Gamma(\mu,c) = \Bigg\{\left(q,\frac{\mu q}{1+\mu q}\right): q \ge 0\Bigg\}.
\]
If we recall the definition of $\psi_{P_X}(\gamma)$ as in (10) of \cite{miolane2018}, that is
\[
\psi_{P_X}(\gamma)=\mathbb{E}\log\left(\int e^{\gamma xX + \sqrt{\gamma}xZ-\frac{\gamma x^2}{2}}d\,P_X\right),
\]
then it is easy to show that
\[
\psi_{P_V}(\gamma)= \frac{\gamma}{2}-\frac{1}{2}\log(1+\gamma),
\]
and
\[
\psi_{P_U}(\gamma)= \frac{\gamma}{2}-\mathsf{I}(\gamma),
\]
where $\mathsf{I}(\gamma)$ is defined in \eqref{eq:inf_scalar}. Let us define
\begin{multline}
\mathcal{F}(q) = \psi_{P_U}\left(\frac{\mu^2}{c}\frac{q}{1+\mu q}\right) + \frac{1}{c}\psi_{P_V}(\mu q) -\frac{\mu^2}{2c} \frac{q^2}{1+\mu q}\\
= \frac{\mu}{2c}+\frac{1}{2c}-\frac{\mu+1}{2c(1+\mu q)}-\mathsf{I}\left(\frac{\mu^2}{c}\frac{q}{1+\mu q}\right)-\frac{1}{2c}\log(1+\mu q).
\end{multline}
Now if $\gamma^*$ is the \revsag{supremum} of $\mathcal{F}(q)$, then it must satisfy, $\mathcal{F}^\prime(\gamma^*)=0$, which further implies
\[
\gamma^* = 1-\smmse\left(\frac{\mu^2}{c}\,\frac{\gamma^*}{1+\mu\gamma^*}\right).
\]
Using Corollary 1 of \cite{miolane2018}, we get
\begin{multline}
\lim_{n \rightarrow \infty} \frac{1}{n}I(\bm v^{*}(\tx)^\top; \bm B) = \frac{\mu}{2c}-\frac{\mu}{2c}-\frac{1}{2c}+\frac{\mu+1}{2c(1+\mu \gamma^*)}+\mathsf{I}\left(\frac{\mu^2}{c}\frac{\gamma^*}{1+\mu \gamma^*}\right)+\frac{1}{2c}\log(1+\mu \gamma^*)\\
= \frac{\mu+1}{2c(1+\mu \gamma^*)}+\mathsf{I}\left(\frac{\mu^2}{c}\frac{\gamma^*}{1+\mu \gamma^*}\right)+\frac{1}{2c}\log(1+\mu \gamma^*)-\frac{1}{2c}\\
= \upkappa(\mu,\gamma_*) + \frac{1}{2c}\log(1+\mu).
\end{multline}
We note that given $\bm{v^{*}}(\tx)^\top$, $\bm{B}$ is \revsag{equal} in distribution to $(\bm{y}_1,\bm{y}_2,\ldots,\bm{y}_p)$, where \newline $\bm{y}_i \sim \mathbf{N}_n\left(\sqrt{\mu/n}\bm{v^{*}}(\tx)^\top,\bm{I}_n\right)$. This implies
\[
H(\bm{B}|\bm{v^{*}}(\tx)^\top) = \frac{p}{2}\log\mbox{det}\left(2\pi e \mathbf{I}_n\right).
\]
Also note that, given $\tx$, $\bm{B}\overset{d}{=}(\bm{b}_1,\bm{b}_2,\ldots,\bm{b}_p)$, where $\bm{b}_i \sim \mathbf{N}_n\left(\mathbf{0},(\mu/n)\tx(\tx)^\top+\bm{I}_n\right)$. This implies
\[
H(\bm{B}|\bm{v^{*}}(\tx)^\top) = \frac{p}{2}\log\mbox{det}\left(2\pi e\left(\frac{\mu}{n}\tx(\tx)^\top+\bm{I}_n\right)\right).
\]
Next we get
\[
\frac{1}{n}\left[I(\tx;\bm{B})-I(\bm{v^{*}}(\tx)^\top;\bm{B})\right]=\frac{1}{n}\left[H(\bm{B}|\bm{v^{*}}(\tx)^\top)-H(\bm{B}|\tx)\right] =-\frac{p}{2n}\log(1+\mu).
\]
Thus, we have
\begin{equation}
\begin{split}
\lim_{n \rightarrow \infty}\frac{1}{n}I\left(\tx; \bm{B}\right) &= \lim_{n \rightarrow \infty}\frac{1}{n}\left[I(\tx;\bm{B})-I(\bm{v^{*}}(\tx)^\top;\bm{B})\right] + \lim_{n \rightarrow \infty}\frac{1}{n}I(\bm{v^{*}}(\tx)^\top;\bm{B})\\
& = \upkappa(\mu,\gamma_*).
\end{split}
\end{equation}
\end{proof}
\begin{lem}
\label{lem:mse_2_3}
Let us consider the function $\xi$ defined in~\eqref{eq:def_xi}. Then for all $\lambda,\mu >0$
\[
\xi(z_{*}(\lambda,\mu),\lambda,\mu)= \xi(\gamma_{*},0,\mu)+\int_{0}^{\lambda}\frac{1}{4}\left(1-z^2_{*}(t,\mu)\right)\,dt.
\]
\end{lem}
\begin{proof}
From \eqref{eq:rec_final} it is easy to see that $z_{*}(\lambda,\mu,\varepsilon)=\gamma_*$ where $\gamma_*$ is the unique non-negative solution to \eqref{eq:rec_final_1}. Then we have
\begin{align}
&\frac{\partial \xi(\gamma,\lambda,\mu)}{\partial \gamma}\bigg|_{(z_*(\lambda,\mu),\lambda)}\\
&= \frac{1}{2}\Bigg(\lambda+\frac{\mu^2}{c}\frac{1}{(\mu z_*(\lambda,\mu)+1)^2}\Bigg)\Bigg\{z_*(\lambda,\mu)-1+\;\mbox{mmse}\Bigg(\lambda z^2_*(\lambda,\mu)+\frac{\mu^2}{c}\frac{z_*(\lambda,\mu)}{\mu z_*(\lambda,\mu)+1}\Bigg)\Bigg\}
\end{align}
also
\[
\frac{\partial \xi(\gamma,\lambda,\mu)}{\partial \lambda}\bigg|_{(z_*(\lambda,\mu),\lambda)} = \frac{1}{4}\left(1-z^2_*(\lambda,\mu)\right).
\]
This implies using~\eqref{eq:rec_final}
\[
\frac{d \xi(\gamma,\lambda,\mu)}{d \lambda}\bigg|_{(z_*(\lambda,\mu),\lambda)} = \frac{1}{4}\left(1-z^2_*(\lambda,\mu)\right).
\]
\end{proof}
\section{Proof of Results in Section \ref{AMP}}
\subsection{Proof of Lemma \ref{lem:prop_pl}}
{
Let us consider two $\bm x, \bm y \in \mathbb{R}^k$, where $\bm x=(x_1,\ldots,x_{k})$ and $\bm y=(y_1,\ldots,y_k)$. Then
\[
|f(\bm x,z)| \le |f(\bm x,z)-f(\bm 0,z)| + |f(\bm 0,z)|.
\]
Since $f$ is partially Lipschitz, we have using \eqref{eq:varphi}
\[
|f(\bm x,z)-f(\bm 0,z)| \le C\|\bm x\|,
\]
for some constant $C>0$. Again, using \eqref{eq:varphi_1}
\[
|f(\bm 0,z)| \le C(1+|z|).
\]
Hence, we have
\[
|f(\bm x,z)| \le C(1+\|\bm x\|+|z|).
\]
\begin{enumerate}
\item Now let us observe that
\begin{align}
&|f(\bm x,z)g(\bm x,z)-f(\bm y,z)g(\bm y,z)|\\
& = |f(\bm x,z)g(\bm x,z)-f(\bm y,z)g(\bm x,z)+f(\bm y,z)g(\bm x,z)-f(\bm y,z)g(\bm y,z)|\\
& \le |f(\bm x,z)-f(\bm y,z)||g(\bm x,z)|+|g(\bm x,z)-g(\bm y,z)||f(\bm y,z)|\\
& \le C(1+\|\bm x\|+\|\bm y\|+|z|)\|\bm x-\bm y\|.
\end{align}
Also
\begin{align}
|f(\bm 0,z)g(\bm 0,z)| \le C^2(1+|z|)^2 \le C_1(1+|z|^2).
\end{align}
\item Let us denote $\bm x^\prime=(x_1,\ldots,x_{r-1},x_{r+1},\ldots,x_k)$ and $\bm y^\prime=(y_1,\ldots,y_{r-1},y_{r+1},\ldots,y_k)$. Next note that
\begin{align}
&|H(\bm x^\prime,z)-H(\bm y^\prime,z)|\\
&\le \mathbb{E}_X\Big[\Big|\phi(x_1,\ldots,x_{r-1},X,x_{r+1},\ldots,x_k,z)-\phi(y_1,\ldots,y_{r-1},X,y_{r+1},\ldots,y_k,z)\Big|\Big]\\
& \le C\;\mathbb{E}_X\Big[(1+\|\bm x^\prime\|+\|\bm y^\prime\|+|X|+|z|)\|\bm x^\prime-\bm y^\prime\|\Big]\\
& \le C_1(1+\|\bm x^\prime\|+\|\bm y^\prime\|+|z|)\|\bm x^\prime-\bm y^\prime\|.
\end{align}
Next, note that
\begin{align}
|H(\bm 0,z)| &\le \mathbb{E}_X\Big[\Big|\phi(0,\ldots,0,X,0,\ldots,0,z)-\phi(0,\ldots,0,0,0,\ldots,0,z)\Big|\Big]\\
& \quad\quad +\Big|\phi(0,\ldots,0,0,0,\ldots,0,z)\Big|\\
& \le C\;\mathbb{E}_X\Big[(1+|X|+|z|)|X|\Big]+C(1+|z|^2)\\
& \le C_1(1+|z|+|z|^2)\\
& \le C_2(1+|z|^2).
\end{align}
\end{enumerate}
}

\subsection{Proof of Theorem \ref{thm:thm_6_1}, Conditioning Technique and the Main Technical Lemma}
\label{AMP_1}
To prove Theorem \ref{thm:thm_6_1}, we apply the same device used in \cite{BM11journal}. We begin by observing that for the Gaussian matrix $\bm{L}$ defined at the beginning of Section \ref{sec:amp-side} and a fixed vector $\bm v$, $\bm{L}\bm{v}$ is a centered Gaussian vector with i.i.d.~entries and variance $\langle \bm{v}, \bm{v} \rangle_p$. 
Similarly, $\bm{N}\bm{v}$ is a centered Gaussian vector with the covariance matrix $\bm{\Sigma}=\bm{I}_n+\frac{1}{n}\bm{v}\bm{v}^\top$. 
However, $\bm{L}^\top\bm{m}^{t}$ is not a centered Gaussian by the previous argument as $\bm m^t$ is not independent of $\bm{L}$. We can argue similarly for the other terms. To resolve this problem we adopt the conditioning technique developed in \cite{Bol12} and later used in \cite{BM11journal}, \cite{JM12} and \cite{Berthier}.

\subsubsection{Conditioning Technique} 
With $\bm q^t$ and $\bm m^t$ defined in \eqref{eq:func_AMP},
the AMP orbits can be written as 

\begin{align}
\label{eq:AMP_new_b}
\bm{b}^{t} &= \bm{L} \bm{q}^t-p_{t}\bm{m}^{t-1}, \\
\label{eq:AMP_new_h}
\bm{h}^{t+1} &= \bm{L}^\top \bm{m}^t -c_t\bm{q}^t,
\end{align}
and
\begin{align}
\label{eq:AMP_new_x}
\bm{y}^{t+1} &= \bm{N} \bm{q}^t-d_{t}\bm{q}^{t-1}.
\end{align}
\paragraph{The Asymmetric Orbit.}
Let us observe that, to construct $\bm{h}^{t+1}$ we need to know $\mathcal{A}_{h,t+1}\newline=\{\bm{h}^{1},\ldots,\bm{h}^{t},\bm{y}^{1},\ldots,\bm{y}^{t},\bm{b}^{0},\ldots,\bm{b}^{t}, \bm{m}^{0},\ldots,\bm{m}^{t},\bm{q}^{0},\ldots,\bm{q}^{t},\bm{\xi}_0,\bm{x}_0,\bm{\omega}_0,\bm{v}_0\}$. 
Let the sigma \newline algebra generated by these random variables be denoted by 
% $\mathcal{G}_{t,t+1}$.
{$\mathcal{G}_{t+1,t}$}. 
Since $\bm{m}^{j}$'s and $\bm{q}^{j}$'s are functions of $\bm{h}^{j}, \bm{y}^{j}, \bm{q}^{j}, \bm{\xi}_0, \bm{x}_0, \bm{\omega}_0$ and $\bm{v}_0$; $\mathcal{G}_{t+1,t}$ is the sigma-algebra generated by $\{\bm{h}^{1},\ldots,\bm{h}^{t},\bm{y}^{1},\ldots\newline,\bm{y}^{t},\bm{b}^{0},\ldots,\bm{b}^{t},\bm{\xi}_0,\bm{x}_0,\bm{\omega}_0,\bm{v}_0\}$. 
\revsag{Further, since $\bm{h}^{t+1}$ depends on $\bm{y}^{1},\ldots,\bm{y}^{t}$ through $\bm{m}^{0},\ldots,\bm{m}^{t}$ and $\bm{q}^{0},\ldots,\bm{q}^{t}$, the conditional distribution of $\bm{L}$ given {$\mathcal{G}_{t+1,t}$} is equal to the conditional distribution of $\bm L$ given 
\begin{equation}
\underbracket{[\bm{h}^{1}+c_0\bm{q}^{0}|\,\ldots|\,\bm{h}^{t}+c_{t-1}\bm{q}^{t-1}]}_{=\widebar{\bm{H}}_t} = \bm{L}^\top\underbracket{[\bm{m}^{0}|\,\ldots|\,\bm{m}^{t-1}]}_{=\bm{M}_t},
\end{equation}
and
\begin{equation}
\underbracket{[\bm{b}^{0}|\,\ldots|\,\bm{b}^{t}+p_{t}\bm{m}^{t-1}]}_{=\widebar{\bm{B}}_{t+1}} = \bm{L}\underbracket{[\bm{q}^{0}|\,\ldots|\,\bm{q}^{t}]}_{=\bm{Q}_{t+1}}.
\end{equation}}
Using Lemma 11 and Lemma 12 of \cite{BM11journal}, we get
\begin{equation}
\label{eq:e_t+1}
\bm{L}|_{\mathcal{G}_{t+1,t}} \overset{d}{=} \mathcal{E}_{t+1,t}+\mathcal{P}_{t+1,t}(\widetilde{\bm{L}}),
\end{equation}
where
\begin{multline}
 \mathcal{E}_{t+1,t} = \widebar{\bm{B}}_{t+1}(\bm{Q}_{t+1}^\top \bm{Q}_{t+1})^{-1}\bm{Q}_{t+1}^\top + \bm{M}_t(\bm{M}_t^\top \bm{M}_t)^{-1}\widebar{\bm{H}}^\top_t\\-\bm{M}_t(\bm{M}^\top_t\bm{M}_t)^{-1}\bm{M}^\top_t\widebar{\bm{B}}_{t+1}(\bm{Q}_{t+1}^\top \bm{Q}_{t+1})^{-1}\bm{Q}_{t+1}^\top,
\end{multline}
and
\begin{equation}
 \mathcal{P}_{t+1,t}(\widetilde{\bm{L}}) = P^{\perp}_{\bm{M}_t}\widetilde{\bm{L}}P^{\perp}_{\bm{Q}_{t+1}}.
\end{equation}
Here $\widetilde{\bm{L}}$ is an independent copy of $\bm{L}$ and $P^{\perp}_{\bm{M}_t}, P^{\perp}_{\bm{Q}_{t+1}}$ are the orthogonal projectors on to the orthogonal complements of the column spaces of $\bm{M}_t$ and $\bm{Q}_{t+1}$, respectively.

Next, if we consider $\bm{b}^t$ we need to know $\mathcal{A}_{b,t}=\{\bm{h}^{1},\ldots,\bm{h}^{t},\bm{y}^{1},\ldots,\bm{y}^{t},\bm{b}^{0},\ldots,\bm{b}^{t-1},\bm{m}^{0},\ldots,\newline\bm{m}^{t-1},\bm{q}^{0},\ldots,\bm{q}^{t},\bm{\xi}_0,\bm{x}_0,\bm{\omega}_0,\bm{v}_0\}$. Let the sigma algebra generated by the above mentioned variables be denoted by $\mathcal{G}_{t,t}$. \revsag{The conditional distribution of $\bm L$ given $\mathcal{G}_{t,t}$ is equal to the conditional distribution of $\bm L$ given $\widebar{\bm{H}}_t = \bm{L}^\top \bm{M}_t$ and $\widebar{\bm{B}}_t=\bm{L}\bm{Q}_t$. By Lemmas 11 and 12 of \cite{BM11journal} we get
\begin{equation}
\label{eq:E_t_t}
\bm{L}|_{\mathcal{G}_{t,t}} \overset{d}{=} \mathcal{E}_{t,t}+\mathcal{P}_{t,t}(\widetilde{\bm{L}}),
\end{equation}
where
\begin{equation}
 \mathcal{E}_{t,t} = \widebar{\bm{B}}_{t}(\bm{Q}_{t}^\top \bm{Q}_{t})^{-1}\bm{Q}_{t}^\top + \bm{M}_t(\bm{M}_t^\top \bm{M}_t)^{-1}\widebar{\bm{H}}^\top_t-\bm{M}_t(\bm{M}^\top_t\bm{M}_t)^{-1}\bm{M}^\top_t\widebar{\bm{B}}_{t}(\bm{Q}_{t}^\top \bm{Q}_{t})^{-1}\bm{Q}_{t}^\top,
\end{equation}
and
\begin{equation}
 \mathcal{P}_{t,t}(\widetilde{\bm{L}}) = P^{\perp}_{\bm{M}_t}\widetilde{\bm{L}}P^{\perp}_{\bm{Q}_{t}}.
\end{equation}}

\paragraph{The Symmetric Orbit}
Now, we consider the second orbit characterized by $\bm B$. Observe that the distribution of $\bm{y}^{t+1}$ depends on the sigma algebra generated by 
$\mathcal{A}_{{ y},t+1}=\{\bm{h}^{1},\ldots,\bm{h}^{t},\bm{y}^{1},\ldots\newline\bm{y}^{t},\bm{b}^{0},\ldots,\bm{b}^{t},\bm{m}^{0},\ldots,\bm{m}^{t-1},\bm{q}^{0},\ldots,\bm{q}^{t},\bm \xi_0,\bm{x}_0,\bm \omega_0,\bm{v}_0\}$. \revsag{This implies that we must consider the distribution of $\bm{N}$ given $\mathcal{G}_{t+1,t}$, or equivalently given 
\begin{equation}
\underbracket{[\bm{y}^{1}|\,\ldots|\,\bm{y}^{t}+d_{t-1}\bm{q}^{t-2}]}_{=\widebar{\bm{Y}}_t} = \bm{N}\underbracket{[\bm{q}^{0}|\,\ldots|\,\bm{q}^{t-1}]}_{=\bm{Q}_{t}}.
\end{equation}
Now, using Lemma 3 of \cite{JM12}, we get
\begin{equation}
\label{eq:cond_n}
\bm{N}|_{\mathcal{G}_{t+1,t}} \overset{d}{=} \mathcal{F}_{t+1,t}+\mathcal{P}_{t+1,t}(\widetilde{\bm{N}}),
\end{equation}
where
\begin{equation}
\label{eq:F_t}
 \mathcal{F}_{t+1,t} = \widebar{\bm{Y}}_{t}(\bm{Q}_{t}^\top \bm{Q}_{t})^{-1}\bm{Q}_{t}^\top + \bm{Q}_{t}(\bm{Q}_{t}^\top \bm{Q}_{t})^{-1}\widebar{\bm{Y}}^\top_t-\bm{Q}_{t}(\bm{Q}^\top_{t}\bm{Q}_{t})^{-1}\bm{Q}^\top_{t}\widebar{\bm{Y}}_{t}(\bm{Q}_{t}^\top \bm{Q}_{t})^{-1}\bm{Q}_{t}^\top,
\end{equation}
and
\begin{equation}
 \mathcal{P}_{t+1,t}(\widetilde{\bm{N}}) = P^{\perp}_{\bm{Q}_{t}}\widetilde{\bm{N}}P^{\perp}_{\bm{Q}_{t}}.
\end{equation}}
Here $\widetilde{\bm{N}}$ is an independent copy of $\bm{N}$ and $P^{\perp}_{\bm{Q}_{t}}$ is the orthogonal projector to the orthogonal complement of the column space of $\bm{Q}_{t}$. Using the above conditioning technique and the following main technical lemma (that is, Lemma \ref{lem:AMP_lem_tech}), the proof of Theorem \ref{thm:thm_6_1} is immediate.

\subsubsection{Main Technical Lemma} 
Let us denote the projection of $\bm{m}^t$ on the column space of $\bm{M}_t$ by $\bm{m}^t_{\|}$ and its ortho-complement by $\bm{m}^t_{\perp}$. Similarly $\bm{q}^t_{\|}$ denotes the projection of $\bm{q}^t$ onto the column space of $\bm{Q}_t$ and $\bm{q}^t_{\perp}$ be its ortho-complement. 
%Finally $\bm{r}^t_{\|}$ denotes the projection of $\bm{r}^t$ onto the column space of $\bm{R}_t$ and $\bm{r}^t_{\perp}$ be its ortho-complement. 
This implies, if we define
\[
\bm{\upalpha}_t = (\upalpha^{t}_{0},\ldots,\upalpha^{t}_{t-1}) = \left[\frac{\bm{M}^\top_t\bm{M}_t}{p}\right]^{-1}\frac{\bm{M}^\top_t\bm{m}^t}{p},
\]
and
\[
\bm{\upbeta}_t = (\upbeta^{t}_{0},\ldots,\upbeta^{t}_{t-1}) = \left[\frac{\bm{Q}^\top_t\bm{Q}_t}{n}\right]^{-1}\frac{\bm{Q}^\top_t\bm{q}^t}{n},
\]
%and
%\[
%\bm{\upgamma}_t = (\upgamma^{t}_{0},\ldots,\upgamma^{t}_{t-1}) = \left[\frac{\bm{R}^\top_t\bm{R}_t}{n}\right]^{-1}\frac{\bm{R}^\top_t\bm{r}^t}{n},
%\]
then we have
\[
\bm{m}^t_{\|} = \sum_{i=0}^{t-1}\upalpha^t_i\bm{m}^{i}, \hspace{0.3in} \bm{m}^{t}_{\perp}=\bm{m}^{t}-\bm{m}^t_{\|};
\]
and
\[
\bm{q}^t_{\|} = \sum_{i=0}^{t-1}\upbeta^t_i\bm{q}^{i}, \hspace{0.3in} \bm{q}^{t}_{\perp}=\bm{q}^{t}-\bm{q}^t_{\|}.
\]
%\[
%\bm{r}^t_{\|} = \sum_{i=0}^{t-1}\upgamma^t_i\bm{r}^{i}, \hspace{0.3in} \bm{r}^{t}_{\perp}=\bm{r}^{t}-\bm{r}^t_{\|}.
%\]
Finally, for two sequences of random vectors $\bm{x}_n,\bm{y}_n$, by $\bm{x}_n \overset{P}{\simeq} \bm{y}_n$ we mean $\bm{x}_n-\bm{y}_n \overset{P}{\rightarrow}0$. With all these defined, let us now state the following general result.
\begin{lem}
Suppose that the conditions of Theorem \ref{thm:thm_6_1} hold. 
Then for all $t \in \mathbb{N} \cup \{0\}$, we get
%$f_0$ is defined in a way such that $\vartheta^2_0= c \lim_{p \rightarrow \infty}\langle \bm{q}^0, \bm{q}^0 \rangle_n$ and $\sigma^2_1= \lim_{n \rightarrow \infty}\langle \bm{r}^0, \bm{r}^0 \rangle_n$, when $p/n \rightarrow 1/c$. Then for all $t \in \mathbb{N} \cup \{0\}$, we get
\begin{enumerate}[label=(\alph*)]
\item
%\[ 
%\bm{h}^{t+1}|_{\mathcal{G}_{t+1,t}} \overset{d}{=} \sum_{i=0}^{t-1}\upalpha^{t}_i\bm{h}^{i+1} + \widetilde{\bm{L}}^\top\bm{m}^{t}_{\perp} - P_{\bm{Q}_{t+1}}\widetilde{\bm{L}}^\top\bm{m}^{t}_{\perp},
%\]
% \[ 
%\bm{x}^{t+1}|_{\mathcal{G}_{t+1,t}} \overset{d}{=} \sum_{i=0}^{t-1}\upgamma^{t}_i\bm{x}^{i+1} + \widetilde{\bm{N}}\bm{r}^{t}_{\perp} - P_{\bm{R}_{t}}\widetilde{\bm{N}}\bm{r}^{t}_{\perp},
%\]
% \[ 
%\bm{b}^{t}|_{\mathcal{G}_{t,t}} \overset{d}{=} \sum_{i=0}^{t-1}\upbeta^{t}_i\bm{b}^{i} + \widetilde{\bm{L}}\bm{q}^{t}_{\perp} - P_{\bm{M}_{t}}\widetilde{\bm{L}}\bm{q}^{t}_{\perp} .
%\]
\begin{align*}
	 % \[
\bm{h}^{t+1}|_{\mathcal{G}_{t+1,t}} 
& \overset{d}{=} \sum_{i=0}^{t-1}\upalpha^{t}_i\bm{h}^{i+1} + \widetilde{\bm{L}}^\top\bm{m}^{t}_{\perp} + \widetilde{\bm{Q}}_{t+1}o(1),\\
\bm{b}^{t}|_{\mathcal{G}_{t,t}} 
& \overset{d}{=} \sum_{i=0}^{t-1}\upbeta^{t}_i\bm{b}^{i} + \widetilde{\bm{L}}\bm{q}^{t}_{\perp} + \widetilde{\bm{M}}_{t}o(1),
% \]
\end{align*}
and
\begin{align*}
\bm{y}^{t+1}|_{\mathcal{G}_{t+1,t}} \overset{d}{=} \sum_{i=0}^{t-1}\upbeta^{t}_i\bm{y}^{i+1} + \widetilde{\bm{N}}\bm{q}^{t}_{\perp} + \widetilde{\bm{Q}}_{t}o(1),\\
\end{align*}
where $\widetilde{\bm{Q}}_t$(alternatively, $\widetilde{\bm{M}}_t$) is a matrix whose columns form an orthogonal basis of $\bm{Q}_t$\newline (respectively, $\bm{M}_t$), and
{$\widetilde{\bm{Q}}^\top_t\widetilde{\bm{Q}}_t=n\bm{I}_{t}$ ($\widetilde{\bm{M}}^\top_t\widetilde{\bm{M}}_t=p\bm{I}_{t}$).}
% $\widetilde{\bm{Q}}_t\widetilde{\bm{Q}}^\top_t=\widetilde{\bm{Q}}^\top_t\widetilde{\bm{Q}}_t=n\bm{I}_{n}$ ($\widetilde{\bm{M}}_t\widetilde{\bm{M}}^\top_t=\widetilde{\bm{M}}^\top_t\widetilde{\bm{M}}_t=p\bm{I}_{p}$).

\item For all partially pseudo-Lipschitz functions $\phi_h:\mathbb{R}^{2t+4}\rightarrow \mathbb{R}$, $\phi_b:\mathbb{R}^{t+3}\rightarrow\mathbb{R}$
\begin{align}
 &\lim_{n \rightarrow \infty}\frac{1}{n}\sum_{i=1}^{n}\phi_h(h^{1}_i,\ldots,h^{t+1}_i,y^{1}_i,\ldots,y^{t+1}_i,\xi_{0,i},x_{0,i})\\
 &\overset{a.s.}{=}\mathbb{E}\left\{\phi_h(\tau_0Z_0,\ldots,\tau_{t}Z_t,\sigma_1\widetilde{Z}_1,\ldots,\sigma_{t+1}\widetilde{Z}_{t+1},\widetilde{\Xi}_0,\widetilde{X}_0)\right\},   
\end{align}

\[
\lim_{p \rightarrow \infty}\frac{1}{p}\sum_{i=1}^{p}\phi_b(b^{0}_i,\ldots,b^{t}_i,\omega_{0,i},v_{0,i}) \overset{a.s.}{=}\mathbb{E}\left\{\phi_b(\vartheta_0\widecheck{Z}_0,\ldots,\vartheta_{t}\widecheck{Z}_t,\widetilde{\Omega}_0,\widetilde{V}_0)\right\},
\]
where $(Z_0,\ldots,Z_t)$, $(\widetilde{Z}_1,\ldots,\widetilde{Z}_t)$,  $(\widecheck{Z}_1,\ldots,\widecheck{Z}_t)$, { $(\widetilde{\Xi}_0, \widetilde{X}_0)$} and {$(\widetilde{\Omega}_0, \widetilde{V}_0)$} are mutually independent random vectors. Marginally, $Z_i, \widetilde{Z}_i, \widecheck{Z}_i \sim N(0,1)$, { $(\widetilde{\Xi}_0,\widetilde{X}_0) \sim P_{\xi,x}$} and { $(\widetilde{\Omega}_0,\widetilde{V}_0) \sim P_{\omega,v}$}.
% \nb{change the statement to conform with theorem 7.1. 20210405zm}

\item For all $0 \le k,\ell \le t$, the following equations hold and all the limits exist, are bounded and have degenerate distributions:
\begin{align*}
% \[
\lim_{n \rightarrow \infty}\langle \bm{h}^{k+1}, \bm{h}^{\ell+1} \rangle_n 
& \overset{a.s.}{=} \lim_{p \rightarrow \infty} \langle \bm{m}^{k}, \bm{m}^{\ell} \rangle_p,\\
\lim_{p \rightarrow \infty}\langle \bm{b}^{k}, \bm{b}^{\ell} \rangle_p  
& \overset{a.s.}{=} c\,\lim_{n \rightarrow \infty} \langle \bm{q}^{k}, \bm{q}^{\ell} \rangle_n,
% \]
\end{align*}
and
\begin{align*}
\lim_{n \rightarrow \infty}\langle \bm{y}^{k+1}, \bm{y}^{\ell+1} \rangle_n 
\overset{a.s.}{=} \lim_{n \rightarrow \infty} \langle \bm{q}^{k}, \bm{q}^{\ell} \rangle_n.
\end{align*}

\item For all $0 \le k,\ell \le t$ and for any partially Lipschitz functions $\varphi: \mathbb{R}^4 \rightarrow \mathbb{R}$, $\psi: \mathbb{R}^3 \rightarrow \mathbb{R}$, with $\varphi^\prime_1$ being the derivative of $\varphi$ with respect to the first coordinate, $\varphi^\prime_2$ being the derivative of $\varphi$ with respect to the second coordinate and $\psi^\prime$ being the derivative of $\psi$ with respect to the first coordinate, where $\varphi^\prime_1$, $\varphi^\prime_2$ and $\psi^\prime$ being partially Lipschitz, the following equations hold and all the limits exist, are bounded and have degenerate distributions:
\begin{align*}
% \[
\lim_{n \rightarrow \infty}\langle\bm{h}^{k+1},\varphi(\bm{h}^{\ell+1},\bm{y}^{\ell+1},\bm \xi_0,\bm{x}_0)\rangle_n 
& \overset{a.s.}{=} \lim_{n \rightarrow \infty}\langle\bm{h}^{k+1},\bm{h}^{\ell+1}\rangle_n \langle\varphi^\prime_1(\bm{h}^{\ell+1},\bm{y}^{\ell+1},\bm \xi_0,\bm{x}_0)\rangle_n,\\
\lim_{p \rightarrow \infty}\langle\bm{b}^{r},\psi(\bm{b}^{s},\bm \omega_0,\bm{v}_0)\rangle_p 
& \overset{a.s.}{=} \lim_{p \rightarrow \infty}\langle\bm{b}^{r},\bm{b}^{s}\rangle_p \langle\psi^\prime(\bm{b}^{s},\bm \omega_0,\bm{v}_0)\rangle_p,
% \]
\end{align*}
and
\begin{align*}
\lim_{n \rightarrow \infty}\langle\bm{y}^{r+1},\varphi(\bm{h}^{s+1},\bm{y}^{s+1},\bm \xi_0,\bm{x}_0)\rangle_n & \overset{a.s.}{=} \lim_{n \rightarrow \infty}\langle\bm{y}^{r+1},\bm{y}^{s+1}\rangle_n \langle\varphi^\prime_2(\bm{h}^{s+1},\bm{y}^{s+1},\bm \xi_0,\bm{x}_0)\rangle_n.
\end{align*}

\item The following relations hold almost surely
\[
\limsup_{n \rightarrow \infty}\frac{1}{n}\sum_{i=1}^{n}(h^{t+1}_i)^2 < \infty,
\quad
% \]
% \[
\limsup_{p \rightarrow \infty}\frac{1}{p}\sum_{i=1}^{n}(b^{t}_i)^2 < \infty,
\quad
% \]
% \[
\limsup_{n \rightarrow \infty}\frac{1}{n}\sum_{i=1}^{n}(y^{t+1}_i)^2 < \infty.
\]
\item For all $0 \le r \le t$:
\[
\lim_{n \rightarrow \infty}
% \frac{1}{n}
\langle\bm{h}^{t+1},\bm{q}^0\rangle_n \overset{a.s.}{=} 0.
\]
\item For all $0 \le s \le t-1$, $0 \le k
%,\ell 
\le t$, the following limits exist, and there exist strictly positive constants $\rho_k,
%\varsigma_\ell
, \varkappa_s$ such that the following relations hold almost surely:
\[
\lim_{n \rightarrow \infty}\langle\bm{q}^k_{\perp},\bm{q}^k_{\perp}\rangle_n > \rho_k,
\quad
% \]
%\[
%\lim_{n \rightarrow \infty}\langle\bm{r}^\ell_{\perp},\bm{r}^\ell_{\perp}\rangle_n > \varsigma_\ell,
%\]
% \[
\lim_{p \rightarrow \infty}\langle\bm{m}^s_{\perp},\bm{m}^s_{\perp}\rangle_p > \varkappa_s.
\]
\end{enumerate}
\label{lem:AMP_lem_tech}
\end{lem}

\subsubsection{Proof of Theorem \ref{thm:thm_6_1}}
\begin{proof}
The desired result is a direct consequence of Lemma \ref{lem:AMP_lem_tech} claim (b).
\end{proof}

\subsection{Proof of Lemma \ref{lem:AMP_lem_tech}}
We shall prove this lemma by induction. We shall first show that the statements (a)-(g) hold true for $\bm b^0, \bm h^1$ and $\bm y^1$. Then assuming that the result holds true for $0 \le s \le t-1$, we shall show that the statements (a)-(g) hold $\bm b^t, \bm h^{t+1}$ and $\bm y^{t+1}$.

First, let us observe that  if ${\sff}_t(x,y,\xi_0,x_0)$ is free of $x,y$ almost surely with respect to $\xi_0$ and $x_0$ or if ${\sfg}_t(u,\omega_0,v_0)$ is free of $u$ almost surely with respect to $\omega_0$ and $v_0$, then the lemma is immediate. So we assume that these degenerate cases do not arise in the rest of this proof.

\paragraph{The Base Case ($\bm b^0, \bm h^1$ and $\bm y^1$).}  The proofs of the assumptions $(a)-(g)$ for $\bm b^0$ follows using exactly the same arguments as in $\mathcal{B}_0$ of \cite{BM11journal}. So we skip the details. Next, let us observe that as $\bm{y}^1=\bm{N}\bm{q}^0$, and $\bm{Q}_0$ is empty matrix, we have $\bm{q}^0_{\perp}=\bm{q}^0$. Thus using the definition of $\mathcal{G}_{1,0}$, we conclude
\[
\bm{y}^1|_{\mathcal{G}_{1,0}} \overset{d}{=} \bm{N}\bm{q}^0_{\perp}.
\]
Again, the assertions (a), (c), (e) and (f) for $\bm h^1$ follows immediately using the techniques of $\mathcal{H}_1$ of \cite{BM11journal}. Now, let us consider these assertions for $\bm y^1$. For the assertion $(a)$, let us observe that, as $\bm{y}^1=\bm{N}\bm{q}^0$, and $\bm{Q}_0$ is empty matrix, we have $\bm{q}^0_{\perp}=\bm{q}^0$. Thus using the definition of $\mathcal{G}_{1,0}$, we conclude
\[
\bm{y}^1|_{\mathcal{G}_{1,0}} \overset{d}{=} \bm{N}\bm{q}^0_{\perp}.
\]
Next, for the assertion (c), using Lemma \ref{lem:key_help_lem}(c), we have
\[
\lim_{n \rightarrow \infty} \langle\bm{y}^1,\bm{y}^1\rangle_n|_{\mathcal{G}_{1,0}} \overset{d}{=} \lim_{n \rightarrow \infty}\frac{1}{n}\,\|\bm{N}\bm{q}^0\|^2 \overset{a.s.}{=} \lim_{n \rightarrow \infty}\langle\bm{q}^0,\bm{q}^0\rangle_n \overset{a.s.}{=} \sigma^2_1.
\]
Then for assertion (e), by Lemma \ref{lem:key_help_lem} (d), we have
\[
\lim_{n \rightarrow \infty}\frac{1}{n}\sum_{i=1}^{n}(y^1_i)^{2} = \lim_{n \rightarrow \infty}\frac{1}{n}\sum_{i=1}^{n}[\bm{N}\bm{q}^0]^2_i \overset{a.s.}{=} \lim_{n \rightarrow \infty}\langle\bm q^0,\bm q^0\rangle_n\;\lim_{n \rightarrow \infty}\frac{1}{n}\sum_{i=1}^{n}z^2_i,
\]
which implies the assertion. Next, we prove the assertion (b) for the pair $(\bm h^1, \bm y^1)$. Let us observe that, as $\phi_h$ is partially pseudo-Lipschitz, by Lemma \ref{lem:key_help_lem} (d) { with $r=1$ and $m=2$, we have}
%\nb{need to spell out more detail. Need part (e) for instance. 20210331zm}
\[
\lim_{n \rightarrow \infty}\frac{1}{n}\sum_{i=1}^{n}\left[\phi_h(h^1_i, { y}^1_i,\xi_{0,i},x_{0,i})-\phi_h(h^1_i,z_i,\xi_{0,i},x_{0,i})\right] \overset{a.s.}{=} 0,
\]
where $\bm{z} \sim N_n(0,(\|\bm{q}^0\|^2/n)\bm{I}_n)$. Now conditional on $\mathcal{G}_{1,0}$, we have
\[
\frac{1}{n}\sum_{i=1}^{n}\phi_h(h^1_i,z_i,\xi_{0,i},x_{0,i})
{ \Big|_{\mathcal{G}_{1,0}}} \overset{d}{=} \frac{1}{n}\sum_{i=1}^{n}\phi_h([\widetilde{\bm{L}}^\top\bm{m}^0]_i+o(1)q^0_i,z_i,\xi_{0,i},x_{0,i}).
\]
Then using the techniques used to prove $\mathcal{H}_1$(b) of \cite{BM11journal}, we can show
\[
\frac{1}{n}\sum_{i=1}^{n}\left[\mathbb{E}\left\{\phi_h([\widetilde{\bm{L}}^\top\bm{m}^0]_i,z_i,\xi_{0,i},x_{0,i})|\mathcal{G}_{1,0}\right\}-\mathbb{E}\left\{\phi_h(\tau_0Z_i,\sigma_1\widetilde{Z}_i,\xi_{0,i},x_{0,i})|\mathcal{G}_{1,0}\right\}\right] \overset{a.s.}{=} 0.
\]
It is easy to see that 
\begin{align}
\mathbb{E}\left\{\phi_h(\tau_0Z_i,\sigma_1\widetilde{Z}_i,\xi_{0,i},x_{0,i})|\mathcal{G}_{1,0}\right\} =\mathbb{E}_{Z_i,\widetilde{Z}_i}\left\{\phi_h(\tau_0Z_i,\sigma_1\widetilde{Z}_i,\xi_{0,i},x_{0,i})\right\},
\end{align}
where the expectation is taken with respect to $Z_i,\widetilde{Z}_i$, treating $\xi_{0,i},x_{0,i}$ as constants. 
Now using SLLN for i.i.d $\psi(\xi_{0,i},x_{0,i})=\mathbb{E}_{Z,\widetilde{Z}}\left\{\phi_h(\tau_0Z,\sigma_1\widetilde{Z},\xi_{0,i},x_{0,i})\right\}$ we get
\[
\frac{1}{n}\sum_{i=1}^{n}\mathbb{E}\left\{\phi_h([\widetilde{\bm{L}}^\top\bm{m}^0]_i,z_i,\xi_{0,i},x_{0,i})|\mathcal{G}_{1,0}\right\}\overset{a.s.}{=} \mathbb{E}\left\{\phi_h(\tau_0Z,\sigma_1\widetilde{Z},\widetilde{\Xi}_0,\widetilde{X}_0)\right\}.
\]
On the right side of the last display, we have that $Z$, $\widetilde{Z}$ and $(\widetilde{\Xi}_0,\widetilde{X}_0)$ are mutually independent.
By our foregoing arguments, the randomness of $Z$ comes from that of $\bm{\widetilde{L}}$, the randomness of $\widetilde{Z}$ comes from that of $\bm{N}$, and both are independent of $(\bm{\xi}_0, \bm{x}_0)$.
To prove the assertion (d) for $\bm h^{1}$ and $\bm y^{1}$, we use part (b) for partially pseudo-Lipschitz function $\phi_h(x,z,\xi_{0,i},x_{0,i})=x\varphi(x,z,\xi_{0,i},x_{0,i})$ (by Lemma \ref{lem:prop_pl}(1)) to obtain $\lim_{n \rightarrow \infty}\langle\bm{h}^1,\varphi(\bm{h}^1,\bm{x}^1,\bm \xi_0,\bm{x}_0)\rangle_n \overset{a.s.}{=}\mathbb{E}\left\{\tau_0Z\varphi(\tau_0Z,\sigma_1\widetilde{Z},\widetilde{\Xi}_0,\widetilde{X}_0)\right\}$. 
Now, \revsag{we have using Lemma 4 of \cite{BM11journal}}
\begin{align*}
\mathbb{E}\left\{\tau_0Z\varphi(\tau_0Z,\sigma_1\widetilde{Z},\widetilde{\Xi}_0,\widetilde{X}_0)\right\} & = \mathbb{E}\left\{\mathbb{E}\left\{\tau_0Z\varphi(\tau_0Z,\sigma_1\widetilde{Z},\widetilde{\Xi}_0,\widetilde{X}_0)|\widetilde{Z},\widetilde{\Xi}_0,\widetilde{X}_0\right\}\right\}\\
														& = \mathbb{E}\left\{\mathbb{E}\left\{\tau^2_0\varphi^{\prime}_1(\tau_0Z,\sigma_1\widetilde{Z},\widetilde{\Xi}_0,\widetilde{X}_0)|\widetilde{Z},\widetilde{\Xi}_0,\widetilde{X}_0\right\}\right\}\\
% & \overset{(1)}{=} \mathbb{E}\left\{\tau^2_0\mathbb{E}\left\{\varphi^{\prime}_1(\tau_0Z,\sigma_1\widetilde{Z},\widetilde{\Xi}_0,\widetilde{X}_0)|\widetilde{Z},\widetilde{\Xi}_0,\widetilde{X}_0\right\}\right\}\\
														& = \tau^2_0\, \mathbb{E}\left\{\varphi^{\prime}_1(\tau_0Z,\sigma_1\widetilde{Z},\widetilde{\Xi}_0,\widetilde{X}_0)\right\}.
\end{align*}
The second equality holds since $Z$ is independent of $\widetilde{Z},\widetilde{\Xi}_0$ and $\widetilde{X}_0$.
Note that \newline $\tau^2_0=\lim_{n \rightarrow \infty} \langle\bm{h}^1,\bm{h}^1\rangle_n$. Using part (b) and the fact that $\varphi^\prime_1$ is partially Lipschitz we get $\lim_{n \rightarrow \infty}\langle\varphi^{\prime}_1(\bm{h}^1,\bm{x}^1,\bm \xi_0,\bm{x}_0)\rangle_n \overset{a.s.}{=} \mathbb{E}\left\{\varphi^{\prime}_1(\tau_0Z,\sigma_1\widetilde{Z},\widetilde{\Xi}_0,\widetilde{X}_0)\right\}$. The assertion about \newline $\lim_{n \rightarrow \infty}\langle\bm{y}^1,\varphi(\bm{h}^1,\bm{y}^1,\bm \xi_0,\bm{x}_0)\rangle_n$ follows similarly. Finally, since $t=0$ and $\bm{q}^0=\bm{q}^0_\perp$, the assertion (g) follows from \[\lim_{p \rightarrow \infty}\langle\bm{q}^0,\bm{q}^0\rangle_n=\frac{1}{c}\vartheta^2_0 < \infty.\]
\paragraph{Inductive Step.} Let us assume that the assertions (a)-(g) for $\bm b^s, \bm h^{s+1}$ and $\bm y^{s+1}$ for $0 \le s \le t-1$. We shall show that the assertions hold for $s=t$. The assertions for $\bm b^t$ follows exactly \revsag{using the same arguments used to prove $\mathcal{B}_t$, (a)-(g) of \cite{BM11journal}. }

Next, we consider the assertion (g) for $\bm h^{t+1}$ and $\bm y^{t+1}$. \revsag{Applying the induction
hypothesis to the partially pseudo-Lipschitz function $\phi_b(h^r_i,y^r_i,\xi_{0,i},x_{0,i})={ \sff}_{r}(h^r_i,y^r_i,\xi_{0,i},x_{0,i})$ and \newline ${\sff}_{s}(h^s_i,y^s_i,\xi_{0,i},x_{0,i})$ { (by Lemma \ref{lem:prop_pl}(1))} such that} for $1 \le r,s \le t$, 
we have almost surely
\[
\lim_{n \rightarrow \infty}\langle\bm{q}^r,\bm{q}^s\rangle_n = \mathbb{E}\left\{{ \sff}_{r}(\tau_{r-1}Z_{r-1},\sigma_{r}\widetilde{Z}_{r},\widetilde{\Xi}_0,\widetilde{X}_0)
{ \sff}_{s}(\tau_{s-1}Z_{s-1},\sigma_{s}\widetilde{Z}_{s},\widetilde{\Xi}_0,\widetilde{X}_0)\right\}.
\]
Further 
\[
\langle\bm{q}^t_\perp,\bm{q}^t_\perp\rangle_n = \langle\bm{q}^t,\bm{q}^t\rangle_n - \frac{(\bm{q}^{t})^\top \bm{Q}_t}{n}\left[\frac{\bm{Q}^\top_t \bm{Q}_t}{n}\right]^{-1}\frac{\bm{Q}^\top_t\bm{q}^{t}}{n}.
\]
Using induction hypotheses part (g), we have $\lim_{p \rightarrow \infty}\langle \bm{q}^r_\perp, \bm{q}^r_\perp\rangle_n > \rho_r$ for all $r \le t-1$. Now using Lemma 9 of \cite{BM11journal}, for large enough $n$ the smallest eigenvalue of matrix $\bm{Q}^\top_t \bm{Q}_t/n$ is larger than positive constant $c^\prime$ independent of $n$. By Lemma 10 of \cite{BM11journal}, $\bm{Q}^\top_t \bm{Q}_t/n$ converges to an invertible limit. Hence, we have
\[
\lim_{n \rightarrow \infty}\langle \bm{q}^t_\perp, \bm{q}^t_\perp\rangle_n \overset{a.s.}{=} \mathbb{E}\left\{[{ \sff}_t(\tau_{t-1}Z_{t-1},\sigma_t\widetilde{Z}_{t},\widetilde{\Xi}_0,\widetilde{X}_0)]^2\right\}-u^\top C^{-1}u
\]  
with $u \in \mathbb{R}^t$ and $C \in \mathbb{R}^{t}\times \mathbb{R}^t$ such that $1\le r,s\le t$:
\[
u_r = \mathbb{E}\left\{{ \sff}_r(\tau_{r-1}Z_{r-1},\sigma_r\widetilde{Z}_{r},\widetilde{\Xi}_0,\widetilde{X}_0){ \sff}_t(\tau_{t-1}Z_{t-1},\sigma_t\widetilde{Z}_{t},\widetilde{\Xi}_0,\widetilde{X}_0)\right\},
\]
and
\[
C_{r,s}=\mathbb{E}\left\{{ \sff}_r(\tau_{r-1}Z_{r-1},\sigma_r\widetilde{Z}_{r},\widetilde{\Xi}_0,\widetilde{X}_0){ \sff}_s(\tau_{s-1}Z_{s-1},\sigma_s\widetilde{Z}_{s},\widetilde{\Xi}_0,\widetilde{X}_0)\right\}.
\]
{ If we show $\mbox{Var}[\tau_{r-1}Z_{ r-1}|\tau_0Z_0,\ldots,\tau_{r-2}Z_{ r-2},\sigma_1\widetilde{Z}_1,\ldots,\sigma_{r-1}\widetilde{Z}_{r-1}]$ and $\mbox{Var}[\sigma_r\widetilde{Z}_r|\tau_0Z_0,\ldots,\newline\tau_{r-2}Z_{ r-2},\sigma_1\widetilde{Z}_1,\ldots,\sigma_{r-1}\widetilde{Z}_{r-1}]$} are strictly positive for $1 \le r \le t$, then using Lemma \ref{lem:lem_10} the result follows. Using the induction hypotheses, part (b), and \revsag{the techniques similar to the proof of $\mathcal{B}_t$ (g) of \cite{BM11journal}}, we have for all $1 \le r,s \le t$:
\begin{align}
\lim_{n \rightarrow \infty}\langle \bm{h}^r_\perp, \bm{h}^r_\perp \rangle_n &= \lim_{n \rightarrow \infty}\left(\langle \bm{h}^r, \bm{h}^r \rangle_n - \frac{(\bm{h}^r)^\top \bm{H}_r}{n}\left[\frac{(\bm{H}_r)^\top \bm{H}_r}{n}\right]^{-1}\frac{\bm{H}^\top_r\bm{h}^r}{n}\right)\\
& = \mbox{Var}[\tau_{r-1}Z_{r-1}|\tau_0Z_0,\ldots,\tau_{r-2}Z_{r-2},
\sigma_1 {\widetilde{Z}}_1,\ldots,\sigma_{r-1}{ \widetilde{Z}}_{r-1}].
\end{align}
Next, using part (c) of the induction hypotheses,  we have almost surely
\begin{align}
\lim_{n \rightarrow \infty}\langle \bm{h}^r_\perp, \bm{h}^r_\perp \rangle_n &= \lim_{n \rightarrow \infty}\left(\langle \bm{h}^r, \bm{h}^r \rangle_n - \frac{(\bm{h}^r)^\top \bm{H}_r}{n}\left[\frac{(\bm{H}_r)^\top \bm{H}_r}{n}\right]^{-1}\frac{H^\top_r\bm{h}^r}{n}\right)\\
&= \lim_{p \rightarrow \infty}\left(\langle \bm{m}^{r-1}, \bm{m}^{r-1} \rangle_p - \frac{(\bm{m}^{r-1})\bm{M}^\top _{r-1}}{p}\left[\frac{(\bm{M}_{r-1})^\top \bm{M}_{r-1}}{p}\right]^{-1}\frac{\bm{M}^\top_{r-1}\bm{m}^{r-1}}{p}\right)\\
&= \lim_{p \rightarrow \infty} \langle \bm{m}^{r-1}_\perp, \bm{m}^{r-1}_\perp \rangle_p.
\end{align} 
Using part (g) of the induction hypotheses, we have $\lim_{p \rightarrow \infty} \langle \bm{m}^{r-1}_\perp, \bm{m}^{r-1}_\perp \rangle_p > \varkappa_{r-1} > 0$, and hence the result follows.
\\
The assertion (a) for $\bm h^{t+1}$ follows using the techniques \revsag{used to prove $\mathcal{H}_{t+1}$(a) of \cite{BM11journal}}. To prove the same for $\bm y^{t+1}$, we consider $\mathcal{F}_{t+1,t}$ defined in \eqref{eq:F_t}, we get
\[
\mathcal{F}_{t+1,t}\bm{q}^t_\perp = \bm{Q}_t(\bm{Q}^\top_t\bm{Q}_t)^{-1}\bm{\widebar{Y}}^\top_t\bm{q}^t_\perp.
\]
Further, note that using $\bm{Q}^\top_t\bm{\widebar{Y}}_t=\bm{\widebar{Y}}^\top_t\bm{Q}_t$ we get
\[
\mathcal{F}_{t+1,t}\bm{q}^t_\|=\bm{\widebar{Y}}_t(\bm{Q}^\top_t\bm{Q}_t)^{-1}\bm{Q}^\top_t\bm{q}^t_\|.
\]
Combining these two equations we get
\[
\mathcal{F}_{t+1,t}\bm{q}^t = \bm{Q}_t(\bm{Q}^\top_t\bm{Q}_t)^{-1}\bm{\widebar{Y}}^\top_t\bm{q}^t_\perp+\bm{\widebar{Y}}_t(\bm{Q}^\top_t\bm{Q}_t)^{-1}\bm{Q}^\top_t\bm{q}^t_\|.
\]
Then using \eqref{eq:cond_n} we get
\[
\bm{y}^{t+1}|_{\mathcal{G}_{t+1,t}} \overset{d}{=} \bm{Q}_t(\bm{Q}^\top_t\bm{Q}_t)^{-1}\bm{\widebar{Y}}^\top_t\bm{q}^t_\perp+\bm{\widebar{Y}}_t(\bm{Q}^\top_t\bm{Q}_t)^{-1}\bm{Q}^\top_t\bm{q}^t_\| + P^\perp_{\bm{Q}_t}\widetilde{\bm{N}}P^\perp_{\bm{Q}_t}\bm{q}^t-d_t\bm{q}^{t-1}.
\]
Now $\bm{\widebar{Y}}_t=\bm{Y}_t+[0|\bm{Q}_{t-1}]\bm{D}_t$, where $\bm{Y}_t=[\bm{y}^1|\ldots|\bm{y}^{t}]$ and $\bm{D}_t=\mbox{diag}(d_0,\ldots,d_{t-1})$. As $\bm{\widebar{Y}}^\top_t\bm{q}^t_\perp=\bm{Y}^\top_t\bm{q}^t_\perp$, so we need to show
\[
\bm{Q}_t(\bm{Q}^\top_t\bm{Q}_t)^{-1}\bm{Y}^\top_t\bm{q}^t_\perp+[0|\bm{Q}_{t-1}]\bm{D}_t(\bm{Q}^\top_t\bm{Q}_t)^{-1}\bm{Q}^\top_t\bm{q}^t-d_t\bm{q}^{t-1} = \bm{Q}_to(1),
\]
or equivalently
\[
[0|\bm{Q}_{t-1}]\bm{D}_t\bm{\upbeta}^t+ \bm{Q}_t(\bm{Q}^\top_t\bm{Q}_t)^{-1}\bm{Y}^\top_t\bm{q}^t_\perp-d_t\bm{q}^{t-1} = \bm{Q}_to(1),
\]
We need to show that the coefficients of $\bm{q}^{\ell-1}$ \revsag{converge} to zero for $\ell=1,\ldots,t$. Now the coefficient of $\bm{q}^{\ell-1}$ is given by
\[
[\bm{Q}_t(\bm{Q}^\top_t\bm{Q}_t)^{-1}\bm{Y}^\top_t\bm{q}^t_\perp]_\ell-d_l(-\upbeta^t_\ell)^{\mathbb{I}_{\ell \neq t}}=\sum_{k=1}^{t}\left[\left(\frac{\bm{Q}^\top_t\bm{Q}_t}{n}\right)^{-1}\right]_{\ell,k}\langle \bm{y}^k,\bm{q}^{t}-\sum_{s=0}^{t-1}\upbeta^t_s\bm{q}^s\rangle_n-d_\ell(-\upbeta^t_\ell)^{\mathbb{I}_{\ell \neq t}}.
\]
Denoting $\bm{Q}^\top_t\bm{Q}_t/n$ by $\bm{G}$, we get
\[
\lim_{n \rightarrow \infty} \mbox{Coefficient of $\bm{q}^{\ell-1}$} = \lim_{n \rightarrow \infty} \left\{\sum_{k=1}^{t}(G^{-1})_{\ell,k}\langle \bm{y}^k,\bm{q}^{t}-\sum_{s=0}^{t-1}\upbeta^t_s\bm{q}^s\rangle_n-d_\ell(-\upbeta^t_\ell)^{\mathbb{I}_{\ell \neq t}}\right\}.
\]
Using parts (c) and (d) of the induction hypotheses, for $k =1,\ldots,t$, we get
\begin{align}
\lim_{n \rightarrow \infty} \mbox{Coefficient of $\bm{q}^{\ell-1}$} &\overset{a.s.}{=} \lim_{n \rightarrow \infty} \left\{\sum_{k=1}^{t}(G^{-1})_{\ell,k}\langle \bm{y}^k,d_t\bm{x}^{t}-\sum_{s=0}^{t-1}\upbeta^t_sd_s\bm{y}^s\rangle_n-d_\ell(-\upbeta^t_\ell)^{\mathbb{I}_{\ell \neq t}}\right\}\\
& \overset{a.s.}{=} \lim_{n \rightarrow \infty} \left\{\sum_{k=1}^{t}(G^{-1})_{\ell,k}[G_{k,t}d_t-\sum_{s=0}^{t-1}\upbeta^t_sG_{k,s}d_s]-d_\ell(-\upbeta^t_\ell)^{\mathbb{I}_{\ell \neq t}}\right\}\\
& =  \lim_{n \rightarrow \infty} \left\{d_t\mathbb{I}_{t=\ell}-\sum_{s=0}^{t-1}\upbeta^t_sd_s\mathbb{I}_{\ell=s}-d_\ell(-\upbeta^t_\ell)^{\mathbb{I}_{\ell \neq t}}\right\}\\
& = 0.
\end{align}
This implies
\[
\bm{y}^{t+1}|_{\mathcal{G}_{t+1,t}} \overset{d}{=} \sum_{i=0}^{t-1}\upbeta^t_i\bm{y}^{i+1}+\widetilde{\bm{N}}\bm{q}^t_\perp-P_{\bm{Q}_t}\widetilde{\bm{N}}\bm{q}^t_\perp+\bm{Q}_to(1).
\]
Using the fact that
\[
\widetilde{\bm{N}}=\frac{1}{2}\left(\bm{N}^\top_1+\bm{N}_1\right),
\]
where all entries of $\bm{N}_1$ are distributed as $N(0,1/n)$. Hence
\[
P_{\bm{Q}_t}\widetilde{\bm{N}}\bm{q}^t_\perp = \frac{1}{2}\left(\sum_{i=1}^{t}\langle\widetilde{\bm{q}}_i, \bm{N}^\top_1\bm{q}^t_\perp\rangle_n\widetilde{\bm{q}}_i\right) + \frac{1}{2}\left(\sum_{i=1}^{t}\langle\widetilde{\bm{q}}_i, \bm{N}_1\bm{q}^t_\perp\rangle_n\widetilde{\bm{q}}_i\right),
\]
where $\widetilde{\bm{q}}_i$ are columns of $\widetilde{\bm Q}_t$. Using Lemma \ref{lem:key_help_lem} (b) along with arguments similar to the proof of $\mathcal{H}_{t+1}$(a) of \cite{BM11journal} and $\langle\bm{q}^t_\perp,\bm{q}^t_\perp\rangle_n<\infty$, we get
\[
P_{\bm{Q}_t}\widetilde{\bm{N}}\bm{q}^t_\perp = \widetilde{\bm{Q}}_to(1).
\]
Hence we have the result. 
\\
Using the induction hypotheses and \revsag{the proof of $\mathcal{H}_{t+1}$(d) of \cite{BM11journal}}, one can prove the assertion (d) for both $\bm h^{t+1}$ and $\bm y^{t+1}$.
\\
The assertion (e) for $\bm h^{t+1}$ \revsag{follows using the proof of $\mathcal{H}_{t+1}$(e) of \cite{BM11journal}}. To show the same for $\bm y^{t+1}$ we condition on $\mathcal{G}_{t+1,t}$, and using the assertion $(a)$ we get
\[
\frac{1}{n}\sum_{i=1}^{n}(y^{t+1}_l)^2 \le \frac{C}{n}\sum_{i=1}^{n}\left(\sum_{r=0}^{t-1}\upbeta^{t}_{r}\bm{y}^{r+1}_i\right)^2 + \frac{C}{n}\sum_{i=1}^n\left([P^\perp_{\bm{Q}_t}\bm{\widetilde{N}}^\top\bm{q}^t_\perp]_i\right)^2+\,o(1)\frac{C}{n}\sum_{r=0}^{t-1}\sum_{i=1}^{n}\left([\bm{q}^r]_i\right)^2.
\]
Now \revsag{using the techniques described in $\mathcal{B}_t$ (e) of \cite{BM11journal}}, we can show
\[
\frac{C}{n}\sum_{i=1}^{n}\left(\sum_{r=0}^{t-1}\upbeta^{t}_{r}\bm{y}^{r+1}_i\right)^2 < \infty,
\]
and 
\[
\frac{C}{n}\sum_{r=0}^{t-1}\sum_{i=1}^{n}\left([\bm{q}^r]_i\right)^2 < \infty.
\]
Finally
\[
\frac{C}{n}\sum_{i=1}^n\left([P^\perp_{\bm{Q}_t}\bm{\widetilde{N}}^\top\bm{q}^t_\perp]_i\right)^2 \le O\left(\frac{C}{n}\sum_{i=1}^n\left([\bm{\widetilde{N}}^\top\bm{q}^t_\perp]_i\right)^2\right) +O\left(\frac{C}{n}\sum_{i=1}^n\left([P_{\bm{Q}_t}\bm{\widetilde{N}}^\top\bm{q}^t_\perp]_i\right)^2\right).
\]
Using Lemma \ref{lem:key_help_lem} (d) with $\varphi_n(\bm{x})=(\|\bm{x}\|^2)/n$ and $\langle\bm{q}^t_\perp,\bm{q}^t_\perp\rangle_n < \langle\bm{q}^t,\bm{q}^t\rangle_n<\infty$ \revsag{, we show that the first term is finite}. We \revsag{show that the second term is finite using Lemma \ref{lem:key_help_lem} (b).} Hence, we have
\[
\frac{1}{n}\sum_{i=1}^{n}(y^{t+1}_l)^2 < \infty.
\]
To show part (b) for $\bm h^{t+1}$ and $\bm y^{t+1}$ we use part (a) to get
\begin{align}
&\phi_h(h^1_i,\ldots,h^{t+1}_i,y^1_i,\ldots,y^{t+1}_i,\xi_{0,i},x_{0,i})\big|_{\mathcal{G}_{t+1,t}}\\
&\overset{d}{=} \phi_h\Big(h^1_i,\ldots,h^{t}_i,\left[\sum_{i=0}^{t-1}\upalpha^t_i\bm{h}^{i+1} + \widetilde{\bm{L}}^\top\bm{m}^t_\perp+\widetilde{\bm{Q}}_{t+1}o(1)\right]_i,y^1_i,\ldots,y^t_i,\\
&\hskip 1.5in \left[\sum_{i=0}^{t-1}\upbeta^t_i\bm{y}^{i+1}+\widetilde{N}\bm{q}^{t}_{\perp}+\widetilde{\bm{Q}}_to(1)\right]_i\hspace{-0.03in},\xi_{0,i},x_{0,i}\hspace{-0.03in}\Big).
\end{align}
Firstly, using Lemma \ref{lem:key_help_lem} (d){ with $r=2t+1$ and $m=2$}, we get
%\nb{i think more details are needed, as the constant $L$ seems to be changing with $n$. similar to step 2. 20210331 zm}
\begin{align}
&\frac{1}{n}\sum_{i=1}^{n}\phi_h\Big(h^1_i,\ldots,h^{t}_i,\left[\sum_{i=0}^{t-1}\upalpha^t_i\bm{h}^{i+1} + \widetilde{\bm{L}}^\top\bm{m}^t_\perp+\widetilde{\bm{Q}}_{t+1}o(1)\right]_i,y^1_i,\ldots,y^t_i,\\
& \hskip 1.5in\left[\sum_{i=0}^{t-1}\upbeta^t_i\bm{y}^{i+1}+\widetilde{N}\bm{q}^{t}_{\perp}+\widetilde{\bm{Q}}_to(1)\right]_i,\xi_{0,i},x_{0,i}\Big)\\
& -\frac{1}{n}\sum_{i=1}^{n}\phi_h\Big(h^1_i,\ldots,h^{t}_i,\left[\sum_{i=0}^{t-1}\upalpha^t_i\bm{h}^{i+1} + \widetilde{\bm{L}}^\top\bm{m}^t_\perp+\widetilde{\bm{Q}}_{t+1}o(1)\right]_i,y^1_i,\ldots,y^t_i,\\
& \hskip 1.5in \left[\sum_{i=0}^{t-1}\upbeta^t_i\bm{y}^{i+1}+\hspace{-0.03in}\bm{z}\frac{\|\bm{q}^t_\perp\|}{\sqrt{n}}+\widetilde{\bm{Q}}_to(1)\right]_i\hspace{-0.05in},\xi_{0,i},x_{0,i}\Big)\\
&\hskip 3.5in \overset{a.s.}{\longrightarrow} 0,
\end{align}
where $\bm{z} \sim N_n(0,\bm{I}_n)$ is independent of everything else. 
Now \revsag{using the techniques used in $\mathcal{B}_t$ (b) of \cite{BM11journal}} we can remove the terms $\widetilde{\bm{Q}}_to(1)$ and $\widetilde{\bm{Q}}_{t+1}o(1)$. So it is enough to consider { 
\[
\widetilde{X}_{n,i}=\phi_h\left(h^1_i,\ldots,h^{t}_i,\left[\sum_{i=0}^{t-1}\upalpha_i\bm{h}^{i+1} + \widetilde{\bm{L}}^\top\frac{\|\bm{m}^t_\perp\|}{\sqrt{p}}\right]_i,y^1_i,\ldots,y^t_i,\left[\sum_{i=0}^{t-1}\upbeta^t_i\bm{y}^{i+1}+\bm{z}{ \frac{\|\bm{q}^t_\perp\|}{\sqrt{n}}}\right]_i,\xi_{0,i},x_{0,i}\right).
\]
It is easy to verify the conditions of Theorem 3 of \cite{BM11journal} conditionally on $\mathcal{G}_{t+1,t}$ and hence we get { for $\bm{z}_1,\bm{z}_2\stackrel{iid}{\sim}N_n(0, \bm{I}_n)$,} given $\mathcal{G}_{t+1,t}$
\begin{align}
&\frac{1}{n}\sum_{i=1}^{n}\Bigg\{\phi_h\left(h^1_i,\ldots,h^{t}_i,\left[\sum_{i=0}^{t-1}\upalpha^t_i\bm{h}^{i+1} + \widetilde{\bm{L}}^\top\bm{m}^t_\perp\right]_i,y^1_i,\ldots,y^t_i,\left[\sum_{i=0}^{t-1}\upbeta^t_i\bm{y}^{i+1}+\hspace{-0.03in}\bm{z}\frac{\|\bm{q}^t_\perp\|}{\sqrt{n}}\right]_i\hspace{-0.05in},\xi_{0,i},x_{0,i}\right)\\
&-\mathbb{E}_{\bm{z}_1,\bm{z}_2}\Bigg[\phi_h\Bigg(h^1_i,\ldots,h^{t}_i,\left[\sum_{i=0}^{t-1}\upalpha^t_i\bm{h}^{i+1} + \hspace{-0.03in}\bm{z}_1\frac{\|\bm{m}^t_\perp\|}{\sqrt{p}}\right]_i,y^1_i,\ldots,y^t_i,\left[\sum_{i=0}^{t-1}\upbeta^t_i\bm{y}^{i+1}+\hspace{-0.03in}\bm{z}_2\frac{\|\bm{q}^t_\perp\|}{\sqrt{n}}\right]_i,\hspace{-0.05in}\\
&\hskip 3in\xi_{0,i},x_{0,i}\Bigg)\Bigg]\Bigg\} \overset{a.s.}{\longrightarrow} 0.
\end{align}
It follows that we also have this marginally. Let $\updelta_t=\lim_{n \rightarrow \infty}\frac{\|\bm{m}^t_\perp\|}{\sqrt{n}}$ and $\uprho_t=\lim_{n \rightarrow \infty}\frac{\|\bm{q}^t_\perp\|}{\sqrt{n}}$. Then by partially pseudo-Lipschitz property of $\phi_h$
\begin{align}
&\frac{1}{n}\sum_{i=1}^{n}\Bigg\{\mathbb{E}_{\bm{z}_1,\bm{z}_2}\Bigg[\phi_h\Bigg(h^1_i,\ldots,h^{t}_i,\left[\sum_{i=0}^{t-1}\upalpha^t_i\bm{h}^{i+1} + \hspace{-0.03in}\bm{z}_1\frac{\|\bm{m}^t_\perp\|}{\sqrt{p}}\right]_i,\\
&\hskip 1.8in y^1_i,\ldots,y^t_i,\left[\sum_{i=0}^{t-1}\upbeta^t_i\bm{y}^{i+1}+\hspace{-0.03in}\bm{z}_2\frac{\|\bm{q}^t_\perp\|}{\sqrt{n}}\right]_i\hspace{-0.05in},\xi_{0,i},x_{0,i}\Bigg)\Bigg]\\&-\mathbb{E}_{\bm{z}_1,\bm{z}_2}\Bigg[\phi_h\Bigg(h^1_i,\ldots,h^{t}_i,\left[\sum_{i=0}^{t-1}\upalpha^t_i\bm{h}^{i+1} + \hspace{-0.03in}\bm{z}_1\updelta_t\right]_i,y^1_i,\ldots,y^t_i,\left[\sum_{i=0}^{t-1}\upbeta^t_i\bm{y}^{i+1}+\hspace{-0.03in}\bm{z}_2\uprho_t\right]_i\hspace{-0.05in},\\
&\hskip 3in \xi_{0,i},x_{0,i}\Bigg)\Bigg]\Bigg\} \overset{a.s.}{\longrightarrow} 0.
\end{align}
Now consider the partially pseudo-Lipschitz function 
\begin{multline}
\widehat{\phi}_h(h^1_i,\ldots,h^{t}_i,y^1_i,\ldots,y^t_i,\xi_{0,i}, x_{0,i})\\=\mathbb{E}_{\bm{z}_1,\bm{z}_2}\left\{\phi_h\left(h^1_i,\ldots,h^{t}_i,\left[\sum_{i=0}^{t-1}\upalpha^t_i\bm{h}^{i+1} + \updelta_t\bm{z}_1\right]_i,y^1_i,\ldots,y^t_i,\left[\sum_{i=0}^{t-1}\upbeta^t_i\bm{y}^{i+1}+\hspace{-0.03in}\uprho_t\bm{z}_2\right]_i\hspace{-0.05in},\xi_{0,i},x_{0,i}\right)\right\}.
\end{multline}
That it is partially pseudo-Lipschitz follows by{ Lemma \ref{lem:prop_pl}(2)}. By the induction hypothesis part (b), we get
\begin{align}
&\lim_{n \rightarrow \infty}\frac{1}{n}\sum_{i=1}^{n}\widehat{\phi}_h(h^1_i,\ldots,h^{t}_i,y^1_i,\ldots,y^t_i,\xi_{0,i}, x_{0,i})\\
&=\mathbb{E}\Bigg\{\phi_h\Bigg(\tau_0Z_0,\ldots,\tau_{t-1}Z_{t-1},\sum_{i=0}^{t-1}\upalpha^t_i\tau_{i}Z_i + \updelta_t Z,\sigma_1\widetilde{Z}_1,\ldots,\sigma_t\widetilde{Z}_t,\\
&\hskip 1.8in\sum_{i=0}^{t-1}\upbeta^t_i\sigma_{i+1}\widetilde{Z}_{i+1}+\uprho_t\widetilde{Z},\widetilde{\Xi}_0,\widetilde{X}_0\Bigg)\Bigg\},
\end{align}
}
As both ${ \tau_t Z_t} = \sum_{i=0}^{t-1}\upalpha^t_i\tau_{i}Z_i + \updelta_t Z$ and 
${ \sigma_{t+1}\widetilde{Z}_{t+1}} = \sum_{i=0}^{t-1}\upbeta^t_i\sigma_{i+1}\widetilde{Z}_{i+1}+\uprho_t\widetilde{Z}$ are centered Gaussians, it is enough to show { their} variances are $\tau^2_t$ and $\sigma^2_{t+1}$ respectively. \revsag{Proceeding as in $\mathcal{B}_t$ (b) of \cite{BM11journal}}
\begin{align}
\mathbb{E}\left\{\sum_{i=0}^{t-1}\upalpha^t_i\tau_{i}Z_i + \updelta_t Z\right\}^2 &\overset{a.s.}{=} \lim_{n \rightarrow \infty} \langle\bm{h}^{t+1},\bm{h}^{t+1}\rangle_n\\
& \overset{a.s.}{=} \lim_{n \rightarrow \infty} \langle\bm{m}^{t},\bm{m}^{t}\rangle_p\\
& \overset{a.s.}{=}\mathbb{E}\left\{{ \sfg}_t(\vartheta_t\widecheck{Z}_t,\widetilde{\Omega}_0,\widetilde{V}_0)^2\right\}\\
& = \tau^2_t.
\end{align}
Similarly we have
\begin{align}
\mathbb{E}\left\{\sum_{i=0}^{t-1}\upbeta^t_{i}\sigma_{i+1}\widetilde{Z}_{i+1} + \uprho_t \widetilde{Z}\right\}^2 &\overset{a.s.}{=} \lim_{n \rightarrow \infty} \langle\bm{y}^{t+1},\bm{y}^{t+1}\rangle_n\\
& \overset{a.s.}{=} \lim_{n \rightarrow \infty} \langle\bm{q}^{t},\bm{q}^{t}\rangle_n\\
& \overset{a.s.}{=}\mathbb{E}\left\{{ \sff}_t(\tau_{t-1}Z_{t-1},\sigma_t\widetilde{Z}_t,\widetilde{\Xi}_0,\widetilde{X}_0)^2\right\}\\
& = \sigma^2_{t+1}.
\end{align}
\revsag{Last but not least}, using induction hypotheses, and the fact that $Z$ and $\widetilde{Z}$ in the definition of $Z_t$ and $\widetilde{Z}_t$ are independent of everything else, we obtain that $(Z_0,\dots,Z_t)$, $(\widetilde{Z}_1,\dots, \widetilde{Z}_{t+1})$ and $(\widetilde{\Xi}_0,\widetilde{X}_0)$ are mutually independent.
This completes the proof of part (b). 
\\ 
Finally, the assertion (d) for $\bm h^{t+1}$ and $\bm y^{t+1}$ follows using the arguments similar to the base case.

\subsection{Proof of Theorem \ref{thm:slln_shifted}}
Let us define the following AMP which is easier to analyze. We shall show that the iterates of this AMP is asymptotically close to the iterates of the original AMP given by \eqref{eq:AMP_shift_main_1} and \eqref{eq:AMP_shift_main_1_1} { with generic $\{f_t,g_t:t\geq 0\}$ satisfying the condition of Theorem \ref{thm:slln_shifted}}. 
Let us define $\widetilde{\bm{u}}^0=\widetilde{\bm{y}}^0=0$. Then for $t \in \mathbb{N}\cup\{0\}$, we define
\begin{align}
\label{eq:AMP_shift}
\widetilde{\bm{v}}^{t} &= \frac{\bm{R}}{\sqrt{p}} f_{t}(\alpha_{t-1}\bm{x}^*+\widetilde{\bm{u}}^t,\mu_{t}\bm{x}^*+\widetilde{\bm{x}}^t,\bm{x}_0(\varepsilon))-\widetilde{p}_{t}g_{t-1}(\beta_{t-1}\bm{v}^*+\widetilde{\bm{v}}^{t-1},\bm{v}_0(\varepsilon))),\\
\widetilde{\bm{u}}^{t+1} &= \frac{\bm{R}^\top}{\sqrt{p}} g_t(\beta_t\bm{v}^*+\widetilde{\bm{v}}^t,\bm{v}_0(\varepsilon))-\widetilde{c}_tf_{t}(\alpha_{t-1}\bm{x}^*+\widetilde{\bm{u}}^t,\mu_{t}\bm{x}^*+\widetilde{\bm{x}}^t,\bm{x}_0(\varepsilon)),\\
\widetilde{\bm{x}}^{t+1} &= \frac{\bm{Z}}{\sqrt{n}} f_{t}(\alpha_{t-1}\bm{x}^*+\widetilde{\bm{u}}^t,\mu_{t}\bm{x}^*+\widetilde{\bm{x}}^t,\bm{x}_0(\varepsilon))-\widetilde{d}_{t}f_{t-1}(\alpha_{t-2}\bm{x}^*+\widetilde{\bm{u}}^{t-1},\mu_{t-1}\bm{x}^*+\widetilde{\bm{x}}^{t-1},\bm{x}_0(\varepsilon)),
\end{align}
where
\vspace{-0.1in}
\begin{align}
\widetilde{c}_t &= \frac{1}{p}\sum_{i=1}^{p} \frac{\partial g_t}{\partial v}(\beta_tv^*_{i}+\widetilde{v}^t_i,v_{0,i}(\varepsilon)),\\
\widetilde{p}_t &= \frac{c}{n}\sum_{i=1}^{n} \frac{\partial f_t}{\partial u}(\alpha_{t-1}x^*_{i}+\widetilde{u}^t_i,\mu_{t}x^*_{i}+\widetilde{x}^t_i,x_{0,i}(\varepsilon)),\\
\widetilde{d}_t &= \frac{1}{n}\sum_{i=1}^{n} \frac{\partial f_t}{\partial y}(\alpha_{t-1}x^*_{i}+\widetilde{u}^t_i,\mu_{t}x^*_{i}+\widetilde{x}^t_i,x_{0,i}(\varepsilon)).
\end{align}
Let us observe that $f(u,v,x,y)=f_t(\alpha_{t-1}x+u,\mu_{t}x+v,y)$ and $g(u,x,y)=g_{t-1}(\beta_{t-1}x+u,y)$ are partially Lipschitz functions for all $t$. {The iterates defined in \eqref{eq:AMP_shift} is of the form \eqref{eq:orig_AMP}. 
Hence using Theorem \ref{thm:thm_6_1} for any partially pseudo-Lipschitz functions $\widehat{\phi}$ and $\widehat{\psi}$}, we get
\begin{equation}
\label{eq:slln_3}
\lim_{n \rightarrow \infty} \frac{1}{n}\sum_{i=1}^{n}\widehat{\phi}(\widetilde{u}^{t}_i, \widetilde{x}^t_i, x^*_i, x_{0,i}(\varepsilon)) \overset{a.s.}{=} \mathbb{E}\left\{\widehat{\phi}\left(\tau_{t-1}Z_1, \sigma_{t}Z_2, X_0, X_0(\varepsilon)\right)\right\},
\end{equation}
and
\begin{equation}
\label{eq:slln_4}
\lim_{p \rightarrow \infty} \frac{1}{p}\sum_{i=1}^{p}\widehat{\psi}(\widetilde{v}^{t}_i, v^*_i, v_{0,i}(\varepsilon)) \overset{a.s.}{=} \mathbb{E}\left\{\widehat{\psi}\left(\vartheta_{t}Z_3, V_0, V_0(\varepsilon)\right)\right\}.
\end{equation}
Define $\widehat{\phi}(x,y,z,w)=\phi(\alpha_{t-1}z+{ x},\mu_t z+{ y},z,w)$ and $\widehat{\psi}(x,y,{ r})=\psi(\beta_{t}y+{ x},y,{ r})$. { It is not hard to observe that these functions are partially pseudo-Lipschitz}. Then we obtain
\[
\lim_{n \rightarrow \infty} \frac{1}{n}\sum_{i=1}^{n}\phi(\alpha_{t-1}x^*_i+\widetilde{u}^{t}_i, \mu_tx^*_i+\widetilde{x}^t_i, x_{0,i}(\varepsilon)) \overset{a.s.}{=} \mathbb{E}\left\{\phi\left(\alpha_{t-1}X_0+\tau_{t-1}Z_1, \mu_tX_0+\sigma_{t}Z_2, X_0(\varepsilon)\right)\right\},
\]
and
\[
\lim_{p \rightarrow \infty} \frac{1}{p}\sum_{i=1}^{p}\psi(\beta_tv^*_i+\widetilde{v}^{t}_i, v_{0,i}(\varepsilon)) \overset{a.s.}{=} \mathbb{E}\left\{\psi\left(\beta_tV_0+\vartheta_{t}Z_3, V_0(\varepsilon)\right)\right\}.
\]
Hence, it is enough to show that
\[
\lim_{n \rightarrow \infty} \frac{1}{n}\sum_{i=1}^{n}\left[\phi(\alpha_{t-1}x^*_i+\widetilde{u}^{t}_i, \mu_tx^*_i+\widetilde{x}^t_i, x_{0,i}(\varepsilon))-\phi(u^{t}_i, x^t_i, x_{0,i}(\varepsilon))\right] \overset{a.s.}{=} 0,
\]
and
\[
\lim_{p \rightarrow \infty} \frac{1}{p}\sum_{i=1}^{p}\left[\psi(\beta_{t}v^*_i+\widetilde{v}^{t}_i, v_{0,i}(\varepsilon))-\psi(v^{t}_i, v_{0,i}(\varepsilon))\right] \overset{a.s.}{=} 0.
\]

We shall prove the last two displays by induction on the following hypotheses:
\begin{enumerate}
\item $\lim_{n \rightarrow \infty} \frac{1}{n}\sum_{i=1}^{n}\left[\phi(\alpha_{t-1}x^*_i+\widetilde{u}^{t}_i, \mu_tx^*_i+\widetilde{x}^t_i, x_{0,i}(\varepsilon))-\phi(u^{t}_i, x^t_i, x_{0,i}(\varepsilon))\right] \overset{a.s.}{=} 0,
$
\item $\lim_{p \rightarrow \infty} \frac{1}{p}\sum_{i=1}^{p}\left[\psi(\beta_{t}v^*_i+\widetilde{v}^{t}_i, v_{0,i}(\varepsilon))-\psi(v^{t}_i, v_{0,i}(\varepsilon))\right] \overset{a.s.}{=} 0,$
\item $\lim_{n \rightarrow \infty}\frac{\|\bm{\Delta}^t_1\|^2}{n} \overset{a.s.}{=} 0$,
\item $\lim_{n \rightarrow \infty}\frac{\|\bm{\Delta}^t_2\|^2}{n} \overset{a.s.}{=} 0$,
\item $\lim_{p \rightarrow \infty}\frac{\|\bm{\Delta}^t_3\|^2}{p} \overset{a.s.}{=} 0$,
\item $\lim_{n \rightarrow \infty}\frac{\|\alpha_{t-1}\bm{x}^*+\widetilde{\bm{u}}^{t}\|^2}{n} < \infty$ { a.s.},
\item $\lim_{n \rightarrow \infty}\frac{\|\mu_{t}\bm{x}^*+\widetilde{\bm{x}}^{t}\|^2}{n} < \infty$ { a.s.},
\item $\lim_{p \rightarrow \infty}\frac{\|\beta_{t}\bm{v}^*+\widetilde{\bm{v}}^{t}\|^2}{p} < \infty$ { a.s.},
\end{enumerate}
where $\bm{\Delta}^t_1=\bm{x}^t-\mu_{t}\bm{x}^*-\widetilde{\bm{x}}^{t}$, $\bm{\Delta}^t_2=\bm{u}^t-\alpha_{t-1}\bm{x}^*-\widetilde{\bm{u}}^{t}$, and $\bm{\Delta}^t_3=\bm{v}^t-\beta_{t}\bm{v}^*-\widetilde{\bm{v}}^{t}$. 
\paragraph{Step 1: The $t=0$ case} Using $\alpha_{-1}=\mu_0=0$ and $\bm{u}^0=\bm{\widetilde{u}}^0=\bm{y}^0=\bm{\widetilde{y}}^0=0$, \revsag{hypotheses} $(1),(3),(4),(6)$ and $(7)$ follows. Now note that
\[
\bm{v}^0=\sqrt{\frac{\mu}{np}}\bm{v}^*(\bm{x}^*)^\top f_0(\bm{0},\bm{0},\bm{x}_0(\varepsilon))+\frac{\bm{R}}{\sqrt{p}}f_0(\bm{0},\bm{0},\bm{x}_0(\varepsilon)),
\]
and
\[
\bm{\widetilde{v}}^0=\frac{\bm{R}}{\sqrt{p}}f_0(\bm{0},\bm{0},\bm{x}_0(\varepsilon)).
\]
We have for $i \in [p]$
\[
\widetilde{v}^0_i = \|f_0(\bm{0},\bm{0},\bm{x}_0(\varepsilon))\|\frac{{ z}_i}{\sqrt{p}},
\]
where ${ z}_1,\ldots,{ z}_p$ are i.i.d $N(0,1)$ and ${ \bm{z}=(z_1,\ldots,z_p)^\top}$. Then we get
\[
\frac{\|\beta_0\bm{v}^*+\bm{\widetilde{v}}^0\|^2}{p} = \beta^2_0 \frac{\|\bm{v}^*\|^2}{p} + \frac{\|f_0(\bm{0},\bm{0},\bm{x}_0(\varepsilon))\|^2}{p}\frac{\|{\bm{z}}\|^2}{p} + 2\,\beta_0\frac{\|f_0(\bm{0},\bm{0},\bm{x}_0(\varepsilon))\|}{\sqrt{p}}\langle
{\bm{z}},\bm{v}^*\rangle_p.
\]
By SLLN, we get that all the terms are finite. Hence Hypothesis (8) follows. We further note that
\[
\lim_{p \rightarrow \infty}\frac{1}{p}\;\|\bm{\Delta}^{0}_3\|^2 = \lim_{p \rightarrow \infty}\left(\sqrt{\frac{\mu}{np}}(\bm{x}^*)^\top f_0(\bm{u}^0,\bm{y}^0,\bm{x}_0(\varepsilon))-\beta_0\right)^2\;\frac{\|\bm{v}^*\|^2}{p}.
\]
Using SLLN, definition of $\beta_0$ and $p/n \rightarrow 1/c$, we get Hypothesis (5). Again note that
\[
\frac{\|\bm{v}^0\|}{\sqrt{p}} \le 
\Bigg|
\sqrt{\frac{\mu}{np}}(\bm{x}^*)^\top f_0(\bm{0},\bm{0},\bm{x}_0(\varepsilon)) 
\Bigg|\frac{\|\bm{v}^*\|}{\sqrt{p}} + \frac{\|{\bm{z}}\|}{\sqrt{p}}\frac{\|f_0(\bm{0},\bm{0},\bm{x}_0(\varepsilon))\|}{\sqrt{p}}.
\]
Then using SLLN we get $\lim_{p \rightarrow \infty}\|\bm{v}^0\|/\sqrt{p} < \infty$ almost surely. 
Now using \revsag{the} partially pseudo-Lipschitz property of $\psi$, we get
\begin{multline}
\Bigg|\frac{1}{p}\;\sum_{i=1}^{p}\left[\psi(\beta_0{ v^*_{0,i}}+\widetilde{v}^0_i,v_{0,i}(\varepsilon))-\psi(v^0_i,v_{0,i}(\varepsilon))\right]\Bigg| \\
\le \left(1+\frac{\|{\beta_0}\bm{v}^*+\bm{\widetilde{v}}^0\|}{\sqrt{p}}+\frac{\|\bm{v}^0\|}{\sqrt{p}}+\frac{\|\bm{v}_0(\varepsilon)\|}{\sqrt{p}}\right) \frac{\|\bm{v}^*\|}{\sqrt{p}}
\Bigg|\sqrt{\frac{\mu}{np}}(\bm{x}^*)^\top f_0(\bm{u}^0,\bm{x}^0,\bm{x}_0(\varepsilon))-\beta_0\Bigg| \overset{a.s.}{\rightarrow} 0,
\end{multline}
by SLLN, definition of $\beta_0$ and $p/n \rightarrow 1/c$. This shows Hypothesis (2).

\medskip

Let the hypotheses hold for $\ell = 0,\ldots,t-1$. Now we show the hypotheses for $\ell=t$ to complete the induction.

\paragraph{Step 2: Hypothesis (6), (7) and (8)} First consider Hypotheis (6) for $\ell=t$. If we consider \revsag{the} partially pseudo-Lipschitz function $\widehat{\phi}(x, y, z, { w})=(\alpha_{t-1}z+x)^2$, then using \eqref{eq:slln_3} this Hypothesis follows. 
Similarly using $\widehat{\phi}(x, y, z, { w})=(\mu_{t}z+y)^2$, Hypothesis (7) follows. 
Finally using $\widehat{\psi}(x, y, { r})=(\beta_{t}y+x)^2$ and \eqref{eq:slln_4}, Hypothesis (8) follows. 
\paragraph{Step 3: Hypothesis (3)} Consider Hypothesis (3) for $\ell = t$. It can be observed that
\begin{align}
&\Delta^t_{1,i}\\
&= \left(\sqrt{\lambda}\langle\bm{x}^*,f_{t-1}(\bm{u}^{t-1},\bm{x}^{t-1},\bm{x}_0(\varepsilon))\rangle_n-\mu_t\right)x^*_i \\
&\hskip 1.5in+\widetilde{d}_{t-1}f_{t-2}(\alpha_{t-3}x^*_i+\widetilde{u}^{t-2}_i,\mu_{t-2}x^*_i+\widetilde{x}^{t-2}_i,\bm{x}_{0,i}(\varepsilon))\\
&\hspace{0.3in}+\frac{1}{\sqrt{n}}\langle\bm{Z}_{i,*},f_{t-1}(\bm{u}^{t-1},\bm{x}^{t-1},\bm{x}_0(\varepsilon))-f_{t-1}(\alpha_{t-2}\bm{x}^*+\widetilde{\bm{x}}^{t-1},\mu_{t-1}\bm{x}^*+\widetilde{\bm{x}}^{t-1},\bm{x}_0(\varepsilon))\rangle_n\\
&\hskip 1.5in-d_{t-1}f_{t-2}({ u^{t-2}_i}, { x^{t-2}_i},\bm{x}_{0,i}(\varepsilon)).
\end{align}
Thus using \revsag{the} Jensen's inequality, we have for constant $L_1>0$
\begin{align}
&\frac{1}{n}\|\Delta^t_{1}\|^2 \\
&\le L_1\left(\sqrt{\lambda}\langle\bm{x}^*,f_{t-1}(\bm{u}^{t-1},\bm{x}^{t-1},\bm{x}_0(\varepsilon))\rangle_n-\mu_t\right)^2 \\
&+ L_1|\widetilde{d}_{t-1}-d_{t-1}|^2\frac{1}{n}\sum_{i=1}^{n}f^2_{t-2}(\alpha_{t-3}x^*_i+\widetilde{u}^{t-2}_i,\mu_{t-2}x^*_i+\widetilde{x}^{t-2}_i,\bm{x}_{0,i}(\varepsilon))\\
&+\frac{L_1}{n^2}\|\bm{Z}\|^2_{\mbox{\tiny{op}}}\|f_{t-1}(\bm{u}^{t-1},\bm{x}^{t-1},\bm{x}_0(\varepsilon))-f_{t-1}(\alpha_{t-2}\bm{x}^*+\widetilde{\bm{u}}^{t-1},\mu_{t-1}\bm{x}^*+\widetilde{\bm{x}}^{t-1},\bm{x}_0(\varepsilon))\|^2\\
&+\frac{L_1}{n}|d_{t-1}|^2\|f_{t-2}(\bm{u}^{t-2},\bm{x}^{t-2},\bm{x}_0(\varepsilon))-f_{t-2}(\alpha_{t-3}\bm{x}^*+\widetilde{\bm{u}}^{t-2},\mu_{t-2}\bm{x}^*+\widetilde{\bm{x}}^{t-2},\bm{x}_0(\varepsilon))\|^2.
\end{align}
Using the partially pseudo-Lipschitz function { $\widehat\phi(x,y,z,w)=zf_{t-1}(\alpha_{t-1}z+x, \mu_t z+y, w)$} ({by Lemma \ref{lem:prop_pl}(1)}), definition of $\mu_t$ and Hypothesis (1) for $\ell=t-1$ we have
\begin{equation}
\label{eq:s_1_1}
\sqrt{\lambda}\langle\bm{x}^*,f_{t-1}(\bm{u}^{t-1},\bm{x}^{t-1},\bm{x}_0(\varepsilon))\rangle_n-\mu_t \overset{a.s.}{\rightarrow} 0.
\end{equation}
From \cite{anderson_guionnet_zeitouni_2009}, we have
\[
\limsup_{n \rightarrow \infty} \frac{1}{n}\|\bm{Z}\|^2_{\mbox{\tiny{op}}} < \infty \quad \mbox{{ a.s.}}
\]
Using \revsag{the} partially Lipschitz property of $f_{t-1}$, we have
\[
\|f_{t-1}(\bm{u}^{t-1},\bm{x}^{t-1},\bm{x}_0(\varepsilon))-f_{t-1}(\alpha_{t-2}\bm{x}^*+\bm{\widetilde{u}}^{t-1},\mu_{t-1}\bm{x}^*+\bm{\widetilde{x}}^{t-1},\bm{x}_0(\varepsilon))\|^2 \le L_1(\|\Delta^{t-1}_1\|^2+\|\Delta^{t-1}_2\|^2).
\]
By induction hypothesis (3) and (4) for $\ell=t-1$ we get
\begin{multline}
\frac{1}{n}\|f_{t-1}(\bm{u}^{t-1},\bm{x}^{t-1},\bm{x}_0(\varepsilon))-f_{t-1}(\alpha_{t-2}\bm{x}^*+\bm{\widetilde{u}}^{t-1},\mu_{t-1}\bm{x}^*+\bm{\widetilde{x}}^{t-1},\bm{x}_0(\varepsilon))\|^2 \\
\le L_1\left(\frac{\|\Delta^{t-1}_1\|^2}{n}+\frac{\|\Delta^{t-1}_2\|^2}{n}\right)
\overset{a.s.}{\rightarrow} 0.
\end{multline}
This implies
\begin{multline}
\label{eq:s_1_2}
\frac{1}{n^2}\|\bm{Z}\|^2_{\mbox{\tiny{op}}}\|f_{t-1}(\bm{u}^{t-1},\bm{x}^{t-1},\bm{x}_0(\varepsilon))-f_{t-1}(\alpha_{t-2}\bm{x}^*+\bm{\widetilde{u}}^{t-1},\mu_{t-1}\bm{x}^*+\bm{\widetilde{x}}^{t-1},\bm{x}_0(\varepsilon))\|^2 \overset{a.s.}{\rightarrow} 0.
\end{multline}
Since $\widehat\psi(x,y,r)=f^2_{t-2}(x,y,r)$ is partially pseudo-Lipschitz ({ by Lemma \ref{lem:prop_pl}(1)}), using \eqref{eq:slln_3}, we have almost surely
\[
\limsup_{n \rightarrow \infty} \frac{1}{n}\sum_{i=1}^nf^2_{t-2}(\alpha_{t-3}x^*_i+\widetilde{u}^{t-2}_i,\mu_{t-2}x^*_i+\widetilde{x}^{t-2}_i,x_{0,i}(\varepsilon)) < \infty.
\] 
As $f^{(2)}_{t-1}(x,y,r)=\frac{\partial f_{t-1}}{\partial y}(x,y,r)$ is partially Lipschitz, we have using Hypothesis (1)
\begin{multline}
|\widetilde{d}_{t-1}-d_{t-1}| = \Bigg|\frac{1}{n}\sum_{i=1}^{n}[f^{(2)}_{t-1}(\alpha_{t-2}x^*_i+\widetilde{u}^{t-1}_i,\mu_{t-1}x^*_i+\widetilde{x}^{t-1}_i,x_{0,i}(\varepsilon))-f^{(2)}_{t-1}(u^{t-1}_i,x^{t-1}_i,x_{0,i}(\varepsilon))]\Bigg|\\
\overset{a.s}{\rightarrow} 0.
\end{multline}
This implies
\begin{equation}
\label{eq:s_1_3}
|\widetilde{d}_{t-1}-d_{t-1}|^2\frac{1}{n}\sum_{i=1}^{n}f^2_{t-2}(\alpha_{t-3}x^*_i+\widetilde{u}^{t-2}_i,\mu_{t-2}x^*_i+\widetilde{x}^{t-2}_i,\bm{x}_{0,i}(\varepsilon)) \overset{a.s.}{\rightarrow} 0.
\end{equation}
Again, using similar arguments, we can show that 
\[
|d_{t-1}|^2 < \infty \quad \mbox{a.s.}
\]
Again using the induction hypothesis (3) and (4) for $\ell=t-2$ we get
\begin{multline}
\frac{1}{n}\|f_{t-2}(\bm{u}^{t-2},\bm{x}^{t-2},\bm{x}_0)-f_{t-2}(\alpha_{t-3}\bm{x}^*+\bm{\widetilde{u}}^{t-2},\mu_{t-2}\bm{x}^*+\bm{\widetilde{x}}^{t-2},\bm{x}_0(\varepsilon))\|^2 \\
\le L_1\left(\frac{\|\Delta^{t-2}_1\|^2}{n}+\frac{\|\Delta^{t-2}_2\|^2}{n}\right)
\overset{a.s.}{\rightarrow} 0.
\end{multline}
Thus, we have the following.
\begin{equation}
\label{eq:s_1_4}
|d_{t-1}|^2\frac{1}{n}\|f_{t-2}(\bm{u}^{t-2},\bm{x}^{t-2},\bm{x}_0(\varepsilon))-f_{t-2}(\alpha_{t-3}\bm{x}^*+\bm{\widetilde{u}}^{t-2},\mu_{t-2}\bm{x}^*+\bm{\widetilde{x}}^{t-2},\bm{x}_0(\varepsilon))\|^2 \overset{a.s.}{\rightarrow} 0.
\end{equation}
Using \eqref{eq:s_1_1}, \eqref{eq:s_1_2}, \eqref{eq:s_1_3}, and \eqref{eq:s_1_4} we have
\[
\lim_{n \rightarrow \infty}\frac{\|\Delta^t_1\|^2}{n} \overset{a.s.}{=} 0.
\]
\paragraph{Step 4: Hypothesis (4)} Next we try to prove Hypothesis {(4)} for $\ell = t$. We observe that
\begin{align}
\Delta^t_{2,i}&= \left((\sqrt{p\mu/n})\langle\bm{v}^*,g_{t-1}(\bm{v}^{t-1},\bm{v}_0(\varepsilon))\rangle_p-\alpha_{t-1}\right)x^*_i\\
&\hskip 0.5in +\widetilde{c}_{t-1}f_{t-1}(\alpha_{t-2}x^*_i+\widetilde{u}^{t-1}_i,\mu_{t-1}x^*_i+\widetilde{x}^{t-1}_i,{x}_{0,i}(\varepsilon))\\
&\hspace{0.7in}+\langle\bm{L}_{*,i},g_{t-1}(\bm{v}^{t-1},\bm{v}_0(\varepsilon))-g_{t-1}(\beta_{t-1}\bm{v}^*+\widetilde{\bm{v}}^{t-1},\bm{v}_0(\varepsilon))\rangle\\
&\hspace{0.7in}-c_{t-1}f_{t-1}(u^{t-1}_i,x^{t-1}_i,{x}_{0,i}(\varepsilon)).
\end{align}
If $\bm{L}=\bm{V}/\sqrt{n}$, then by \revsag{the} Jensen's inequality, we have for constant $L_2>0$
\begin{align}
\frac{\|\Delta^t_{2}\|^2}{n} &\le L_2\left(\sqrt{\frac{\mu p}{n}}\langle\bm{v}^*,g_{t-1}(\bm{v}^{t-1},\bm{v}_0(\varepsilon))\rangle_p-\alpha_{t-1}\right)^2 \\&+ L_2|\widetilde{c}_{t-1}-c_{t-1}|^2\frac{1}{n}\sum_{i=1}^{n}f^2_{t-1}(\alpha_{t-2}x^*_i+\widetilde{u}^{t-1}_i,\mu_{t-1}x^*_i+\widetilde{x}^{t-1}_i,\bm{x}_{0,i}(\varepsilon))\\&+\frac{L_2}{np}\uplambda_{\mbox{\tiny{max}}}(\bm{V}\bm{V}^\top)\|g_{t-1}(\bm{v}^{t-1},\bm{v}_0(\varepsilon))-g_{t-1}(\beta_{t-1}\bm{v}^*+\widetilde{\bm{v}}^{t-1},\bm{v}_0(\varepsilon))\|^2\\&+\frac{L_2}{n}|c_{t-1}|^2\|f_{t-1}(\bm{u}^{t-1},\bm{x}^{t-1},\bm{x}_0(\varepsilon))-f_{t-1}(\alpha_{t-1}\bm{x}^*+\widetilde{\bm{u}}^{t-1},\mu_{t-1}\bm{x}^*+\widetilde{\bm{x}}^{t-1},\bm{x}_0(\varepsilon))\|^2.
\end{align}
Using the partially pseudo-Lipschitz function ${ \widehat\psi(x,y,r)}=yg_{t-1}(\beta_{t-1}y+x,r)$({ by Lemma \ref{lem:prop_pl}(1)}), definition of $\alpha_{t-1}$ and Hypothesis (2) for $\ell=t-1$, we have
\begin{equation}
\label{eq:s_2_1}
\sqrt{\frac{\mu p}{n}}\langle\bm{v}^*,g_{t-1}(\bm{v}^{t-1},\bm{v}_0(\varepsilon))\rangle_p-\alpha_{t-1} \overset{a.s.}{\rightarrow} 0.
\end{equation}
Since $p/n \rightarrow 1/c$, using Corollary 5.35 of \cite{vershynin_2012}, we have
\[
\limsup_{n \rightarrow \infty} \frac{\uplambda_{\mbox{\tiny{max}}}(\bm{V}\bm{V}^\top)}{n}< \infty \quad \mbox{{ a.s.}}
\]
Using \revsag{the} partially Lipschitz property of $g_{t-1}$, we have
\[
\|g_{t-1}(\bm{v}^{t-1},\bm{v}_0(\varepsilon))-g_{t-1}(\beta_{t-1}\bm{v}^*+\bm{\widetilde{v}}^{t-1},\bm{v}_0(\varepsilon))\|^2 \le L_2\,\|\Delta^{t-1}_3\|^2
\]
By induction hypothesis (5) for $\ell=t-1$, we get
\begin{equation}
\frac{1}{p}\|g_{t-1}(\bm{v}^{t-1},\bm{v}_0(\varepsilon))-g_{t-1}(\beta_{t-1}\bm{v}^*+\bm{\widetilde{v}}^{t-1},\bm{v}_0(\varepsilon))\|^2 \le L_2\left(\frac{\|\Delta^{t-1}_3\|^2}{p}\right)\overset{a.s.}{\rightarrow} 0.
\end{equation}
This implies
\begin{equation}
\label{eq:s_2_2}
\limsup_{n \rightarrow \infty} \frac{\uplambda_{\mbox{\tiny{max}}}(\bm{V}\bm{V}^\top)}{np}\|g_{t-1}(\bm{v}^{t-1},\bm{v}_0(\varepsilon))-g_{t-1}(\beta_{t-1}\bm{v}^*+\bm{\widetilde{v}}^{t-1},\bm{v}_0(\varepsilon))\|^2 \overset{a.s.}{=}0.
\end{equation}
Since ${\widehat\psi(x,y,r)}=f^2_{t-1}(x,y,r)$ is partially pseudo-Lipschitz (by Lemma \ref{lem:prop_pl} (1)), using \eqref{eq:slln_3}, we have almost surely
\[
\limsup_{n \rightarrow \infty} \frac{1}{n}\sum_{i=1}^nf^2_{t-1}(\alpha_{t-2}x^*_i+\widetilde{u}^{t-1}_i,\mu_{t-1}x^*_i+\widetilde{x}^{t-1}_i,x_{0,i}(\varepsilon)) < \infty.
\] 
As $g^{\prime}_{t-1}(x,y)=\frac{\partial g_{t-1}}{\partial x}(x,y)$ is partially Lipschitz, we have using Hypothesis (2) for $\ell=t-1$
\begin{equation}
|\widetilde{c}_{t-1}-c_{t-1}| =\Bigg|\frac{1}{p}\sum_{i=1}^{p}[g^{\prime}_{t-1}(\beta_{t-1}v^*_i+\widetilde{v}^{t-1}_i,v_{0,i}(\varepsilon))-g^{\prime}_{t-1}(v^{t-1}_i,v_{0,i}(\varepsilon))]\Bigg|
\overset{a.s}{\rightarrow} 0.
\end{equation}
This implies 
\begin{equation}
\label{eq:s_2_3}
|\widetilde{c}_{t-1}-c_{t-1}|^2\frac{1}{n}\sum_{i=1}^{n}f^2_{t-1}(\alpha_{t-2}x^*_i+\widetilde{u}^{t-1}_i,\mu_{t-1}x^*_i+\widetilde{x}^{t-1}_i,x_{0,i}(\varepsilon)) \overset{a.s.}{\rightarrow} 0.
\end{equation}
Using similar arguments, we can show 
\[
|c_{t-1}|^2 < \infty \quad \mbox{a.s.}
\]
Using the induction hypotheses (3) and (4) for $\ell=t-1$ we get
\begin{multline}
\frac{1}{n}\|f_{t-1}(\bm{u}^{t-1},\bm{x}^{t-1},\bm{x}_0(\varepsilon))-f_{t-1}(\alpha_{t-2}\bm{x}^*+\bm{\widetilde{u}}^{t-1},\mu_{t-1}\bm{x}^*+\bm{\widetilde{x}}^{t-1},\bm{x}_0(\varepsilon))\|^2 \\
\le L_1\left(\frac{\|\Delta^{t-2}_1\|^2}{n}+\frac{\|\Delta^{t-2}_2\|^2}{n}\right)
\overset{a.s.}{\rightarrow} 0.
\end{multline}
Thus we have
\begin{equation}
\label{eq:s_2_4}
|c_{t-1}|^2\frac{1}{n}\|f_{t-2}(\bm{u}^{t-2},\bm{x}^{t-2},\bm{x}_0(\varepsilon))-f_{t-2}(\alpha_{t-3}\bm{x}^*+\bm{\widetilde{u}}^{t-2},\mu_{t-2}\bm{x}^*+\bm{\widetilde{x}}^{t-2},\bm{x}_0(\varepsilon))\|^2 \overset{a.s.}{\rightarrow} 0.
\end{equation}
Using \eqref{eq:s_2_1}, \eqref{eq:s_2_2}, \eqref{eq:s_2_3}, and \eqref{eq:s_2_4}, we have
\[
\lim_{n \rightarrow \infty}\frac{\|\Delta^t_2\|^2}{n} \overset{a.s.}{=} 0.
\]
\paragraph{Step 5: Hypothesis (1)}
Now observe that using \revsag{the} partially pseudo-Lipschitz property of $\phi$, we have for a constant $C_1>0$
\begin{equation}
\label{eq:phi_pseudo_lips}	
\begin{aligned}
& \left|\phi(\alpha_{t-1}x^*_i+\widetilde{u}^{t}_i, \mu_tx^*_i+\widetilde{x}^t_i, x_{0,i}(\varepsilon))-\phi(u^{t}_i, x^t_i, x_{0,i}(\varepsilon))\right| \\
& ~~~~\le C_1(|\Delta^t_{1,i}|+|\Delta^t_{2,i}|)(1+|u^{t}_i|+|x^t_i|+|\alpha_{t-1}x^*_i+\widetilde{u}^{t}_i|+|\mu_tx^*_i+\widetilde{y}^t_i|+|x_{0,i}(\varepsilon)|)\\
& ~~~~\le 2C_1(|\Delta^t_{1,i}|+|\Delta^t_{2,i}|)(1+|\Delta^{t}_{1,i}|+|\Delta^t_{2,i}|+|\alpha_{t-1}x^*_i+\widetilde{u}^{t}_i|+|\mu_tx^*_i+\widetilde{x}^t_i|+|x_{0,i}(\varepsilon)|).
\end{aligned}
\end{equation}
Using \eqref{eq:phi_pseudo_lips} and the Cauchy Schwarz inequality, we get
\begin{multline}
\frac{1}{n}\sum_{i=1}^{n}\left|\phi(\alpha_{t-1}x^*_i+\widetilde{u}^{t}_i, \mu_tx^*_i+\widetilde{x}^t_i, x_{0,i}(\varepsilon))-\phi(u^{t}_i, x^t_i, x_{0,i}(\varepsilon))\right|\\\le\frac{2L}{n}\sum_{i=1}^{n}\{|\Delta^{t}_{1,i}|+|\Delta^{t}_{1,i}|^2+|\Delta^{t}_{1,i}||\alpha_{t-1}x^*_i+\widetilde{u}^{t}_i|+|\Delta^{t}_{1,i}||\Delta^{t}_{2,i}|\\\hspace{0.8in}+|\Delta^{t}_{1,i}||\mu_tx^*_i+\widetilde{x}^t_i|+|\Delta^{t}_{1,i}||x_{0,i}(\varepsilon)|+|\Delta^{t}_{2,i}|+|\Delta^{t}_{2,i}|^2\\\hspace{1.8in}+|\Delta^{t}_{2,i}||\alpha_{t-1}x^*_i+\widetilde{u}^{t}_i|+|\Delta^{t}_{1,i}||\Delta^{t}_{2,i}|+|\Delta^{t}_{2,i}||\mu_tx^*_i+\widetilde{x}^t_i|+|\Delta^{t}_{2,i}||x_{0,i}(\varepsilon)|\}\\
\le\frac{2L}{n}\{\sqrt{n}\|\Delta^{t}_{1}\|+\|\Delta^{t}_{1}\|^2+\|\Delta^{t}_{1}\|\|\alpha_{t-1}\bm{x}^*+\widetilde{\bm{u}}^{t}\|+2\|\Delta^{t}_{1}\|\|\Delta^{t}_{2}\|\\+\|\Delta^{t}_{1}\|\|\mu_t\bm{x}^*+\widetilde{\bm{x}}^t\|+\|\Delta^{t}_{1}\|\|\bm{x}_{0}(\varepsilon)\|+\sqrt{n}\|\Delta^t_2\|+\|\Delta^t_2\|^2\\+\|\Delta^{t}_{2}\|\|\alpha_{t-1}\bm{x}^*+\widetilde{\bm{u}}^{t}\|+\|\Delta^{t}_{2}\|\|\mu_t\bm{x}^*+\widetilde{\bm{x}}^t\|+\|\Delta^{t}_{2}\|\|\bm{x}_{0}(\varepsilon)\|\}\\.
\end{multline}
Thus, using Hypotheses (3) and (4) for $\ell = t$, we have
\[
\frac{1}{n}\sum_{i=1}^{n}\left|\phi(\alpha_{t-1}x^*_i+\widetilde{u}^{t}_i, \mu_tx^*_i+\widetilde{x}^t_i, x_{0,i}(\varepsilon))-\phi(u^{t}_i, x^t_i, x_{0,i}(\varepsilon))\right| \overset{a.s.}{\rightarrow} 0.
\]
\paragraph{Step 6: Hypothesis (5)} Now we try to prove Hypothesis {(5)} for $\ell = t$. We first observe that
\begin{multline}
\Delta^t_{3,i} = \left((\sqrt{n\mu/p})\langle\bm{x}^*,f_{t}(\bm{u}^{t},\bm{x}^t,\bm{x}_0(\varepsilon))\rangle_n-\beta_{t}\right)v^*_i + \widetilde{p}_{t}g_{t-1}(\beta_{t-1}v^*_i+\widetilde{v}^{t-1}_i, v_{0,i}(\varepsilon))\\\hspace{0.7in}+\langle\bm{L}_{i,*},f_{t}(\bm{u}^{t},\bm{x}^t,\bm{x}_0(\varepsilon))-f_{t}(\alpha_{t-1}\bm{x}^*+\widetilde{\bm{u}}^{t-1},\mu_{t}\bm{x}^*+\widetilde{\bm{x}}^{t},\bm{x}_0(\varepsilon))\rangle-p_{t}g_{t-1}(v^{t-1}_i,{v}_{0,i}(\varepsilon)).
\end{multline}
By Jensen's inequality, \revsag{there exists a} constant $L_3>0$ \revsag{such that}
\begin{multline}
\frac{1}{p}\|\Delta^t_{3}\|^2 \le L_3\left(\sqrt{\frac{n\mu}{p}}\langle\bm{x}^*,f_{t}(\bm{u}^{t},\bm{x}^t,\bm{x}_0(\varepsilon))\rangle_n-\beta_{t}\right)^2\left(\frac{1}{p}\sum_{i=1}^{p}v^2_{0,i}(\varepsilon)\right) \\\hspace{0.5in}+L_3\frac{\uplambda_{\mbox{\tiny{max}}}(\bm{V}^\top\bm{V})}{p^2}\,\|f_{t}(\bm{u}^{t},\bm{x}^{t},\bm{x}_0(\varepsilon))-f_{t}(\alpha_{t-1} x^*_i+\widetilde{\bm{u}}^{t},\mu_t x^*_i+\widetilde{\bm{x}}^{t},\bm{x}_0(\varepsilon))\|^2\\\hspace{0.7in}+\frac{L_3}{p}\,|p_{t}|^2\|g_{t-1}(\bm{v}^{t-1},\bm{v}_0(\varepsilon))-g_{t-1}(\beta_{t-1}\bm{v}^*+\widetilde{\bm{v}}^{t-1},\bm{v}_0(\varepsilon))\|^2\\\quad+ L_3|\widetilde{p}_{t}-p_{t}|^2\frac{1}{p}\sum_{i=1}^{n}g^2_{t-1}(\beta_{t-1}v^*_i+\widetilde{v}^{t-1}_i,{v}_{0,i}(\varepsilon)).
\end{multline}
Using the partially pseudo-Lipschitz function { $\widehat\phi(x,y,z,w)=zf_{t}(\alpha_{t-1}z+x,\mu_t z+y,w)$} ({by Lemma \ref{lem:prop_pl}(2))}, definition of $\beta_{t}$ and Hypothesis (1) for $\ell=t-1$ we have
\begin{equation}
\label{eq:s_3_1}
\sqrt{\frac{n\mu}{p}}\langle\bm{x}^*,f_{t}(\bm{u}^{t},\bm{x}^t,\bm{x}_0(\varepsilon))\rangle_n-\beta_{t} \overset{a.s.}{\rightarrow} 0.
\end{equation}
Since $p/n \rightarrow 1/c$, using Corollary 5.35 of \cite{vershynin_2012}, we have
\[
\limsup_{n \rightarrow \infty} \frac{\uplambda_{\mbox{\tiny{max}}}(\bm{V}^\top\bm{V})}{p}< \infty \quad \mbox{{ a.s.}}
\]
Using \revsag{the} partially Lipschitz property of $f_{t}$, we have
\[
\|f_{t}(\bm{u}^{t},\bm{x}^{t},\bm{x}_0(\varepsilon))-f_{t}(\alpha_{t-1} \bm{x}^*+\widetilde{\bm{u}}^{t},\mu_t \bm{x}^*+\widetilde{\bm{x}}^{t},\bm{x}_0(\varepsilon))\|^2 \le L_3\,(\|\Delta^{t}_1\|^2 + \|\Delta^{t}_2\|^2)
\]
By induction hypotheses (3) and (4) for $\ell=t$ we get
\begin{equation}
\frac{1}{n}\|f_{t}(\bm{u}^{t},\bm{x}^{t},\bm{x}_0(\varepsilon))-f_{t}(\alpha_{t-1} \bm{x}^*+\widetilde{\bm{u}}^{t},\mu_t \bm{x}^*+\widetilde{\bm{x}}^{t},\bm{x}_0(\varepsilon))\|^2 \le L_3\,\left(\frac{\|\Delta^{t}_1\|^2}{n} + \frac{\|\Delta^{t}_2\|^2}{n}\right)\overset{a.s.}{\rightarrow} 0.
\end{equation}
This implies
\begin{equation}
\label{eq:s_3_2}
\limsup_{n \rightarrow \infty} \frac{\uplambda_{\mbox{\tiny{max}}}(\bm{V}^\top\bm{V})}{p^2}\,\|f_{t}(\bm{u}^{t},\bm{x}^{t},\bm{x}_0(\varepsilon))-f_{t}(\alpha_{t-1} \bm{x}^*+\widetilde{\bm{u}}^{t},\mu_t \bm{x}^*+\widetilde{\bm{x}}^{t},\bm{x}_0(\varepsilon))\|^2 \overset{a.s.}{=}0.
\end{equation}
Since, {$\widehat\psi(x,y,r)=g^2_{t-1}(\beta_{t-1} y+x,r)$} is partially pseudo-Lipschitz ({by Lemma \ref{lem:prop_pl}(1)}), using \eqref{eq:slln_3}, we have almost surely
\[
\limsup_{p \rightarrow \infty} \frac{1}{p}\sum_{i=1}^pg^2_{t-1}(\beta_{t-1}v^*_i+\widetilde{v}^{t-1}_i,v_{0,i}(\varepsilon)) < \infty.
\] 
As $f^{(1)}_{t}(x,y,r)=\frac{\partial f_{t}}{\partial x}(x,y,r)$ is partially Lipschitz, we have using Hypothesis (1) for $\ell = t$
\begin{equation}
\label{eq:lable}
|\widetilde{p}_{t}-p_{t}| = \Bigg|\frac{1}{n}\sum_{i=1}^{n}[f^{(1)}_{t}(\alpha_{t-1}x^*_i+\widetilde{u}^{t}_i,\mu_{t}x^*_i+\widetilde{x}^{t}_i,x_{0,i}(\varepsilon))-f^{(1)}_{t}(u^{t}_i,x^{t}_i,x_{0,i}(\varepsilon))]\Bigg|\\
\overset{a.s}{\rightarrow} 0.
\end{equation}
This implies
\begin{equation}
\label{eq:s_3_3}
|\widetilde{p}_{t}-p_{t}|^2\frac{1}{p}\sum_{i=1}^pg^2_{t-1}(\beta_{t-1}v^*_i+\widetilde{v}^{t-1}_i,v_{0,i}(\varepsilon)) \overset{a.s.}{\rightarrow} 0.
\end{equation}
Using arguments similar to \eqref{eq:lable}, we can show that 
\[
{ \limsup_{n\to\infty}}\,|p_{t}|^2 < \infty \quad \mbox{a.s.}
\]
Then, using the induction hypothesis (5) for $\ell=t-1$, we get
\begin{equation}
\frac{1}{p}\|g_{t-1}(\bm{v}^{t-1},\bm{v}_0(\varepsilon))-g_{t-1}(\beta_{t-1}\bm{v}^*+\widetilde{\bm{v}}^{t-1},\bm{v}_0(\varepsilon))\|^2 \le L_3\,\frac{\|\Delta^{t-1}_3\|^2}{p}\\
\overset{a.s.}{\rightarrow} 0.
\end{equation}
Thus, we have the following.
\begin{equation}
\label{eq:s_3_4}
|p_{t}|^2\frac{1}{p}\|g_{t-1}(\bm{v}^{t-1},\bm{v}_0(\varepsilon))-g_{t-1}(\beta_{t-1}\bm{v}^*+\widetilde{\bm{v}}^{t-1},\bm{v}_0(\varepsilon))\|^2 \overset{a.s.}{\rightarrow} 0.
\end{equation}
Using \eqref{eq:s_3_1}, \eqref{eq:s_3_2}, \eqref{eq:s_3_3}, and \eqref{eq:s_3_4} we have
\[
\lim_{n \rightarrow \infty}\frac{\|\Delta^t_3\|^2}{n} \overset{a.s.}{=} 0.
\]
\paragraph{Step 7: Hypothesis (2)} Again using the partially pseudo-Lipschitz property of $\psi$ and \revsag{the} Cauchy--Schwarz inequality we get
\begin{multline}
\frac{1}{p}\sum_{i=1}^{p}\left|\psi(\beta^{t}v^*_i+\widetilde{v}^{t}_i, v_{0,i}(\varepsilon))-\psi(v^{t}_i, v_{0,i}(\varepsilon))\right|\\\le\frac{2L}{p}\{\sqrt{p}\|\Delta^{t}_{3}\|+\|\Delta^{t}_{3}\|^2+\|\Delta^{t}_{3}\|^2+\|\Delta^{t}_{3}\|\|\beta_t\bm{v}^*+\widetilde{\bm{v}}^t\|+\|\Delta^{t}_{1}\|\|\bm{v}_{0}(\varepsilon)\|.
\end{multline}
Using Hypothesis (5) for $\ell = t$ gives us
\[
\frac{1}{p}\sum_{i=1}^{p}\left|\psi(\beta^{t}v^*_i+\widetilde{v}^{t}_i, v_{0,i}(\varepsilon))-\psi(v^{t}_i, v_{0,i}(\varepsilon))\right| \overset{a.s.}{\rightarrow} 0.
\]

\subsection{Lemmas Used to Prove Results in Section \ref{AMP}}

Variants of the following lemmas have previously appeared in \cite{BM11}.
We include their statement and proof here mainly 
for the manuscript to be self-contained.

\begin{lem}
	\label{lem:key_help_lem}
Consider a sequence of matrices $\bm{A} \sim GOE(n)$ and two sequences of vectors $\bm{u},\bm{v} \in \mathbb{R}^n$ such that $\|\bm{u}\|=\|\bm{v}\|=\sqrt{n}$.
\begin{enumerate}[label=(\alph*)]
\item $\langle\bm{v},\bm{A}\bm{u}\rangle_n \overset{a.s.}{\rightarrow} 0$,
\item Let $\bm{P} \in \mathbb{R}^{n \times n}$ be a sequence of projection matrices such that there exists a constant $t$ that satisfies for all $n$, $\mbox{rank}\,(\bm{P}) \le t$. Then $\frac{1}{n}\|\bm{P}\bm{A}\bm{u}\|^2_2 \overset{a.s.}{\rightarrow} 0$,
\item $\frac{1}{n}\|\bm{A}\bm{u}\|^2_2 \overset{a.s.}{\rightarrow} 1$,
\item { There exists a sequence of random vectors $\bm{z} \sim N(0,\bm{I}_n)$ such that for any sequence of functions $\varphi_n:(\mathbb{R}^n)^r \times \mathbb{R}^n \times (\mathbb{R}^n)^m \rightarrow \mathbb{R}$, $n \ge 1$ satisfying 
\begin{multline}
|\varphi_n(\bm h_1,\ldots,\bm h_r, \bm x, \bm \xi_1,\ldots,\bm \xi_m)-\varphi_n(\bm h_1,\ldots,\bm h_r, \bm y, \bm \xi_1,\ldots,\bm \xi_m)| \\\le L\, \left(1+\sum\limits_{i=1}^{r}\frac{\|\bm{h}_i\|}{\sqrt{n}}+\frac{\|\bm{x}\|}{\sqrt{n}}+\frac{\|\bm{y}\|}{\sqrt{n}}+\sum\limits_{j=1}^{m}\frac{\|\bm{\xi}_j\|}{\sqrt{n}}\right)\frac{\|\bm{x}-\bm{y}\|}{\sqrt{n}},
\end{multline}
where for all $i \in [r]$ and $j \in [m]$,
\[
\limsup\limits_{n \rightarrow \infty}\frac{\|\bm{h}_i\|}{\sqrt{n}} < \infty \quad \quad \limsup\limits_{n \rightarrow \infty}\frac{\|\bm{\xi}\|}{\sqrt{n}} < \infty.
\]
Then we have
\[
\varphi_n(\bm h_1,\ldots,\bm h_r,\bm{A}\bm{u},\bm \xi_1,\ldots,\bm \xi_m)-\varphi_n(\bm h_1,\ldots,\bm h_r,\bm{z},\bm \xi_1,\ldots,\bm \xi_m) \overset{a.s.}{\rightarrow} 0.
\]}
\end{enumerate}
\end{lem}
\begin{proof}
First note that for any fixed $k>0$, if we have a sequence of random variables $X_n$ defined on the same probability space such that $X_n \sim N(0,k/n)$. Then we have the following inequality
\[
\mathbb{P}\left(|X_n| \ge \frac{1}{n^{1/4}}\right) \le 2 \exp\left(-\frac{\sqrt{n}}{2k}\right).
\]
As $\sum\limits_{n=1}^{\infty}\exp\left(-\frac{\sqrt{n}}{2k}\right) < \infty$, using \revsag{the} Borel Cantelli Lemma we have
\[
|X_n| \overset{a.s.}{\rightarrow} 0.
\]
\begin{enumerate}[label=(\alph*)]
\item Recall that $\bm{A}=\bm{G}+\bm{G}^\top$ where $G_{i,j}$ are i.i.d $N(0,1/(2n))$ random variables, thus
\[
\frac{1}{n}\langle\bm{v},\bm{A}\bm{u}\rangle = \frac{1}{n}\langle\bm{v},\bm{G}\bm{u}\rangle+\frac{1}{n}\langle\bm{v},\bm{G}^\top\bm{u}\rangle.
\]
The random variable $\frac{1}{n}\langle\bm{v},\bm{G}\bm{u}\rangle$ is a centered Gaussian random variable with variance $1/2n$. Thus $\frac{1}{n}\langle\bm{v},\bm{G}\bm{u}\rangle \overset{a.s.}{\rightarrow} 0$. Similarly we can show $\frac{1}{n}\langle\bm{v},\bm{G}^\top\bm{u}\rangle \overset{a.s.}{\rightarrow} 0$.
\item Suppose $\bm{v}_1,\ldots,\bm{v}_k$, an orthogonal basis of the image $\bm{P}$, such that $\|\bm{v}_i\|=\sqrt{n}$. As $k$ is bounded by $t$, by part (a)
\[
\frac{1}{n}\|\bm{P}\bm{A}\bm{u}\|^2_2 = \frac{1}{n}\sum\limits_{i=1}^{k}\left(\frac{\langle\bm{v},\bm{A}\bm{u}\rangle}{\|\bm{v}_j\|}\right)^2 = \sum\limits_{i=1}^{k}\left(\frac{1}{n}\langle\bm{v},\bm{A}\bm{u}\rangle\right)^2 \overset{a.s.}{\rightarrow} 0.
\]
\item By (d), we have a sequence of random vectors $\bm{z} \sim N(0,\bm{I}_n)$
\[
\frac{1}{n}\|\bm{A}\bm{u}\|^2_2-\frac{1}{n}\|\bm{z}\|^2_2 \overset{a.s.}{\rightarrow} 0.
\]
As $\frac{1}{n}\|\bm{z}\|^2_2 \overset{a.s.}{\rightarrow} 1$, we can show part (c).
{ \item We shall show this for $r=1$ and $m=1$, the case for higher $r$ and $m$ follows. It is easy to check that $\bm{A}\bm{u}$ is a centered Gaussian vector with covariance matrix $\bm{\Sigma}=\bm{I}_n+\frac{1}{n}\bm{u}\bm{u}^\top$. Thus there exists a Gaussian vector $\bm{z}\sim N(0,\bm{I}_n)$ such that $\bm{A}\bm{u}=\bm{\Sigma}^{1/2}\bm{z}=\bm{z}+(\sqrt{2}-1)\frac{1}{n}\bm{u}\bm{u}^\top\bm{z}$. By the property of $\varphi_n$ we have
\[
|\varphi_n(\bm h, \bm{A}\bm{u}, \bm \xi)-\varphi_n(\bm h, \bm{z}, \bm \xi)| \le L\, \left(1+\frac{\|\bm h\|}{\sqrt{n}}+\frac{\|\bm{A}\bm{u}\|}{\sqrt{n}}+\frac{\|\bm{z}\|}{\sqrt{n}}\right)\frac{\|\bm{A}\bm{u}-\bm{z}\|}{\sqrt{n}}.
\]
The law of large numbers imply, $\|z\|_2/\sqrt{n} \overset{a.s.}{\rightarrow} 1$, and we have $\|\bm{A}\bm{u}\|/\sqrt{n} \le \|\bm{\Sigma}^{1/2}\|_{\mbox{\tiny{op}}}\|\bm{z}\|/\sqrt{n} \newline \le \sqrt{2}\|\bm{z}\|/\sqrt{n} \overset{a.s.}{\rightarrow} \sqrt{2}$. Further
\[
\frac{\|\bm{A}\bm{u}-\bm{z}\|}{\sqrt{n}}=\frac{\|\left(\bm{\Sigma}^{1/2}-\bm{I}_n\right)\bm{z}\|}{\sqrt{n}}=\frac{1}{n^{3/2}}(\sqrt{2}-1)\|\bm{u}\bm{u}^\top\bm{z}\|=(\sqrt{2}-1)\frac{1}{n}\,|\bm{u}^\top\bm{z}|\overset{a.s.}{\rightarrow}0.
\]
The last assertion follows as $\frac{1}{n}\,|\bm{u}^\top\bm{z}|$ is a centered Gaussian with variance $1/n$.}
\end{enumerate}
\end{proof}
\begin{lem}
\label{lem:lem_10}
Let $(Z_1,\ldots,Z_t)$, $(\widetilde{Z}_1,\ldots,\widetilde{Z}_t)$ be sequences Gaussian random variables, where the two sequences are independent. Let $c_1,\ldots,c_t$ and $\tilde{c}_1,\ldots,\tilde{c}_t$ be strictly positive constants such that for all $i=1,\ldots,t$: \[\mbox{Var}(Z_i|Z_1,\ldots,Z_{i-1},\widetilde{Z}_1,\ldots,\widetilde{Z}_{i-1}) > c_i\] and \[\mbox{Var}(\widetilde{Z}_i|Z_1,\ldots,Z_{i-1},\widetilde{Z}_1,\ldots,\widetilde{Z}_{i-1}) > \tilde{c}_i.\] Further assume $\mathbb{E}\left\{Z^2_i\right\} \le K$ for all $i$ and $\mathbb{E}\left\{\widetilde{Z}^2_i\right\} \le L$. Let $Y$ be a random variable in the same probability space. 

Finally let $\ell:\mathbb{R}^3 \rightarrow \mathbb{R}$ be a Lipschitz function, with $(z,y)\mapsto \ell(z,y,Y)$ non-constant with positive probability (with respect to $Y$). Then there exists a positive constant $c^\prime_t$ such that
\[
\mathbb{E}\left\{[\ell(Z_t,\widetilde{Z}_t,Y)^2]\right\}-u^\top C^{-1} u > c^\prime_t,
\]
where $u \in \mathbb{R}^{t-1}$ is given by $u_i=\mathbb{E}\left\{\ell(Z_t,\widetilde{Z}_t,Y)\ell(Z_i,\widetilde{Z}_i,Y)\right\}$, and $C \in \mathbb{R}^{t-1}\times\mathbb{R}^{t-1}$ satisfies $C_{i,j}=\mathbb{E}\left\{\ell(Z_i,\widetilde{Z}_i,Y)\ell(Z_j,\widetilde{Z}_j,Y)\right\}$ for all $1 \le i,j \le t-1$.
\end{lem}
\begin{proof}
Let us denote by $Q$ the covariance of the Gaussian vectors $Z_1,\ldots,Z_t$, and $Q^\prime$ the covariance of the Gaussian vectors $\widetilde{Z}_1,\ldots,\widetilde{Z}_t$. The set of matrices $Q,Q^\prime$ satisfying the constraints with constants $c_1,\ldots,c_t,K$ is compact. So if the thesis does not hold then there must exist a covariance matrix
\begin{equation}
\label{eq:d_2}
\mathbb{E}\left\{[\ell(Z_t,\widetilde{Z}_t,Y)]^2\right\}-u^\top C^{-1} u =0.
\end{equation}
Let $S \in \mathbb{R}^{t \times t}$ be the matrix with the entries $S_{i,j}=\mathbb{E}\left\{\ell(Z_i,\widetilde{Z}_i,Y)\ell(Z_j,\widetilde{Z}_j,Y)\right\}$.
Then \eqref{eq:d_2} implies that $S$ is not invertible by \revsag{the} Schur Complement  Formula. Therefore, there exist non-vanishing constants $a_1,\ldots,a_\ell$ such that
\[
a_1\ell(Z_1,\widetilde{Z}_1,Y)+\ldots+a_t\ell(Z_t,\widetilde{Z}_t,Y) \overset{a.s.}{=}0.
\]
The function $(z_1,\ldots,z_t) \mapsto a_1\ell(z_1,\tilde{z}_1,Y)+\ldots+a_t\ell(z_t,\tilde{z}_t,Y)$ is Lipschitz and non-constant. Hence there is a set $\mathcal{A} \subseteq \mathbb{R}^t$ of positive Lebesgue Measure such that it is non-vanishing on $\mathcal{A}$.  Therefore, $\mathcal{A}$ must have zero measure under the law of $(Z_1,\ldots,Z_t)$ and $(\widetilde{Z}_1,\ldots,\widetilde{Z}_t)$, i.e., $\lambda_{\mbox{\tiny{min}}}(Q)=0$ and $\lambda_{\mbox{\tiny{min}}}(Q^\prime)=0$. This implies there exists $a^\prime_1,\ldots,a^\prime_t$ and $b^\prime_1,\ldots,b^\prime_t$ such that
\[
a^\prime_1Z_1+\ldots+a^\prime_tZ_t \overset{a.s.}{=}0,
% \]
\quad\mbox{and}\quad
% \[
b^\prime_1\widetilde{Z}_1+\ldots+b^\prime_t\widetilde{Z}_2 \overset{a.s.}{=}0.
\]
If $t_{*}=\max\{i \in \{1,\ldots,t\}:a^\prime_i \neq 0\}$ and $s_{*}=\max\{i \in \{1,\ldots,t\}:b^\prime_i \neq 0\}$, this implies
\[
Z_{t_{*}}\overset{a.s.}{=}\sum\limits_{i=1}^{t_{*}-1}(-a^\prime_i/a^\prime_{t_*})Z_i,
% \]
\quad\mbox{and}\quad
% \[
\widetilde{Z}_{t_{*}}\overset{a.s.}{=}\sum\limits_{i=1}^{t_{*}-1}(-b^\prime_i/b^\prime_{t_{*}})\widetilde{Z}_i.
\]
This violates the assumption of the hypothesis.
\end{proof}
\begin{lem}
The vectors
\[
\bm{\upalpha}_t = (\upalpha^{t}_{1},\ldots,\upalpha^{t}_{t-1}) = \left[\frac{M^\top_tM_t}{p}\right]^{-1}\frac{M^\top_t\bm{m}^t}{p},
\]
and
\[
\bm{\upbeta}_t = (\upbeta^{t}_{1},\ldots,\upbeta^{t}_{t-1}) = \left[\frac{Q^\top_tQ_t}{n}\right]^{-1}\frac{Q^\top_t\bm{q}^t}{n},
\]
have finite limits as $p ,n\rightarrow \infty$.
\end{lem}
\begin{proof}
Applying Lemma 9 of \cite{BM11} and $\mathcal{B}_t(g), \mathcal{H}_{t}(g), \mathcal{X}_{t}(g)$ we can obtain that for large $n$ the smallest eigenvalues of $(M^\top_tM_t)/n$, $(Q^\top_tQ_t)/n$ are all strictly positive. By Lemma 10 of \cite{BM11} this implies they converge to invertible limits. Then using $\mathcal{H}_{t}(c),\mathcal{X}_{t}(c)$ and $\mathcal{B}_{t-1}(c)$ we have the result.
\end{proof}

\end{document}